\documentclass[11pt, reqno]{amsart}
\usepackage[utf8]{inputenc}
\usepackage{amsthm}
\usepackage{amsmath,amsfonts,amssymb}
\usepackage[a4paper, margin=1in]{geometry}
\usepackage{upgreek}
\usepackage{float}

\usepackage{bm}
\usepackage{csquotes}
\usepackage{siunitx}
\usepackage{cite}
\usepackage{mwe}
\usepackage{graphicx}
\usepackage{caption}
\usepackage{subcaption}
\usepackage{hyperref}
\usepackage{soul}
\usepackage[export]{adjustbox}
\usepackage[most]{tcolorbox}
\usepackage{mwe,tikz}\usepackage[percent]{overpic}

\numberwithin{equation}{section}
\newtheorem{remark}{Remark}[section]
\newtheorem{defn}{Definition}[section]

\newtheorem{proposition}{Proposition}[section]

\newcommand{\nbb}{\nabla_\Omega}

\newcommand{\bs}[1]{\boldsymbol{#1}}

\newcommand{\R}{\mathbb R}
\newcommand{\M}{\mathcal M}
\newcommand{\ra}{\rightarrow}
\newcommand{\dint}[1]{\int_0^T\int_\Omega{#1}dxdt}
\newcommand{\intlr}[1]{\left({#1}\right)_{\Omega_{t}}}
\newcommand{\intlrh}[1]{\left({#1}\right)_{\Omega_{t,h}}}
\newcommand{\lrangle}[1]{\left\langle{#1}\right\rangle}
\newcommand{\kpo}{{k+1}}
\newcommand{\D}{\mathcal{D}}
\newcommand{\ih}{{i,h}}
\newcommand{\su}{\sigma_u}
\newcommand{\sph}{\sigma_{\phi}}

\begin{document}
\title[MFC Barycenter]
{Efficient Computation of Mean field Control based Barycenters from Reaction-Diffusion Systems}

\author[Vijaywargiya]{Arjun Vijaywargiya}
\author[Fu]{Guosheng Fu}
\address{Department of Applied and Computational Mathematics and
Statistics, University of Notre Dame, USA.}
\email{avijaywa@nd.edu, gfu@nd.edu}
\author[Osher]{Stanley Osher}
\email{sjo@math.ucla.edu}
\address{Department of Mathematics, University of California, Los Angeles}
\author[Li]{Wuchen Li}
\email{wuchen@mailbox.sc.edu}
\address{Department of Mathematics, University of South Carolina}
\keywords{Mean field control; Barycenter; Finite element methods; Reaction-diffusion systems.}
\subjclass{65N30, 65N12, 76S05, 76D07}
\begin{abstract}
    We develop a class of barycenter problems based on 
 mean field control problems in three dimensions with associated reactive-diffusion systems of unnormalized multi-species densities. This problem is the generalization of the Wasserstein barycenter problem for single probability density functions. The primary objective is to present a comprehensive framework for efficiently computing the proposed variational problem: generalized Benamou-Brenier formulas with multiple input density vectors as boundary conditions. Our approach involves the utilization of high-order finite element discretizations of the spacetime domain to achieve improved accuracy. The discrete optimization problem is then solved using the primal-dual hybrid gradient (PDHG) algorithm, a first-order optimization method for effectively addressing a wide range of constrained optimization problems. The efficacy and robustness of our proposed framework are illustrated through several numerical examples in three dimensions, such as the computation of the barycenter of multi-density systems consisting of Gaussian distributions and reactive-diffusive multi-density systems involving 3D voxel densities. Additional examples highlighting computations on 2D embedded surfaces are also provided.
\end{abstract}
\maketitle

\section{Introduction}
Interpolating and averaging multiple unnormalized densities is an essential problem in computer vision, data science, Bayesian computation, and scientific computing \cite{cot, solomon:hal-01188953,Carlier15}. The averaging density is often called the barycenter \cite{AguehCarlier11}. One usually needs to solve optimization problems to obtain the barycenter \cite{CuturiPeyre16,staib2017parallel,Rabin11,CuturiDoucet14,GramfortPeyre15,Carlier15,Alvarez16,BenamouCarlier15,Dvurechensky2018,LasryLions2007}. It is frequently used to minimize information divergences, such as Kullback-Leibler divergence \cite{JMLR:v6:banerjee05b}, among multiple input densities. The minimizer forms the ``center'' of those input densities. In recent years, optimal transport has introduced a class of distances between probability densities, namely Wasserstein distances \cite{Villani2009_optimalb}. Distinguished from the information divergences, a critical feature of Wasserstein distance is that it incorporates the ground costs in the sample space. Thus, the interpolation in terms of Wasserstein distances, namely the Wasserstein barycenter problems, has been widely studied 
\cite{solomon:hal-01188953, cot} in areas involving density functions. 

Meanwhile, mean field control problems (MFC) have been introduced \cite{LasryLions2007,CainesHuangMalhame2006_large,bensoussan2018mean}, which generalize optimal transport and Wasserstein distances. The MFC problems model general optimal control problems of density evolutions, which come with flexible choices of initial/terminal densities, running costs, and interaction costs during the density evolution. In this sense, the optimal functional value in MFC problems also defines generalized type ``distances'' between densities. More recently, \cite{Chizat18,Mielke11,LiLeeOsher22, fu2023generalized} have introduced the MFC problems of reaction-diffusion equations and systems. The MFC formulation models the evolution of multiple species densities and controls the complex interaction behaviors between species. In addition, the MFC problem of reaction-diffusion systems can also handle unbalanced densities from the source terms in the reaction-diffusion systems.

In this paper, we formulate MFC barycenter problems for multi-species unnormalized densities. 
These MFC barycenter problems generalize Wasserstein barycenters. Here the MFC problem utilizes optimal control problems of multi-species densities following the formulation of reaction-diffusion systems, where the running cost is with kinetic and interaction energies among different species of unnormalized densities, 
and the constraint is based on the set of reaction-diffusion systems. 
The MFC barycenter model is discretized using the high-order space-time finite element method \cite{FuLiu23}. We then solve the discretized variational problem by the primal-dual hybrid gradient (PDHG) algorithm \cite{Chambolle,gprox}. 
Numerical examples in three-dimensional and complex two-dimensional surfaces demonstrate the effectiveness of the proposed computational method.

In literature, Wasserstein barycenter problems have been widely studied, with applications in statistical learning and computer visions \cite{solomon:hal-01188953,cot,papadakis2014optimal}. Fast algorithms toward barycenter problems are often based on the Kantorovich formulation of Wasserstein distances with the entropic regularization, namely the Sinkhorn algorithm \cite{cot}.  
On the other hand, the Wasserstein distance can also be defined using the dynamical formulation, known as the Benamou-Brenier formula \cite{benamou2000computational}. The dynamical formulation is a particular optimal control problem in density spaces, which is a special case of MFC problems. Nowadays, MFC has been widely used in population dynamics in financial markets \cite{Gu2011}, propagation of pandemics \cite{lee2021controlling}, etc. In this direction, the MFC associated with reaction-diffusion systems has been proposed in \cite{lee2021controlling}, which models both transport behaviors of populations and the complex interactions with multiple densities \cite{lee2022mean}. We note that the MFC has not been fully used in the modeling and computation of barycenter problems, which is the subject of the current study. 
Our formulation is also motivated by the reaction-diffusion systems from gradient flow formulations, known as Onsager principles \cite{OnsagerMachlup1953_fluctuations}. This formulation allows us to obtain MFC-induced barycenters from complicated reaction terms among different species and handle the unnormalized densities. 

The paper is organized as follows. In Section \ref{sec2}, we briefly review the generalized Wasserstein distance associated with the
reaction-diffusion systems. These motivate us to define the general MFC-based barycenter problems to interpolate unnormalized multi-species densities.  
In Section \ref{sec3}, we introduce numerical algorithms to solve the proposed MFC barycenter problems. We use the space-time high-order finite element method to discretize the MFC problem and apply the PDHG algorithm to compute the resulting discrete optimization problem. In Section \ref{sec4}, we provide numerical examples for densities in three-dimensional and surface domains to demonstrate the quality of proposed MFC barycenters. We conclude in Section \ref{sec5}.

\section{Methodology}\label{sec2}
In this section, we first review the generalized Wasserstein distance associated with reaction-diffusion equations/systems. We then formulate the corresponding generalized Wasserstein barycenter problem. Its detailed numerical discretization will be discussed in the next section.
\subsection{Reactive-diffusive Wasserstein distances}
In line with the preceding work in \cite{FuLiu23, fu2023generalized}, before we set up a mean-field control problem to compute the multi-density barycenter, we need to construct a generalized Wasserstein distance metric that incorporates the reaction-diffusion equations/systems. For a more detailed review of the underlying concepts of metric distances, gradient flow, generalized optimal transport, and mean-field control problems, please refer to \cite{FuOsherLi,LegerLi2021_hopfcole, LiLeeOsher22,Mielke11}. 

We start by constructing the distance for the \textit{scalar} case. We then generalize to the \textit{system} case. Throughout this manuscript, we denote $\Omega$ either as a 3D volume domain in $\mathbb{R}^3$
with boundary $\partial\Omega$, or a closed 2D surface domain embedded in $\mathbb{R}^3$. 
We denote $\nbb$ and $\nbb\cdot$  as the standard gradient and divergence operators when $\Omega$ is a 3D volume domain, or as the surface gradient and divergence operators when $\Omega$ is a surface domain. Definitions of surface gradient and divergence operators follow from the standard literature on 
surface PDEs;  see, e.g., \cite{surface}.

\subsubsection{Scalar distance}
On the spatial domain $\Omega$, consider the following dissipative reaction-diffusion equation for the non-negative density function $\rho:[0,T] \times \Omega \ra \R_+$ \cite{LiLeeOsher22,FuOsherLi}:
\begin{equation}
\label{scalarPDE}
\partial_t \rho=\nbb \cdot\left(V_1(\rho) \nbb \frac{\delta}{\delta \rho} \mathcal{E}(\rho)\right)-V_2(\rho) \frac{\delta}{\delta \rho} \mathcal{E}(\rho), \text { on }[0, T] \times \Omega,
\end{equation}
where $V_1,V_2: \R_+ \ra \R_+$ are both positive mobility functions, and 
$\mathcal E:\mathcal{M} \ra \R$ is an energy functional with $\frac{\delta \mathcal E}{\delta \rho}$ denoting its first variation with respect to the density function $\rho$ in $L^2$ space. We define the space
\begin{equation}
    \mathcal{M}=\left\{\rho \in C(\Omega): \rho \geq 0\right\}, \quad \forall t \geq 0,
\end{equation}
and take the density $\rho$ such that it satisfies
$
\rho(t,\cdot) \in \mathcal{M.}
$
The PDE \eqref{scalarPDE} is supplied with homogeneous Neumann boundary conditions so that
\begin{equation}
\label{bdry}
\left.V_1(\rho) \nbb \frac{\delta}{\delta \rho} \mathcal{E}(\rho) \cdot \boldsymbol{\nu}\right|_{\partial \Omega}=0,
\end{equation}
with $\bm\nu$ denoting the outward unit normal direction on the boundary $\partial \Omega$.
We note that when $\Omega$ has no boundary, such as a periodic volume geometry or a closed surface geometry, the boundary condition 
\eqref{bdry} is ignored since $\partial\Omega=\emptyset$ is an empty set. 
The PDE \eqref{scalarPDE} is dissipative in the energy functional $\mathcal E$, which satisfies the following dissipation law.
\begin{equation}
\frac{d}{d t} \mathcal{E}(\rho)=-\int_{\Omega}\left[\left\|\nbb \frac{\delta}{\delta \rho} \mathcal{E}(\rho)\right\|^2 V_1(\rho)+\left|\frac{\delta}{\delta \rho} \mathcal{E}(\rho)\right|^2 V_2(\rho)\right] d x \leq 0.
\end{equation}
This energy law corresponds to a gradient flow in the metric space $\M$. We describe the metric in the following definition which characterizes the distance between densities $\rho^0, \rho^1 \in \M$. 

\begin{defn}
\label{scalardist}
    (Scalar reactive-diffusive Wasserstein distance) The distance functional
$$
\operatorname{Dist}_{V_1, V_2}: \mathcal{M} \times \mathcal{M} \ra \mathbb{R}_{+}
$$
can be defined by considering the following optimal control problem:
\begin{subequations}
\label{scalardisteq}
    \begin{equation}
\operatorname{Dist}_{V_1, V_2}\left(\rho^0, \rho^1\right)^2:=\inf _{\rho, \boldsymbol{v}_1, v_2} \int_0^T \int_{\Omega}\left[\left\|\boldsymbol{v}_1\right\|^2 V_1(\rho)+\left|v_2\right|^2 V_2(\rho)\right] d x d t, 
\end{equation}
where the constraint satisfies the following equation connecting the initial and terminal densities $\rho^0,\rho^1 \in \M$:
\begin{equation}
\left\{\begin{array}{l}
\partial_t \rho+\nbb \cdot\left(V_1(\rho) \boldsymbol{v}_1\right)=V_2(\rho) v_2, \quad \text { on }[0, T] \times \Omega, \\
\rho(0, x)=\rho^0(x), \quad \rho(T, x)=\rho^1(x),
\end{array}\right.,
\end{equation}
with the homogeneous boundary condition
$
\left.V_1(\rho) \boldsymbol{v}_1 \cdot \boldsymbol{\nu}\right|_{\partial \Omega}=0.
$ Here, $\boldsymbol{v}_1$ is the drift vector field, $v_2$ is the reaction rate, $V_1$ is the drift mobility, and $V_2$ denotes the reaction mobility. 
\end{subequations}
\end{defn}
\begin{remark}
In the above definition, if $V_1(\rho) = \rho$ and $V_2(\rho)=0$, this is the classical Wasserstein-2 distance; see \cite{benamou2000computational}. If $V_1(\rho) = \rho$ and $V_2(\rho)=\rho$, this is the Wasserstein-Fisher-Rao distance for unbalanced densities; see, e.g., \cite{Mielke11,Chizat18,LieroMielkeSavare2018_optimalb}. If $V_1(\rho)=\rho$, $V_2(\rho)=1$, it belongs to the generalized unnormalized Wassrstein distance \cite{LeeLaiLiOsher21}. 
\end{remark}

The control problem in definition \ref{scalardist} can be reformulated by the introduction of a flux function $\bs m(t,x):[0,T] \times \Omega \ra \R^d$ and a source function $s(t,x):[0,T]\times\Omega \ra \R$ such that they satisfy
\begin{equation}
\bs{m}(t, x)=V_1(\rho(t, x)) \boldsymbol{v}_1(t, x), \quad s(t, x)=V_2(\rho(t, x)) v_2(t, x). 
\end{equation}
By using the above change of variables, the distance formula in \eqref{scalardisteq} takes the form:
\begin{subequations}
    \begin{equation}
\operatorname{Dist}_{V_1, V_2}\left(\rho^0, \rho^1\right)^2:=\inf _{\rho, \boldsymbol{m}, s} \int_0^T \int_{\Omega}\left[\frac{\|\boldsymbol{m}\|^2}{V_1(\rho)}+\frac{|s|^2}{V_2(\rho)}\right] d x d t,
\end{equation}
where the constraint is given as
\begin{equation}
\begin{aligned}
& \partial_t \rho(t, x)+\nbb \cdot \boldsymbol{m}(t, x)=s(t, x), \quad \text { in }[0, T] \times \Omega, \\
& \boldsymbol{m} \cdot \boldsymbol{\nu}=0, \quad \text { on }[0, T] \times \partial \Omega, \\
& \rho(0, x)=\rho^0(x), \quad \rho(T, x)=\rho^1(x), \quad \text { in } \Omega .
\end{aligned}
\end{equation}
\end{subequations}
This optimization problem with linear constraints will be convex if we assume that the mobility functions $V_1$ and $V_2$ are positive and concave.

\subsubsection{System distance}
Analogous to the scalar case, the distance functional for the \textit{system} case comes from a reaction-diffusion system. We consider the following reaction-diffusion system with $N$ species undergoing $R$ reactions \cite{Mielke11, fu2023generalized}:
\begin{equation}
    \label{systemPDE}
\partial_t \rho_i=\nbb \cdot\left(V_{1, i}\left(\rho_i\right) \nbb \frac{\delta}{\delta \rho_i} \mathcal{E}_i\left(\rho_i\right)\right)-\sum_{p=1}^R V_{2, p}(\boldsymbol{\rho}) \gamma_{i, p} \sum_{j=1}^N \gamma_{j, p} \frac{\delta}{\delta \rho_j} \mathcal{E}_j\left(\rho_j\right),
\quad \forall 1\le i\le N,
\end{equation}
which is endowed with homogeneous Neumann boundary conditions
$$
\left.V_{1, i}\left(\rho_i\right) \nbb \frac{\delta}{\delta \rho_i} \mathcal{E}_i\left(\rho_i\right) \cdot \boldsymbol{\nu}\right|_{\partial \Omega}= 0, \enspace \text{for } 1\leq i \leq N.
$$
For species $i$, $\rho_i$ denotes the density function which is defined such that $\rho_i(t,\cdot) \in \M$, $\mathcal E_i:\M \ra \R$ is its energy functional, and $V_{1,i}:\R_+\ra\R_+$ is its mobility function which is assumed to be positive. In contrast, $V_{2,p}:\R^N_+\ra\R_+$ denotes the mobility function for the $p$-th reaction for $1\le p\le R$, which is also considered positive. Moreover, $\bs \rho = (\rho_1,\cdots,\rho_N)$ represents the collection of all $N$ densities. Finally, the coefficient matrix $\bs \Gamma = (\gamma_{i,p}) \in \R^{N \times R}$ is chosen such that the following constraint is met
$$
\sum_{i=1}^N \gamma_{i,p} = 0,\quad \enspace \forall 1\leq p \leq R.
$$
This is done to ensure the total mass  conservation:
$$
\frac{d}{dt}\int_\Omega \sum_{i=1}^N \rho_i dx = 0.
$$
Under the above-stated assumptions on all the functions involved, the PDE \eqref{systemPDE} can be shown to satisfy the following energy dissipation law \cite{Mielke11,fu2023generalized}:
\begin{equation}
\frac{d}{d t} \sum_{i=1}^N \mathcal{E}_i\left(\rho_i\right)=-
\int_{\Omega}\left[\sum_{i=1}^N\left\|\nbb \frac{\delta \mathcal{E}_i}{\delta \rho_i}\right\|^2 V_{1, i}\left(\rho_i\right)+\sum_{p=1}^R\left|\sum_{j=1}^N \gamma_{j, p} \frac{\delta \mathcal{E}_j}{\delta \rho_j}\right|^2 V_{2, p}(\boldsymbol{\rho})\right] d x
 \leq 0
\end{equation}
Similar to the \textit{scalar} case, this energy law corresponds to a gradient flow in the metric space $\M^{N}$. This metric characterizes the distance between density vectors $\bs \rho^0, \bs \rho^1 \in \M^N$.

\begin{defn}
    (System reactive-diffusive Wasserstein distance) The distance functional
    $$
\operatorname{ Dist }_{\boldsymbol{V}_1, \boldsymbol{V}_2}: \mathcal{M}^N \times \mathcal{M}^N \rightarrow \mathbb{R}_{+}
    $$
can be defined by constructing the following optimal control problem
\begin{subequations}
    \begin{equation}
\operatorname{Dist}_{\boldsymbol{V}_1, \boldsymbol{V}_2}\left(\boldsymbol{\rho}^0, \boldsymbol{\rho}^1\right)^2:=\inf _{\boldsymbol{\rho}, \underline{\boldsymbol{m}}, \boldsymbol{s}}\int_0^T \int_{\Omega}\left(\sum_{i=1}^N \frac{||\bs m_i||^2}{V_{1, i}\left(\rho_i\right)}+\sum_{p=1}^R \frac{\left|s_p\right|^2}{V_{2, p}(\boldsymbol{\rho})}\right) d x d t,
\end{equation}
where the constraints satisfy the following equations connecting the initial and terminal density vectors $\boldsymbol{\rho}^0,
\boldsymbol{\rho}^1 \in \M^N$:
\begin{equation}
\begin{aligned}
& \left\{\begin{array}{l}
\partial_t \rho_i+\nbb \cdot \boldsymbol{m}_i=\sum_{p=1}^R \gamma_{i, p} s_p, \enspace \forall 1 \leq i \leq N, \\
\left.\boldsymbol{m}_i \cdot \boldsymbol{\nu}\right|_{\partial \Omega}=0,\enspace  \boldsymbol{\rho}(0, x)=\boldsymbol{\rho}^0, \enspace \boldsymbol{\rho}(T, x)=\boldsymbol{\rho}^1, 
\end{array}\right.
\end{aligned}
\end{equation}
with the collection of flux terms $\bs{\underline m} = (\bs m_1, \cdots, \bs m_N)$, and the collection of source terms $\bs s = (s_1,\cdots,s_R)$.
\end{subequations}
\end{defn}

\subsection{The MFC barycenter problem}
For simplicity of presentation, we formulate the multi-density barycenter problem with a cyclical reaction-diffusion system that form a closed graph containing $N$ vertices, representing density species, and $N$ edges, representing strongly reversible pairwise reactions.
Furthermore, we take the forward and backward reaction rates both equal to 1, which corresponds to the coefficient matrix $\bs \Gamma = (\gamma)_{i,p} \in \R^{N\times N}$ being such that
\begin{equation}
\label{reac}
    \gamma_{i,p} = \begin{cases}
         1, &\enspace p = i\\
        -1, &\enspace p = i-1\\
         0, &\enspace \text{otherwise}
    \end{cases}.
\end{equation}
\begin{remark}
We note that while we focuses on the special case of $\bs \Gamma$ in \eqref{reac} in the following MFC barycenter formulation, our proposed numerical scheme can be 
 naturally extended to the case with $N\not= R$ with a more general coefficient matrix $\bs \Gamma$.     
\end{remark}

Given a collection of densities $\bs \rho^0 = (\rho_1^0,\cdots,\rho_N^0)$, the Wasserstein barycenter is a density function ${\varrho}$ that minimizes the sum of Wasserstein distances to each density function $\rho_i^0$ in $\bs \rho^0$; see, e.g., \cite{AguehCarlier11}. 
In other words, we consider the following optimization problem
\begin{equation*}
 \inf_{\varrho\in \mathcal{M}} \quad \frac{1}{N}\sum_{i=1}^N\operatorname{Dist}_{V_1, V_2}\left({\rho}_i^0,  {\varrho}\right)^2, 
\end{equation*}
with $V_1(\rho) = \rho$ and $V_2(\rho)=0$.

We now generalize the above Wasserstein barycenter problem by allowing reaction effects among different species and a more general form of mobility functions $V_{1,i}$ and $V_{2,i}$. We further
include a general potential functional $\bm{\mathcal{F}}\colon \mathcal{M}^N\rightarrow\mathbb{R}$.
This leads to the following MFC barycenter formulation.
\begin{defn}[MFC barycenter problem]
We consider the following minimization problem:
\begin{subequations}\label{bcprob}
\begin{equation}
\label{bcproba}
\inf _{\boldsymbol{\rho}, \underline{\boldsymbol{m}}, \boldsymbol{s}, {\varrho}}
\int_0^T 
\left[
\int_{\Omega}\left(\sum_{i=1}^N \frac{||\bs m_i||^2}{2V_{1, i}\left(\rho_i\right)}+\sum_{i=1}^N \frac{\left|s_i\right|^2}{2V_{2, i}(\boldsymbol{\rho})}\right)dx
-\bm{\mathcal{F}}(\boldsymbol{\rho})
\right] dt
\end{equation}
such that
\begin{equation}\label{bcprobb}
\begin{aligned}
& \left\{\begin{array}{l}
\partial_t \rho_i+\nbb \cdot \boldsymbol{m}_i=s_{i}-s_{i-1}, \enspace \forall 1 \leq i \leq N, \\
\left.\boldsymbol{m}_i \cdot \boldsymbol{\nu}\right|_{\partial \Omega}=0,\enspace  \boldsymbol{\rho}(0, \cdot)=\boldsymbol{\rho}^0, \enspace \boldsymbol{\rho}(T, \cdot)={\varrho}\bs 1, 
\end{array}\right.
\end{aligned}
\end{equation}
where $\bs1=(1,\cdots, 1)\in\mathbb{R}^N$. 
\end{subequations}    
The above minimization problem involving mobilities and interaction energy is often named the MFC problem. In this sense, we call the minimization problem \eqref{bcprob} {\em MFC barycenter problem}, where the minimizer $\varrho$ is the MFC barycenter of the density vectors $\bs \rho^0$.
\end{defn}

In our numerical simulations in Section \ref{sec4}, we make the following choices for the mobility functions $V_{1,i}$ and $V_{2,i}$:
\begin{equation}
\label{v1v2}
    \begin{split}
        V_{1,i}(\rho_i) &=
\rho_i,\\        
        V_{2,i}(\bs \rho) & = V_{2,i}(\rho_i,\rho_{i+1}) = \alpha \frac{\rho_i-\rho_{i+1}}{\log\rho_i-\log\rho_{i+1}} \text{ with } \alpha \geq 0.
    \end{split}
\end{equation}
Moreover, we take a separable interaction functional
$\bm{\mathcal{F}}(\bm\rho)=\sum_{i=1}^N\int_{\Omega}F_i(\rho_i)\,dx,
$
where $F_i: \mathbb{R}_+\rightarrow \mathbb{R}$ is the potential density function for component $i$
taking the following form:
\begin{align}
\label{pot}    
F_i(\rho_i) = -\beta_i\rho_i\log(\rho_i), \quad \beta_i\ge0.
\end{align}
With these choices, the optimization problem \eqref{bcprob} can be shown to be convex and admits a unique minimizer.
Note that if we set $\alpha=0$ and $\beta_i=0$,
the above problem reduces to the computation of the classical Wasserstein-2 barycenter problem \cite{AguehCarlier11} without any reaction effects. The above choices of $V_{1,i}$, $V_{2,i}$ have statistical physics interpolations, such as generalized Onsager principles; see details in \cite{OnsagerMachlup1953_fluctuations, Mielke11}. In this paper, we use these formulations to define and compute generalized barycenter problems.   

\subsection{The unconstrained optimization problem}
We now reformulate the MFC problem \eqref{bcprob} as a saddle-point problem, for which a high-order finite element discretization will be introduced in Section \ref{sec3}. To do so, we multiply the PDE constraints with corresponding Lagrange multipliers $\phi_i$ and add to \eqref{bcproba} to obtain the following unconstrained problem:
\begin{equation}\label{uncon0}
    \begin{aligned}
        \inf_{\boldsymbol{\rho}, \underline{\boldsymbol{m}}, \boldsymbol{s}, {\varrho}}\sup_{\bs \phi}\int_0^T
        \left[\int_{\Omega}\left(\sum_{i=1}^N \frac{||\bs m_i||^2}{2V_{1, i}\left(\rho_i\right)}+\sum_{i=1}^N \frac{\left|s_i\right|^2}{2V_{2, i}(\rho_i,\rho_{i+1})}\right) dx
        -\bm{\mathcal{F}}(\boldsymbol{\rho})
\right]
        dt\\
        +\sum_{i=0}^N \dint {(\partial_t\rho_i+\nbb\cdot\bs m_i+s_{i-1}-s_{i})\phi_i },
    \end{aligned}
\end{equation}
where $\bs \phi = (\phi_1,\cdots,\phi_N)$ denotes the collection of Lagrange multipliers $\phi_i$. Integrating the second integral in \eqref{uncon0} by parts, applying boundary conditions from \eqref{bcprobb}, and using the short-hand notations $(\bs a,\bs b)_{\Omega_t} = \dint{\bs a \cdot \bs b}$ where $\Omega_t = [0,T]\times \Omega$ is the space-time domain, and $\langle \bs a \cdot \bs b \rangle_{t=s} = \int_\Omega \bs a(s,x) \cdot \bs b(s,x) dx$, we obtain
\begin{equation}\label{uncon1}
    \begin{aligned}
        \inf_{\bs\rho, \underline{\bs m}, \bs s, {\varrho}}\sup_{\bs\phi}&
        \sum_{i=1}^N \intlr{\frac{||\bs m_i||^2}{2V_{1, i}\left(\rho_i\right)}+\frac{\left|s_i\right|^2}{2V_{2, i}(\rho_i,\rho_{i+1})}, 1}
        - \int_0^T\bm{\mathcal{F}}(\boldsymbol{\rho})dt
        \\
         &- \intlr{\bs \rho,\partial_t\bs \phi} - \intlr{\underline{\bs m},\underline{\nbb\bs\phi}} - \intlr{\bs s, \hat{\bs\phi}} + \lrangle{{\varrho},\bs \phi\cdot\bs 1}_{t=T} - \lrangle{\bs\rho^0,\bs\phi}_{t=0},
    \end{aligned}
\end{equation}
where $\underline{\bs m} = (\bs m_1, \cdots, \bs m_N)$, $\underline{\nbb\bs\phi} = (\nbb\phi_1,\cdots,\nbb\phi_N)$, and $\hat{\bs\phi} = (\phi_1-\phi_2,\cdots,\phi_N-\phi_1)$. We can further simplify \eqref{uncon1} by an introduction of a few new variables. Let $\bs u_i = (\rho_i,\bs m_i,s_i)$ and $\underline{\bs u} = (\bs u_1,\cdots,\bs u_N)$. Furthermore, let us define the operator $\D$ such that
\begin{equation}\label{dope}
\D\phi_i = (\partial_t \phi_i,\nbb\phi_i, \phi_i-\phi_{i+1}),
\end{equation}
and let $\underline{\D\bs\phi} = (\D\phi_1,\cdots,\D\phi_N)$. Using these new variables, we can rewrite \eqref{uncon1} into the following unconstrained-saddle point form:
\begin{subequations}\label{uncon2}
    \begin{equation}\label{uncon2a}
\inf_{\underline{\bs u},{\varrho}} \sup_{\bs\phi}\quad \underbrace{H(\underline{\bs u}) - \left(\underline{\bs u},\underline{\D\bs\phi}\right)_{\Omega_t} + \lrangle{{\varrho} ,\bs \phi\cdot\bs 1}_{t=T}-\lrangle{\bs\rho^0,\bs\phi}_{t=0}}_{:=\mathcal{L}(\underline{\bs u}, \varrho, \bs\phi)},
\end{equation}
where
\begin{equation}\label{uncon2b}
    H(\underline{\bs u}) := \sum_{i=1}^N\dint{ \left[\frac{||\bs m_i||^2}{2V_{1, i}\left(\rho_i\right)}+\frac{\left|s_i\right|^2}{2V_{2, i}(\rho_i,\rho_{i+1})}\right]}
            - \int_0^T\bm{\mathcal{F}}(\boldsymbol{\rho})dt.
\end{equation}
\end{subequations}
We use the following function spaces for the saddle point problem \eqref{uncon2}, which ensure that all terms in the Lagrangian $\mathcal{L}(\underline{\bs u}, \varrho, \bs\phi)$
stays bounded:
\begin{subequations}
    \label{spaces}
\begin{align}
\bs m_i \in &\; [L^2([0,T]\times \Omega)]^3, \\
s_i \in &\; L^2([0,T]\times \Omega), \\
\label{rho-s}
\rho_i \in &\; \{\mu\in L^\infty([0,T]\times \Omega): \; \mu \ge \epsilon > 0, ~a.e.\}, \\
\label{rhobar-s}
\varrho \in &\; \{\mu\in L^\infty(\Omega): \; \mu \ge \epsilon > 0,~a.e.\}, \\
\phi_i \in &\; H^1([0,T]\times \Omega), 
\end{align}
\end{subequations}
for all $1\le i\le N$.
Note that here we require the density $\rho_i$ and $\varrho$ to be almost every positive and away from {\it zero} in 
\eqref{rho-s} and \eqref{rhobar-s} to avoid the technical issue of division by zero in $H(\underline{\bs u})$.

\subsection{The MFC barycenter system}
We conclude this section by formulating the critical point system for the optimization problem \eqref{uncon2a}. We name the derived PDE system as the {\em mean-field control barycenter system}.

\begin{proposition}(MFC barycenter system) 
    The critical point $(\bs \rho, \underline{\bs m}, \bs s, \varrho, \bs \phi)$ for the saddle point problem \eqref{uncon2}  satisfies the following conditions: for $1\le i\le N$
    \begin{subequations}
    \label{prop21}
        \begin{equation}
        \label{prop21a}
        \frac{\bs m_i}{V_{1,i}(\rho_i)} = \nabla_{\Omega}\phi_i,\quad \enspace \frac{s_i}{V_{2,i}(\rho_i, \rho_{i+1})}=\phi_{i}-\phi_{i+1},
        \end{equation}
    and 
    \begin{equation}
    \label{prop21b}
        \left\{\begin{split}
            &\partial_t\rho_i + \nabla_{\Omega}\cdot(V_{1,i}(\rho_i)\nabla_{\Omega}\phi_i) - V_{2,i}(\rho_i,\rho_{i+1})(\phi_i-\phi_{i+1}) \\
            &\hspace{4cm}+ V_{2,i-1}(\rho_{i-1},\rho_{i})(\phi_{i-1}-\phi_i)=0,\\
            &\partial_t\phi_i + \frac{1}{2}||\nabla_{\Omega}\phi_i||^2V_{1,i}^\prime(\rho_i) + \frac{1}{2}|\phi_{i}-\phi_{i+1}|^2 \frac{\partial}{\partial\rho_i}V_{2,i}(\rho_{i},\rho_{i+1})\\
            &\hspace{4cm} + \frac{1}{2}|\phi_{i-1}-\phi_{i}|^2 \frac{\partial}{\partial\rho_i}V_{2,i-1}(\rho_{i-1},\rho_{i})
            + \frac{\partial}{\partial \rho_i}\bm{\mathcal F}(\rho) = 0,
        \end{split}\right.
    \end{equation}
    with the initial and terminal conditions 
    \begin{equation}
    \label{prop21c}
        \rho_i(0,x) = \rho^0_i(x), \enspace
                \rho_i(T,x) = \varrho(x), 
        \enspace \sum_{i=1}^N\phi_i(T,x) = 0, \quad \text{ on }\Omega,
    \end{equation}
    and the boundary condition 
\begin{align}
                \bs m_i \cdot\bs \nu = 0 \text{ on } [0,T]\times \partial\Omega.
\end{align}
    \end{subequations}
\end{proposition}
\begin{proof}
The critical point system is obtained by setting the  first-order variation of the  Lagrangian
$\mathcal{L}(\underline{\bs u}, \varrho, \bs\phi)$ to be zero. 
Hence, we have 
$\frac{\delta \mathcal{L}}{\delta \bs m_i}
= \frac{\bs m_i}{V_{1,i}(\rho_i)}-
\nabla_{\Omega}\phi_i = 0$
and 
$\frac{\delta \mathcal{L}}{\delta s_i}
= \frac{s_i}{V_{2,i}(\rho_i, \rho_{i+1})}-
(\phi_i-\phi_{i+1}) =0$, which implies \eqref{prop21a}.
Taking variation on $\rho_i$, we obtain 
\begin{align*}
0=    \frac{\delta \mathcal{L}}{\delta \rho_i}
= &
-
\frac{||\bs m_i||^2}{2V_{1, i}^2\left(\rho_i\right)}V'_{1, i}\left(\rho_i\right)
-\frac{\left|s_i\right|^2}{2V_{2, i}^2(\rho_i,\rho_{i+1})}
\partial_{\rho_i}V_{2, i}\left(\rho_i,\rho_{i+1}\right)\\
&\quad
-\frac{\left|s_{i-1}\right|^2}{2V_{2, i-1}^2(\rho_{i-1},\rho_{i})}
\partial_{\rho_i}V_{2, i-1}\left(\rho_{i-1},\rho_{i}\right)
-\partial_{\rho_i}\mathcal{F}(\bs \rho)
-\partial_t\phi_i,
\end{align*}
which is the backward equation for $\phi_i$ in \eqref{prop21b}.
Taking variational on the terminal density $\varrho$, we get the boundary condition for $\bs\phi$:
\[
\frac{\delta \mathcal{L}}{\delta \varrho}
=\sum_{i=1}^N\phi_i(T, x) = 0.
\]
Finally, taking variation on $\phi_i$
and applying integration by parts, we obtain
\begin{align*}
0 = \frac{\delta \mathcal{L}}{\delta \phi_i}(\delta \phi_i) 
= &\;
\left(\partial_t\rho_i+\nabla_{\Omega}\cdot\bs m_i
-(s_{i}-s_{i-1})
, \delta \phi_i\right)_{\Omega_t}
\\
&\quad
+\langle\varrho(x)-\rho_i(T,x), \delta \phi_i(T,x)\rangle_{t=T} \\
&\quad-
\langle\rho_i^0(x)-\rho_i(0,x), \delta \phi_i(0,x)\rangle_{t=0},\\
&\quad -
\int_0^T\!\!\!\int_{\partial\Omega}\bs m_i\cdot
\bs \nu \delta \phi_i ds dt.
\end{align*}
for all $\delta\phi_i\in H^1(\Omega_t)$.
This implies 
$
\partial_t\rho_i+\nabla_{\Omega}\cdot\bs m_i
-(s_{i}-s_{i-1}) = 0
$ in $\Omega_t$, 
$\rho_i(0,x) = \rho_i^0(x)$ and 
$\rho_i(T,x) = \varrho(x)$ on $\Omega$, 
and $\bs m_i \cdot\bs \nu = 0$ on
$[0,T]\times \partial\Omega$.
This completes the proof.

\end{proof}

\section{High order discretizations and optimization algorithm}\label{sec3}
In this section, we present a detailed high-order finite element discretization to the optimization problem \eqref{uncon2}.
The variational structure of the problem \eqref{uncon2}
makes the finite element method an ideal candidate.
We refer interested reader to \cite{benamou2000computational, FuLiu23, FuOsherLi,fu2023generalized} for related work. 

\subsection{Finite-element discretization}
The finite-element discretization for the optimization problem formulated above in \eqref{uncon2} follows the previous work in \cite{FuLiu23, fu2023generalized}. Let $\Omega_h=\{S_n\}_{n=1}^{N_S}$ be a conforming triangulation, containing $N_S$ cells, of the spatial domain $\Omega$ , and $\mathcal I_h = \{I_m\}_{m=1}^{N_T}$ be a uniform discretization of the time domain $[0,T]$ comprising of $N_T$ intervals. The spacetime mesh is taken to be the tensor-product mesh $\Omega_{t,h} := \mathcal I_h \otimes \Omega_h$. We assume the spatial mesh $\Omega_h$ consists of mapped cubic elements when $\Omega$ is a volume domain, or mapped quadrilateral elements when $\Omega$ is a surface domain. 
Given these meshes, we make use of the following finite-element spaces:
\begin{subequations}
    \begin{alignat}{3}
        V_h^k &:=\left\{v \in H^1\left(\Omega_{t}\right):\right.  \left.\left.v\right|_{I_m \times S_n} \in Q^k\left(I_m\right) \otimes Q^k\left(S_{n}\right), \enspace\forall m, n\right\}\\
        W_h^{k-1} &:=\left\{w \in L^2\left(\Omega_{t}\right):\right.  \left.\left.w\right|_{I_m \times S_{n}} \in Q^{k-1}\left(I_m\right) \otimes Q^{k-1}\left(S_{n}\right), \enspace\forall m, n\right\},\\
        M_h^{k-1} &:=\left\{\mu \in L^2(\Omega):\right.  \left.\left.\mu\right|_{S_{n}} \in Q^{k-1}\left(S_{n}\right), \enspace\forall n\right\},
    \end{alignat}
\end{subequations}
where $k \geq 1$ denotes the polynomial degree, and $Q^k(S_n)$ is polynomial space on $S_n$ formed by taking a tensor-product of polynomials of maximum degree $k$ in each coordinate direction. We use the discontinuous $L^2$ spaces to approximate the physical variables, since the system \eqref{uncon2} does not involve their derivatives. On the other hand, to approximate the $N$ components of the Lagrange multiplier $\bs \phi$, we use the $H^1$-conforming space $V_h^k$ as \eqref{uncon2} requires the spacetime gradient of $\bs \phi$. 
Notice that, using the spaces defined above, the discrete version of \eqref{uncon2} takes the following form: Find the optimal point of the following discrete system
\begin{equation}\label{uncon3}
    \inf_{\underline{\bs u_h},{\varrho_h}} \sup_{\bs\phi_h}\quad H_h(\underline{\bs u_h}) - \intlrh{\underline{\bs u_h},\underline{\D\bs\phi_h}} + \lrangle{{\varrho_h} ,\bs \phi_h\cdot\bs 1}_{t=T,h}-\lrangle{\bs\rho^0_h,\bs\phi_h}_{t=0,h},
\end{equation}
where we have the variables $\underline{\bs u_h} = (\bs u_{1,h},\cdots, \bs u_{N,h})$ with $\bs u_\ih = (\rho_\ih,{\bs{m}_\ih}, s_\ih)$ and $\bs \phi_h = (\phi_{1,h},\cdots,\phi_{N,h})$. We take $\bs{m}_\ih \in [W_h^{k-1}]^3$, $\rho_\ih, s_\ih \in W_h^{k-1}$, $\varrho_h \in M_h^{k-1}$ with $\varrho_h \geq 0$ a.e., and $\phi_{i,h} \in V_h^{k}$. Note that the spaces used to approximate the physical variables have a polynomial degree one less than the space used for the Lagrange multipliers. When $k=1$, we obtain a staggered scheme in which the physical variables are stationed at the center while the Lagrange multipliers are placed at the vertices of the mesh.

To evaluate the discrete integrals in \eqref{uncon3}, we use the Gauss-Legendre quadrature rule with $k$ quadrature points per spatial dimension. Since there are $N_T$ time elements and $N_S$ space elements, the total number of quadrature points used to evaluate an integral on $\Omega_{t,h}=\mathcal I \otimes \Omega_h$ is $(kN_T)\times (k^dN_S)$, where $d=2$ for the surface geometry and $d=3$ for the volume geometry. Let $\{\chi_i\}_{i=1}^{k^dN_S}$ be the set of spatial quadrature points with and $\{\tau_i\}_{i=1}^{N_T}$ be the set of temporal quadrature points. Let $\{\lambda_i\}_{i=1}^{k^dN_S}$ and $\{\theta_i\}_{i=1}^{N_T}$ be the sets of corresponding quadrature weights respectively. Using these, the discrete integrals can be evaluated as:
\begin{subequations}\label{dscrtint}
    \begin{equation}\label{dscrtinta}
            \intlrh{\bs a, \bs b }= \sum_{i=1}^{kN_T} \sum_{j=1}^{k^dN_S} \theta_i \lambda_j \bs a(\tau_i,\chi_k) \cdot \bs b(\tau_i,\chi_k),
    \end{equation}
    \begin{equation}\label{dscrtintb}
        \lrangle{a,\bs b \cdot \bs 1}_{t=T,h} = \sum_{j=1}^{k^dN_S}  \lambda_j  a(T,\chi_j)\bs b(T,\chi_j)\cdot \bs 1,
    \end{equation}
    \begin{equation}\label{dscrtintc}
        \lrangle{a,\bs b \cdot \bs 1}_{t=0,h} = \sum_{j=1}^{k^dN_S}  \lambda_j  a(0,\chi_j)\bs b(0,\chi_j)\cdot \bs 1.
    \end{equation}
\end{subequations}
Note that \eqref{dscrtinta} is used to evaluate the spacetime integral in $ H_h(\underline{\bs u_h})$ as well. 
In addition to the Gauss-Legendre quadrature rule, we use tensor-product Gauss-Legendre basis functions for the discontinuous spaces $W_h^{k-1}$ and $M_h^{k-1}$. This means that any function $w_h \in W_h^{k-1}$ has the following form:
\begin{equation}
    w_h = \sum_{i=1}^{kN_T}\sum_{j=1}^{k^dN_S}\text{w}_{i,j} \psi_i(t)\varphi_j(x), \enspace \text{w}_{i,j} \in \R, \enspace \text{w}_{i,j} = w_h(\tau_i,\chi_j).
\end{equation}
For any $i,j$, the basis function satisfies the nodal property, i.e.,
\begin{equation}
    \psi_i(\tau_{i'}) = \delta_{ii'} \text{ and } \enspace \varphi_j(\chi_{j'}) = \delta_{jj'},
\end{equation}
where $\delta_{ij}$ is the Kronecker-delta function with indices $i$ and $j$. Similarly, for any function $v_h \in M_h^{k-1}$, we have
\begin{equation}
v_h = \sum_{j=1}^{k^dN_S}\text{v}_{j} \varphi_j(x), \enspace \text{v}_{j} \in \R, \enspace \text{v}_{j} = v_h(\chi_j).
\end{equation}
Therefore, thanks to these basis functions, the unknowns in problem \eqref{uncon3} have the following forms:
\begin{subequations}\label{GLvars}
\begin{alignat}{4}
    \rho_\ih &= \sum_{j=1}^{kN_T}\sum_{\ell=1}^{k^dN_S}\uprho_{i,j\ell} \psi_j(t) \varphi_\ell(x), \enspace \uprho_{i,j\ell} \in \R_+, \enspace \uprho_{i,j\ell} = \rho_\ih(\tau_j,\chi_\ell),\\
    \varrho_h &= \sum_{j=1}^{kN_T}\sum_{\ell=1}^{k^dN_S} \upvarrho_{j\ell}\psi_j(t) \varphi_\ell(x), \enspace \upvarrho_{j\ell} \in \R_+, \enspace \upvarrho_{j\ell} = \varrho_h(\tau_j,\chi_\ell)\\
    \bs m_\ih &= \sum_{j=1}^{kN_T}\sum_{\ell=1}^{k^dN_S}\textbf{m}_{i,j\ell} \psi_j(t) \varphi_\ell(x), \enspace \textbf{m}_{i,j\ell} \in \R^3, \enspace \textbf{m}_{i,j\ell} = \bs m_\ih(\tau_j,\chi_\ell),\\
    s_\ih &= \sum_{j=1}^{kN_T}\sum_{\ell=1}^{k^dN_S}\textbf{s}_{i,j\ell} \psi_j(t) \varphi_\ell(x), \enspace \textbf{s}_{i,j\ell} \in \R, \enspace \textbf{s}_{i,j\ell} = s_\ih(\tau_j,\chi_\ell)
\end{alignat}
    
\end{subequations}
We highlight here that the combination of Gauss-Legendre quadrature rules and Gauss-Legendre basis functions for the $L^2$-conforming spaces $W_h^{k-1}$ and $M_h^{k-1}$ brings about a notable simplification of the optimization problem \eqref{uncon3} as it decouples the degrees of freedom of the physical variables $\underline{\bs u_h}$ and $\varrho_h$ for a fixed value of the Lagrange multiplier $\bs\phi_h$. Consequently, it becomes possible to independently solve the non-linear optimization problems for $\underline{\bs u_h}$ and $\varrho_h$ at each quadrature point. It is also clear from the first two equations in \eqref{GLvars} that the positivity of densities $\rho_\ih$ and the terminal density $\varrho_h$ are guaranteed at all quadrature points since at these points their respective admissible sets require $\uprho_{i,j\ell}$ and $\upvarrho_{j\ell}$ to be non-negative. On the other hand, the above-described features of the scheme do not depend on the choice of basis functions for the continuous space $V_h$. In our  implementation, we use nodal Gauss-Lobatto basis functions.

\subsection{Primal-Dual Hybrid Gradient Algorithm}
To tackle the saddle-point problem in \eqref{uncon3}, we apply the primal-dual hybrid gradient (PDHG) method\cite{Chambolle}. Given parameters $\sph,\su > 0$, and initial guess $(\underline{\bs u^0_h}, {\varrho}^0_h, \bs \phi^0_h)$, the PDHG method performs proximal gradient ascent on the variable $\bs \phi_h$ and proximal gradient descent on the variable $(\underline{\bs u_h} , {\varrho}_h)$ alternatively. The $(k+1)$-th iteration of the algorithm takes the following form:
\begin{subequations}
    \begin{itemize}
        \item Step 1: Proximal gradient ascent for $\bs\phi^\kpo_h$:
        \begin{equation}\label{pdhga}
            \begin{aligned}
                \bs\phi^\kpo_h = \underset{\bs\phi_h}{\operatorname{argmax}}  &- \intlrh{\underline{\bs u_h^k},\underline{\D\bs\phi_h}} + \lrangle{{\varrho}^k_h,\bs \phi_h\cdot\bs 1}_{t=T,h}-\lrangle{\bs\rho^0_h,\bs\phi_h}_{t=0,h} \\ &- \frac{1}{2\sph}\intlrh{\underline{\D(\bs\phi_h-\bs\phi^k_h)},\underline{\D(\bs\phi_h-\bs\phi^k_h)}} \\& -\frac{1}{2\sph}\lrangle{(\bs\phi_h-\bs\phi_h^k)\cdot\bs 1,(\bs\phi_h-\bs\phi_h^k)\cdot\bs 1}_{t=T,h}
            \end{aligned}   
        \end{equation}
        \item Step 2: Extrapolation for $\tilde{\bs\phi}^\kpo_h$:
        \begin{equation}\label{pdhgb}
            \tilde{\bs\phi}^\kpo_h = 2\bs\phi^\kpo_h-\bs\phi^k_h
        \end{equation}
        \item Step 3: Proximal gradient descent for $(\underline{\bs u_h^\kpo}$ and $\varrho^\kpo_h)$:
        \begin{equation}\label{pdhgc}
            {\varrho}^\kpo_h = \underset{{\varrho}_h}{\operatorname{argmin}} \lrangle{{\varrho}_h,\tilde{\bs\phi}^\kpo_h\cdot \bs 1}_{t=T,h}+\frac{1}{2\sigma_{\underline{\bs u_h}}} \lrangle{{\varrho}_h-{\varrho}^k_h,{\varrho}_h-{\varrho}^k_h}_{t=T,h}
        \end{equation}
        \begin{equation}\label{pdhgd}
            \begin{aligned}
                \underline{\bs u_h^\kpo} = \underset{\underline{\bs u_h}}{\operatorname{argmin}} \enspace H_h(\underline{\bs u_h}) - \intlrh{\underline{\bs u_h},\underline{\D\tilde{\bs\phi}^\kpo_h}} + \frac{1}{2\sigma_{\underline{\bs u_h}}} \intlrh{\underline{\bs u_h}-\underline{\bs u_h^k},\underline{\bs u_h}-\underline{\bs u_h^k}}
            \end{aligned}
        \end{equation}
    \end{itemize}
\end{subequations}

\subsubsection{Step 1: Solving for $\bs\phi^\kpo_h$}
$\\$
If we write \eqref{pdhga} as a minimization problem:
\begin{equation}
\begin{aligned}
    \bs\phi^\kpo_h = \inf_{\bs\phi_h}  &\intlrh{\underline{\bs u_h^k},\underline{\D\bs\phi_h}} - \lrangle{{\varrho}^k_h,\bs \phi_h\cdot\bs 1}_{t=T,h}+\lrangle{\bs\rho^0_h,\bs\phi_h}_{t=0,h} \\ &+ \frac{1}{2\sph}\intlrh{\underline{\D(\bs\phi_h-\bs\phi^k_h)},\underline{\D(\bs\phi_h-\bs\phi^k_h)}}
    \\&+\frac{1}{2\sph}\lrangle{(\bs\phi_h-\bs\phi_h^k)\cdot\bs 1,(\bs\phi_h-\bs\phi_h^k)\cdot\bs 1}_{t=T,h},
\end{aligned}
\end{equation}
it becomes clear that the above is equivalent to solving the following elliptical problem: find $\bs \phi_h\in [V_h^k]^N$ such that for all $\bs\psi_h\in [V_h^k]^N$
\begin{equation}\label{elliptical}
    \begin{aligned}
        &\intlrh{\underline{\D (\bs \phi_h -\bs\phi^k_h)},\underline{\D \bs \psi_h}} + \lrangle{(\bs\phi_h-\bs\phi_h^k)\cdot\bs 1,\bs\psi_h\cdot\bs 1}_{t=T,h} \\
        = &-\sph\intlrh{\underline{\bs u_h^k},\underline{\D\bs\psi_h}} + \sph\lrangle{{\varrho}^k_h,\bs \psi_h\cdot \bs 1}_{t=T,h} 
        -\sph\lrangle{\bs\rho^0_h,\bs\psi_h}_{t=0,h}.
    \end{aligned}
\end{equation}
We will solve \eqref{elliptical} component-wise for each $\phi_\ih$. To do so, we use a Gauss-Seidel type splitting to break the cross-terms. Skipping details for brevity, we get the following solve for component $\delta\phi_\ih = \phi_\ih-\phi_\ih^k$:
for $1\le i\le N$,
find $\delta\phi_\ih\in V_h^k$ such that for all $\psi_\ih\in V_h^k$
\begin{subequations}
    \begin{equation}\label{step1ls}
    \begin{aligned}
        &\intlrh{\partial_t \delta\phi_\ih,\partial_t \psi_\ih}+\intlrh{\nbb \delta\phi_\ih,\nbb \psi_\ih} + 2\intlrh{\delta\phi_\ih,\psi_\ih} + \lrangle{\delta\phi_\ih,\psi_\ih}_{t=T,h}\\=&-\sph\intlrh{\rho_\ih^k,\partial_t\psi_\ih}-\sph\intlrh{\bs m_\ih^k,\nbb\psi_\ih} 
        -\sph\intlrh{s_\ih^k-s_{i-1,h}^k,\psi_\ih}
        \\   
        &+\sph\lrangle{{\varrho}_h^k,\psi_\ih}_{t=T,h}-\sph\lrangle{\rho^0_\ih,\psi_\ih}_{t=0,h} \\&+\intlrh{\delta\phi^{k+1}_{i-1,h}+\delta\phi^k_{i+1,h},\psi_\ih} -\lrangle{\textstyle\sum^{i-1}_{j=1}\delta\phi_{j,h}^{k+1}
        +\textstyle\sum^{N}_{j=i+1}\delta\phi_{j,h}^{k}
        ,\psi_\ih}_{t=T,h},
    \end{aligned}
\end{equation}
and 
\begin{equation}
    \phi_\ih^\kpo = \phi_\ih^k + \delta\phi_\ih.
\end{equation}
\end{subequations}

\subsubsection{Step 2: Extrapolation for $\tilde{\bs\phi}^{k+1}_h$}
$\\$
Having solved for $\bs\phi_h^\kpo$ in the previous step, the extrapolation in this step takes the following simple form
\begin{equation}
    \tilde{\bs\phi}^{k+1}_h = \bs\phi^\kpo_h + \delta\bs\phi_h = 2\bs\phi_h^\kpo - \bs\phi_h^k.
\end{equation}

\subsubsection{Step 3: Solving for $\underline{\bs u_h^{k+1}}$ and ${\varrho}_h^\kpo$}
$\\$
Equation \eqref{pdhgc} is equivalent to the following solve: find ${\varrho_h}$ such that for all $\nu_h$
\begin{equation}
    \lrangle{{\varrho}_h-{\varrho}^k_h,\nu_h}_{t=T,h} = -\su\lrangle{\tilde{\bs\phi}_h^\kpo\cdot\bs 1,\nu_h}_{t=T,h},
\end{equation}
which, by use of a Gauss-Legendre quadrature rule in conjunction with a nodal Gauss-Legendre basis for the space $M_h^{k-1}$, reduces to the following pointwise solve at each integration point:
\begin{equation}
    {\varrho}_h = {\varrho}^k_h - \su \sum_{i=1}^N\tilde{\bs\phi}_\ih^\kpo.
\end{equation}
Finally, to get the solver for $\underline{\bs u_h}$, we set 
\begin{equation}
    \bar{\bs u}_h = \underline{\bs u_h^k} + \su\mathcal D\tilde{\phi}^\kpo_h
\end{equation}
and rewrite \eqref{pdhgd} as 
\begin{equation}\label{minubar}
    \underline{\bs u_h^\kpo} = \underset{\underline{\bs u_h}}{\operatorname{argmin}} \enspace H_h(\underline{\bs u_h}) + \frac{1}{2\su}\intlrh{\underline{\bs u^k_h}-\bar{\bs u}_h,\underline{\bs u^k_h}-\bar{\bs u}_h} .
\end{equation}
It is clear that with the choice of the basis functions for $\underline{\bs u_h}$ and the quadrature rule for the numerical integration in \eqref{minubar}, the above optimization problem can be solved separately in parallel on each quadrature point. 
On each quadrature point, we have $5N$ unknowns from each component of $\underline{\bs u_h}$ ($3N$ for the fluxes $\underline{\bs m_h}$, $N$ for densities ${\bs\rho_h}$, and $N$ for sources $\bs s_h$). 
We further simplify this problem by applying optimization on the fluxes $\bs{\underline m_h}$ and sources $\bs s_h$ to get, for all $1\le i\le N$,
    \begin{equation}\label{optimalms}
    \begin{aligned}
        \bs m_\ih &= \frac{V_{1,i}(\rho_\ih)}{\su+V_{1,i}(\rho_\ih)}\bar{\bs m}_\ih,\\
        s_\ih &= \frac{V_{2,i}(\rho_\ih,\rho_{i+1,h})}{\su+V_{2,i}(\rho_\ih,\rho_{i+1,h})}\bar{s}_\ih.
    \end{aligned} 
    \end{equation}
    Replacing these expressions into \eqref{minubar}, we get the following optimization problem for the $N$-components of density on each quadrature point $(\tau_j, \chi_\ell)$
    for all $1\le j\le kN_T$, $1\le \ell\le k^dN_S$:
    \begin{equation}\label{step2nls}
    \begin{aligned}
       \bs \rho_{j\ell}^\kpo = \underset{\bs \rho_{j\ell}\ge 0}{\operatorname{argmin}} \enspace\sum_{i=1}^N&\frac{||\bar{\bs m}_{i,j\ell}||^2}{2[\su+V_{1,i}(\rho_{i,j\ell})]} + \frac{\bar{s}_{i,j\ell}^2}{2[\su+V_{2,i}(\rho_{i,j\ell},\rho_{i+1,j\ell})]} \\ 
        &\quad+\frac{1}{2\su}(\rho_{i,j\ell}-\bar{\rho}_{i,j\ell})^2 - F_i(\rho_{i,j\ell}).
    \end{aligned}
\end{equation}
Finally, we approximately solve the above problem \eqref{step2nls} sequentially for each component $i$, which results in a single variable minimization problem for $\rho_{i,j\ell}$. 
This minimization problem is solved using Brent's minimization algorithm \cite{brent1973}. We borrow its implementation from the C++ Boost libraries \cite{BoostLibrary}, the details of which can be found on the documentation page \cite{boost_brent_minima}. Once $\rho_\ih$ are obtained for all $i$, we  use the equations in \eqref{optimalms} to recover $
\bs m_\ih$ and $s_\ih$.

\subsection{PDHG parameters, initialization and stopping criteria}
In the simulation, we take the PDHG parameters $\sigma_u=\sigma_\phi = 1$, and 
set initial guess as {\it zero} except for the initial densities $\bs \rho_h$ and $\varrho_h$, where $\bs\rho_h$ is set to be the initial data $\bs\rho^0$, and $\varrho_h$ is taken as the arithmetic average of 
the initial data $\bs\rho^0$. 
We terminate the PDHG algorithm either after a fixed number of iterations or when the $L_1$-norm of the difference between two consecutive terminal density 
$err = \|\varrho_h^k-\varrho_h^{k+1}\|_{L_1(\Omega)} < tol$, where $tol$ is a user defined tolerance. 
The theoretical convergence study of this algorithm is out of the scope of the current work and will be investigated separately in our future work. 

\section{Numerical Results}\label{sec4}
In this section, we present several numerical examples that demonstrate the effectiveness and applicability of our proposed approach \eqref{uncon3} for computing Wasserstein barycenters of 3D volume and 2D surface data with reaction-diffusion effects. 
We take the terminal time $T=1$. The initial densities $\bs \rho^0 = (\rho_1, \ldots, \rho_N)$ are specified, which under the action of the scheme \eqref{uncon3}, all flow to the same unknown terminal density $\rho_T$. 
For all simulations, both PDHG parameters $\sigma_{\phi_h}$ and $\su$ are set to 1.
The specific choices of the mobility functions and the interaction functional of the model are given in \eqref{v1v2} and \eqref{pot}.
The linear systems arising in \eqref{step1ls} are solved using the preconditioned conjugate gradient method with a geometric multigrid preconditioner, while the minimization in \eqref{step2nls} is performed using Brent's algorithm in the range $10^{-6}\leq \rho_{i,h} \leq 40$. The computations are performed using the high-performance finite element C++ library MFEM \cite{MFEM}. The code for these numerical examples can be found in the git repository: \href{https://github.com/avj-jpg/WassersteinBarycenter}{https://github.com/avj-jpg/WassersteinBarycenter}.

\begin{figure}[h!]
    \centering
    \begin{subfigure}{0.32\textwidth}
        \centering
        \includegraphics[width=\textwidth,trim=20mm 20mm 40mm 20mm, clip]{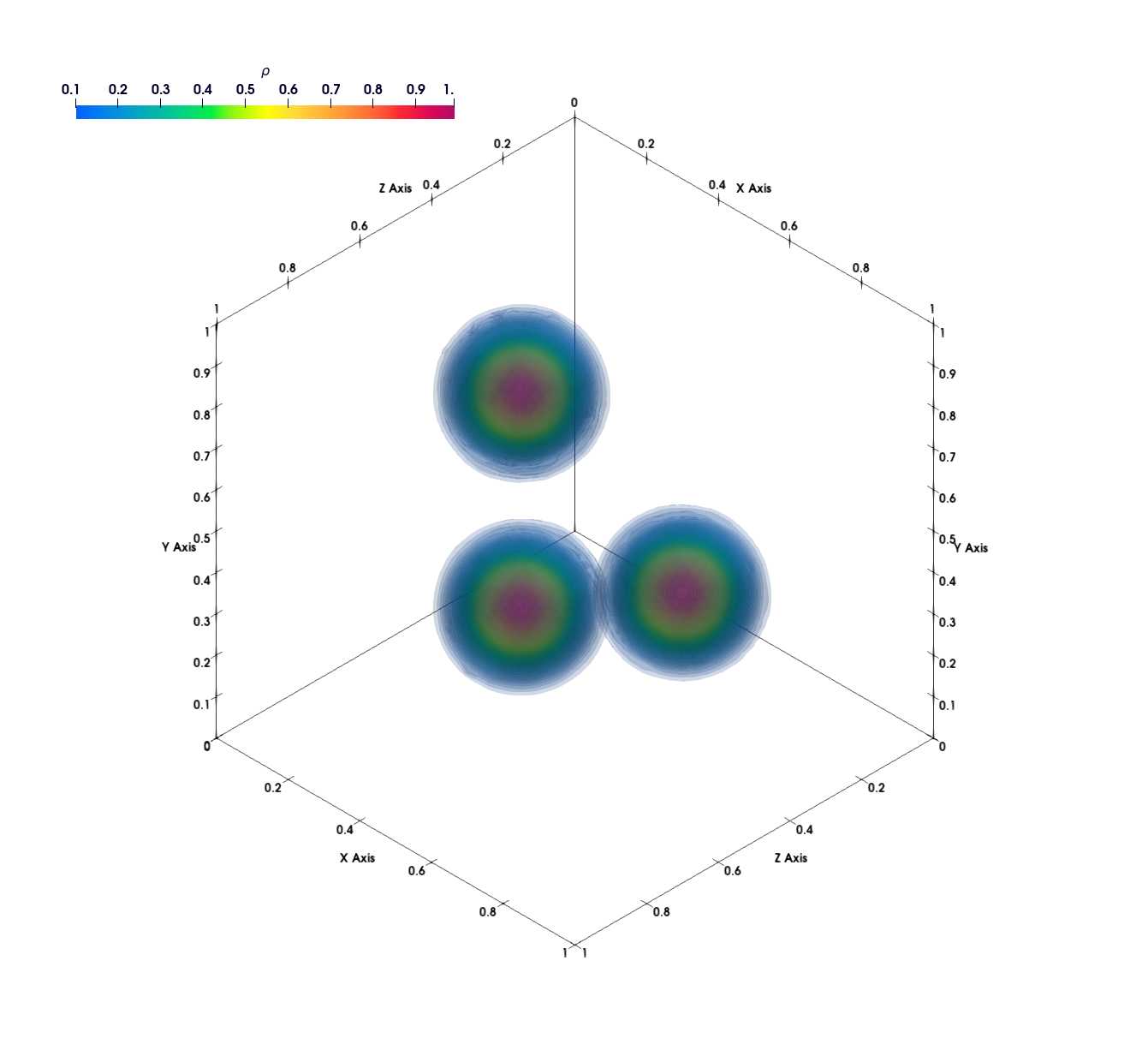}
        \caption{$t=0$}
    \end{subfigure}
    \begin{subfigure}{0.32\textwidth}
        \centering
        \includegraphics[width=\textwidth,trim=20mm 20mm 40mm 20mm, clip]{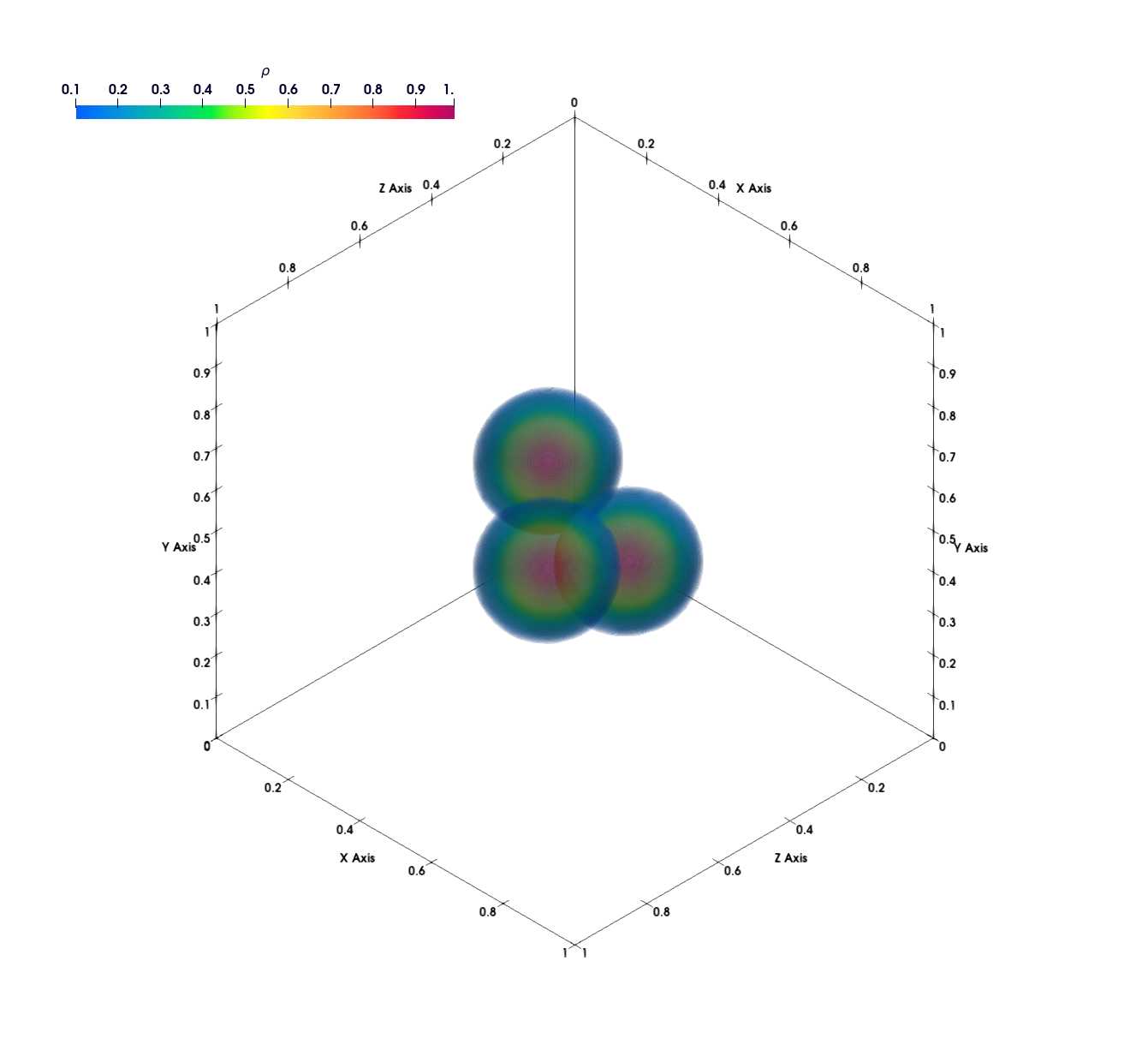}
        \caption{$t=0.5$}
    \end{subfigure}
    \begin{subfigure}{0.32\textwidth}
        \centering
        \includegraphics[width=\textwidth,trim=20mm 20mm 40mm 20mm, clip]{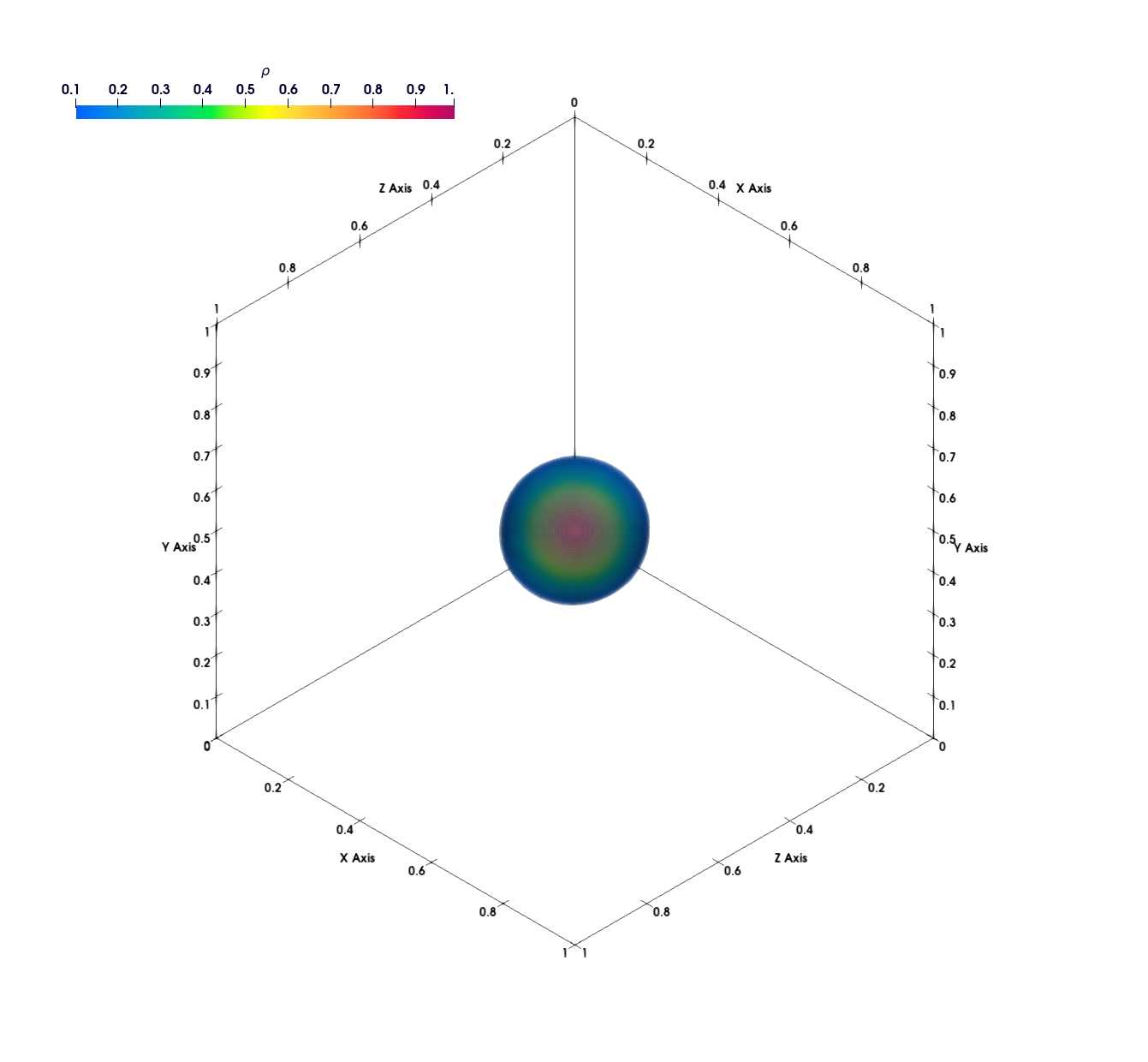}
        \caption{$t=1.0$}
    \end{subfigure}
      
    \caption{The Wasserstein barycenter of 3 Gaussian distributions without any reaction effects.}
    \label{f:gaussian}
\end{figure}

\begin{figure}[h!]
    \centering
\includegraphics[width=0.5\textwidth]{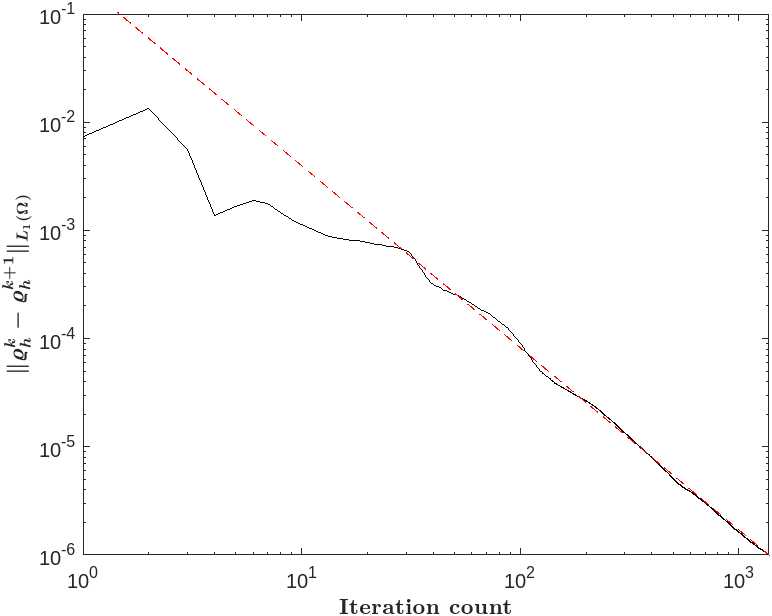}
    \caption{Evolution of the $L_1$-error 
    $err= \|\varrho_h^k-\varrho_h^{k+1}\|_{L_1(\Omega)}$
    against the number of PDHG iterations $k$.
}
    \label{f:conv}
\end{figure}

\subsection{Barycenter of three Gaussian distributions in 3D}
Our first example concerns the computation of the Wasserstein barycenter of three Gaussian initial densities given by
\begin{subequations}
    \begin{equation}
        \rho_1 = e^{-50\left( (x-0.8)^2 + (y-0.5)^2 + (z-0.5)^2 \right)}
    \end{equation}
    \begin{equation}
        \rho_2 = e^{-50\left( (x-0.35)^2 + (y-0.5-0.15\sqrt{3})^2 + (z-0.5)^2 \right)}. 
    \end{equation}
    \begin{equation}
        \rho_2 = e^{-50\left( (x-0.35)^2 + (y-0.5+0.15\sqrt{3})^2 + (z-0.5)^2 \right)}. 
    \end{equation}
\end{subequations}
The spatial domain is a unit cube $\Omega = [0,1]^3$.
We take the interaction strengths $\beta_i$ as well as the reaction strength $\alpha$ to be 0, which results in the 
classical Wasserstein barycenter problem with no reaction effects.
The Wasserstein barycenter is also a Gaussian density with a center located at the Euclidean barycenter of the centers of three Gaussian initial densities:
\begin{equation}
    \rho_T = e^{-50\left( (x-0.5)^2 + (y-0.5)^2 + (z-0.5)^2 \right)}.
\end{equation}
We perform our computations with a polynomial degree $k=4$ on a spacetime grid with $16\times16\times16$ uniform cubical elements for the spatial variable and eight uniform line elements for the time variable. 
The results are visualized in Figure \ref{f:gaussian}, which shows the contours of the three initial densities that evolve toward the terminal density over time. As foreseen, the terminal density is a Gaussian with a center located at $(0.5,0.5,0.5)$. 
We further plot the convergence 
of error $err= \|\varrho_h^k-\varrho_h^{k+1}\|_{L_1(\Omega)}$
against the number of PDHG iterations $k$ in Figure \ref{f:conv}. Clearly, we observe an asymptotic linear convergence rate. 
\begin{figure}[h!]
    \centering
\begin{minipage}[b]{\textwidth}
    \begin{minipage}[b]{\textwidth}
    \hfill
        \begin{subfigure}{\textwidth}
            \centering
            \includegraphics[width=\textwidth,trim=0mm 0mm 0mm 400mm, clip]{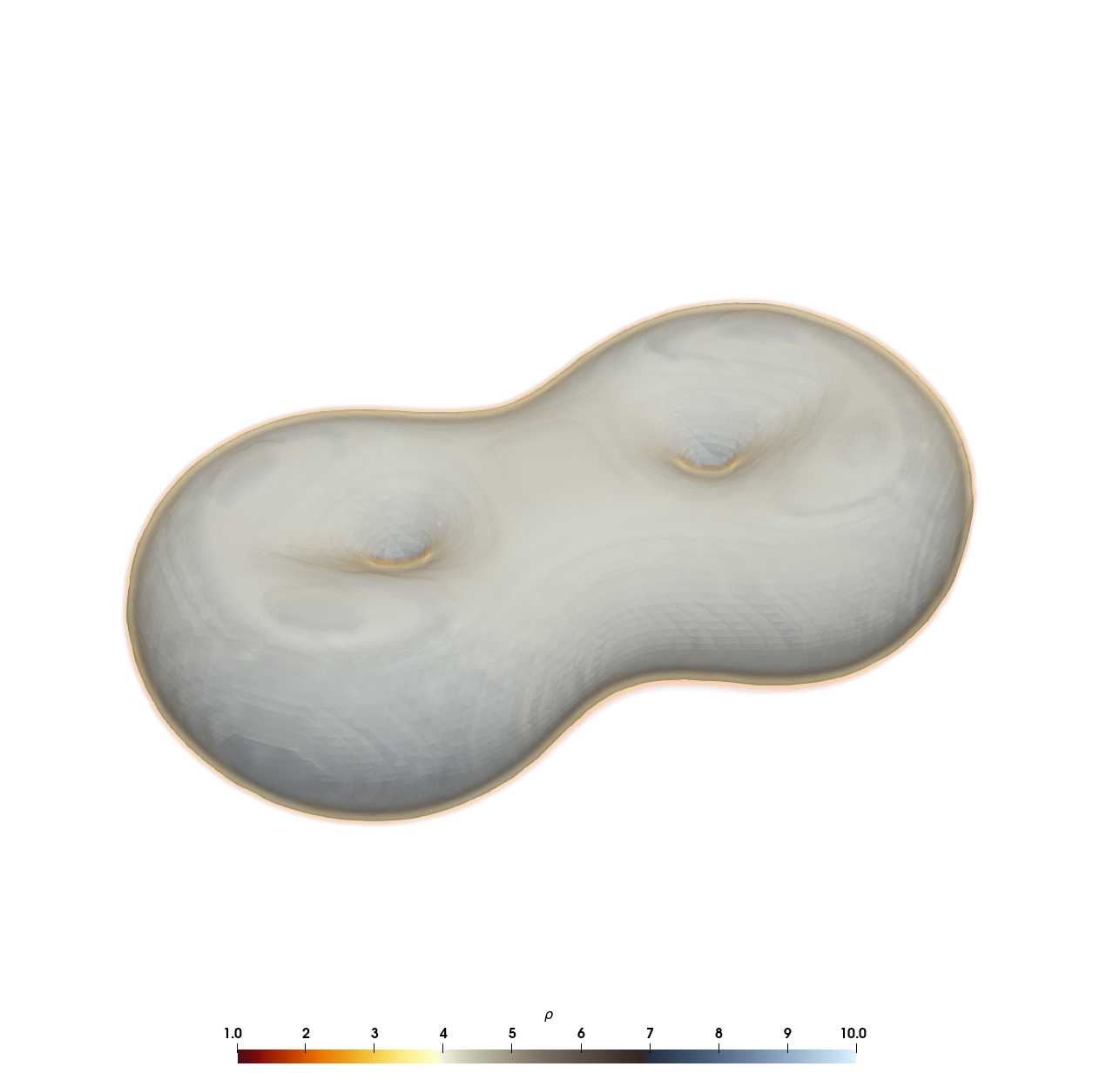}
        \end{subfigure}
    \end{minipage}
    
    \begin{subfigure}{0.16\textwidth}
        \centering
        \includegraphics[width=\textwidth,trim=30mm 40mm 30mm 70mm, clip]{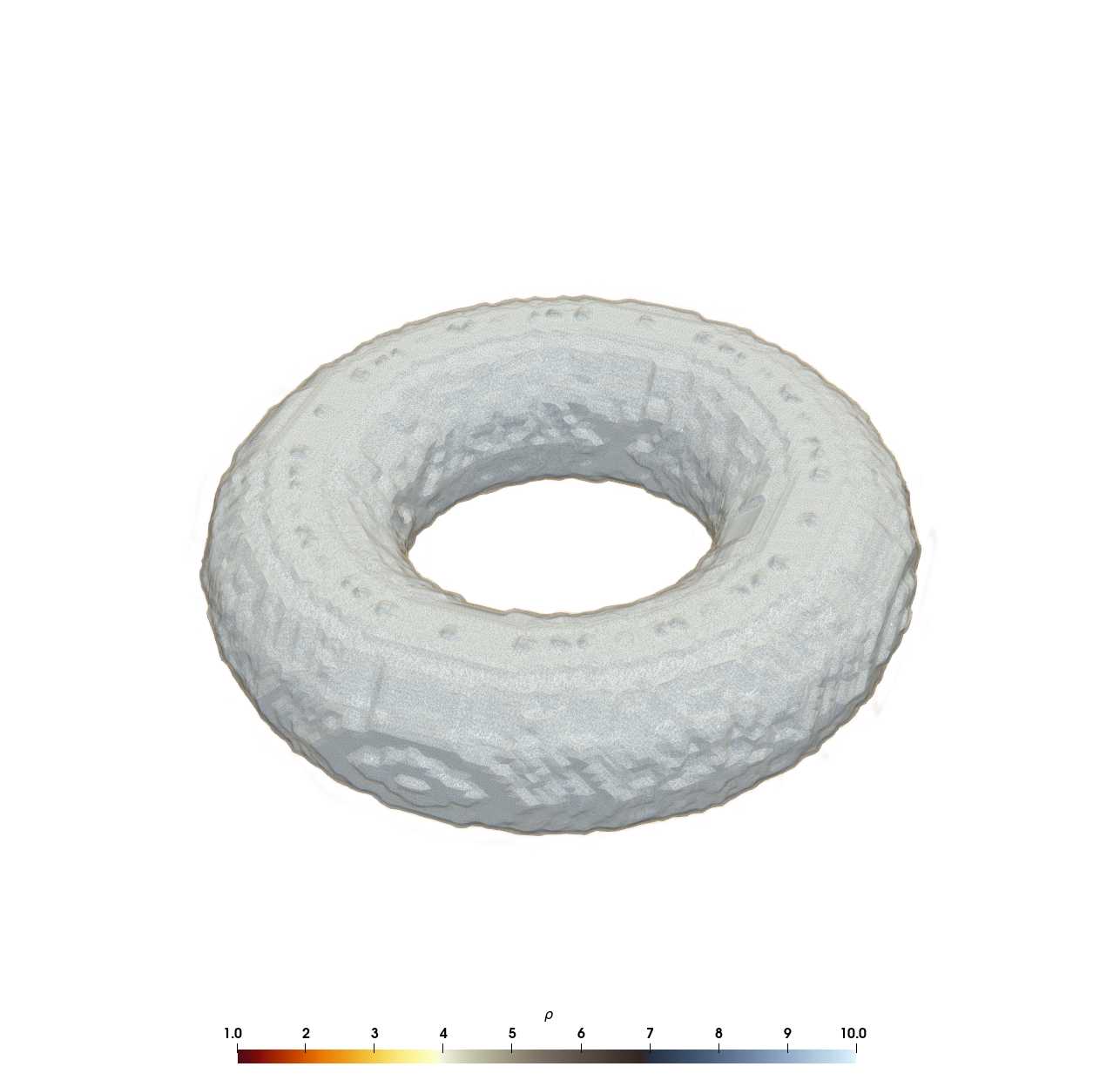}
    \end{subfigure}
    \begin{subfigure}{0.16\textwidth}
        \centering
        \includegraphics[width=\textwidth,trim=30mm 40mm 30mm 70mm, clip]{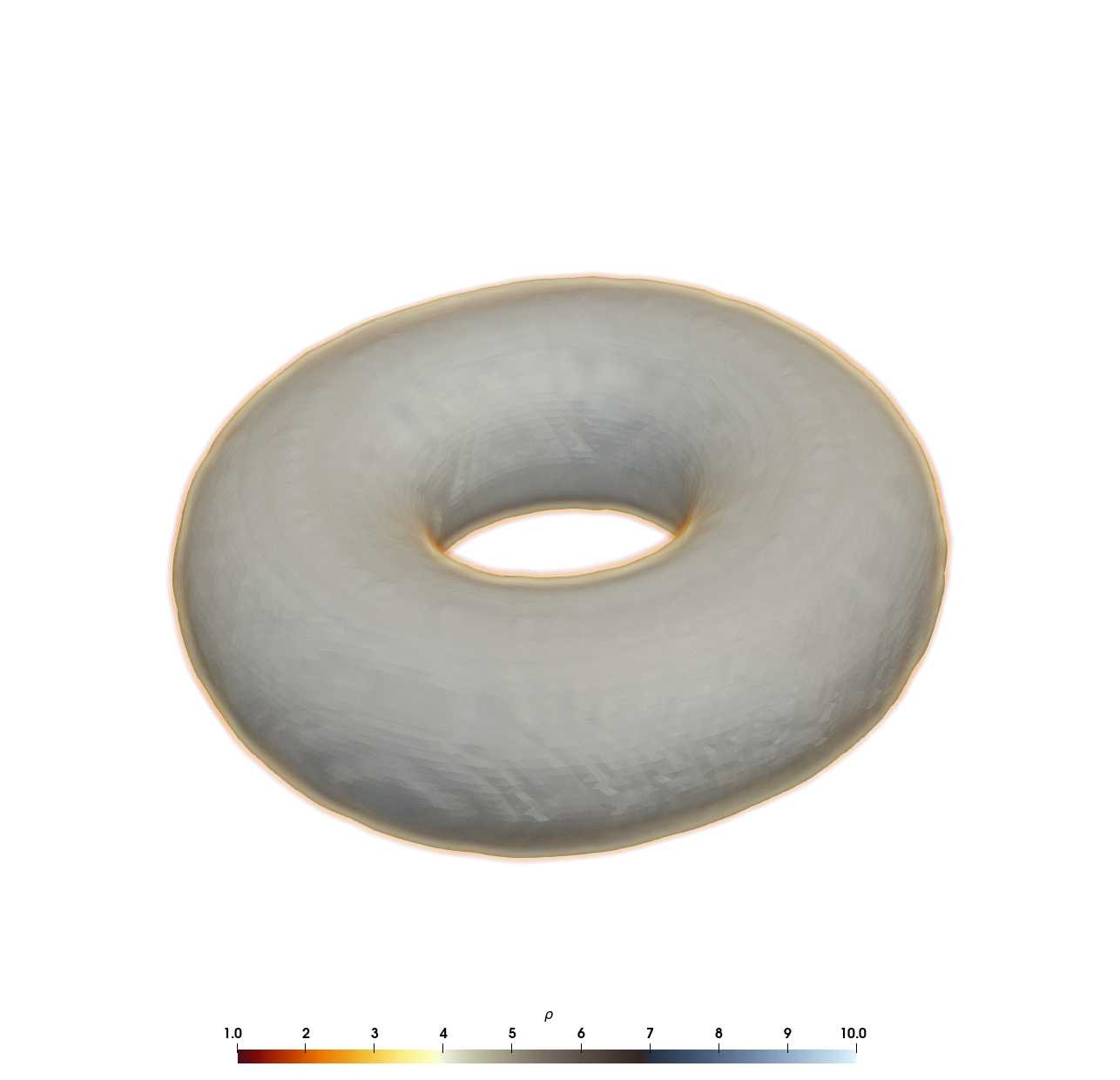}
    \end{subfigure}
    \begin{subfigure}{0.16\textwidth}
        \centering
        \includegraphics[width=\textwidth,trim=30mm 40mm 30mm 70mm, clip]{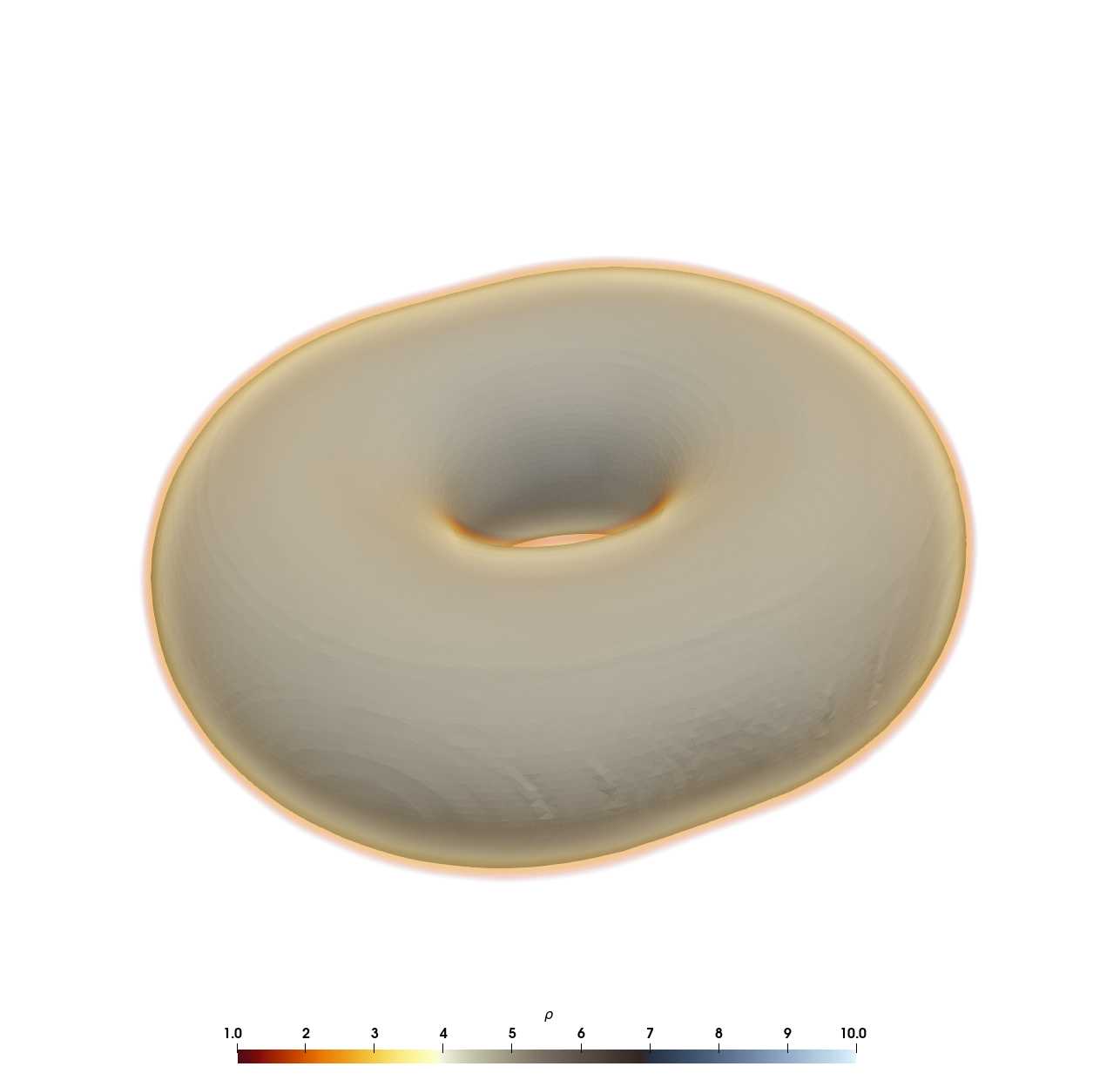}
    \end{subfigure}
    \begin{subfigure}{0.16\textwidth}
        \centering
        \includegraphics[width=\textwidth,trim=30mm 40mm 30mm 70mm, clip]{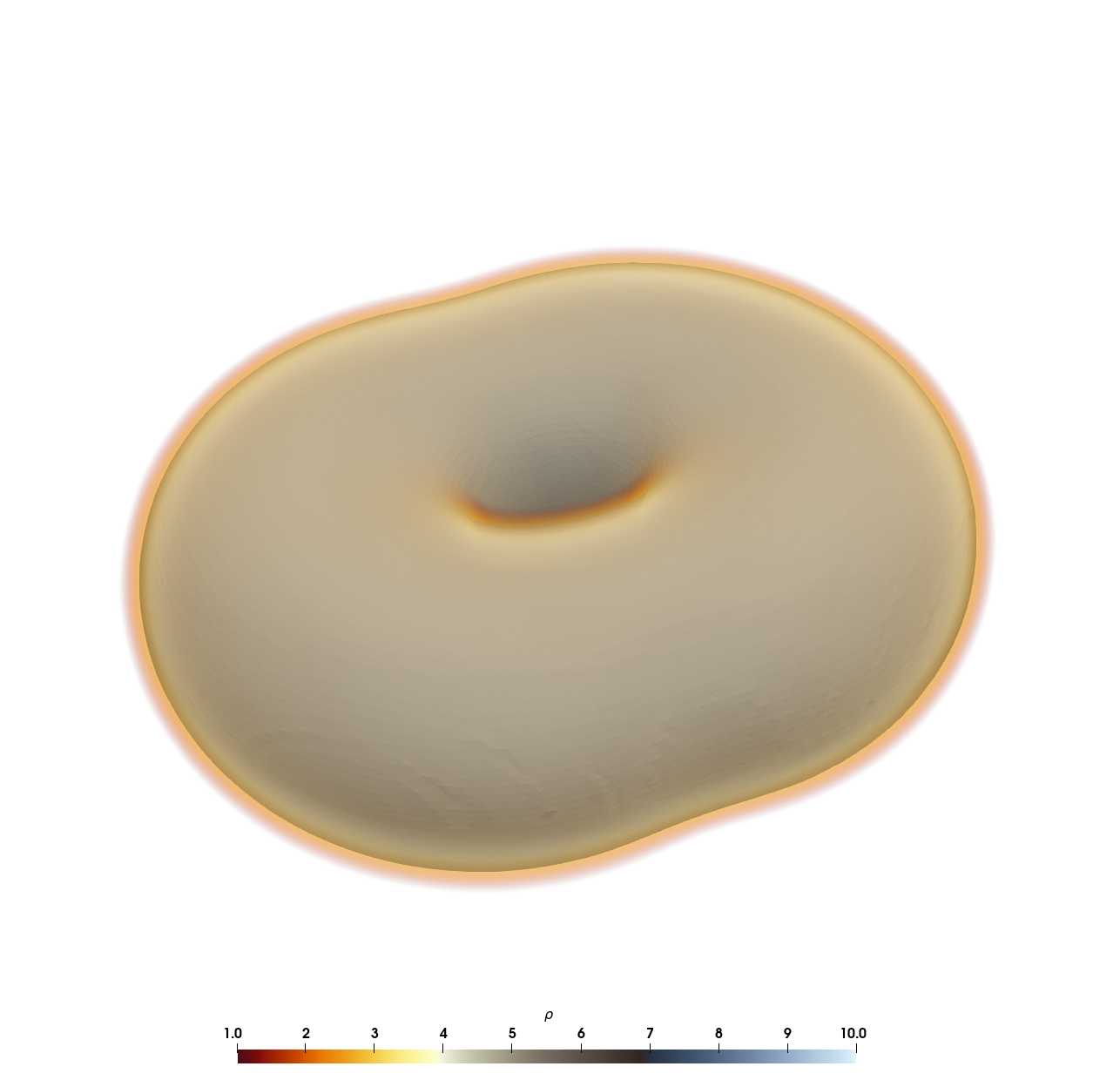}
    \end{subfigure}
    \begin{subfigure}{0.16\textwidth}
        \centering
        \includegraphics[width=\textwidth,trim=30mm 40mm 30mm 70mm, clip]{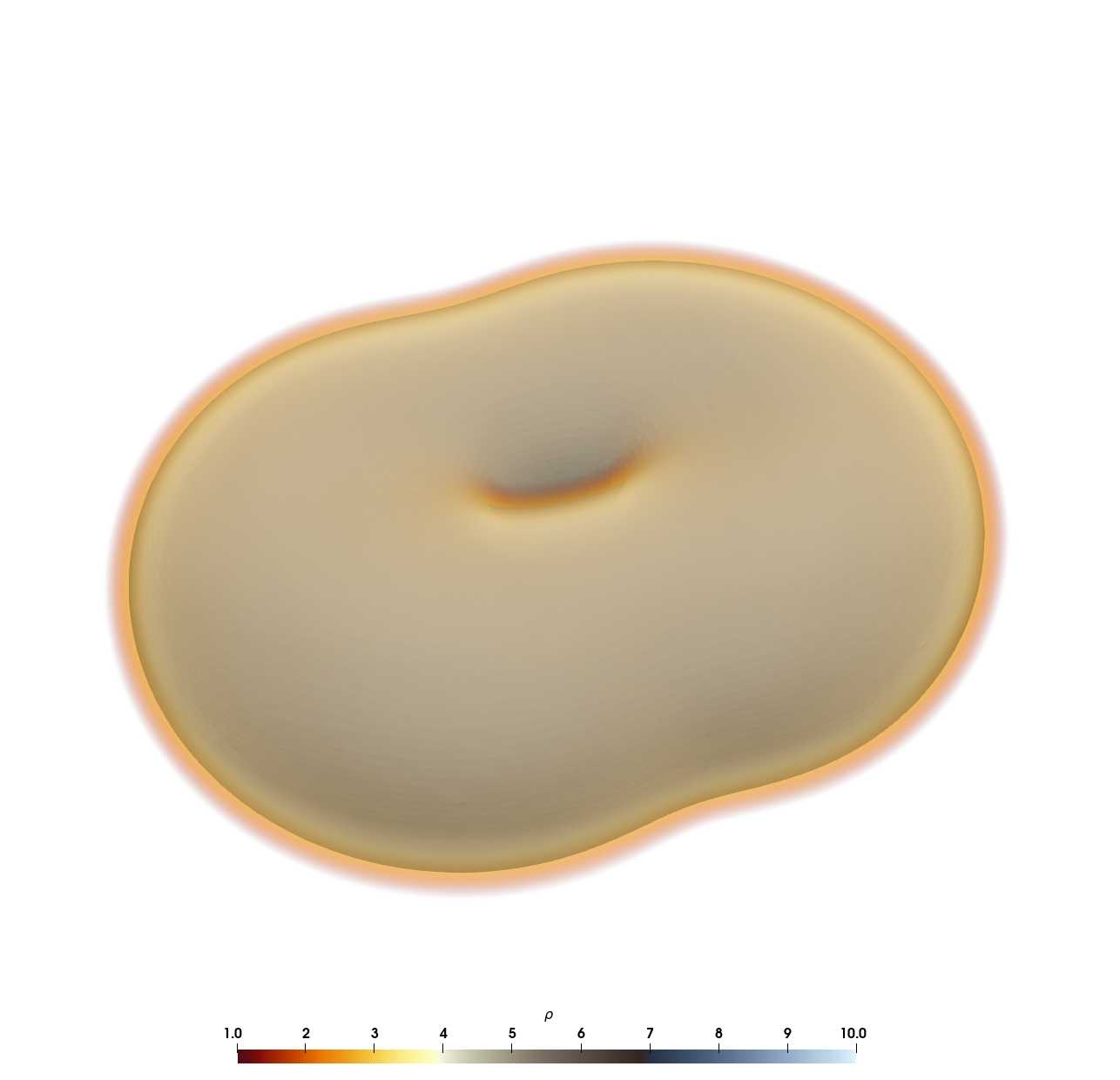}
    \end{subfigure}
    \begin{subfigure}{0.16\textwidth}
        \centering
        \includegraphics[width=\textwidth,trim=30mm 40mm 30mm 70mm, clip]{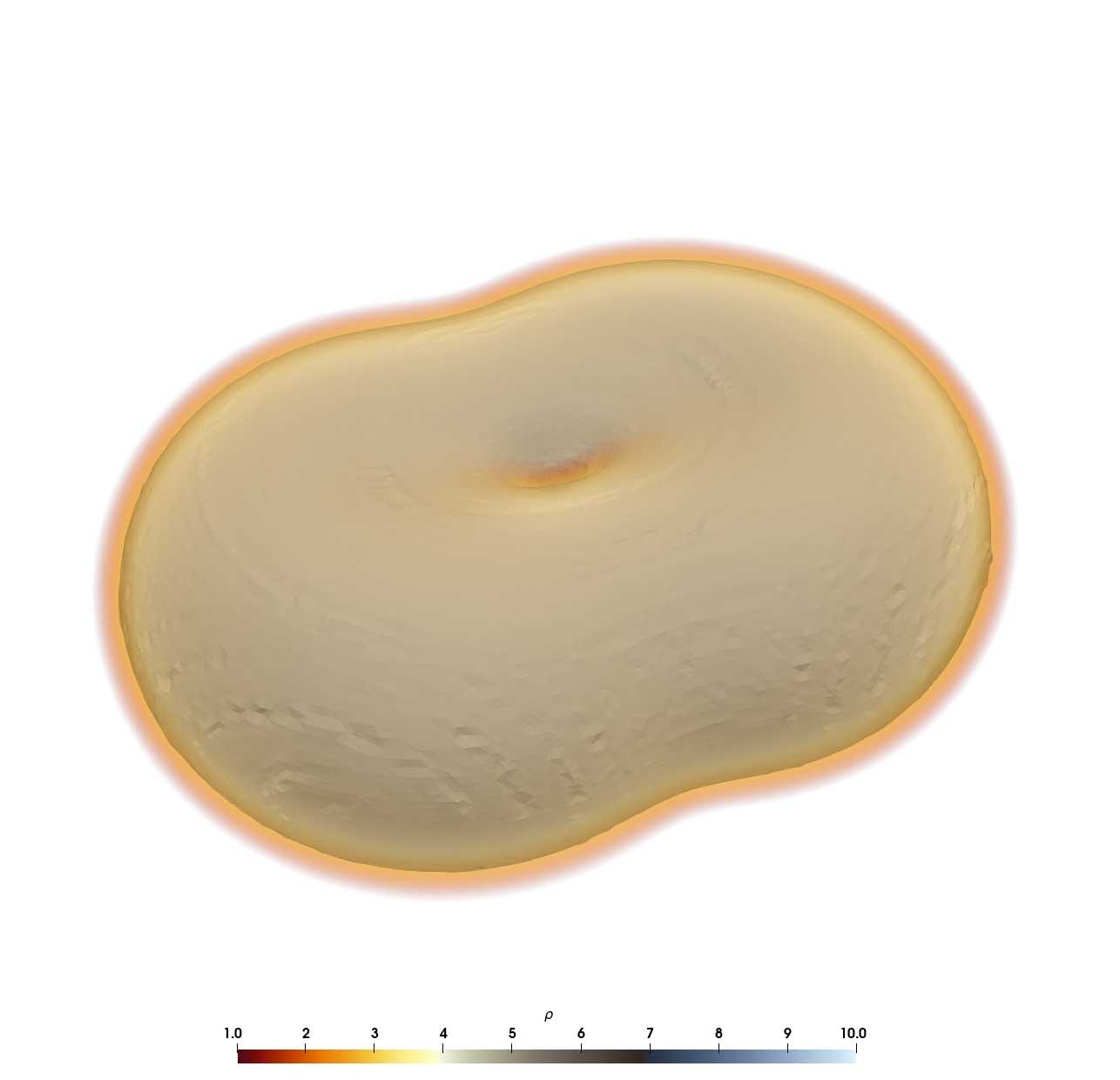}
    \end{subfigure}
    
    \begin{subfigure}{0.16\textwidth}
        \centering
        \includegraphics[width=\textwidth,trim=30mm 40mm 30mm 70mm, clip]{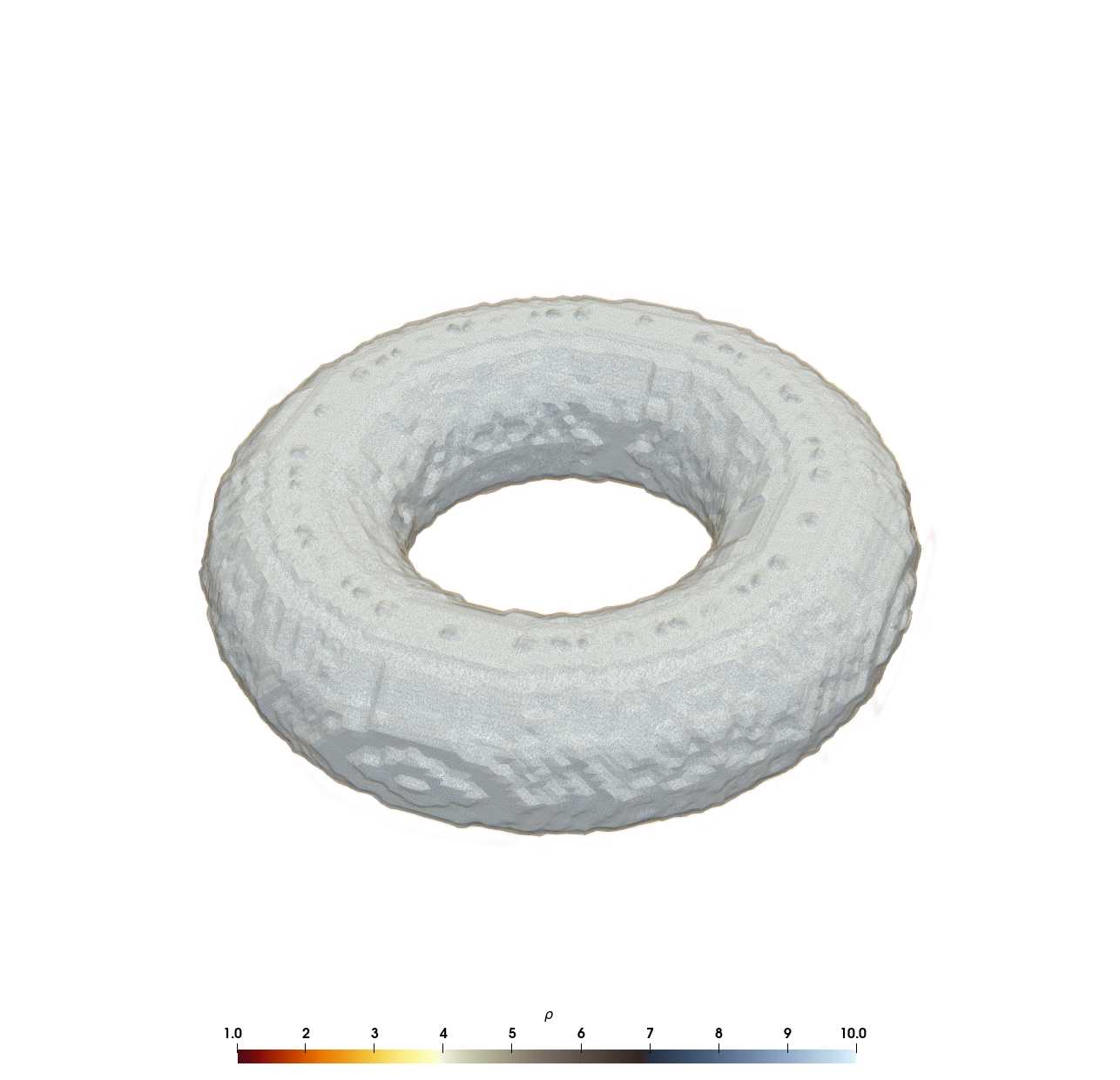}
    \end{subfigure}
    \begin{subfigure}{0.16\textwidth}
        \centering
        \includegraphics[width=\textwidth,trim=30mm 40mm 30mm 70mm, clip]{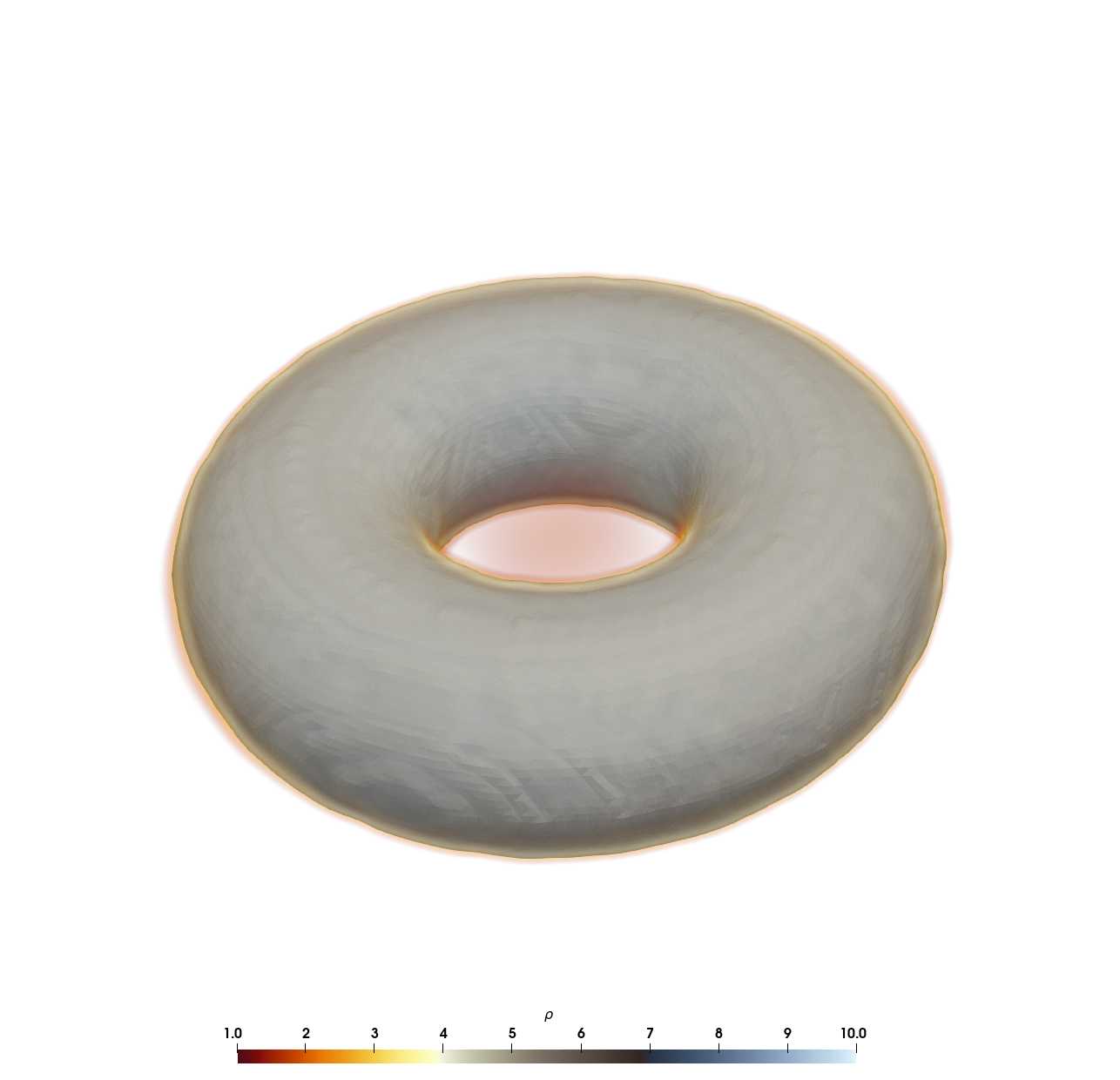}
    \end{subfigure}
    \begin{subfigure}{0.16\textwidth}
        \centering
        \includegraphics[width=\textwidth,trim=30mm 40mm 30mm 70mm, clip]{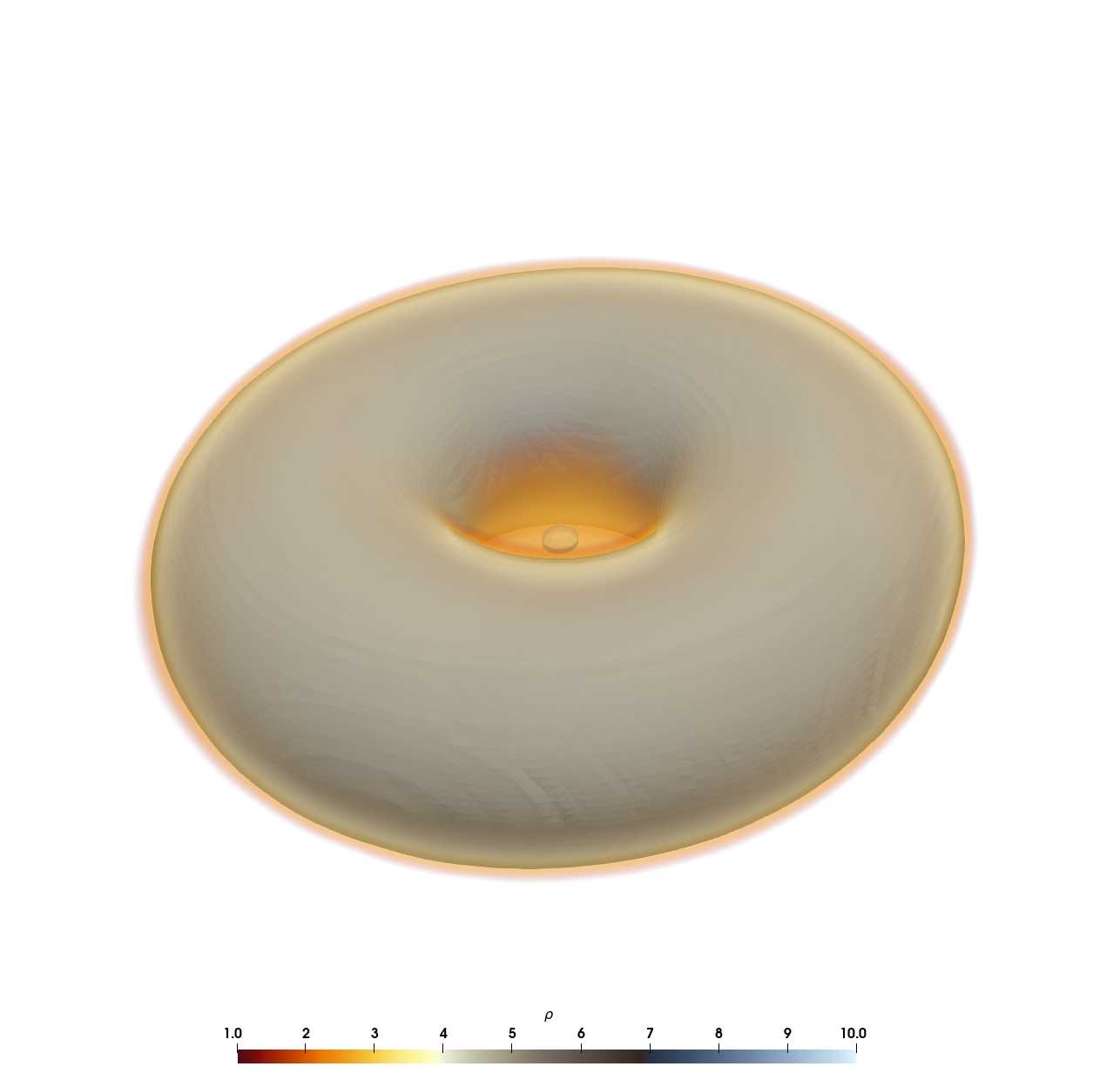}
    \end{subfigure}
    \begin{subfigure}{0.16\textwidth}
        \centering
        \includegraphics[width=\textwidth,trim=30mm 40mm 30mm 70mm, clip]{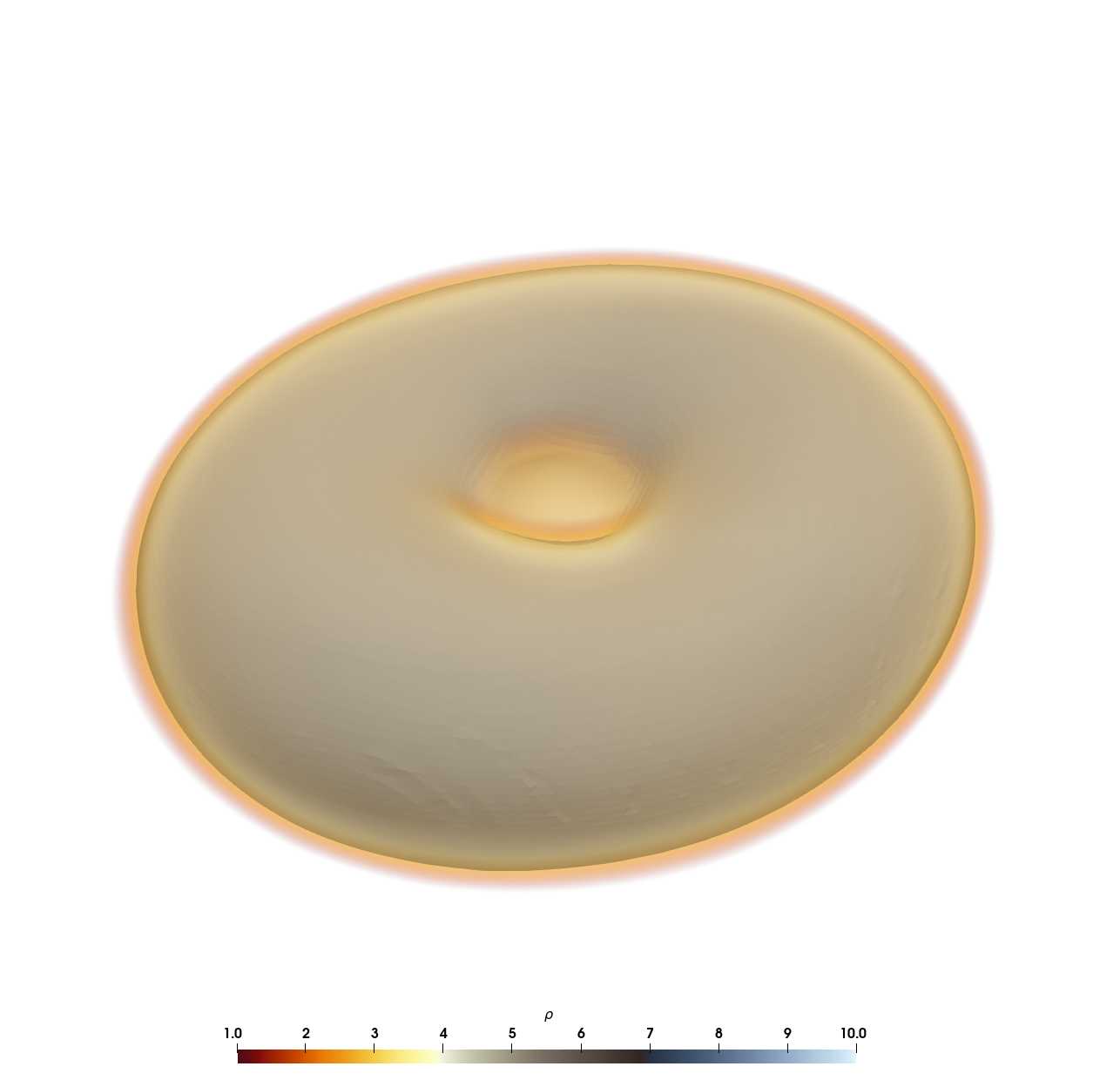}
    \end{subfigure}
    \begin{subfigure}{0.16\textwidth}
        \centering
        \includegraphics[width=\textwidth,trim=30mm 40mm 30mm 70mm, clip]{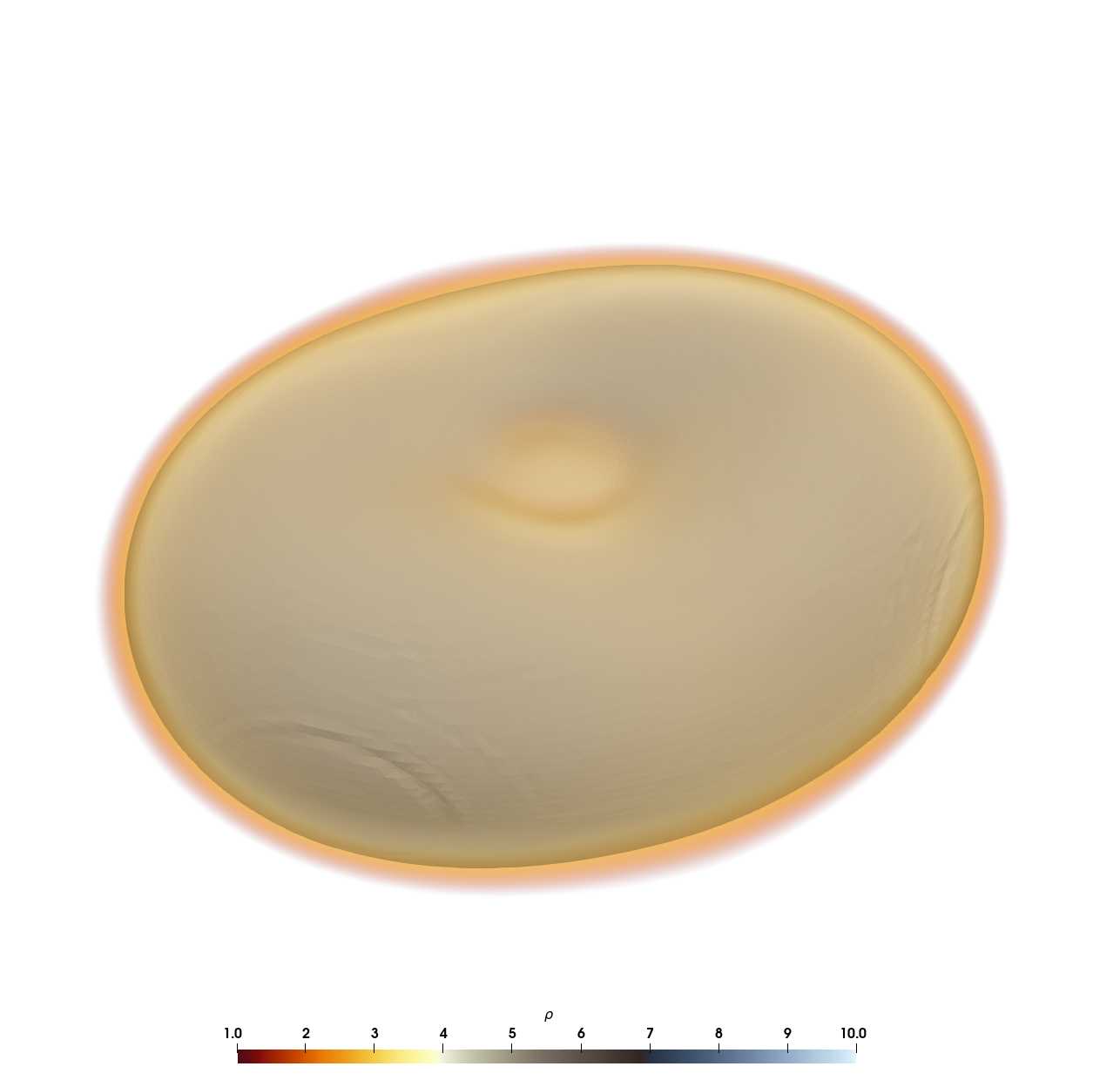}
    \end{subfigure}
    \begin{subfigure}{0.16\textwidth}
        \centering
        \includegraphics[width=\textwidth,trim=30mm 40mm 30mm 70mm, clip]{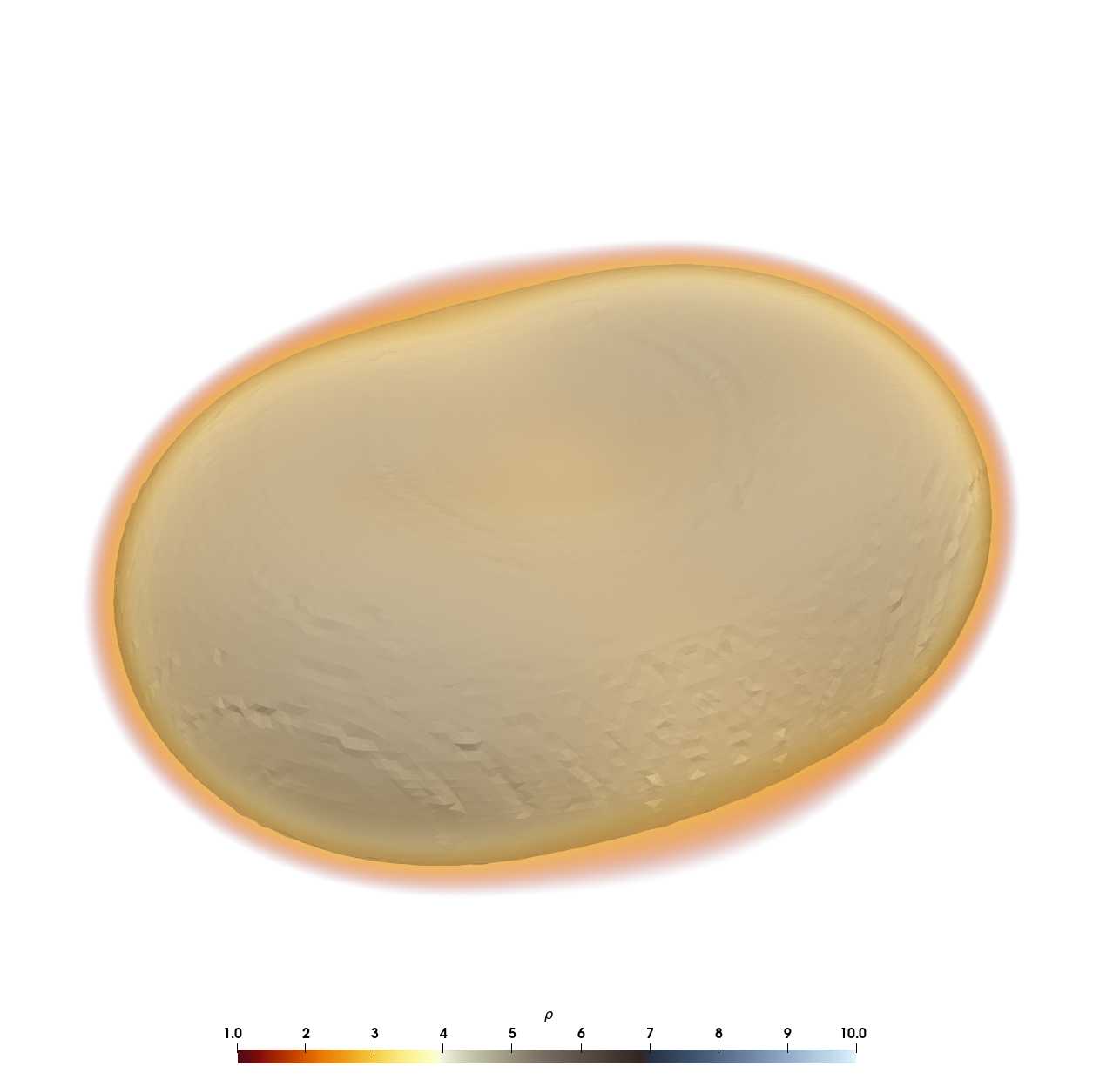}
    \end{subfigure}

    \begin{subfigure}{0.16\textwidth}
        \centering
        \includegraphics[width=\textwidth,trim=30mm 40mm 30mm 70mm, clip]{figures/reaction/torusDoubleTorus/pdhg0.0000..jpg}
        \caption*{$t=0$}
    \end{subfigure}
    \begin{subfigure}{0.16\textwidth}
        \centering
        \includegraphics[width=\textwidth,trim=30mm 40mm 30mm 70mm, clip]{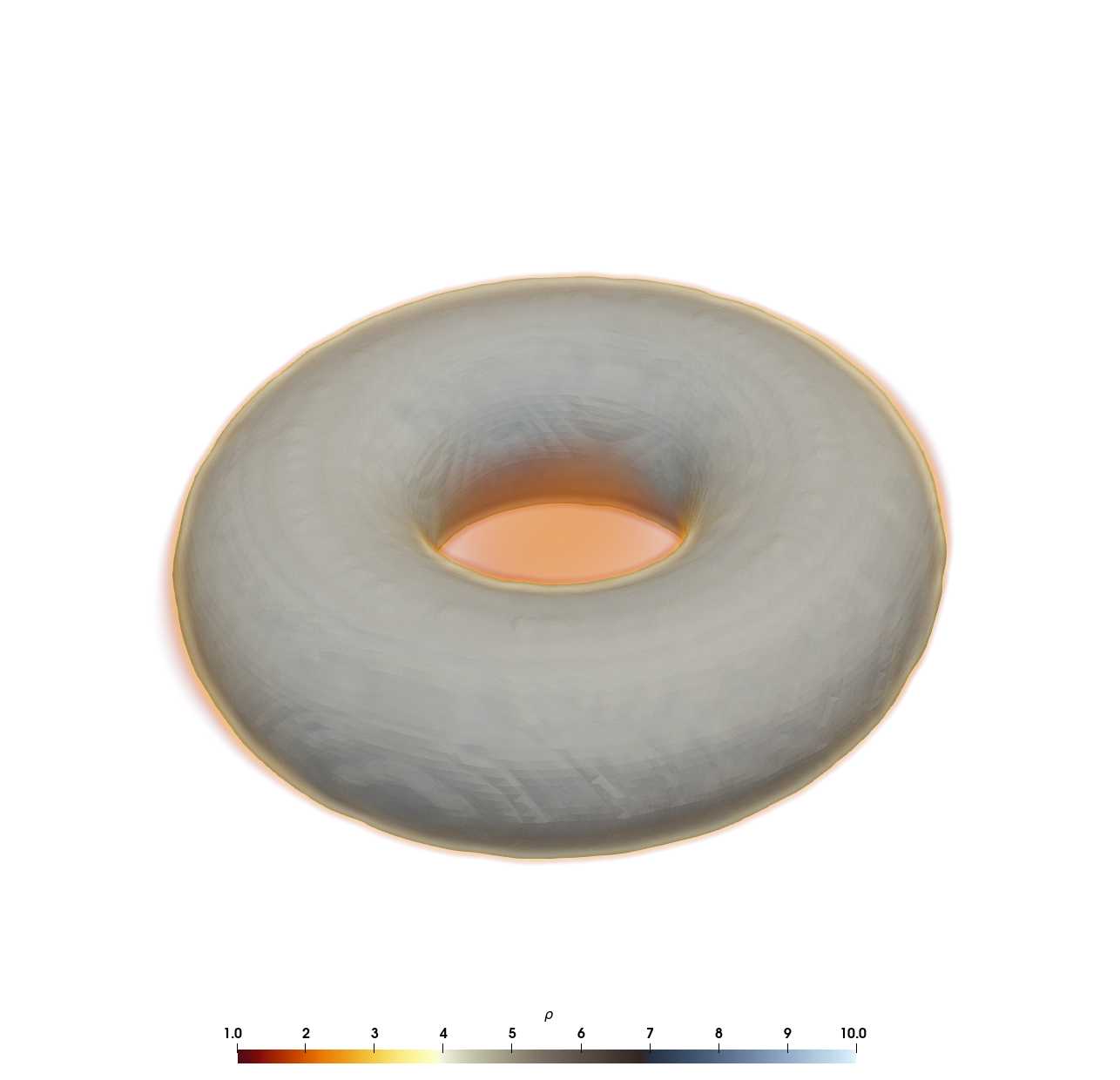}
        \caption*{$t=0.2$}
    \end{subfigure}
    \begin{subfigure}{0.16\textwidth}
        \centering
        \includegraphics[width=\textwidth,trim=30mm 40mm 30mm 70mm, clip]{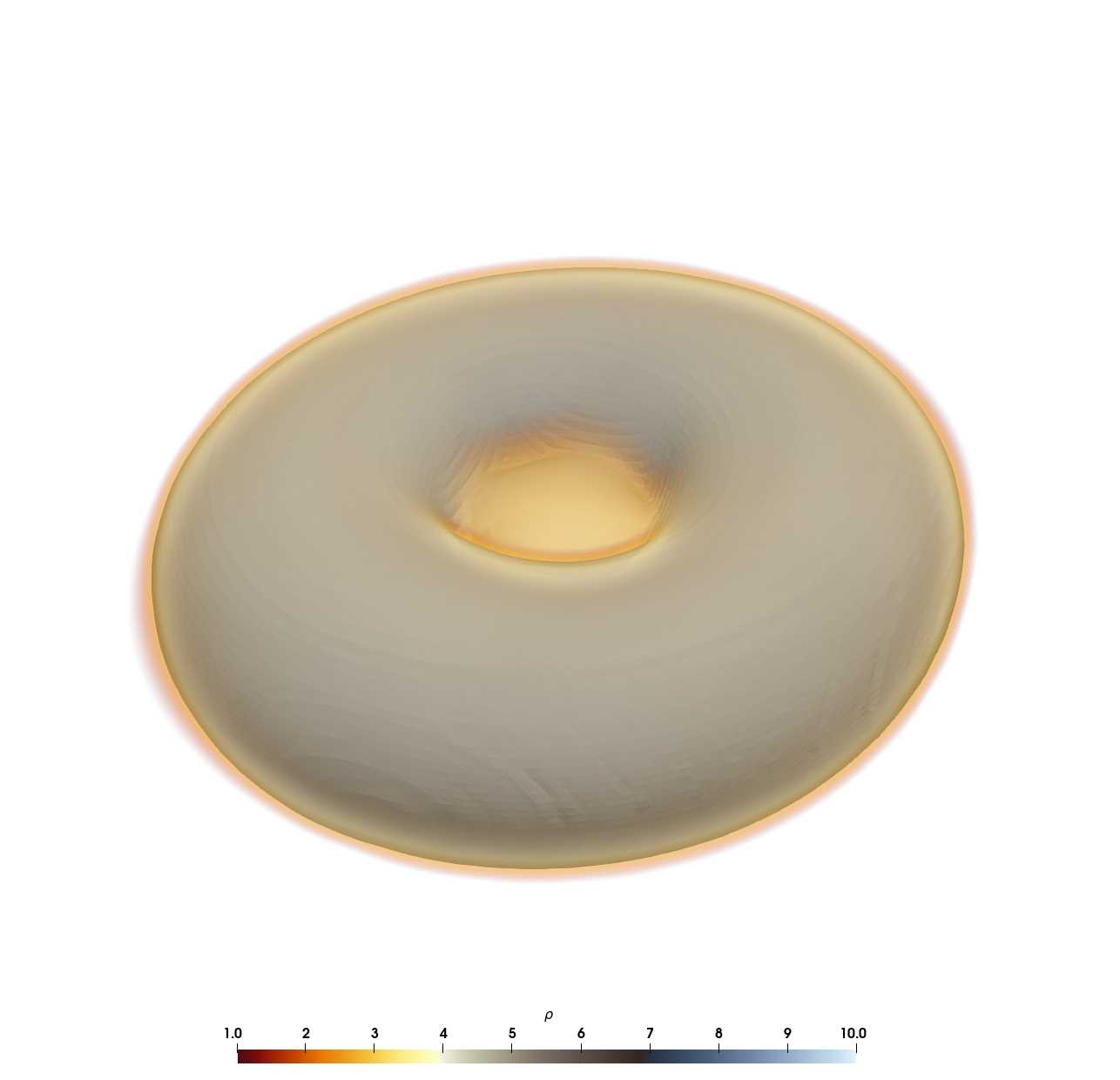}
        \caption*{$t=0.4$}
    \end{subfigure}
    \begin{subfigure}{0.16\textwidth}
        \centering
        \includegraphics[width=\textwidth,trim=30mm 40mm 30mm 70mm, clip]{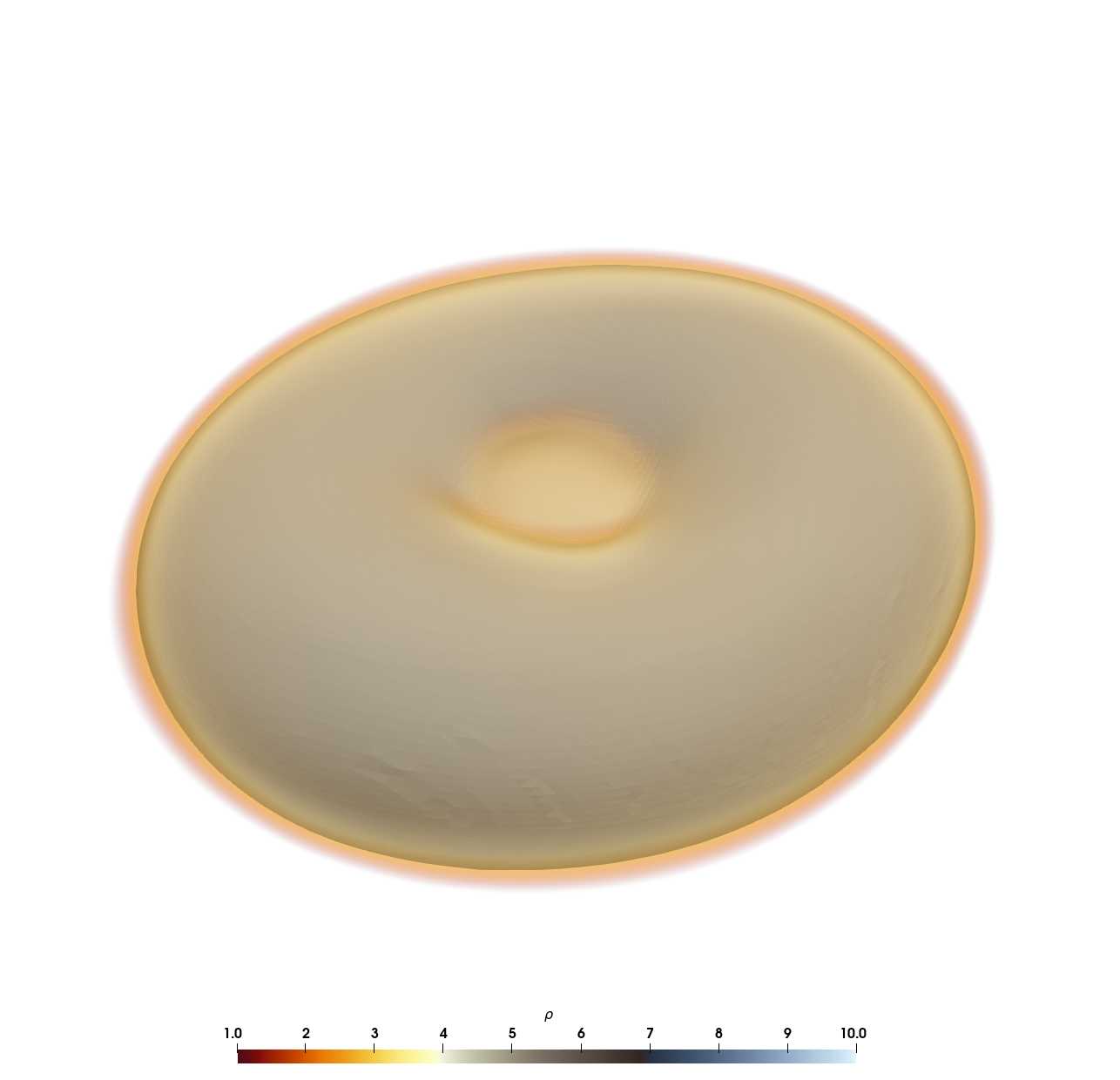}
        \caption*{$t=0.6$}
    \end{subfigure}
    \begin{subfigure}{0.16\textwidth}
        \centering
        \includegraphics[width=\textwidth,trim=30mm 40mm 30mm 70mm, clip]{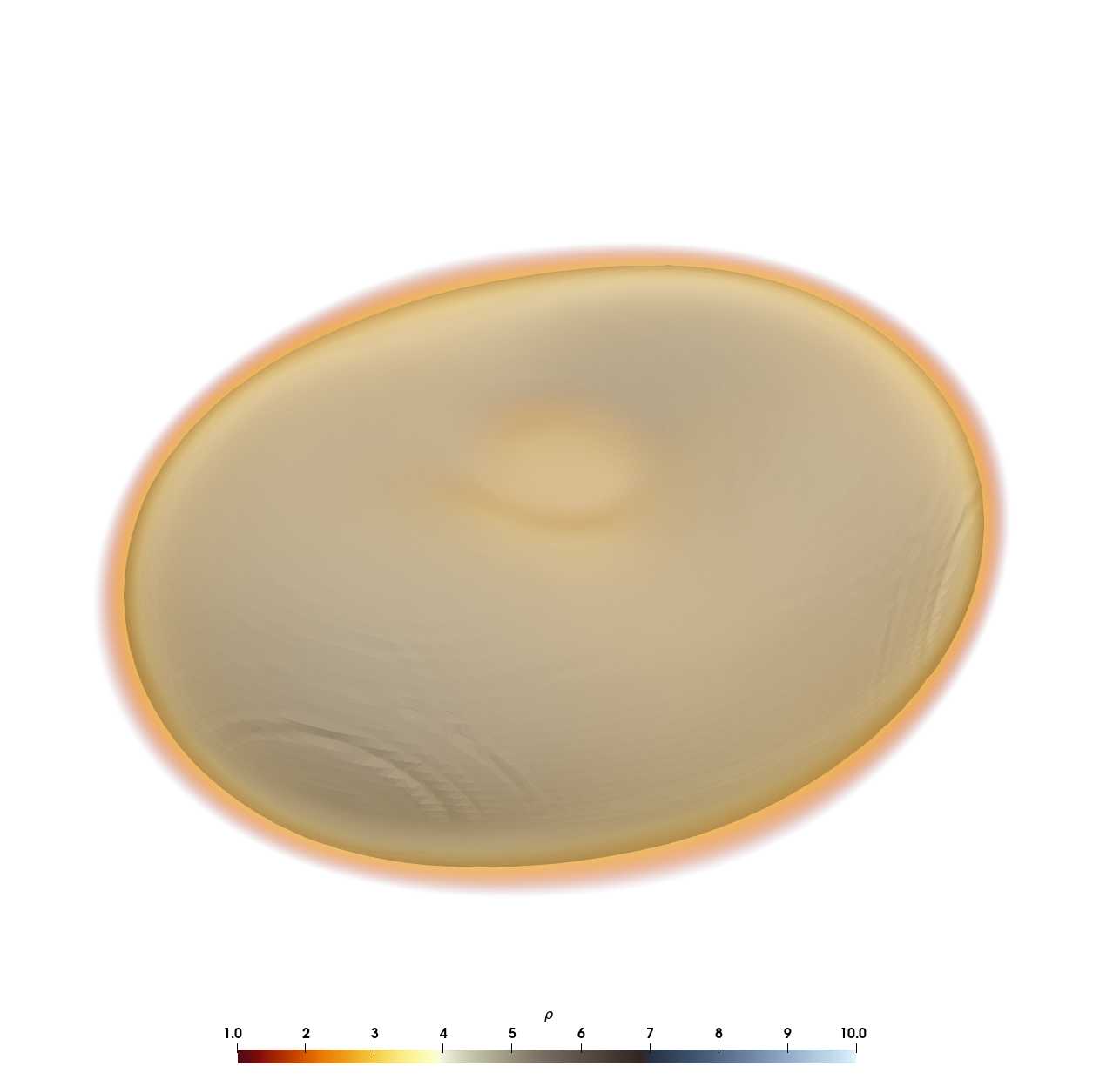}
        \caption*{$t=0.8$}
    \end{subfigure}
    \begin{subfigure}{0.16\textwidth}
        \centering
        \includegraphics[width=\textwidth,trim=30mm 40mm 30mm 70mm, clip]{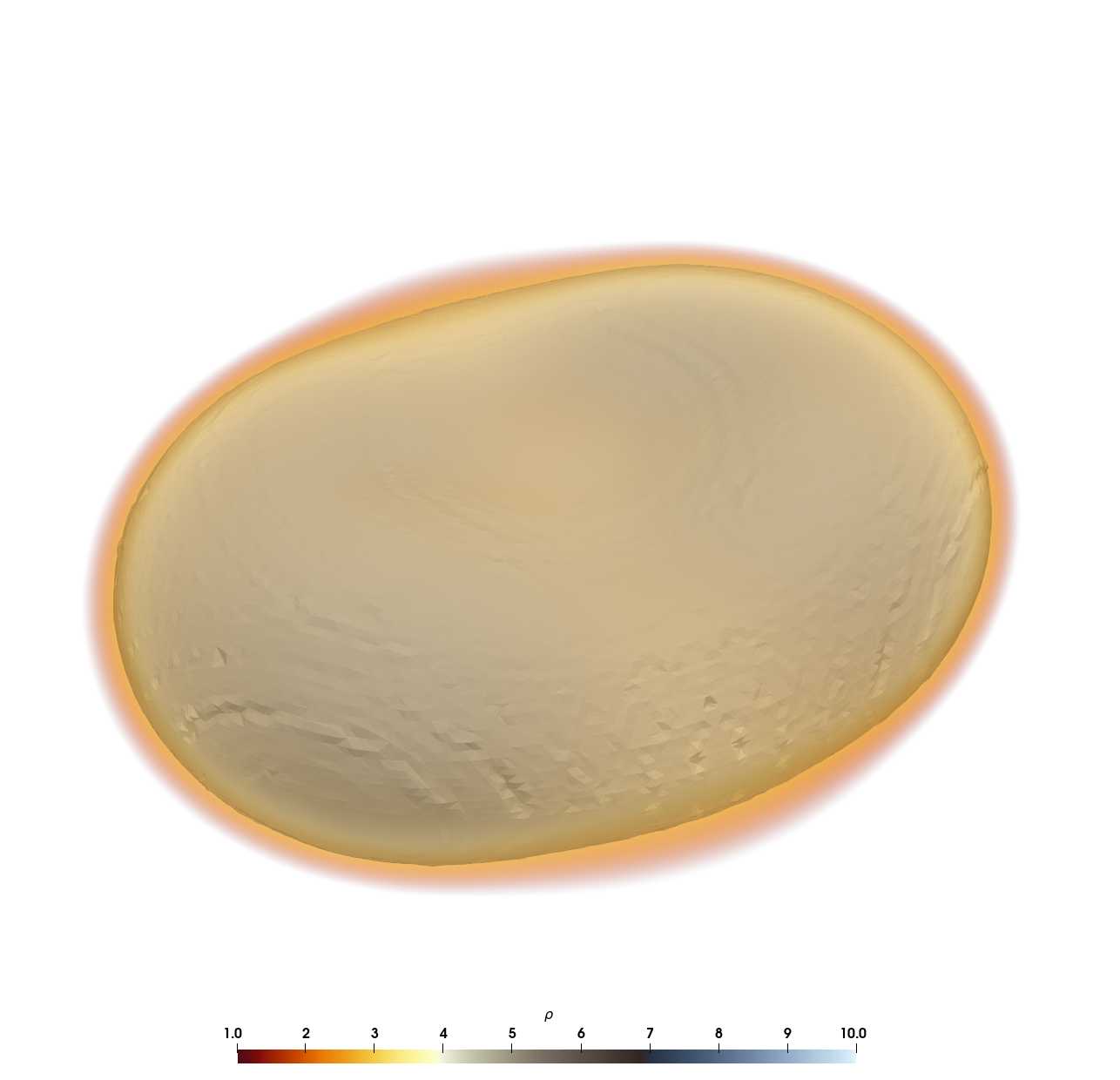}
        \caption*{$t=1.0$}
    \end{subfigure}
\subcaption{$\rho_1$}
\label{f:td1}
\end{minipage}
\begin{minipage}[b]{\textwidth}

    \begin{minipage}[b]{\textwidth}
    \hfill
        \begin{subfigure}{\textwidth}
            \centering
            \includegraphics[width=\textwidth,trim=0mm 0mm 0mm 400mm, clip]{figures/reaction/doubleTorusBunny/pdhg0.0000..jpg}
        \end{subfigure}
    \end{minipage}

    \begin{subfigure}{0.16\textwidth}
        \centering
        \includegraphics[width=\textwidth,trim=30mm 40mm 30mm 70mm, clip]{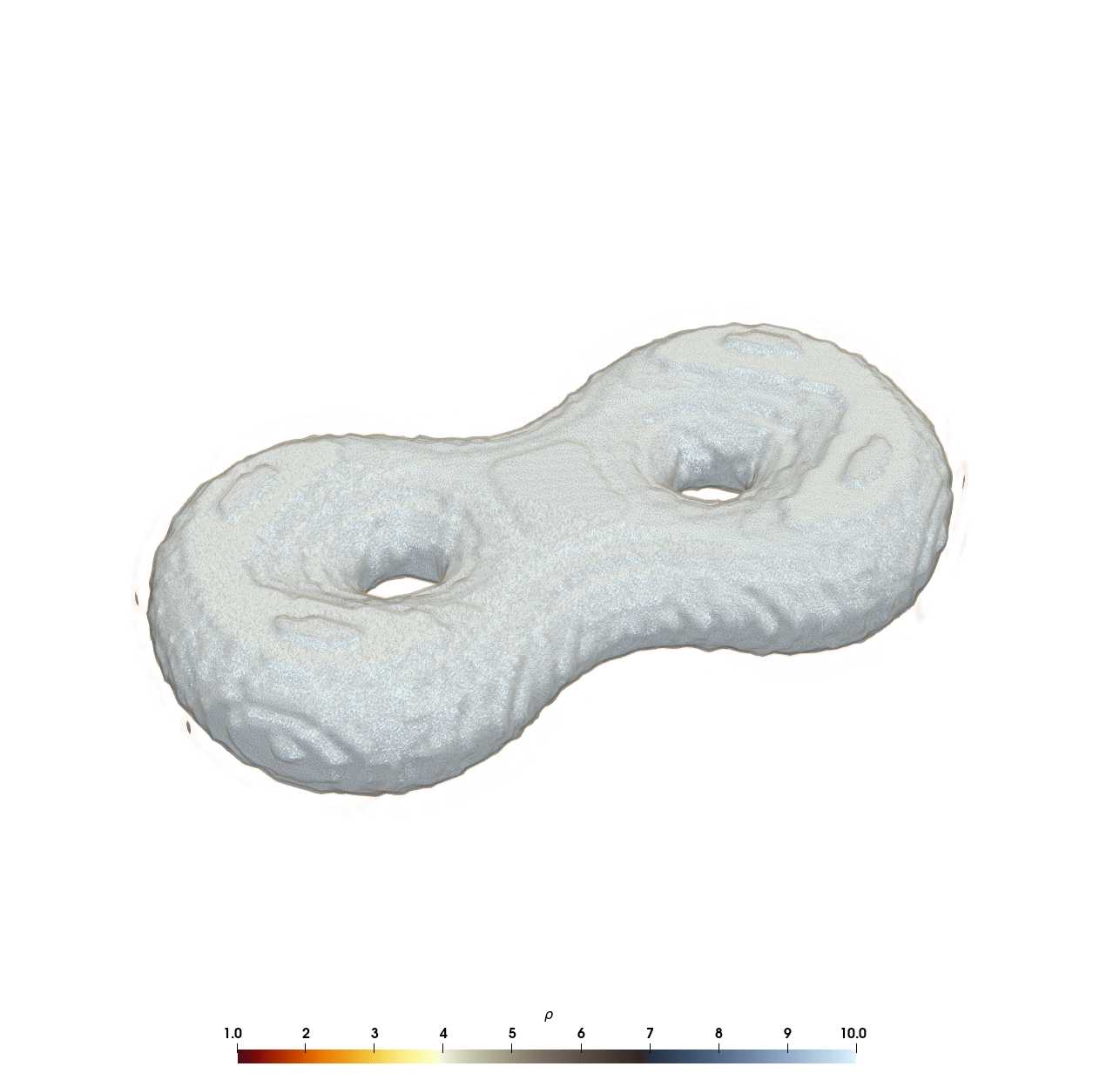}
    \end{subfigure}
    \begin{subfigure}{0.16\textwidth}
        \centering
        \includegraphics[width=\textwidth,trim=30mm 40mm 30mm 70mm, clip]{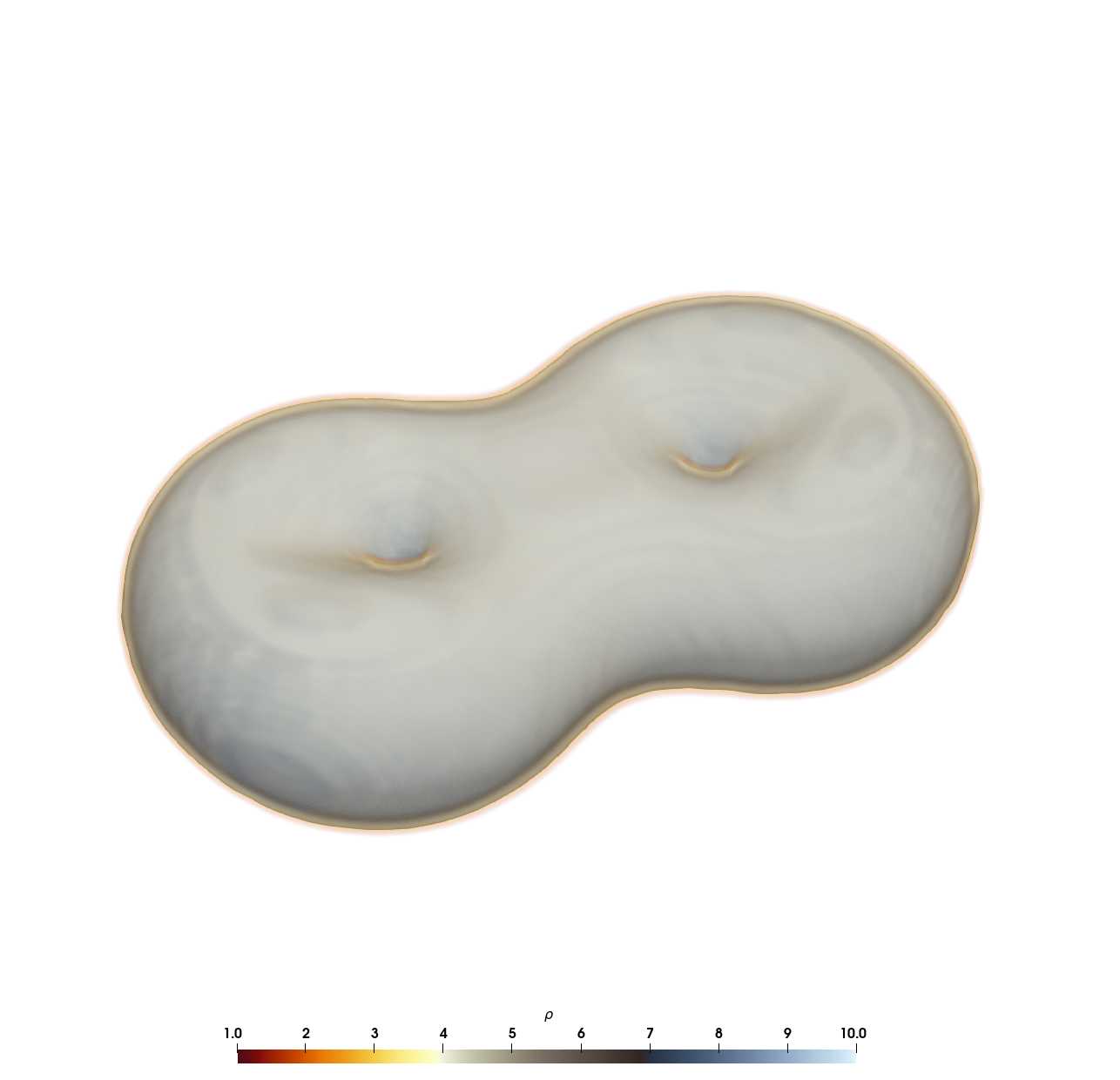}
    \end{subfigure}
    \begin{subfigure}{0.16\textwidth}
        \centering
        \includegraphics[width=\textwidth,trim=30mm 40mm 30mm 70mm, clip]{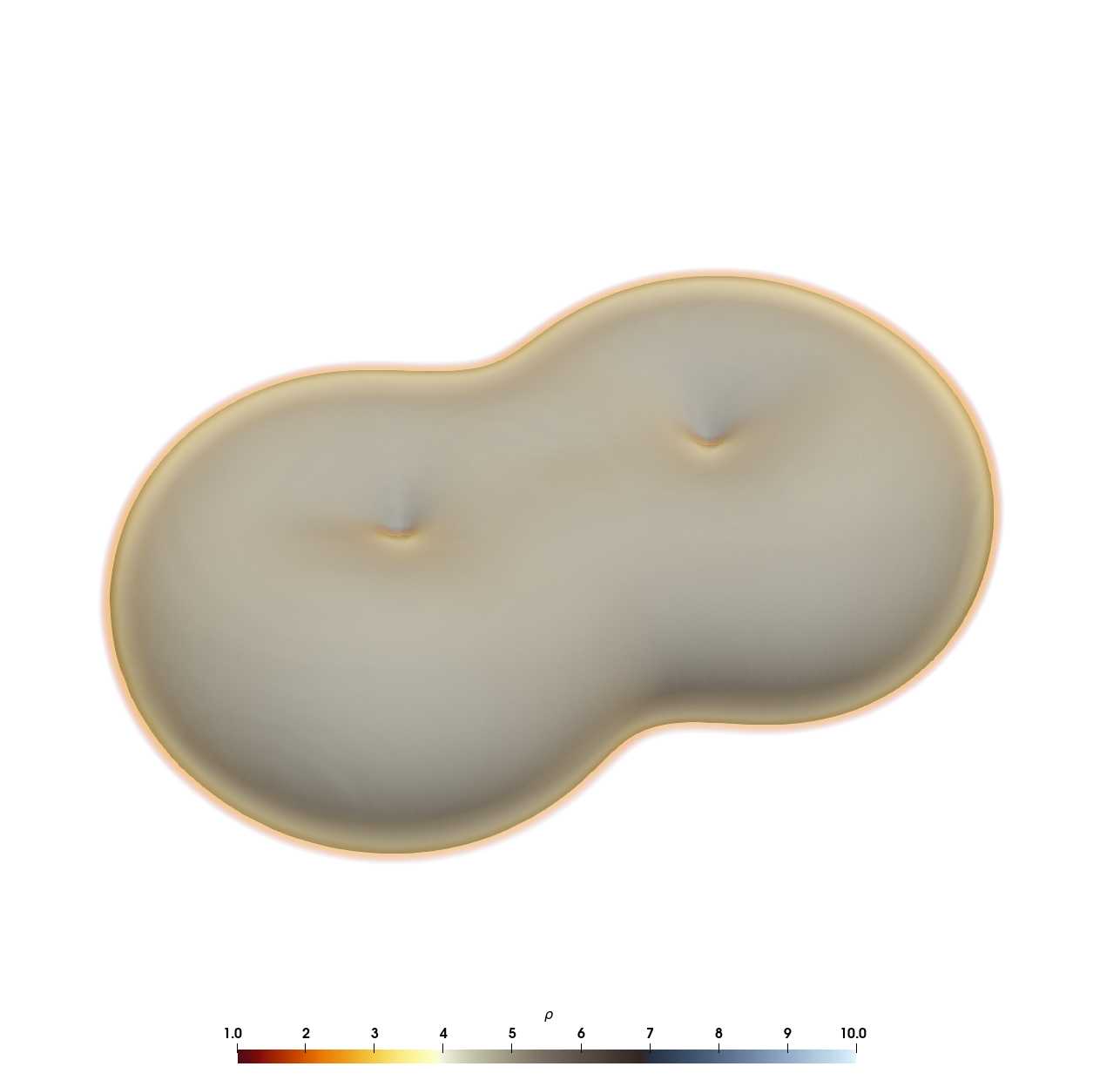}
    \end{subfigure}
    \begin{subfigure}{0.16\textwidth}
        \centering
        \includegraphics[width=\textwidth,trim=30mm 40mm 30mm 70mm, clip]{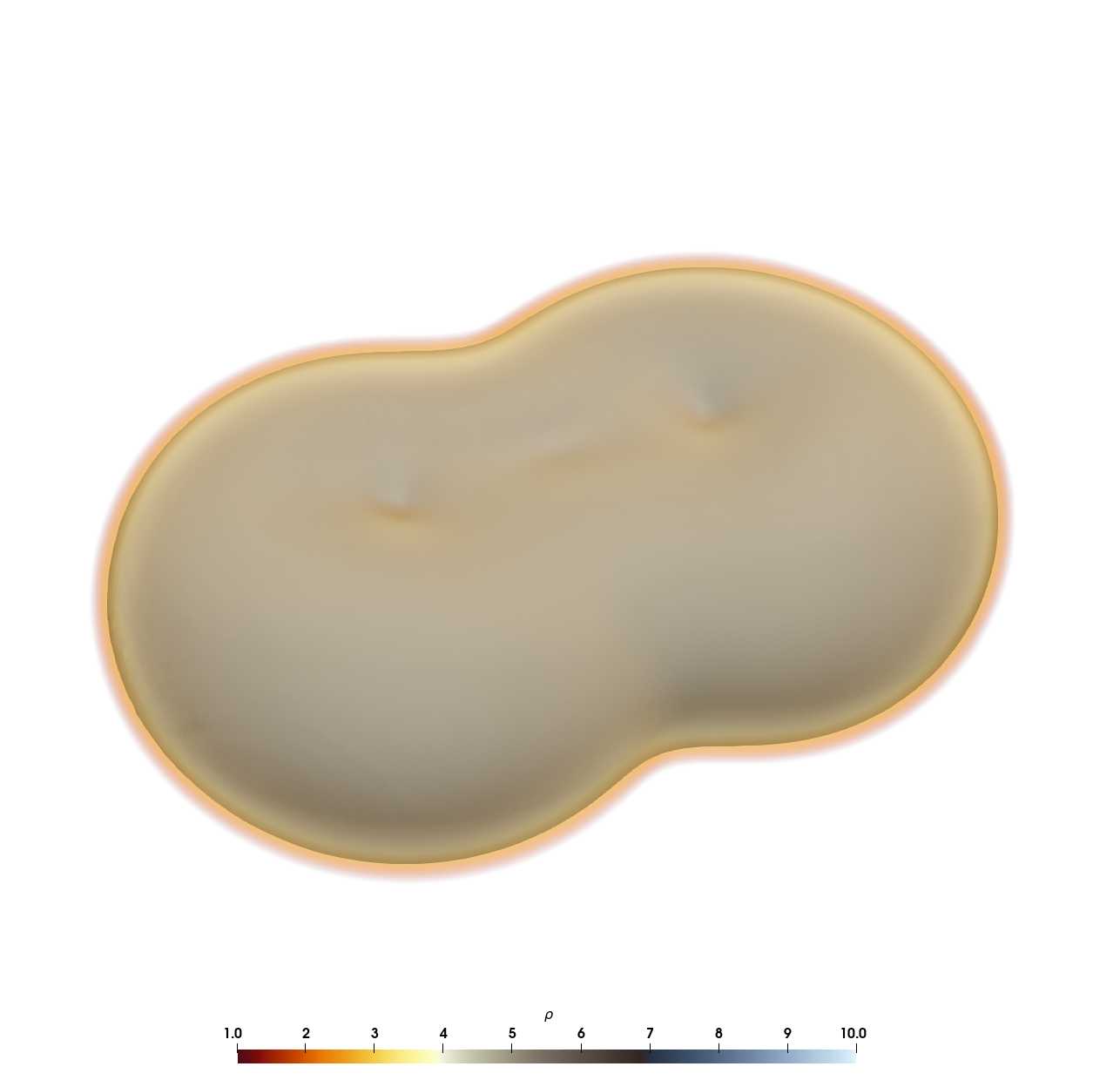}
    \end{subfigure}
    \begin{subfigure}{0.16\textwidth}
        \centering
        \includegraphics[width=\textwidth,trim=30mm 40mm 30mm 70mm, clip]{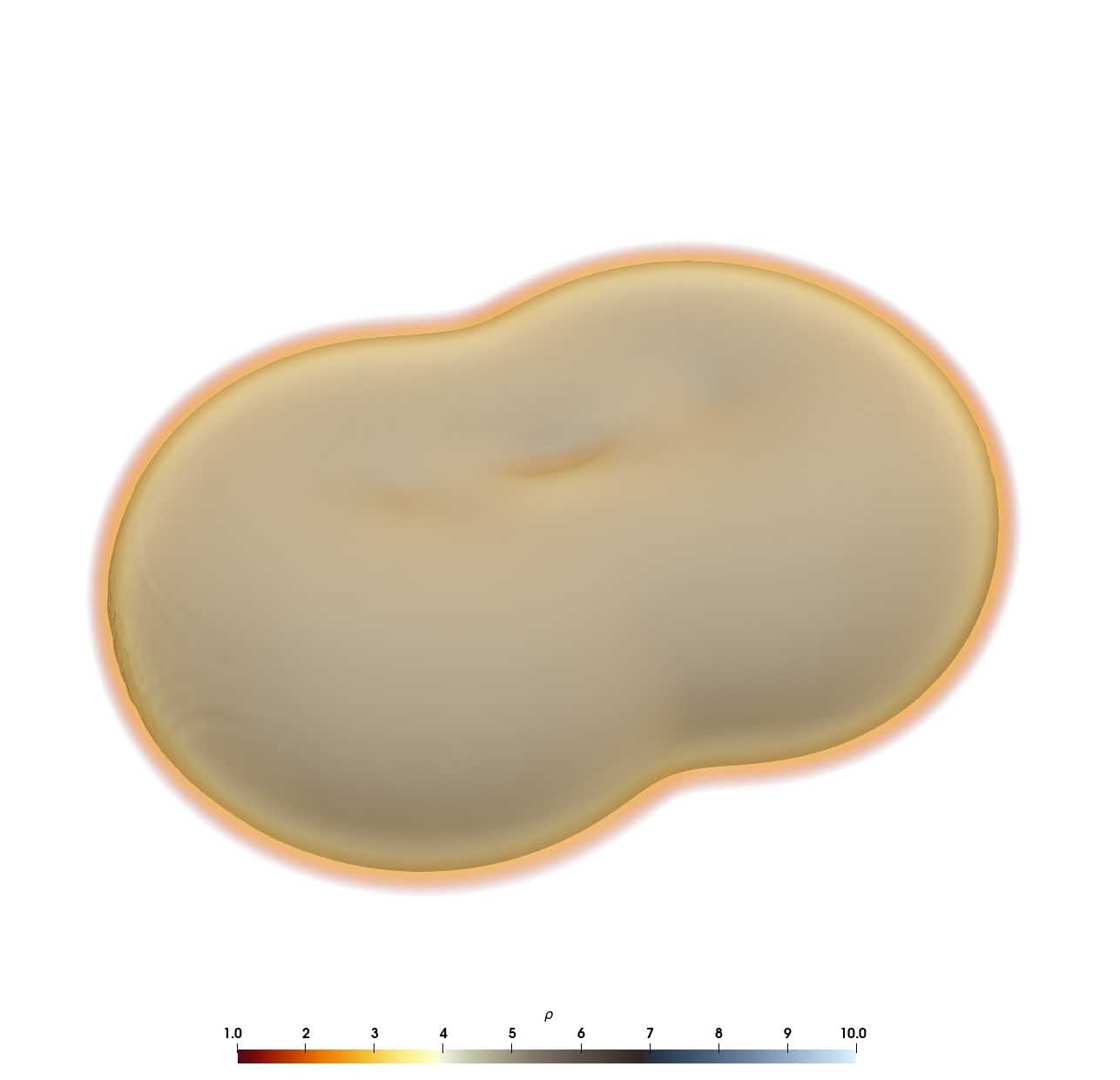}
    \end{subfigure}
    \begin{subfigure}{0.16\textwidth}
        \centering
        \includegraphics[width=\textwidth,trim=30mm 40mm 30mm 70mm, clip]{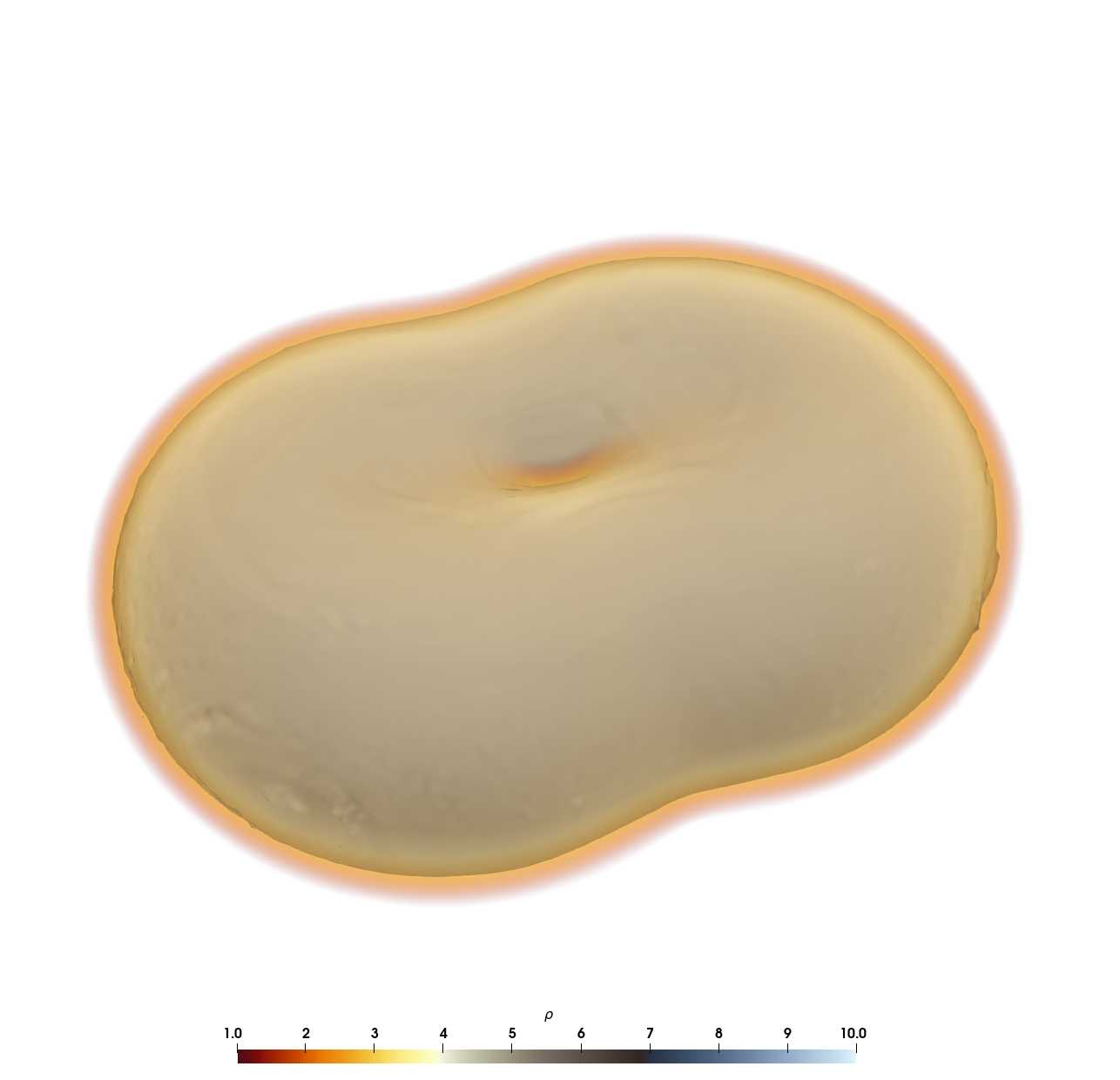}
    \end{subfigure}
    
    \begin{subfigure}{0.16\textwidth}
        \centering
        \includegraphics[width=\textwidth,trim=30mm 40mm 30mm 70mm, clip]{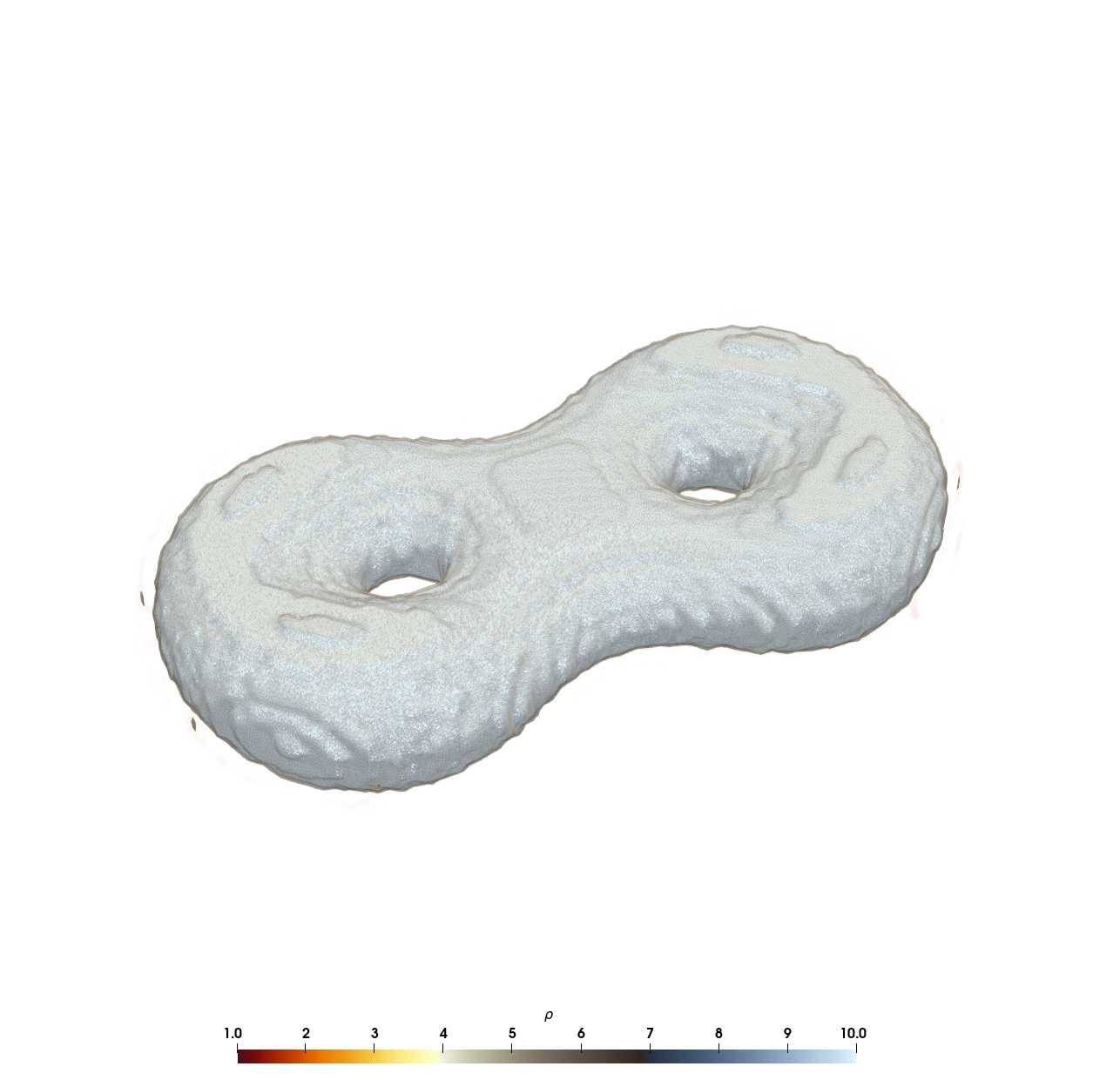}
    \end{subfigure}
    \begin{subfigure}{0.16\textwidth}
        \centering
        \includegraphics[width=\textwidth,trim=30mm 40mm 30mm 70mm, clip]{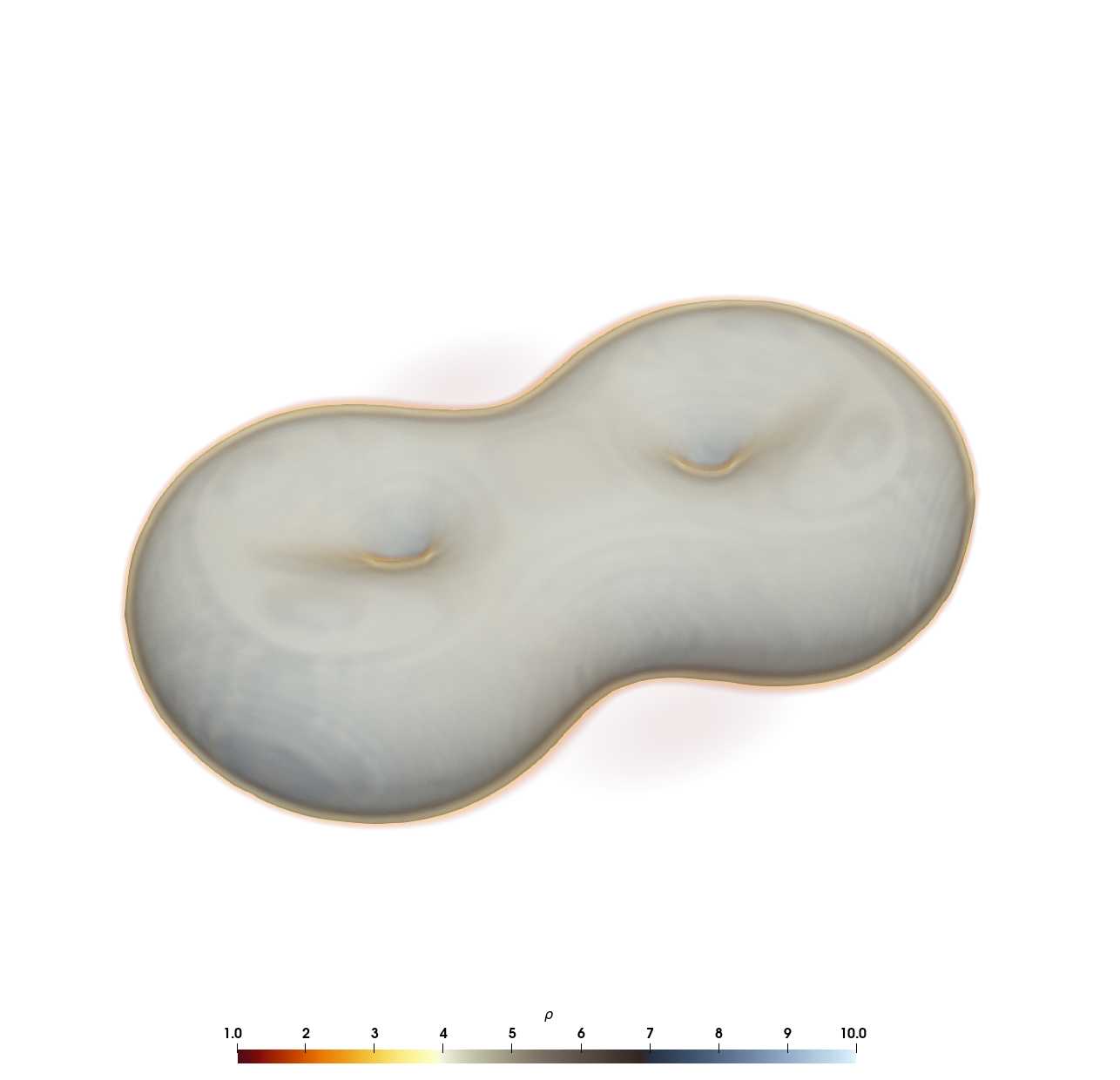}
    \end{subfigure}
    \begin{subfigure}{0.16\textwidth}
        \centering
        \includegraphics[width=\textwidth,trim=30mm 40mm 30mm 70mm, clip]{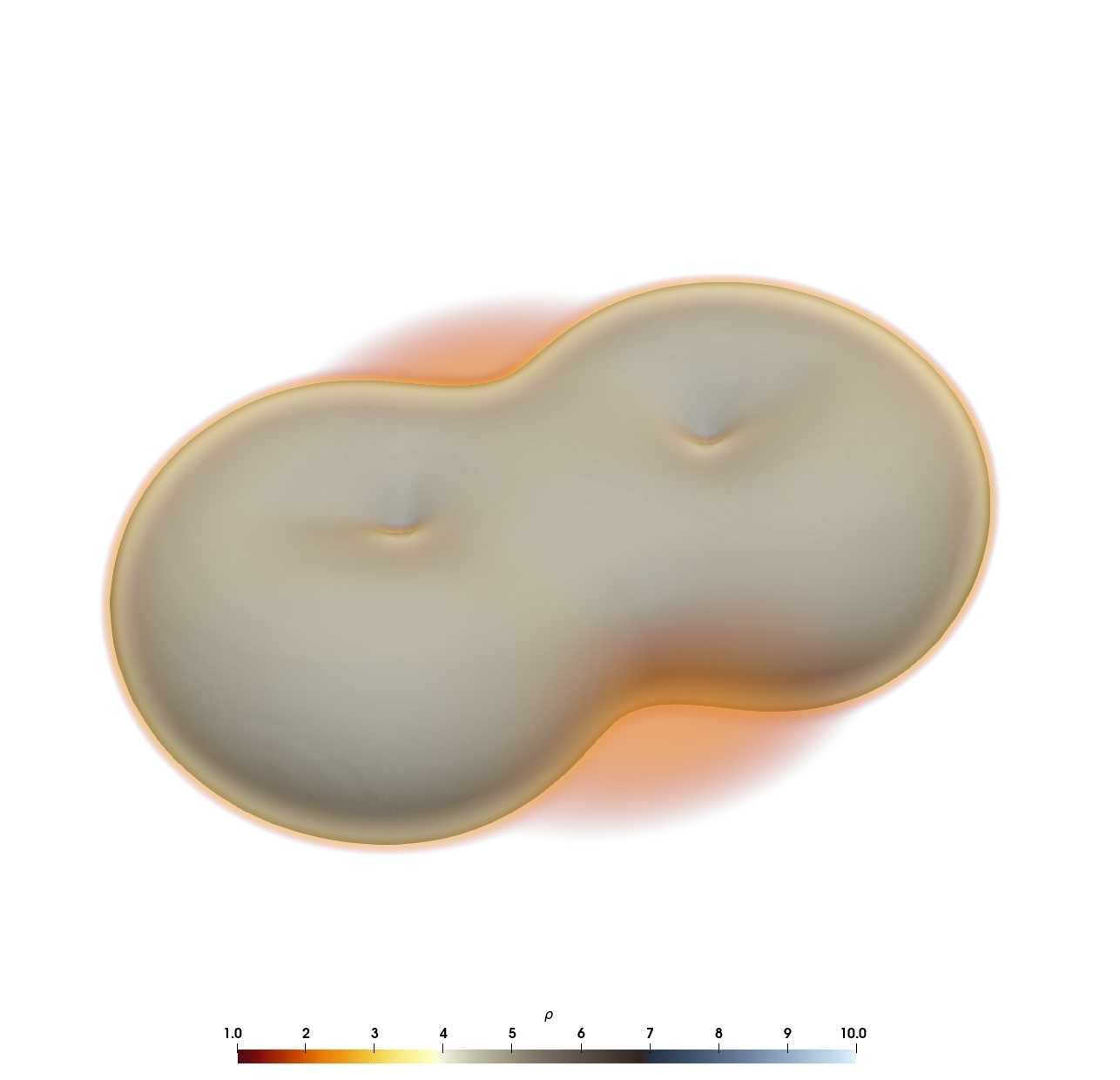}
    \end{subfigure}
    \begin{subfigure}{0.16\textwidth}
        \centering
        \includegraphics[width=\textwidth,trim=30mm 40mm 30mm 70mm, clip]{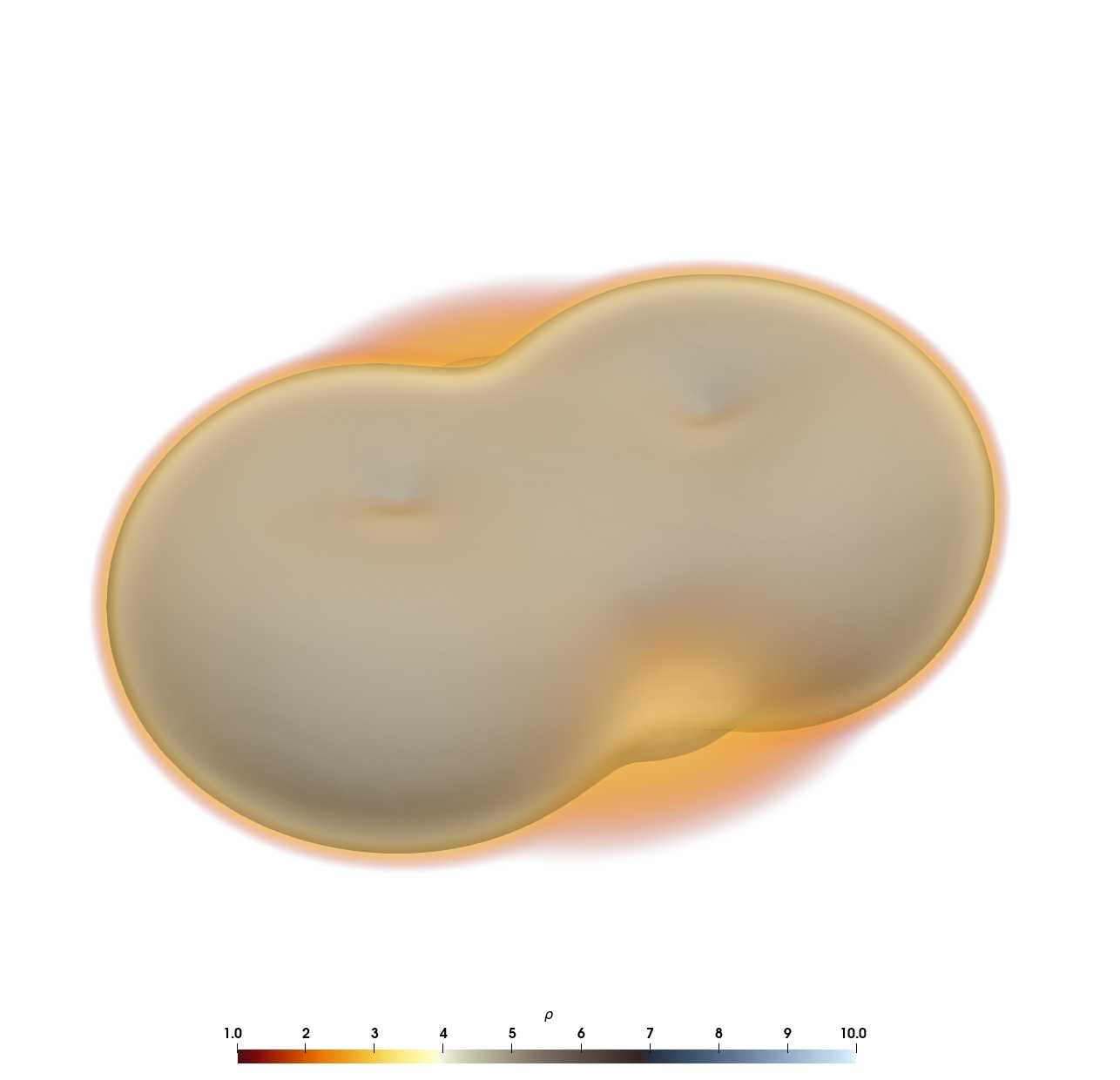}
    \end{subfigure}
    \begin{subfigure}{0.16\textwidth}
        \centering
        \includegraphics[width=\textwidth,trim=30mm 40mm 30mm 70mm, clip]{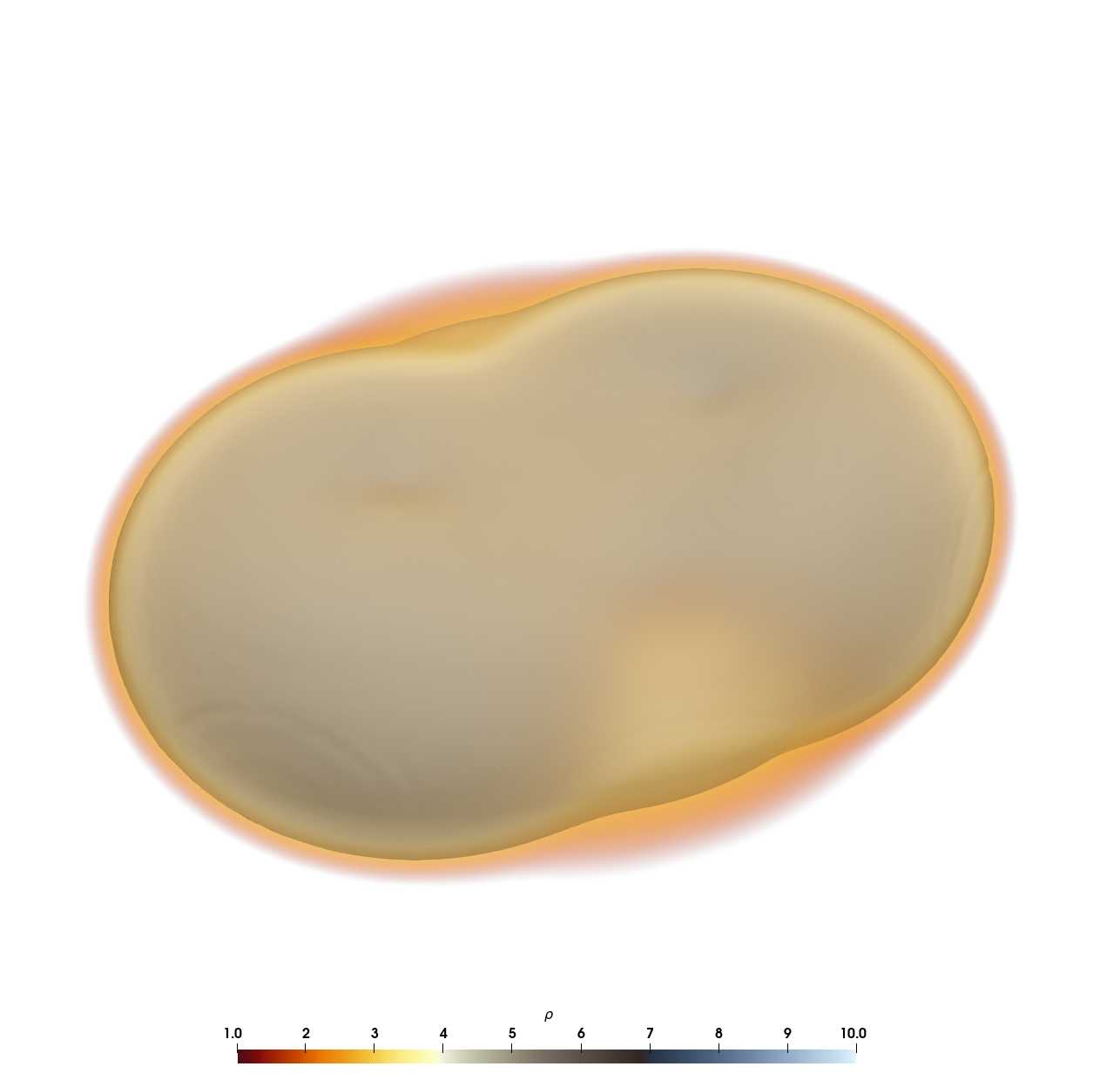}
    \end{subfigure}
    \begin{subfigure}{0.16\textwidth}
        \centering
        \includegraphics[width=\textwidth,trim=30mm 40mm 30mm 70mm, clip]{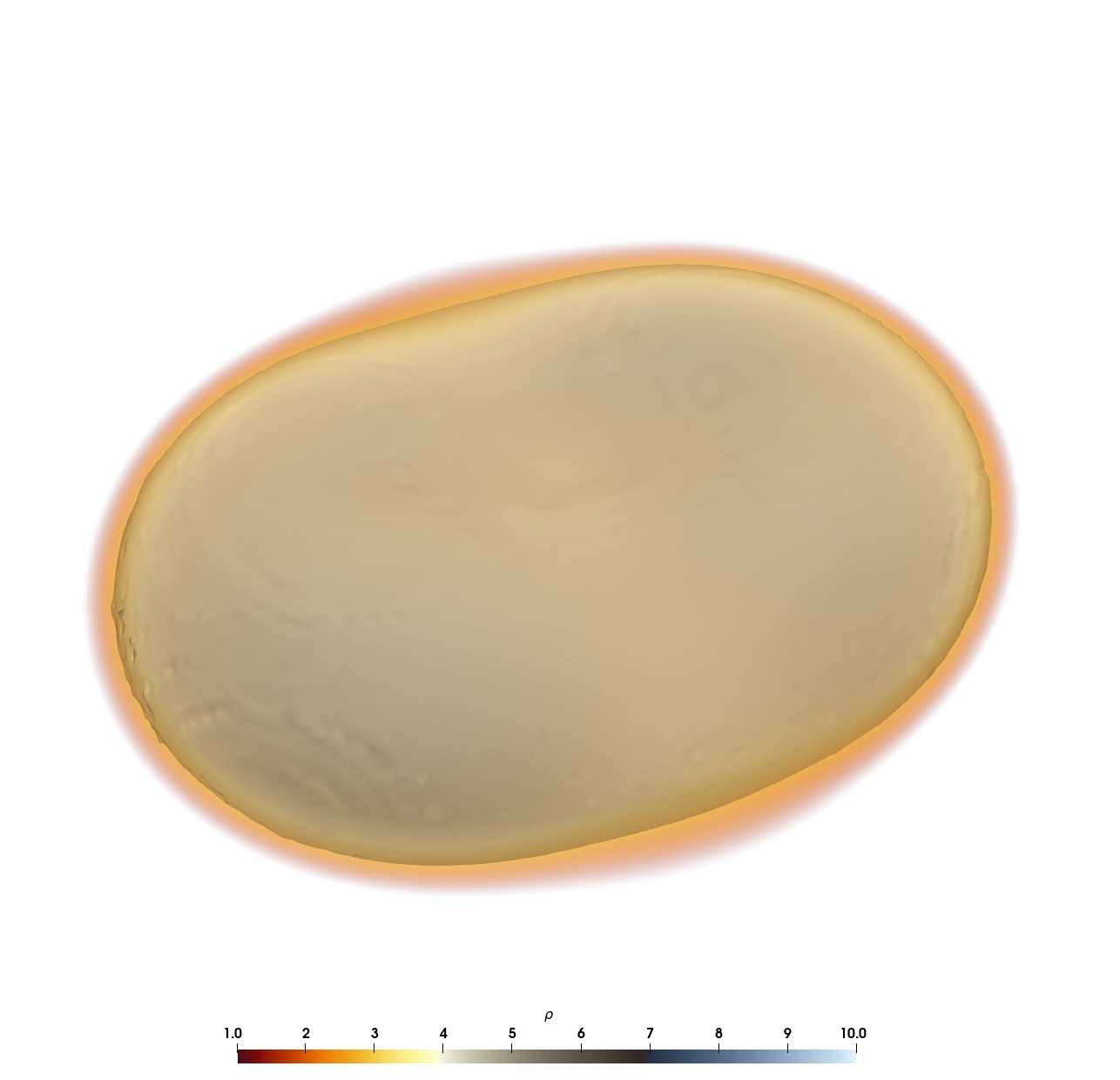}
    \end{subfigure}

    \begin{subfigure}{0.16\textwidth}
        \centering
        \includegraphics[width=\textwidth,trim=30mm 40mm 30mm 70mm, clip]{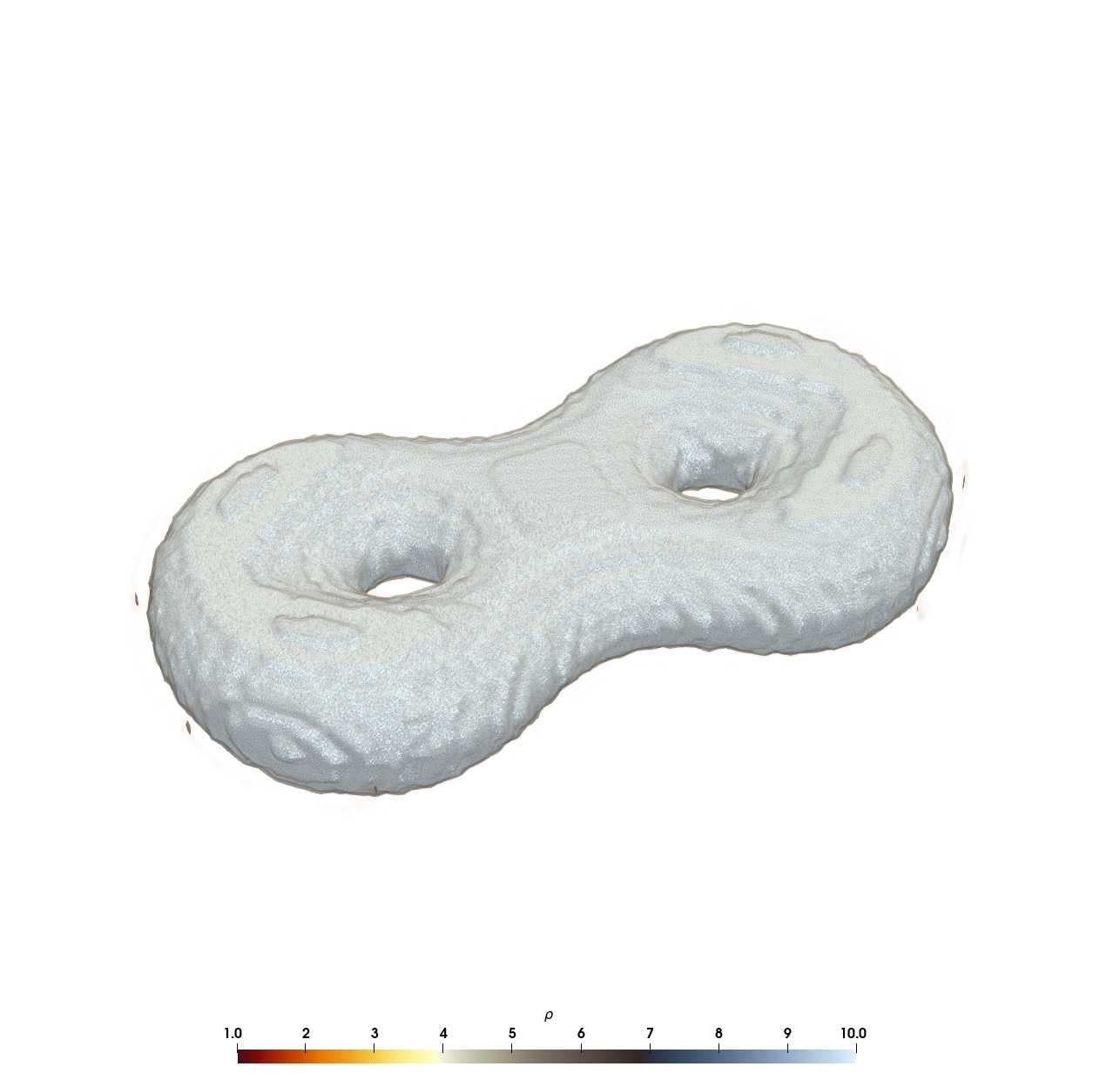}
        \caption*{$t=0$}
    \end{subfigure}
    \begin{subfigure}{0.16\textwidth}
        \centering
        \includegraphics[width=\textwidth,trim=30mm 40mm 30mm 70mm, clip]{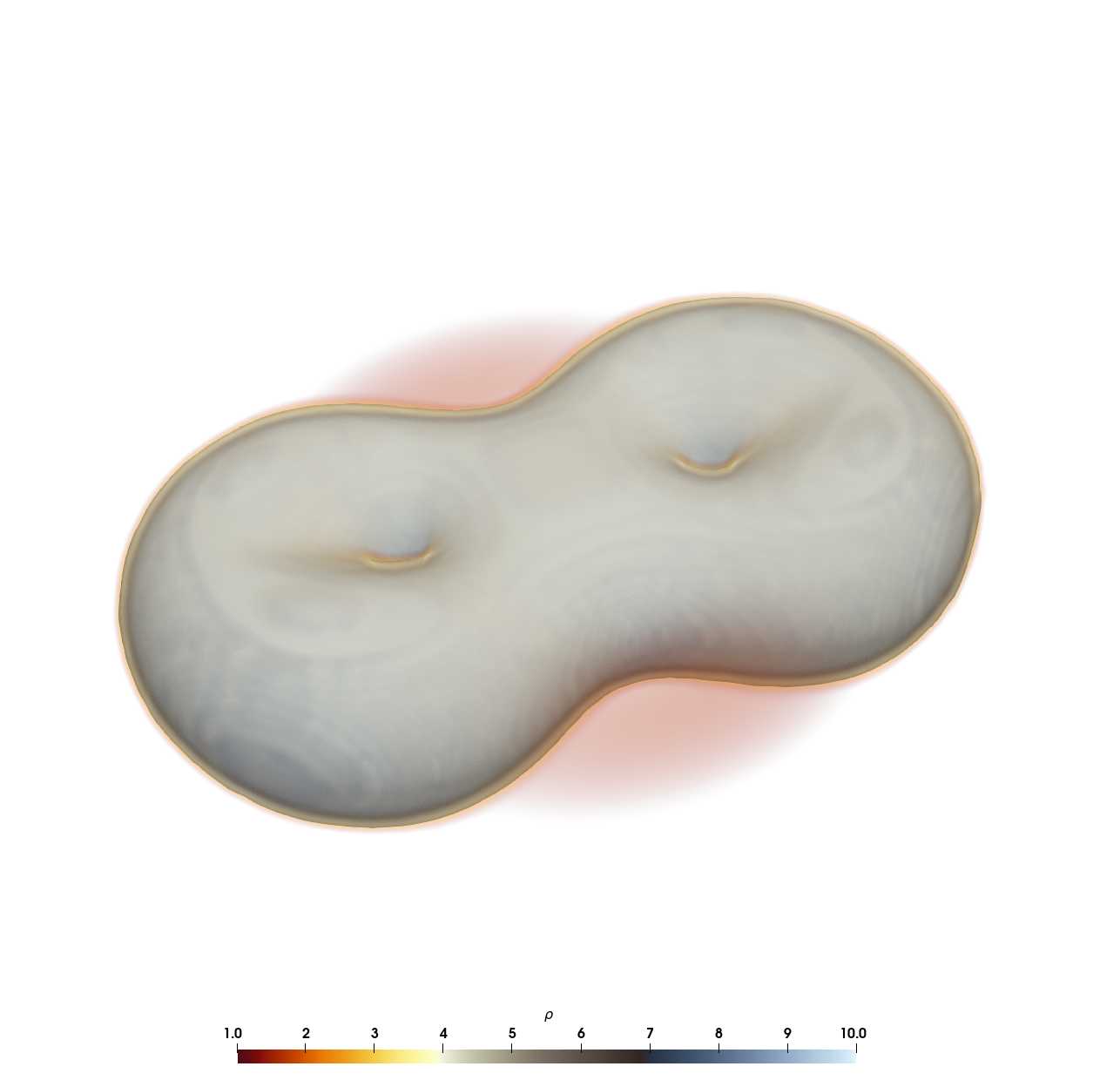}
        \caption*{$t=0.2$}
    \end{subfigure}
    \begin{subfigure}{0.16\textwidth}
        \centering
        \includegraphics[width=\textwidth,trim=30mm 40mm 30mm 70mm, clip]{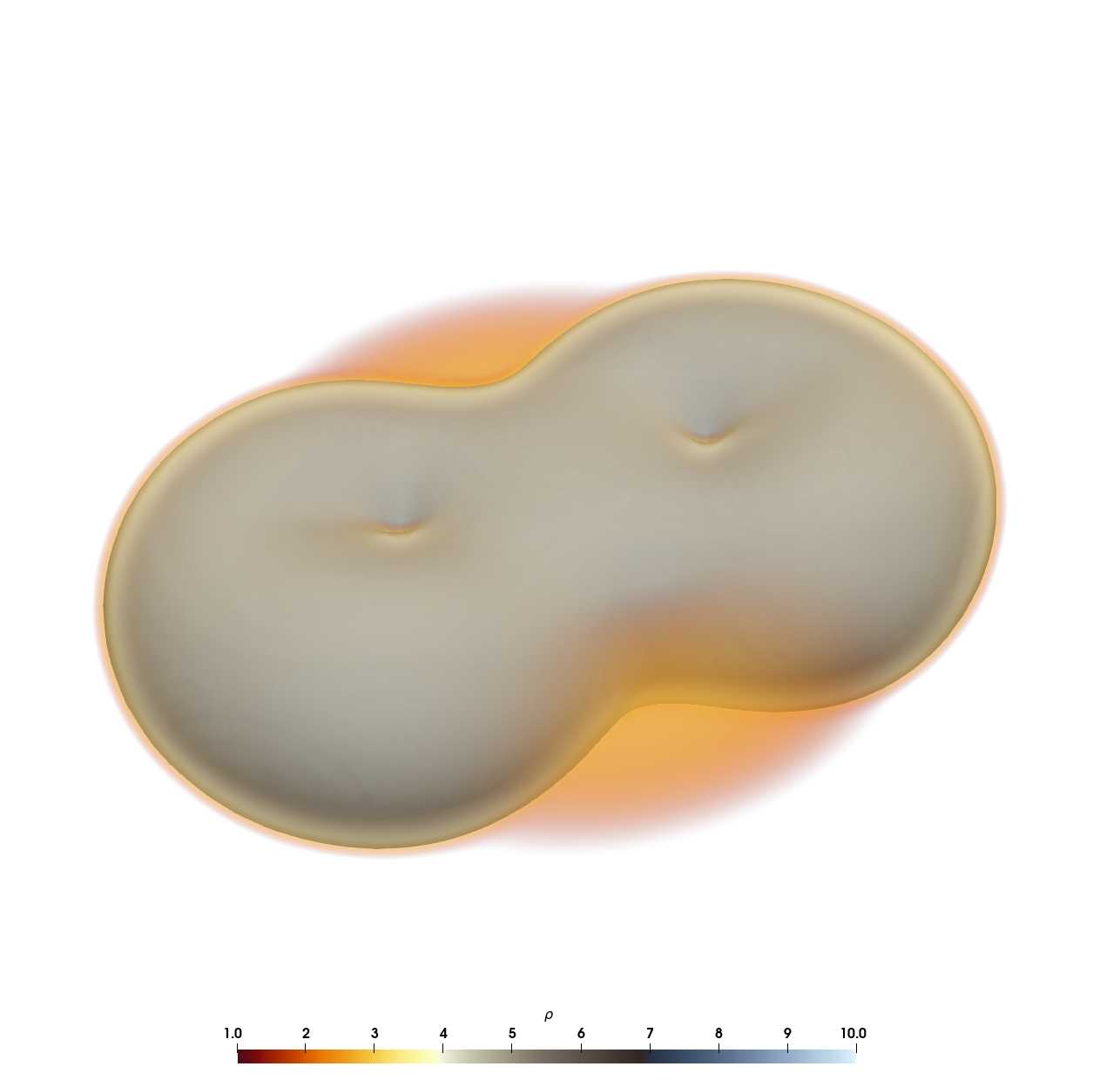}
        \caption*{$t=0.4$}
    \end{subfigure}
    \begin{subfigure}{0.16\textwidth}
        \centering
        \includegraphics[width=\textwidth,trim=30mm 40mm 30mm 70mm, clip]{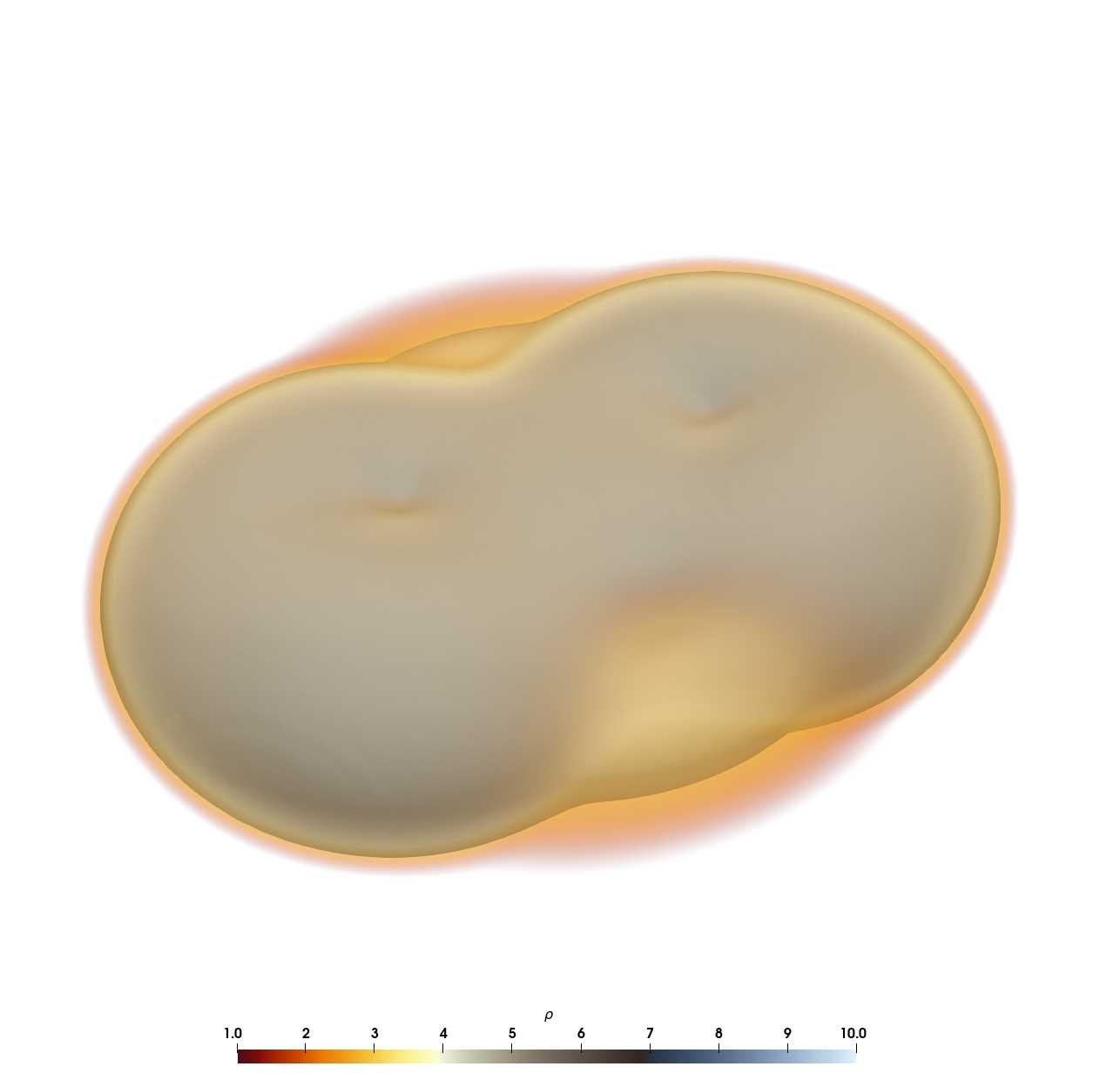}
        \caption*{$t=0.6$}
    \end{subfigure}
    \begin{subfigure}{0.16\textwidth}
        \centering
        \includegraphics[width=\textwidth,trim=30mm 40mm 30mm 70mm, clip]{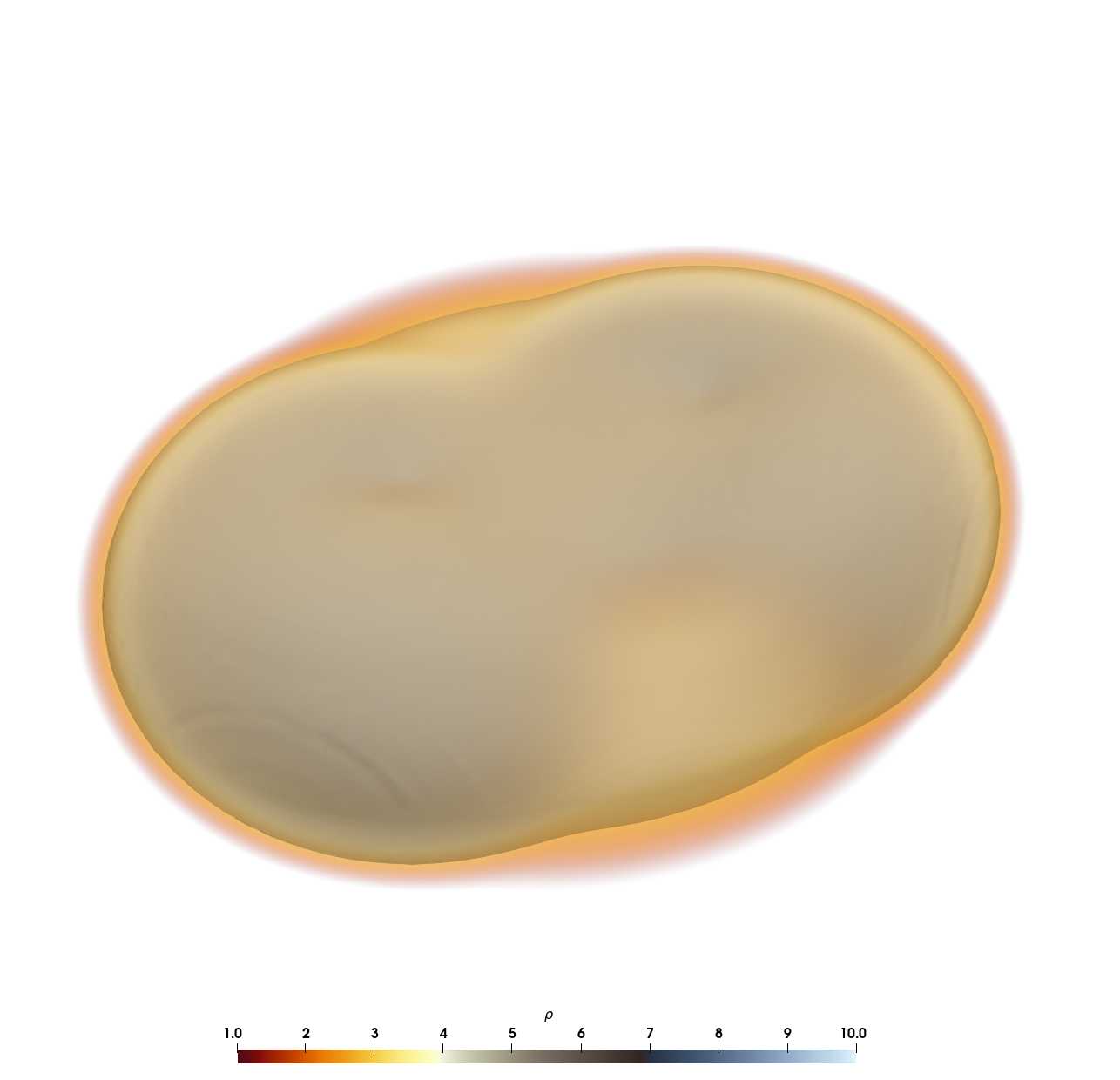}
        \caption*{$t=0.8$}
    \end{subfigure}
    \begin{subfigure}{0.16\textwidth}
        \centering
        \includegraphics[width=\textwidth,trim=30mm 40mm 30mm 70mm, clip]{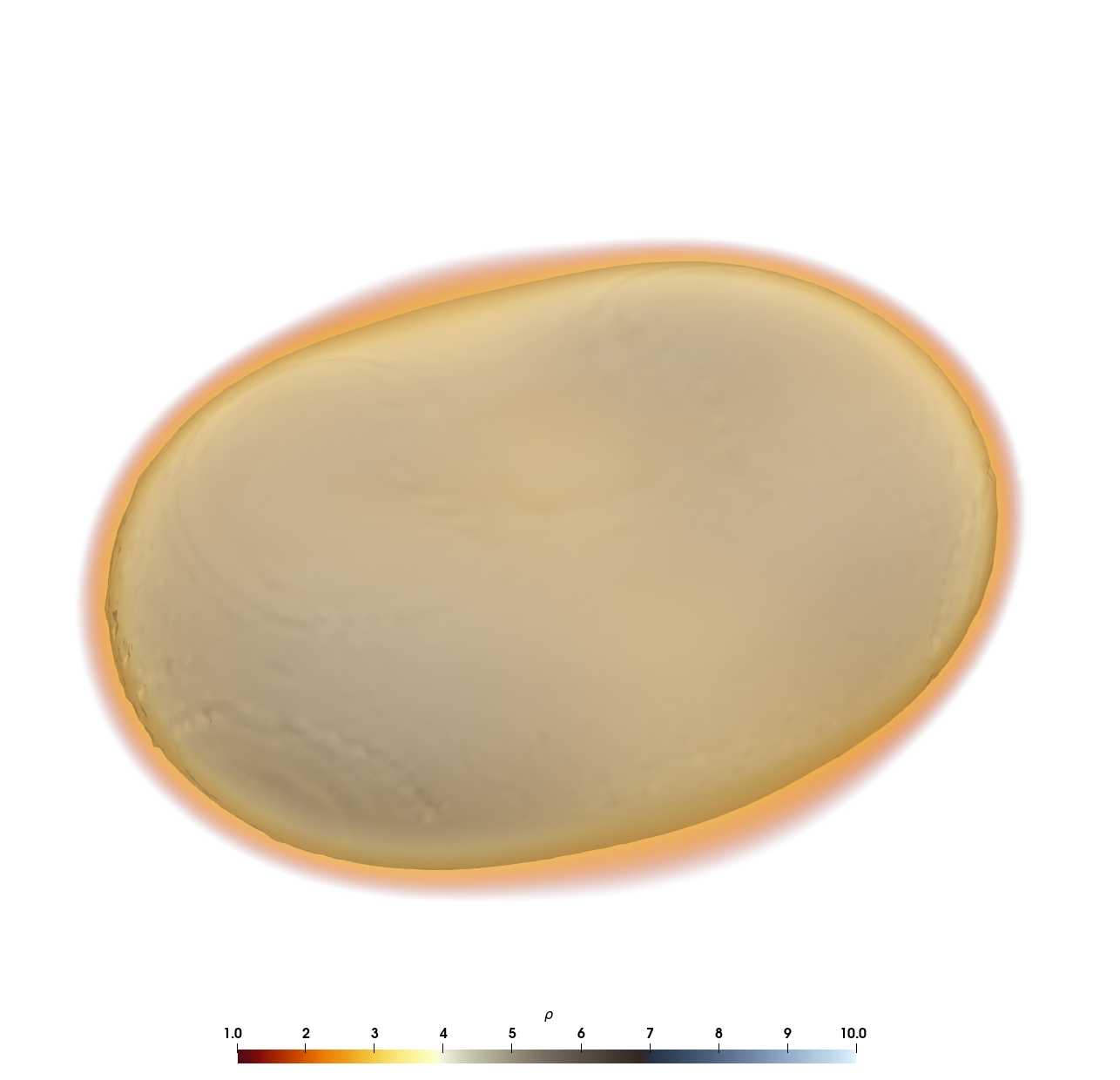}
        \caption*{$t=1.0$}
    \end{subfigure}

\subcaption{$\rho_2$}
\label{f:tdt2}
\end{minipage}
\caption{Shape interpolation in a \textit{torus}-\textit{double torus} system. Left-right shows the time evolution of density. Top-down shows plots at different reaction strengths: $\alpha=0$, $\alpha=50,$ and $\alpha=100$.}
\label{f:tdt}
\end{figure}

\begin{figure}[h!]
    \centering
\begin{minipage}[b]{\textwidth}

    \begin{minipage}[b]{\textwidth}
    \hfill
        \begin{subfigure}{\textwidth}
            \centering
            \includegraphics[width=\textwidth,trim=0mm 0mm 0mm 400mm, clip]{figures/reaction/doubleTorusBunny/pdhg0.0000..jpg}
        \end{subfigure}
    \end{minipage}
    
    \begin{subfigure}{0.16\textwidth}
        \centering
        \includegraphics[width=\textwidth,trim=40mm 40mm 50mm 70mm, clip]{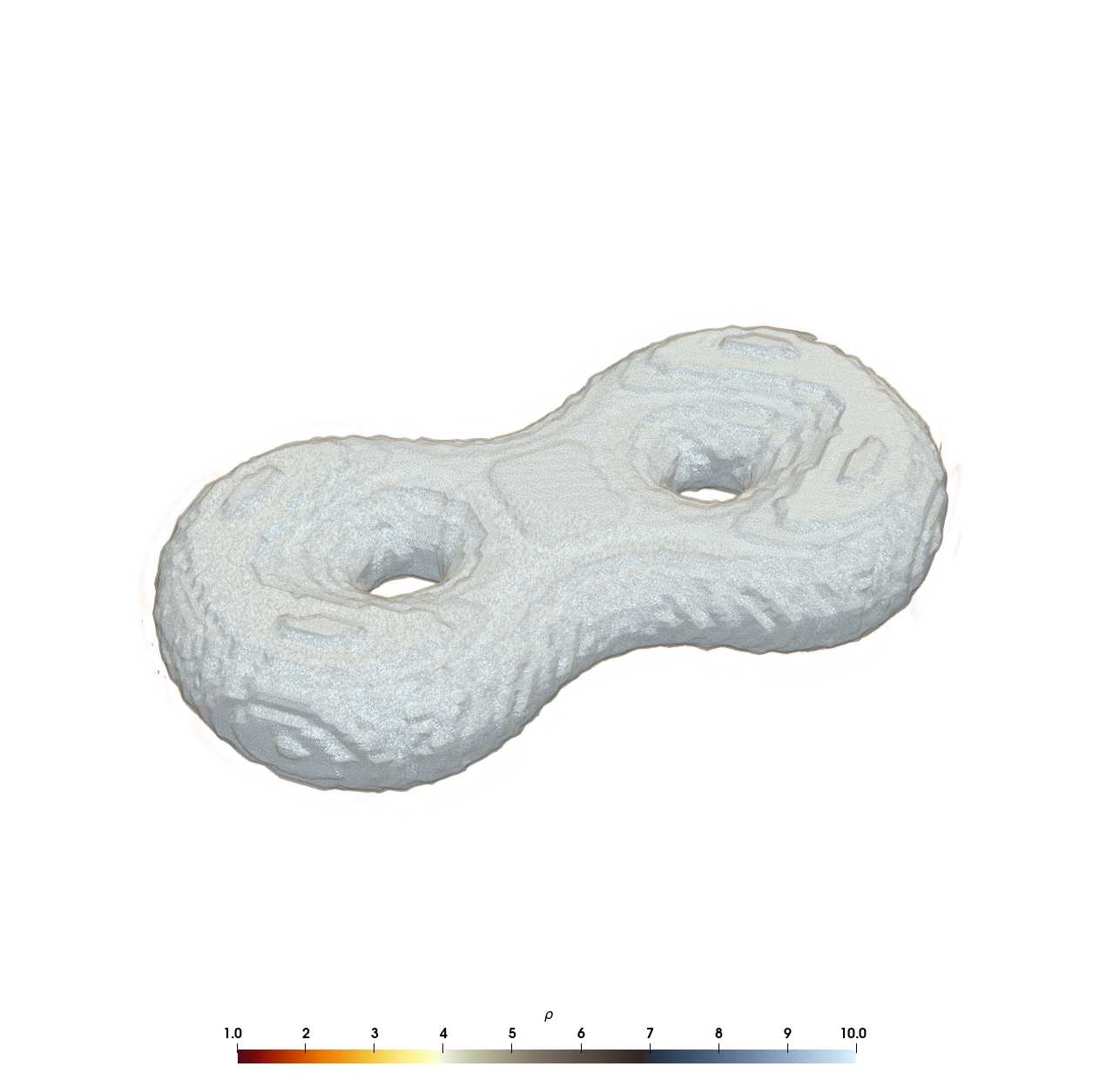}
    \end{subfigure}
    \begin{subfigure}{0.16\textwidth}
        \centering
        \includegraphics[width=\textwidth,trim=40mm 40mm 50mm 70mm, clip]{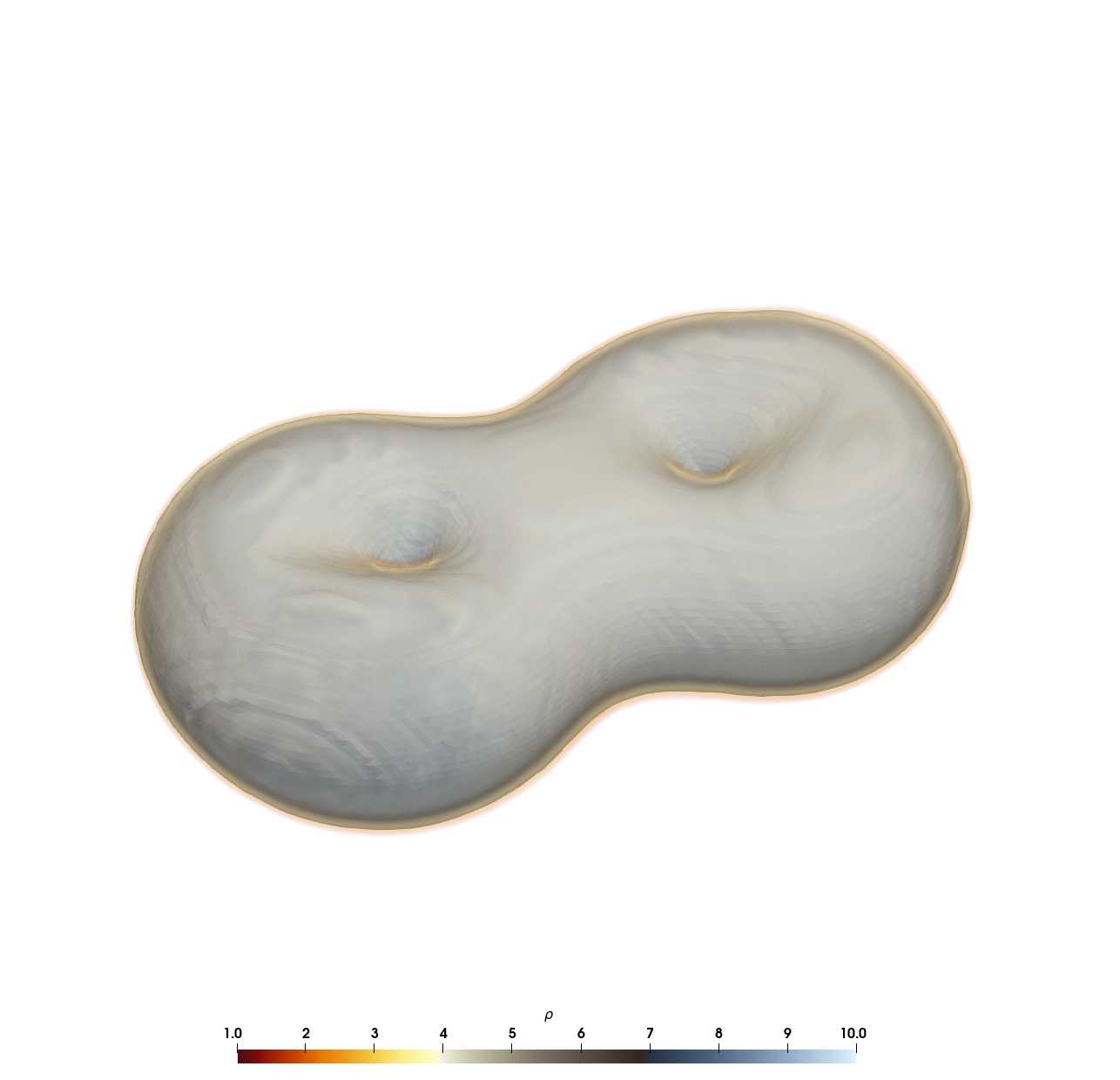}
    \end{subfigure}
    \begin{subfigure}{0.16\textwidth}
        \centering
        \includegraphics[width=\textwidth,trim=40mm 40mm 50mm 70mm, clip]{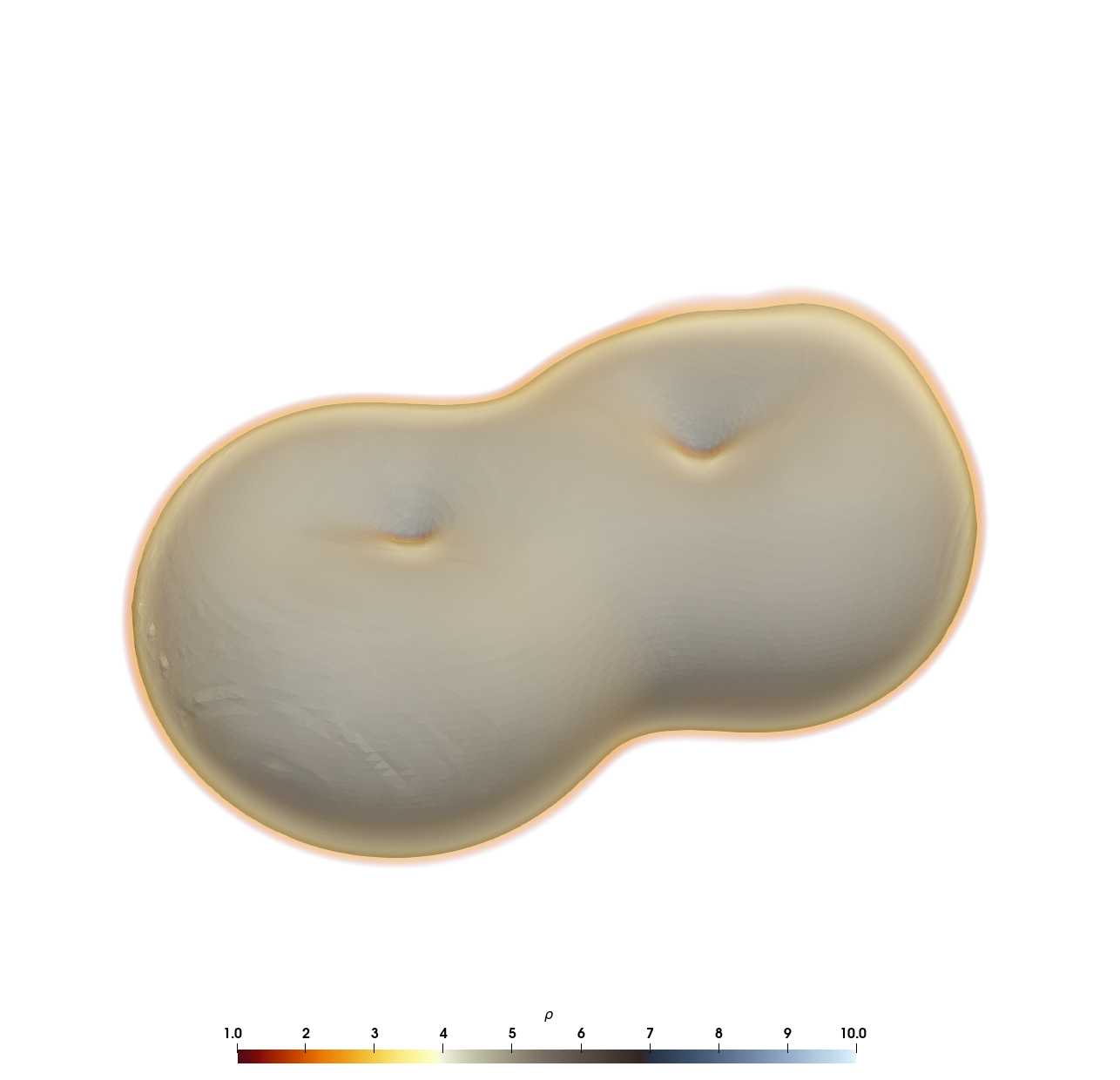}
    \end{subfigure}
    \begin{subfigure}{0.16\textwidth}
        \centering
        \includegraphics[width=\textwidth,trim=40mm 40mm 50mm 70mm, clip]{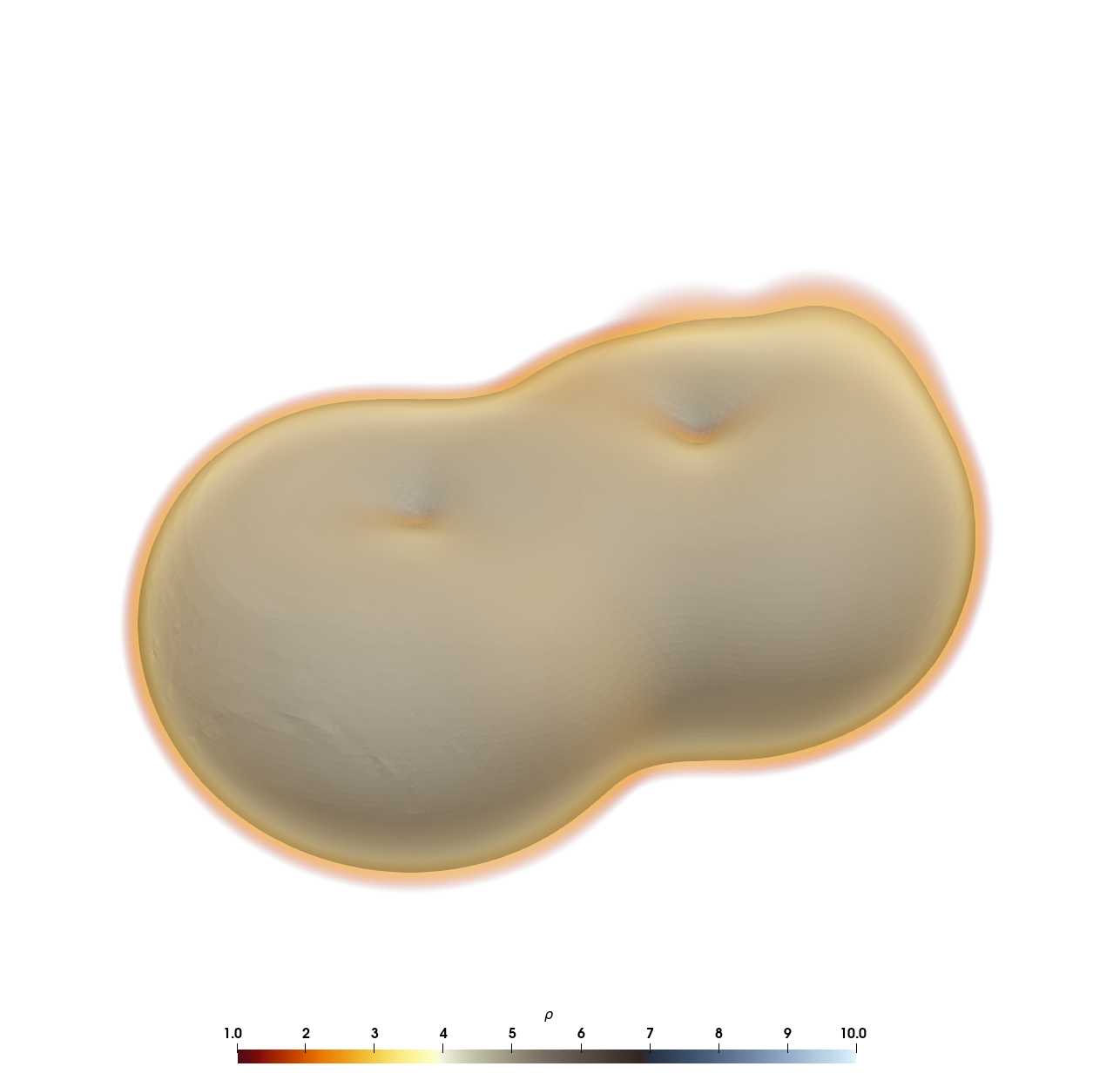}
    \end{subfigure}
    \begin{subfigure}{0.16\textwidth}
        \centering
        \includegraphics[width=\textwidth,trim=40mm 40mm 50mm 70mm, clip]{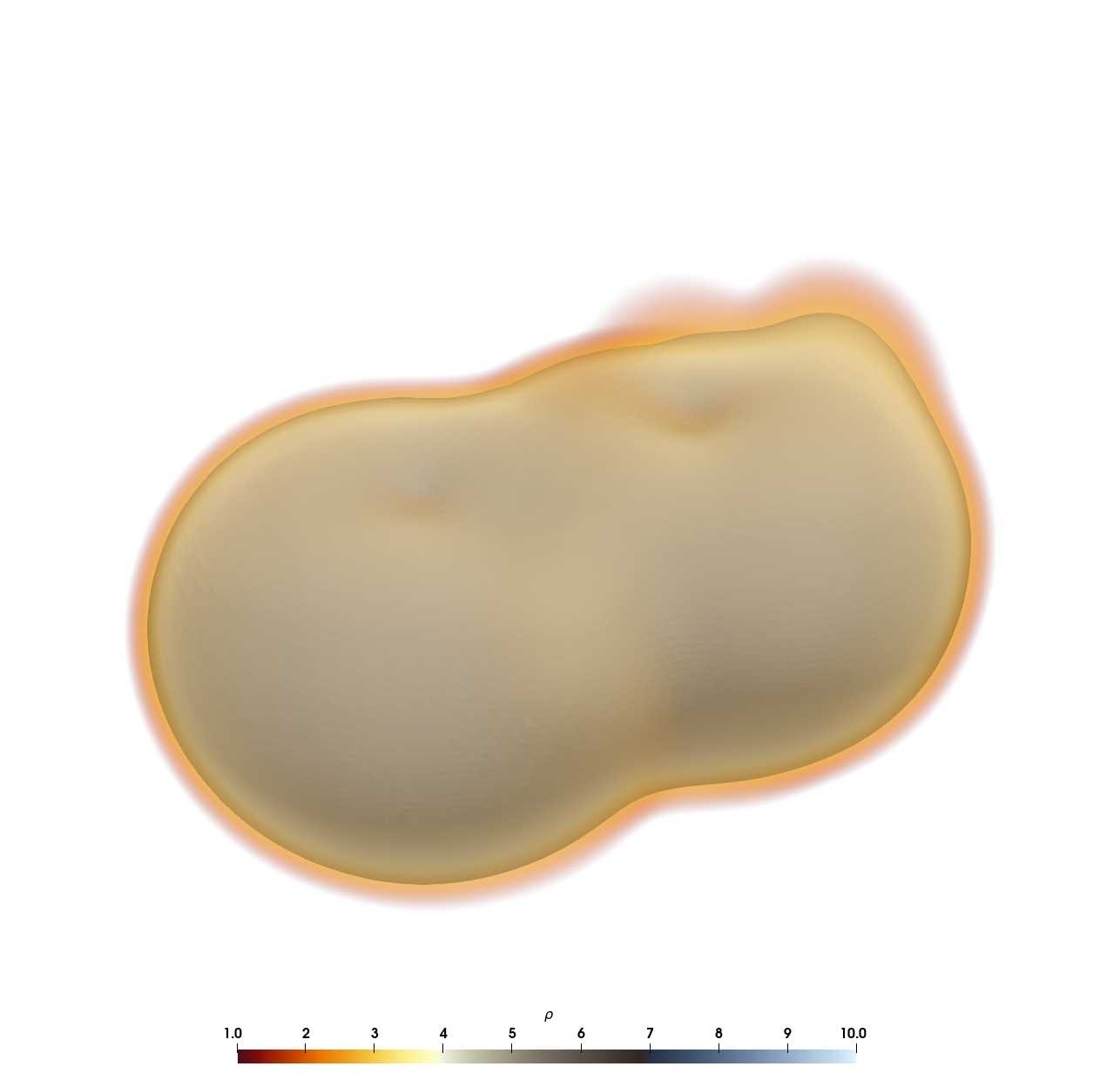}
    \end{subfigure}
    \begin{subfigure}{0.16\textwidth}
        \centering
        \includegraphics[width=\textwidth,trim=40mm 40mm 50mm 70mm, clip]{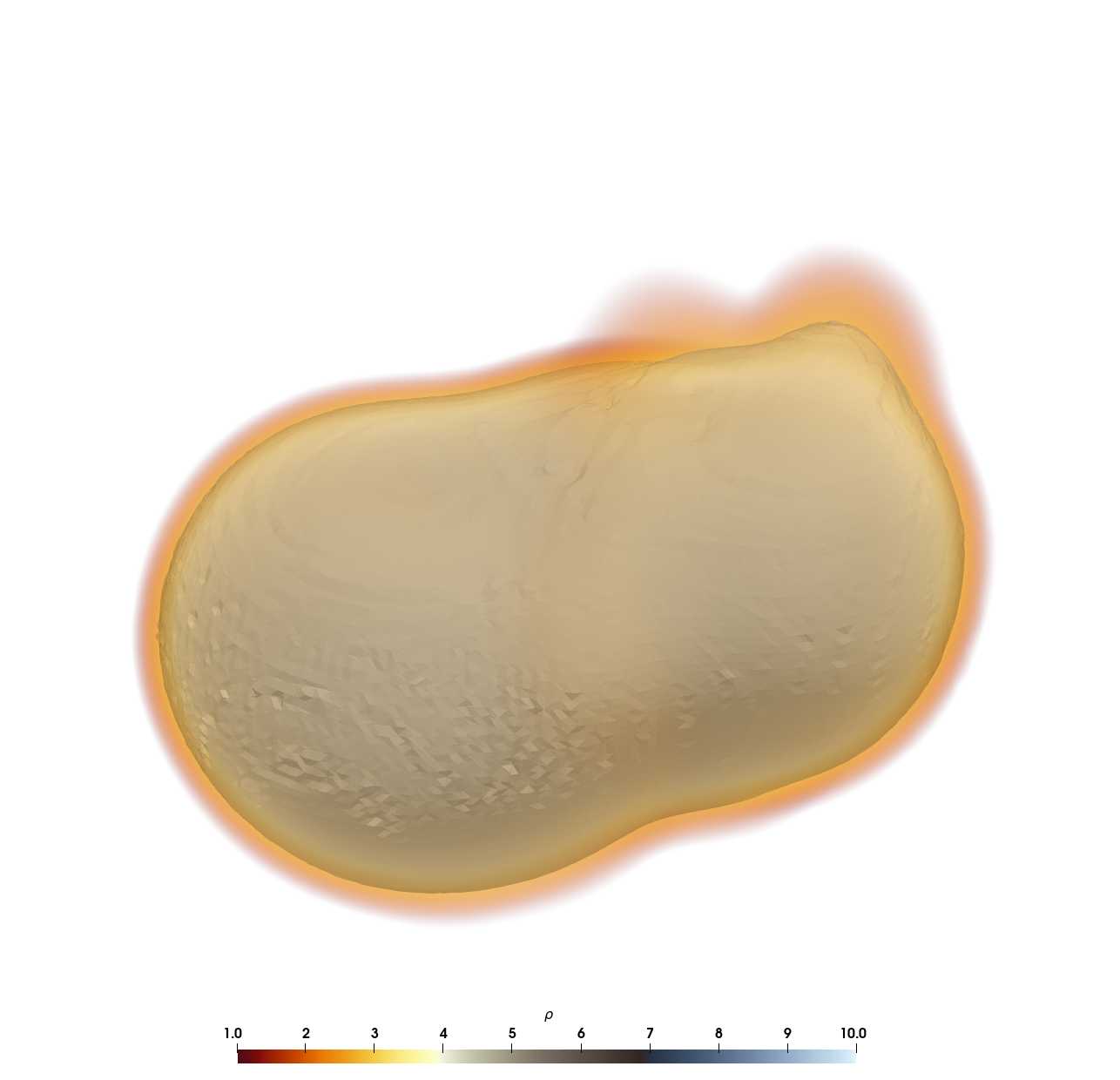}
    \end{subfigure}
    
    \begin{subfigure}{0.16\textwidth}
        \centering
        \includegraphics[width=\textwidth,trim=40mm 40mm 50mm 70mm, clip]{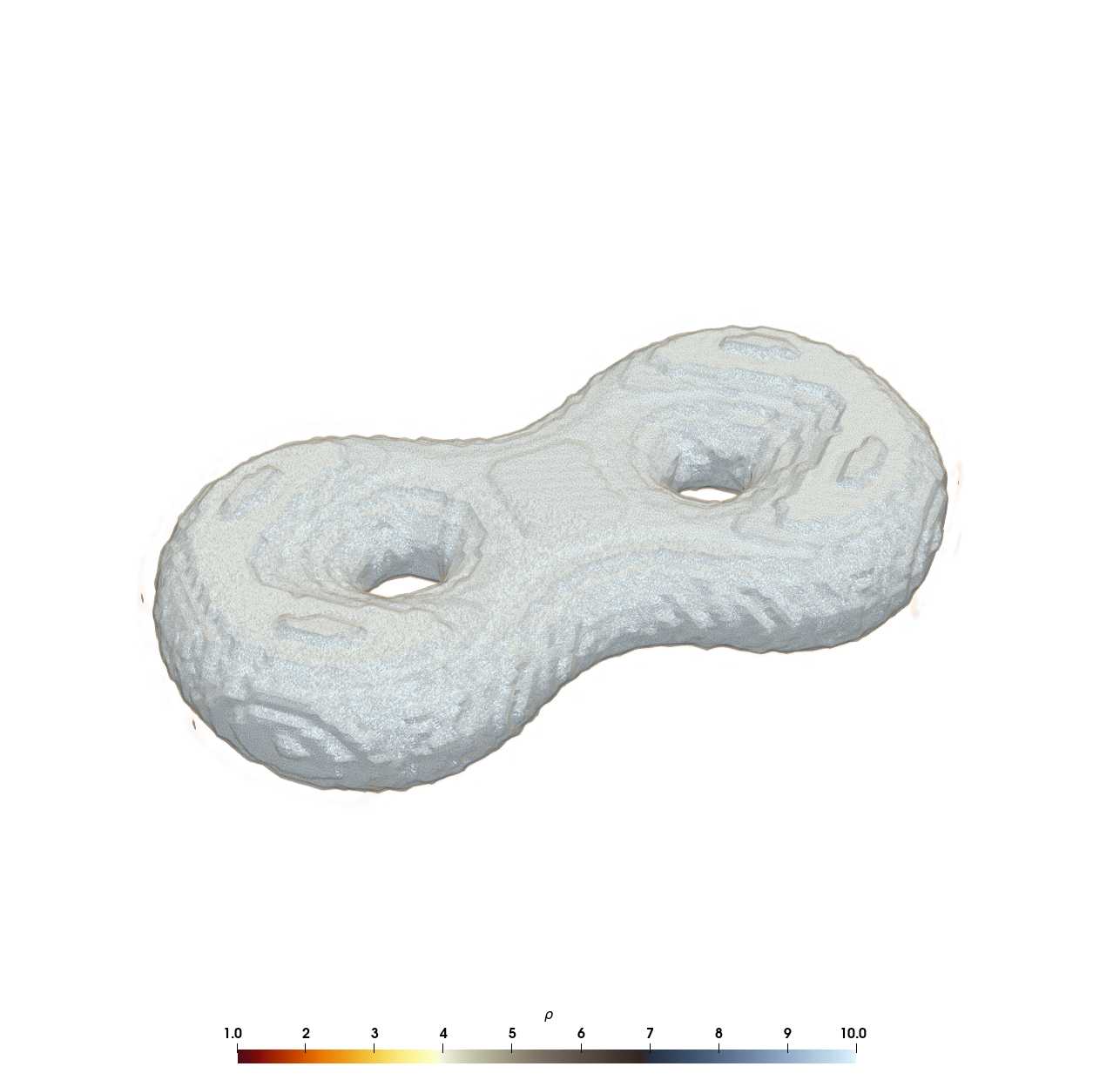}
    \end{subfigure}
    \begin{subfigure}{0.16\textwidth}
        \centering
        \includegraphics[width=\textwidth,trim=40mm 40mm 50mm 70mm, clip]{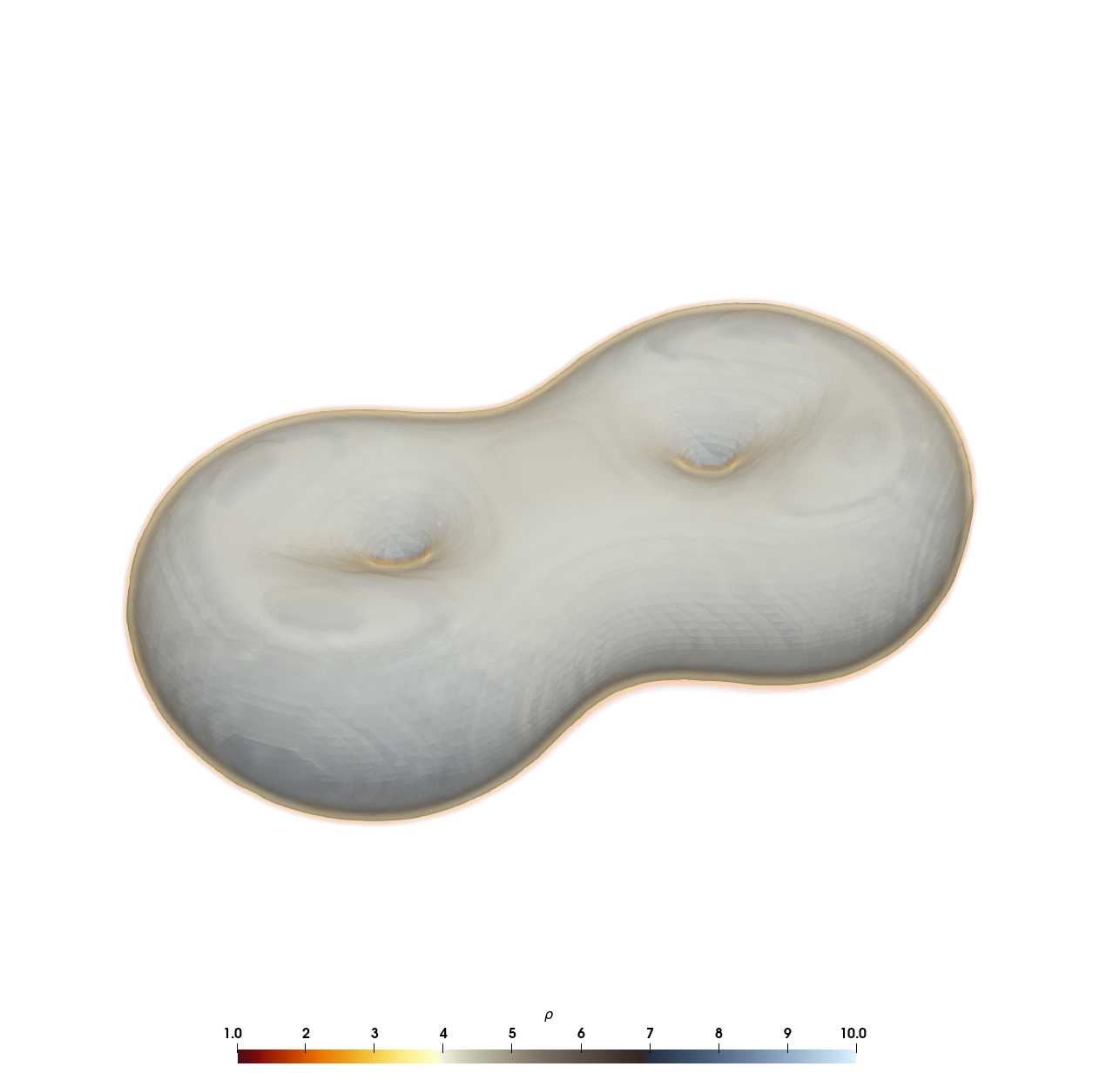}
    \end{subfigure}
    \begin{subfigure}{0.16\textwidth}
        \centering
        \includegraphics[width=\textwidth,trim=40mm 40mm 50mm 70mm, clip]{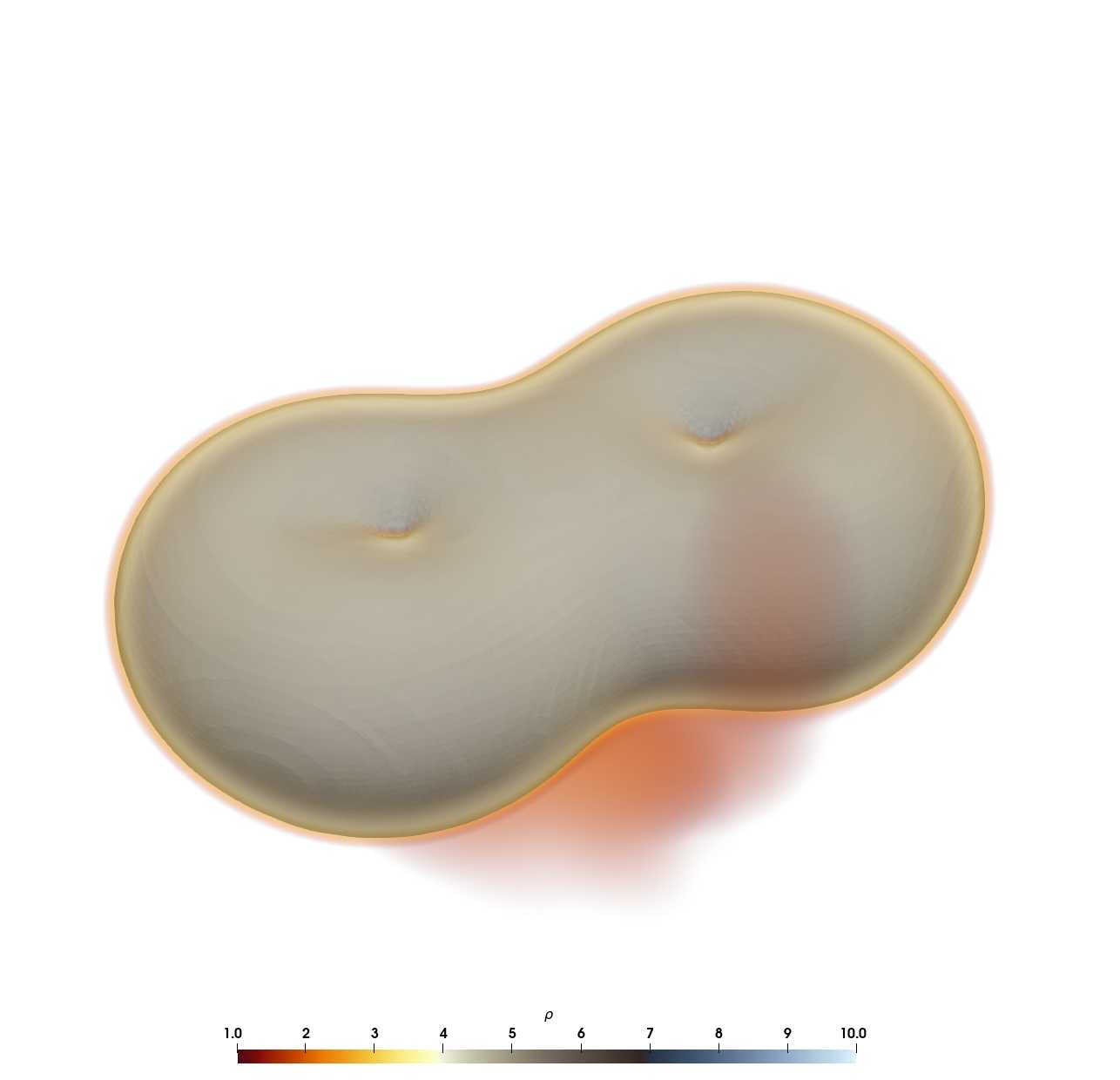}
    \end{subfigure}
    \begin{subfigure}{0.16\textwidth}
        \centering
        \includegraphics[width=\textwidth,trim=40mm 40mm 50mm 70mm, clip]{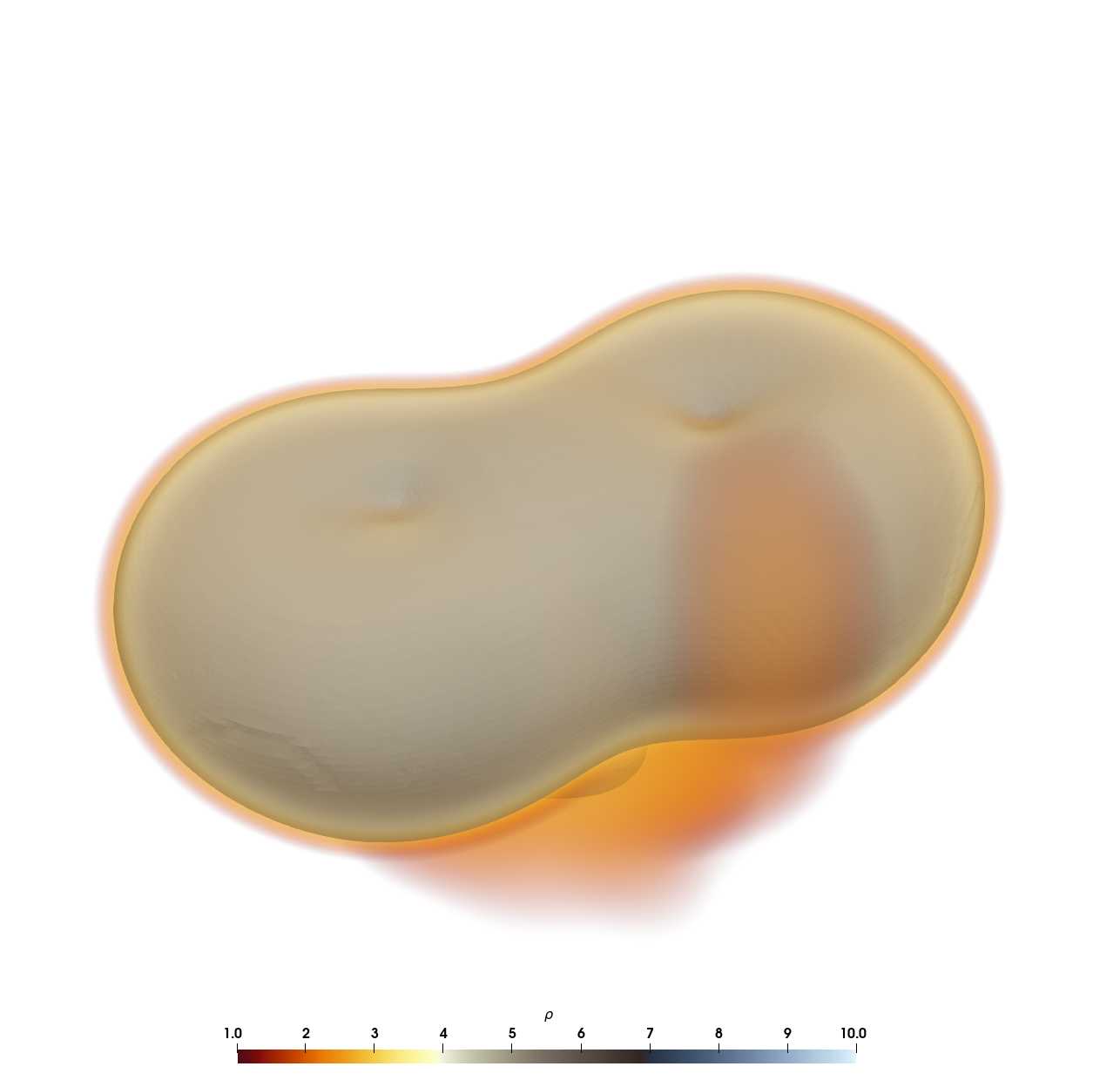}
    \end{subfigure}
    \begin{subfigure}{0.16\textwidth}
        \centering
        \includegraphics[width=\textwidth,trim=40mm 40mm 50mm 70mm, clip]{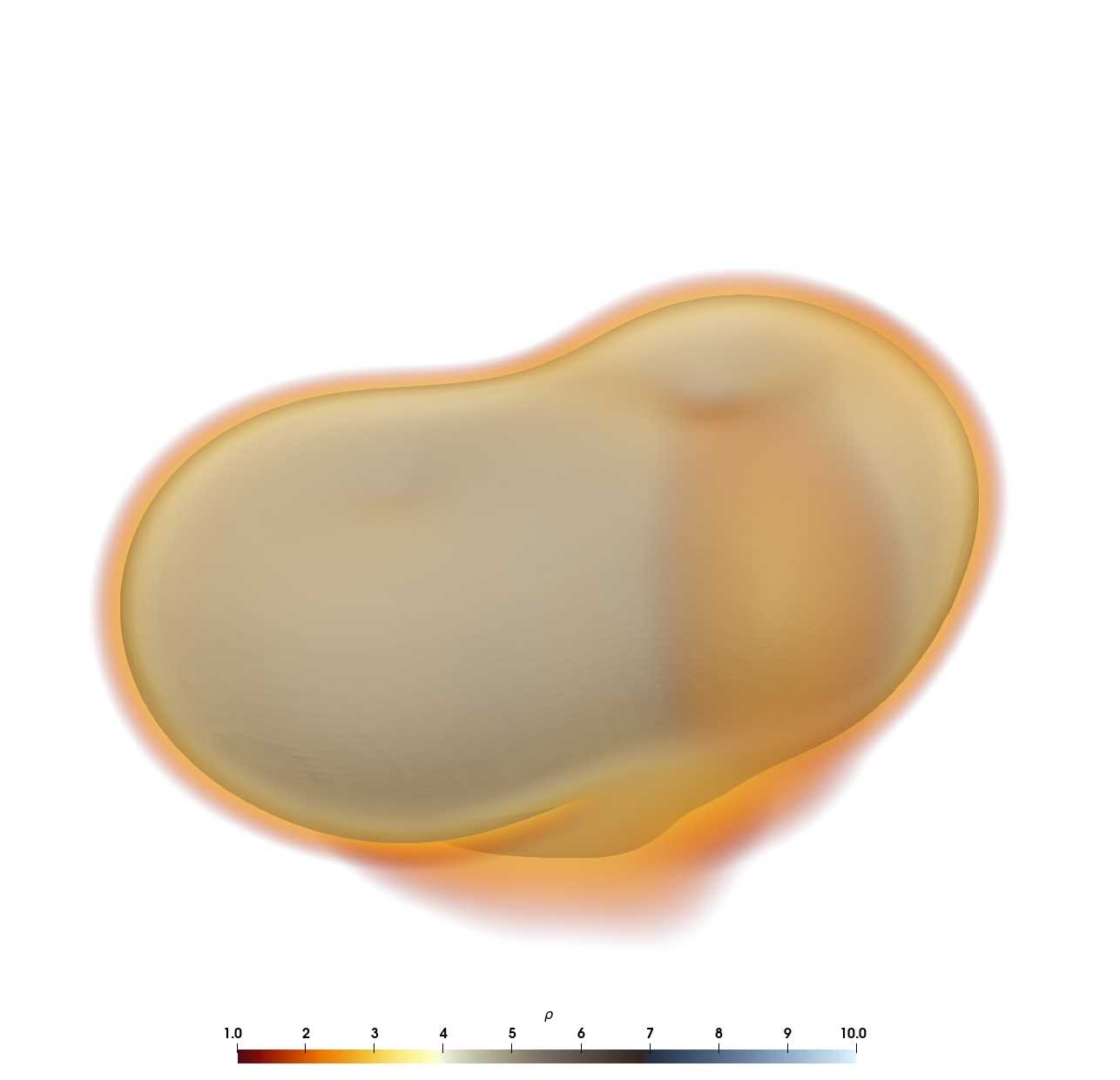}
    \end{subfigure}
    \begin{subfigure}{0.16\textwidth}
        \centering
        \includegraphics[width=\textwidth,trim=40mm 40mm 50mm 70mm, clip]{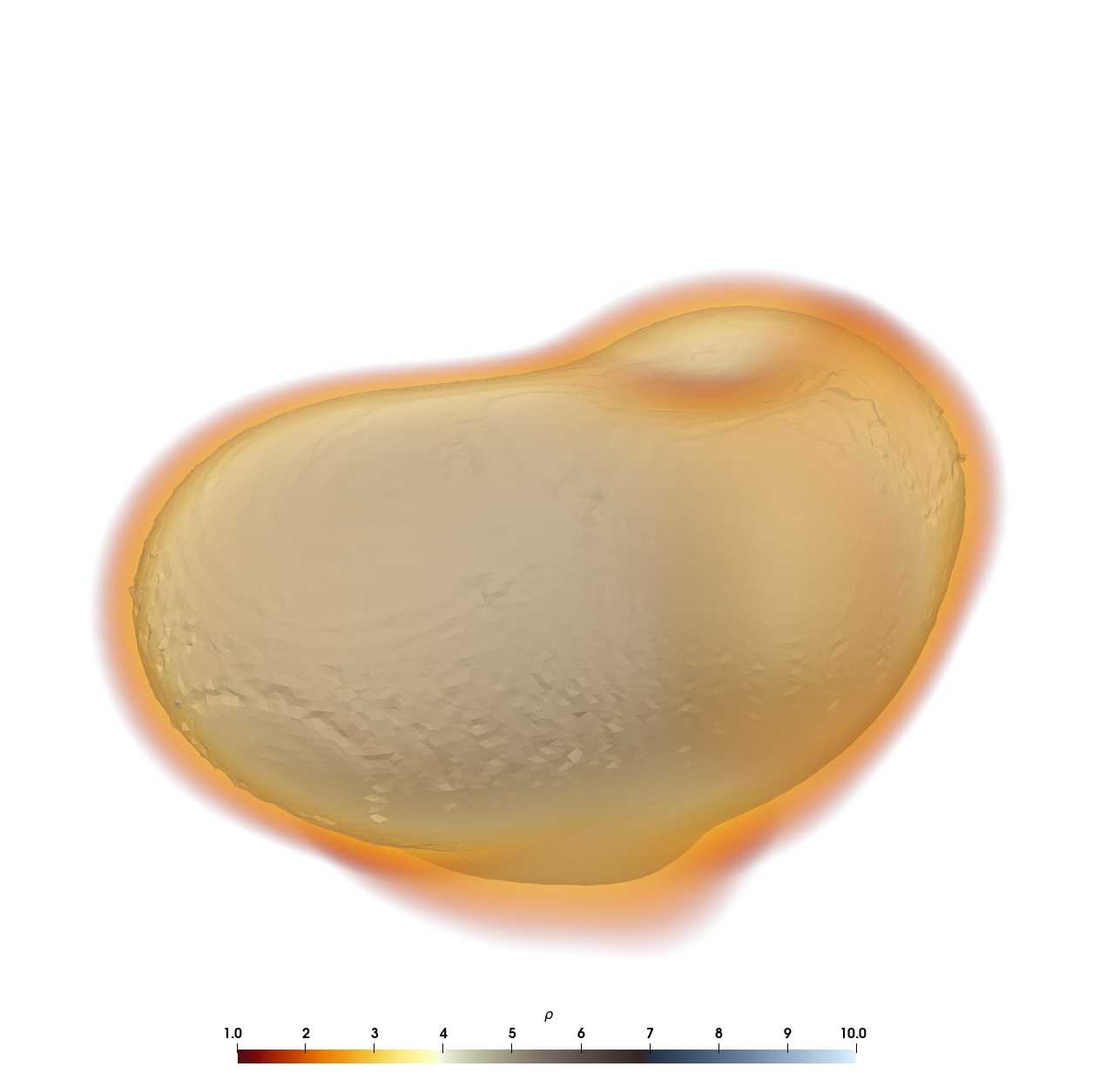}
    \end{subfigure}

    \begin{subfigure}{0.16\textwidth}
        \centering
        \includegraphics[width=\textwidth,trim=40mm 40mm 50mm 70mm, clip]{figures/reaction/100/doubleTorusBunny/pdhg0.0000..jpg}
        \caption*{$t=0.0$}
    \end{subfigure}
    \begin{subfigure}{0.16\textwidth}
        \centering
        \includegraphics[width=\textwidth,trim=40mm 40mm 50mm 70mm, clip]{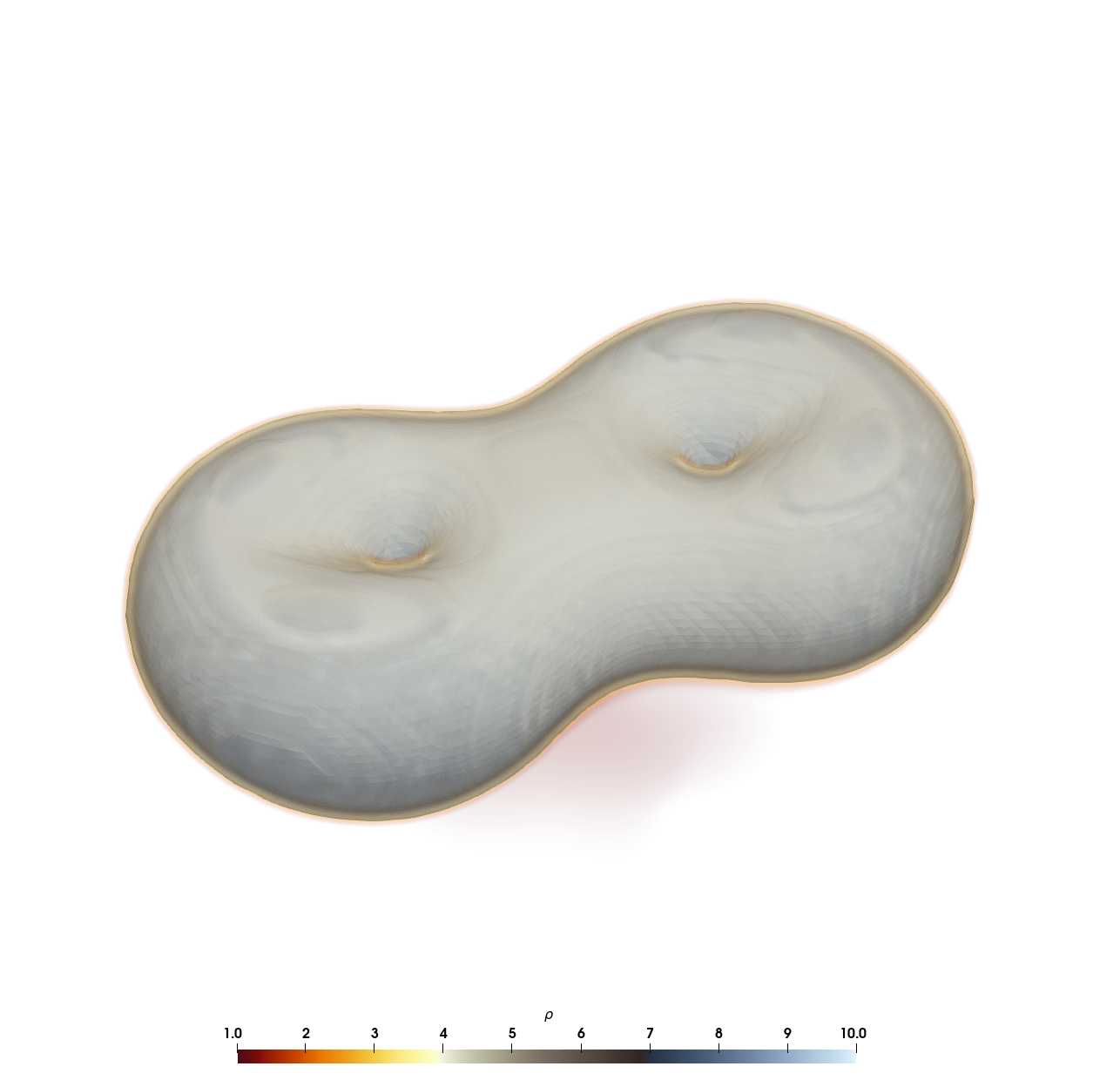}
        \caption*{$t=0.2$}
    \end{subfigure}
    \begin{subfigure}{0.16\textwidth}
        \centering
        \includegraphics[width=\textwidth,trim=40mm 40mm 50mm 70mm, clip]{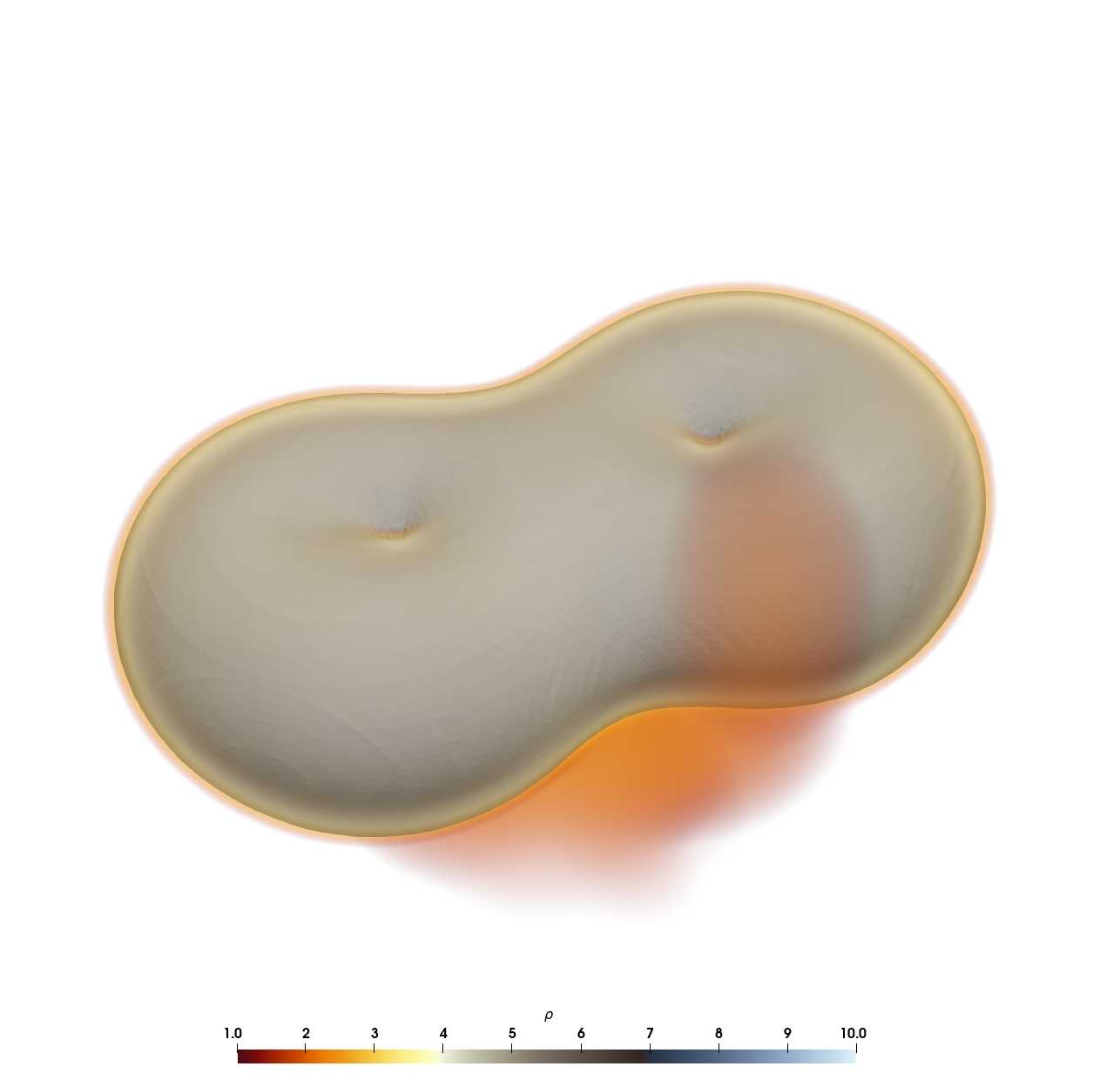}
        \caption*{$t=0.4$}
    \end{subfigure}
    \begin{subfigure}{0.16\textwidth}
        \centering
        \includegraphics[width=\textwidth,trim=40mm 40mm 50mm 70mm, clip]{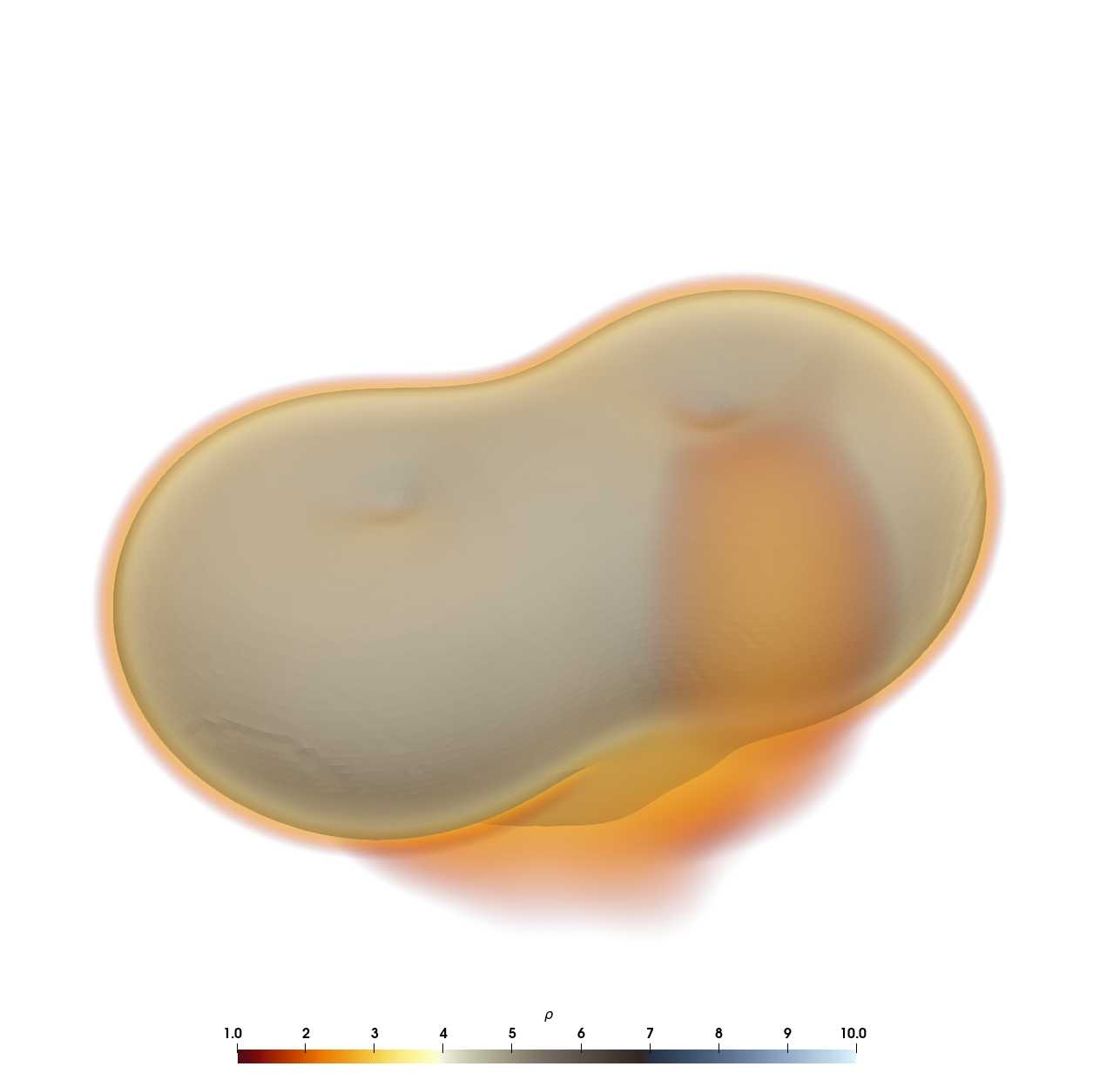}
        \caption*{$t=0.6$}
    \end{subfigure}
    \begin{subfigure}{0.16\textwidth}
        \centering
        \includegraphics[width=\textwidth,trim=40mm 40mm 50mm 70mm, clip]{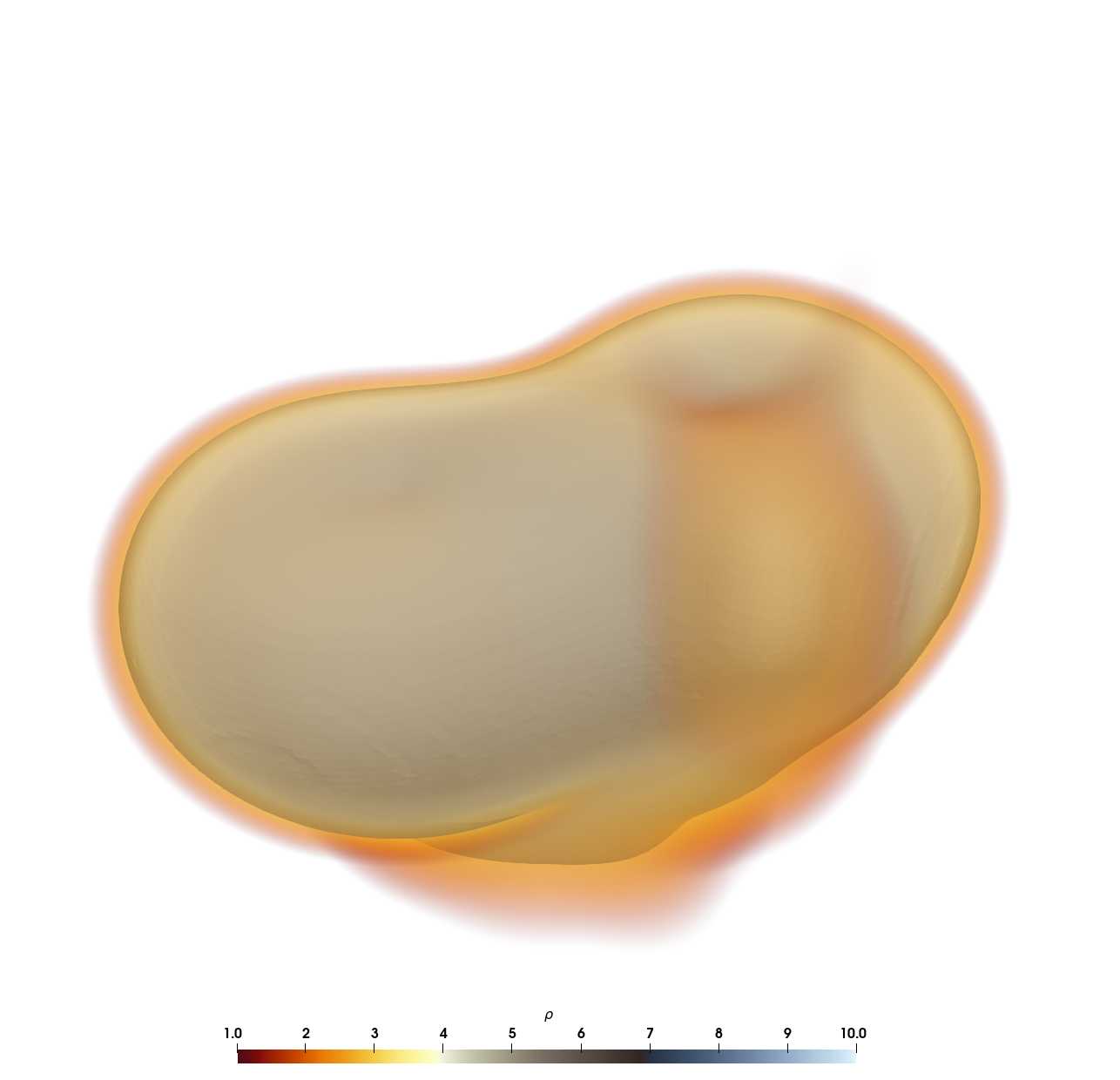}
        \caption*{$t=0.8$}
    \end{subfigure}
    \begin{subfigure}{0.16\textwidth}
        \centering
        \includegraphics[width=\textwidth,trim=40mm 40mm 50mm 70mm, clip]{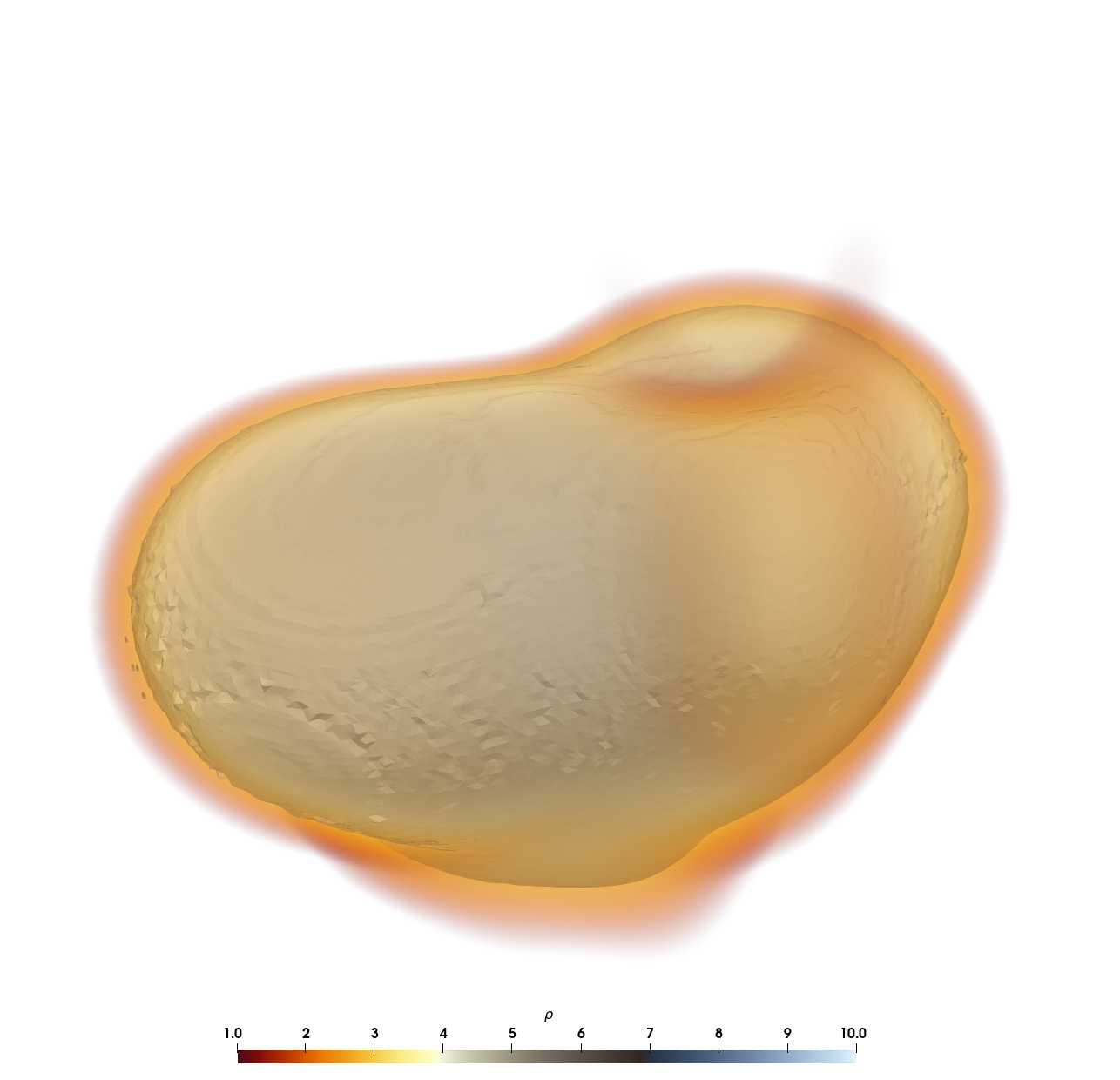}
        \caption*{$t=1.0.$}
    \end{subfigure}

\subcaption{$\rho_1$}
\label{f:dtb1}
\end{minipage}
\begin{minipage}[b]{\textwidth}

    \begin{minipage}[b]{\textwidth}
    \hfill
        \begin{subfigure}{\textwidth}
            \centering
            \includegraphics[width=\textwidth,trim=0mm 0mm 0mm 400mm, clip]{figures/reaction/doubleTorusBunny/pdhg0.0000..jpg}
        \end{subfigure}
    \end{minipage}
    
    \begin{subfigure}{0.16\textwidth}
        \centering
        \includegraphics[width=\textwidth,trim=40mm 40mm 50mm 70mm, clip]{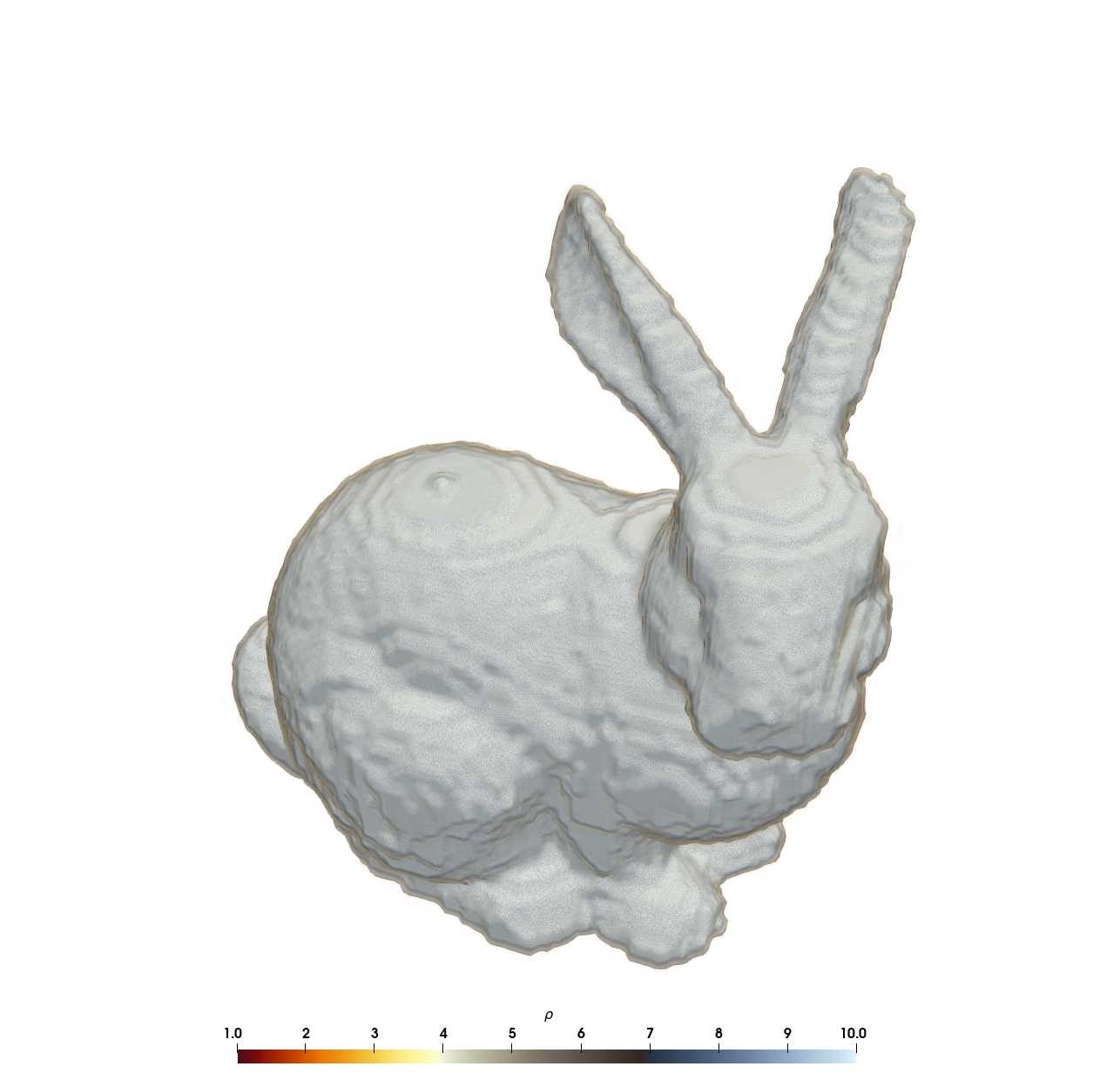}
    \end{subfigure}
    \begin{subfigure}{0.16\textwidth}
        \centering
        \includegraphics[width=\textwidth,trim=40mm 40mm 50mm 70mm, clip]{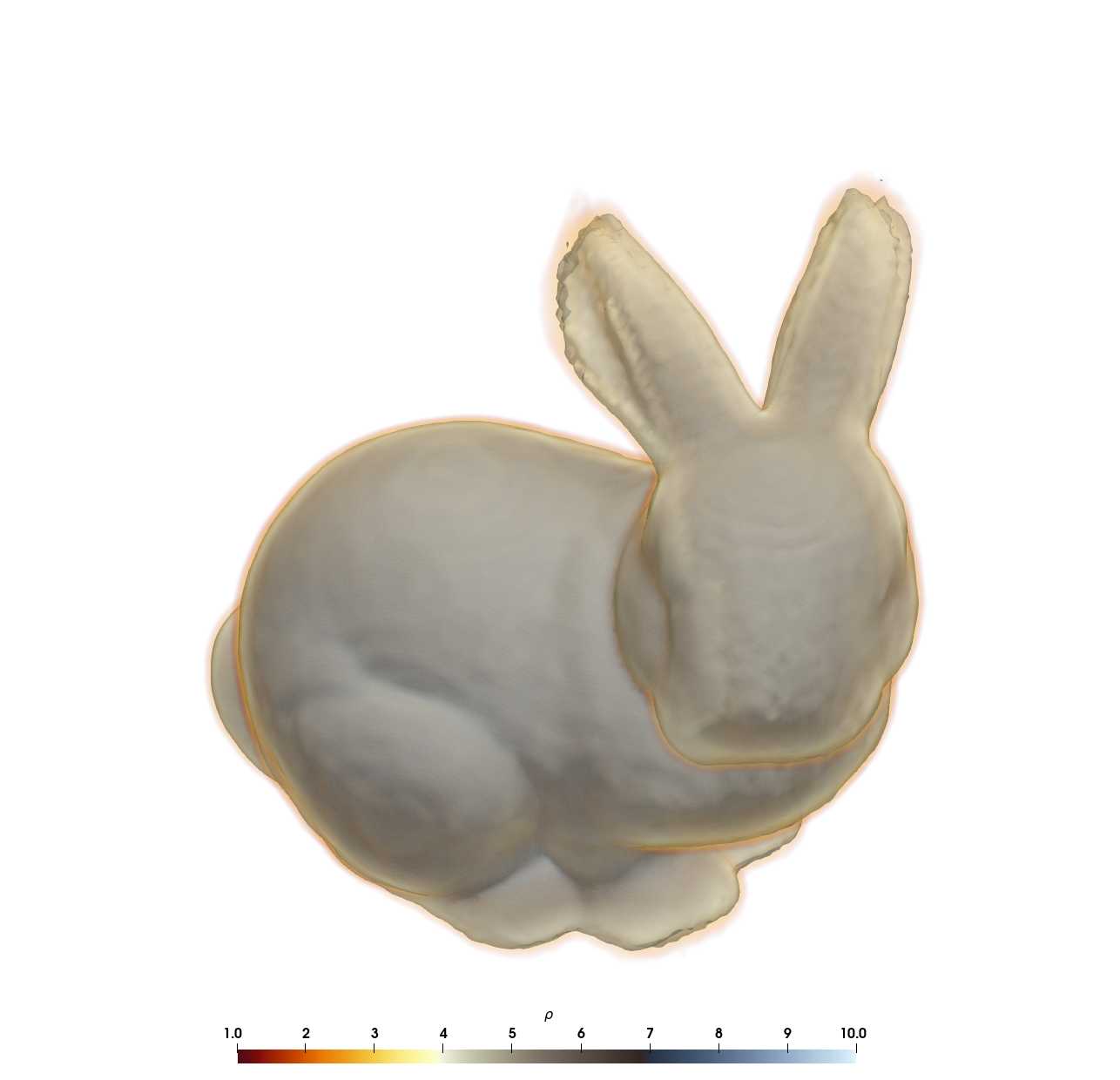}
    \end{subfigure}
    \begin{subfigure}{0.16\textwidth}
        \centering
        \includegraphics[width=\textwidth,trim=40mm 40mm 50mm 70mm, clip]{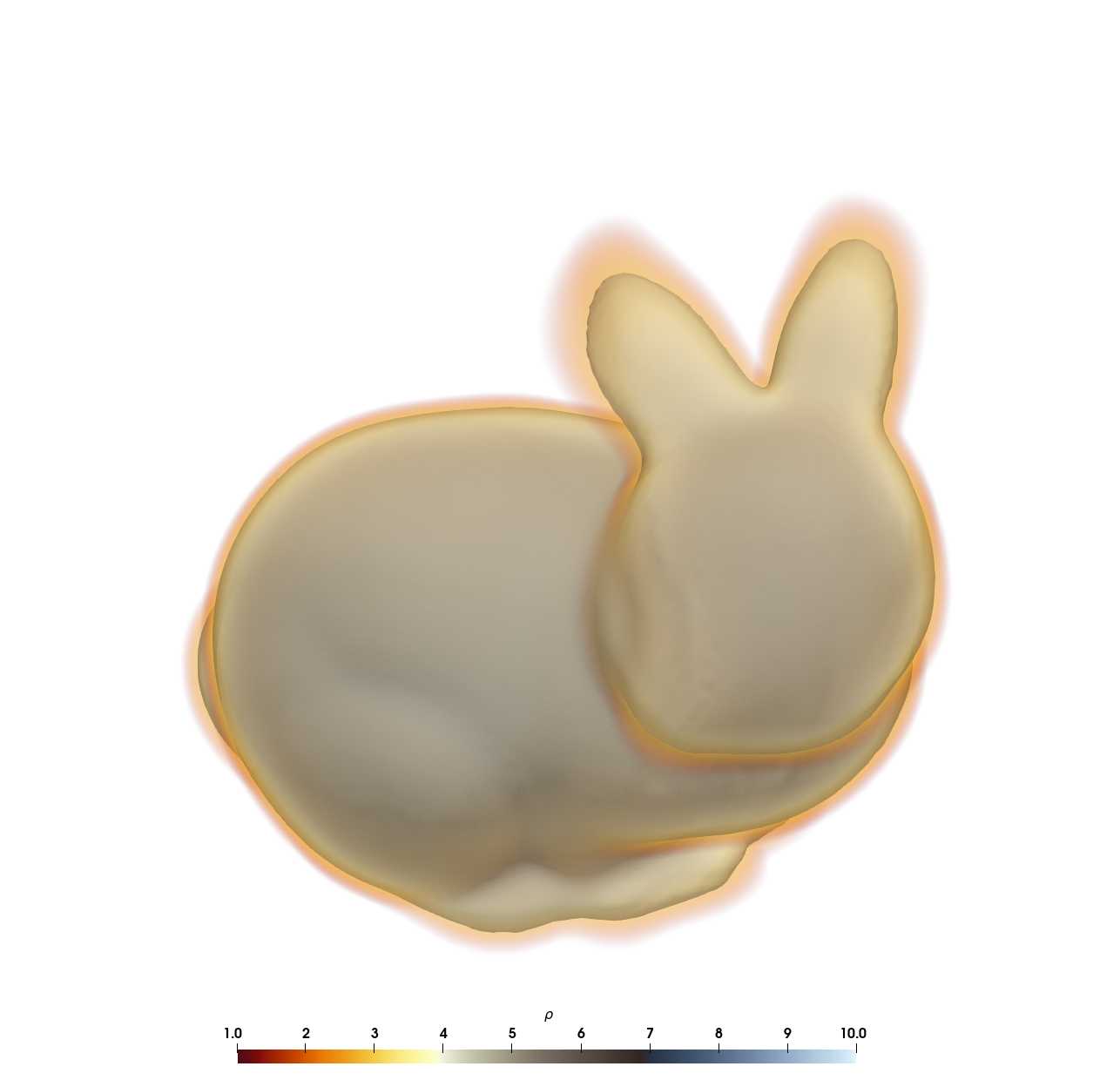}
    \end{subfigure}
    \begin{subfigure}{0.16\textwidth}
        \centering
        \includegraphics[width=\textwidth,trim=40mm 40mm 50mm 70mm, clip]{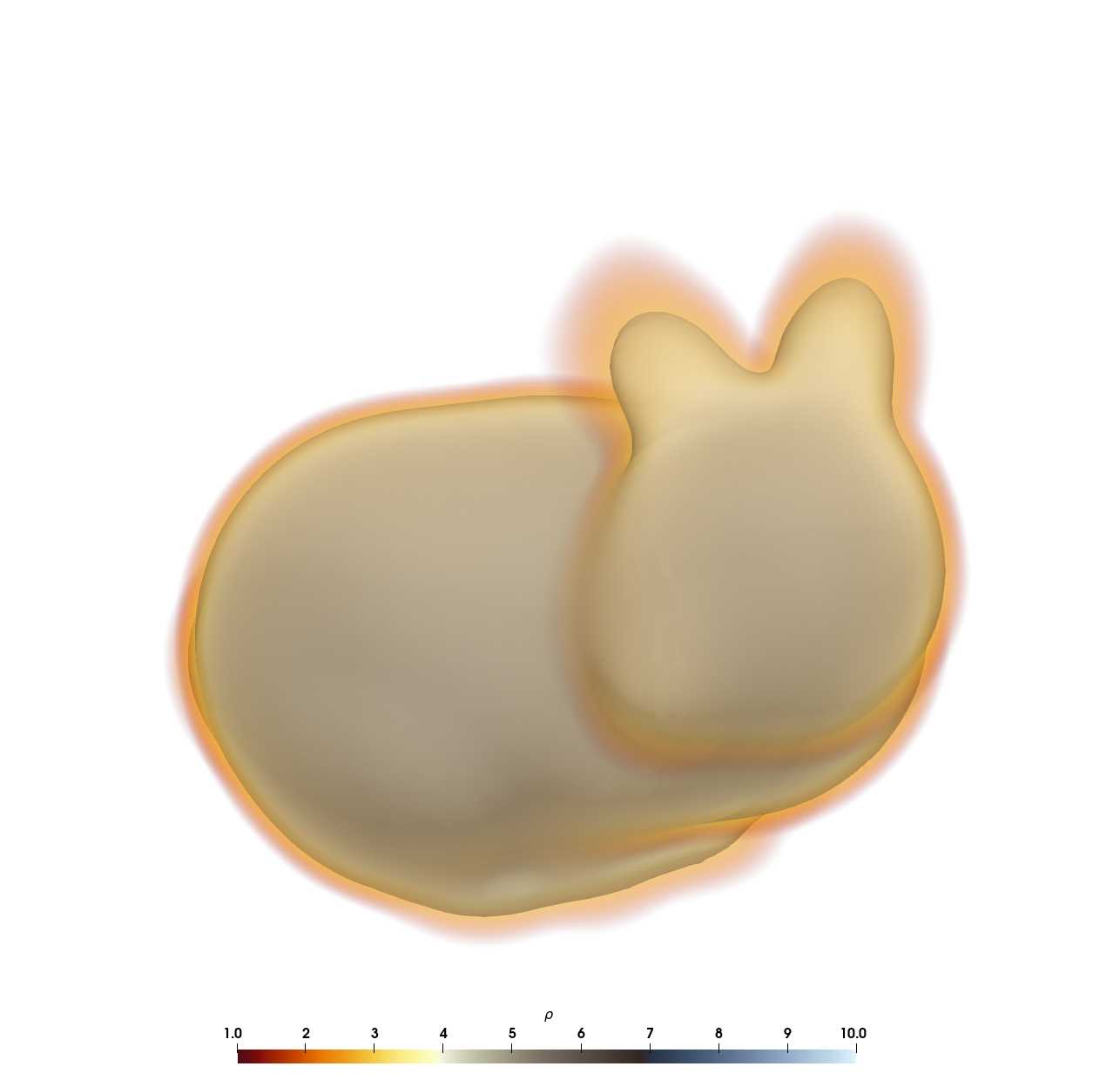}
    \end{subfigure}
    \begin{subfigure}{0.16\textwidth}
        \centering
        \includegraphics[width=\textwidth,trim=40mm 40mm 50mm 70mm, clip]{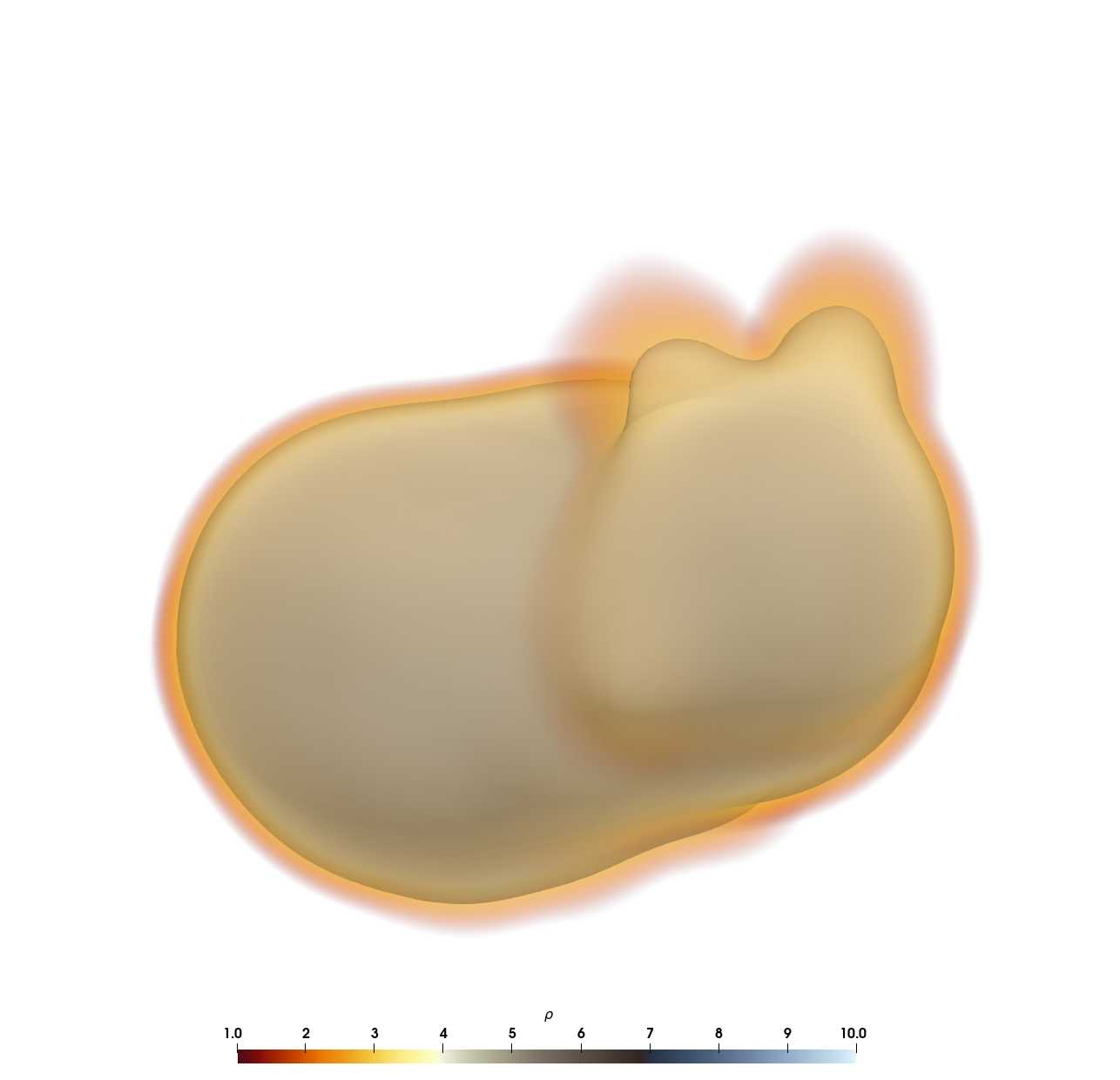}
    \end{subfigure}
    \begin{subfigure}{0.16\textwidth}
        \centering
        \includegraphics[width=\textwidth,trim=40mm 40mm 50mm 70mm, clip]{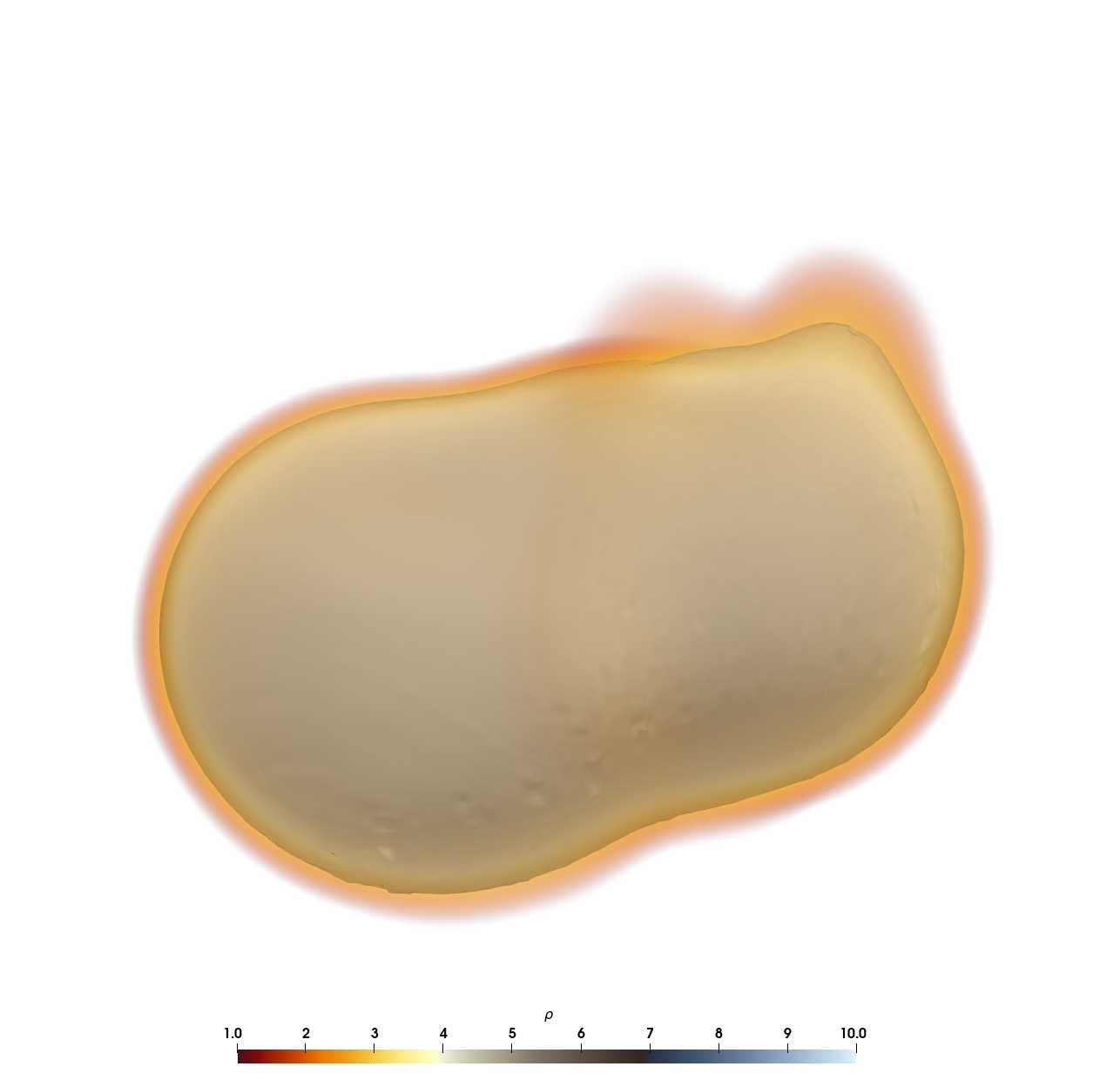}
    \end{subfigure}
    
    \begin{subfigure}{0.16\textwidth}
        \centering
        \includegraphics[width=\textwidth,trim=40mm 40mm 50mm 70mm, clip]{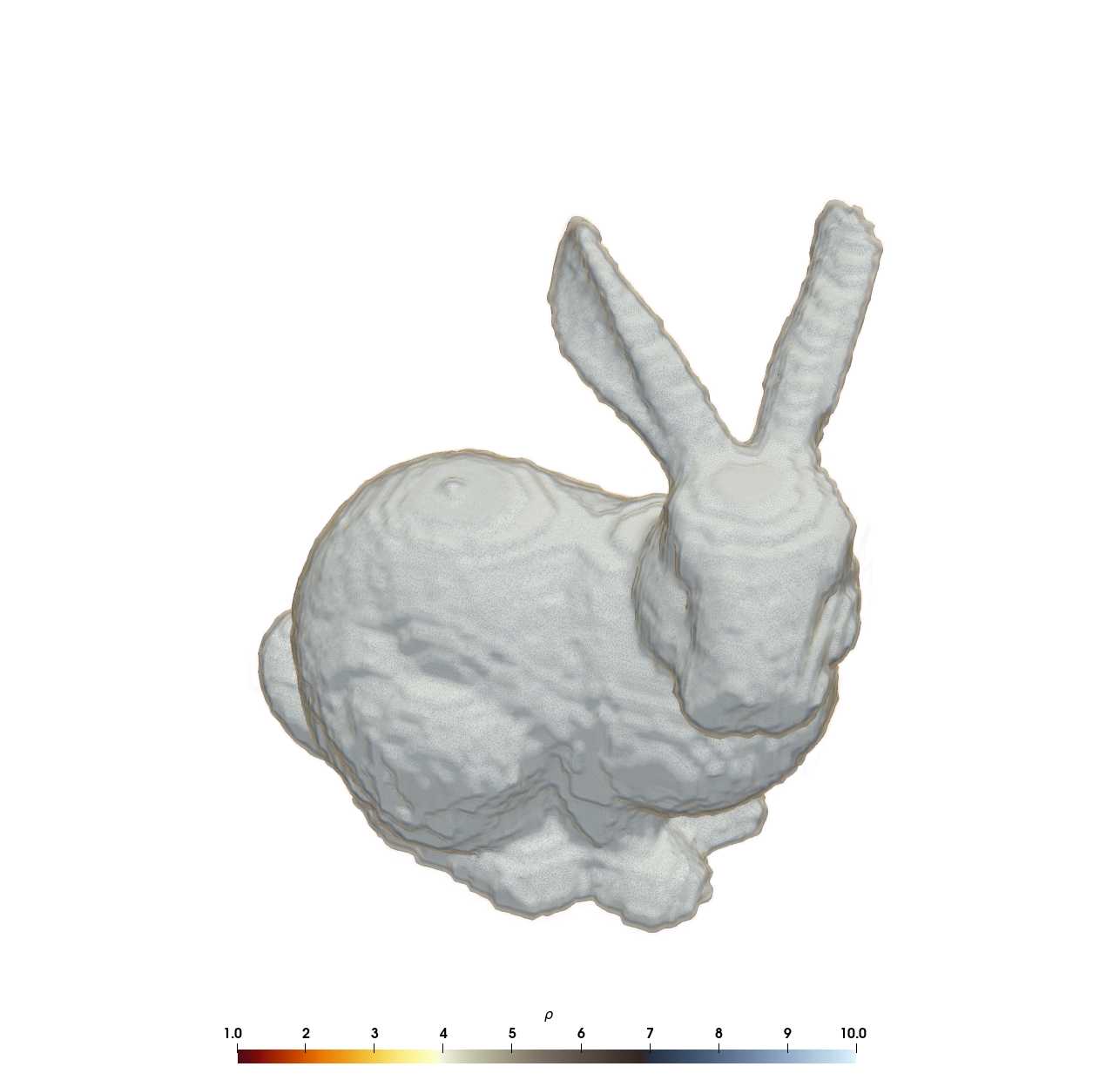}
    \end{subfigure}
    \begin{subfigure}{0.16\textwidth}
        \centering
        \includegraphics[width=\textwidth,trim=40mm 40mm 50mm 70mm, clip]{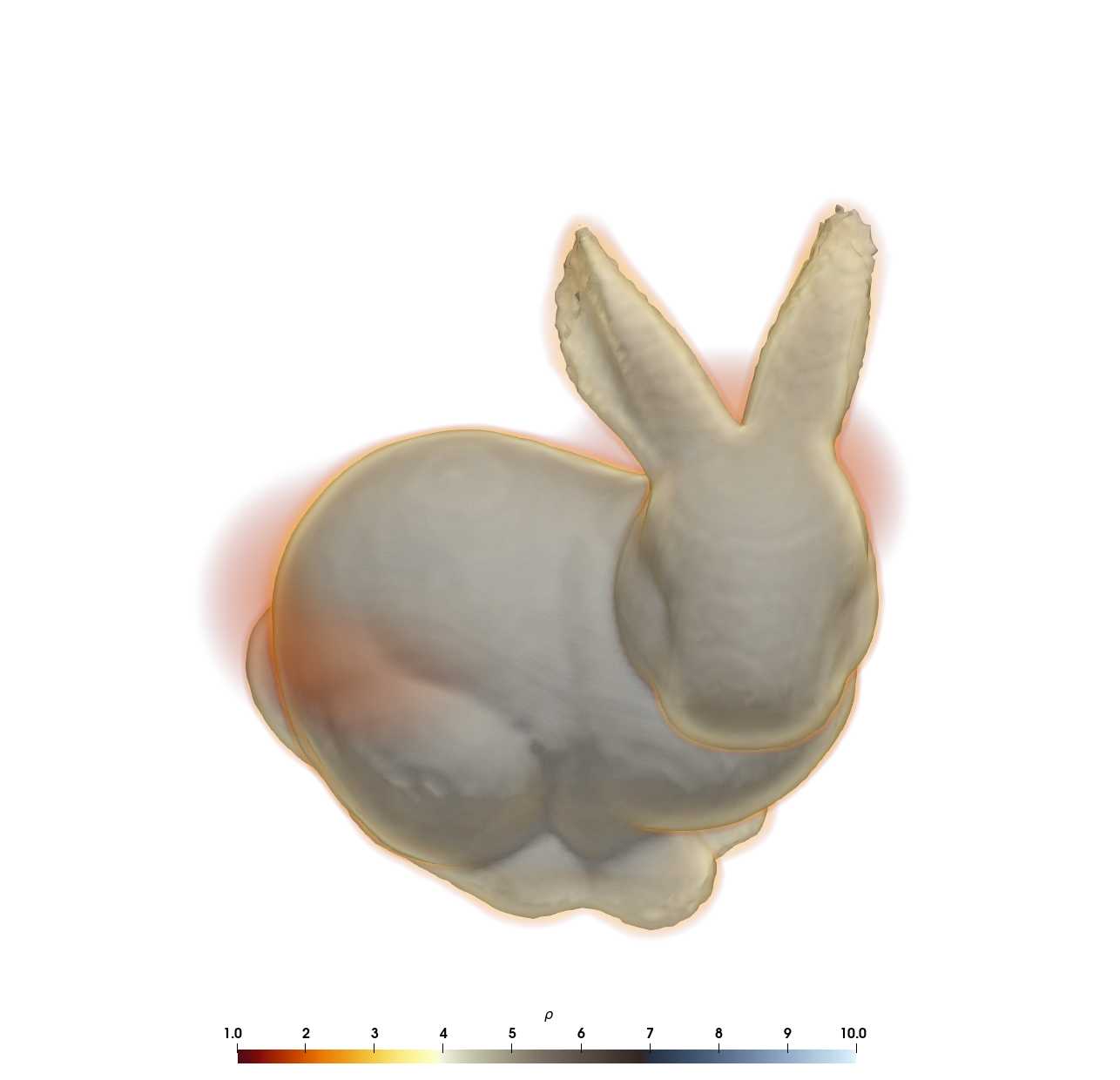}
    \end{subfigure}
    \begin{subfigure}{0.16\textwidth}
        \centering
        \includegraphics[width=\textwidth,trim=40mm 40mm 50mm 70mm, clip]{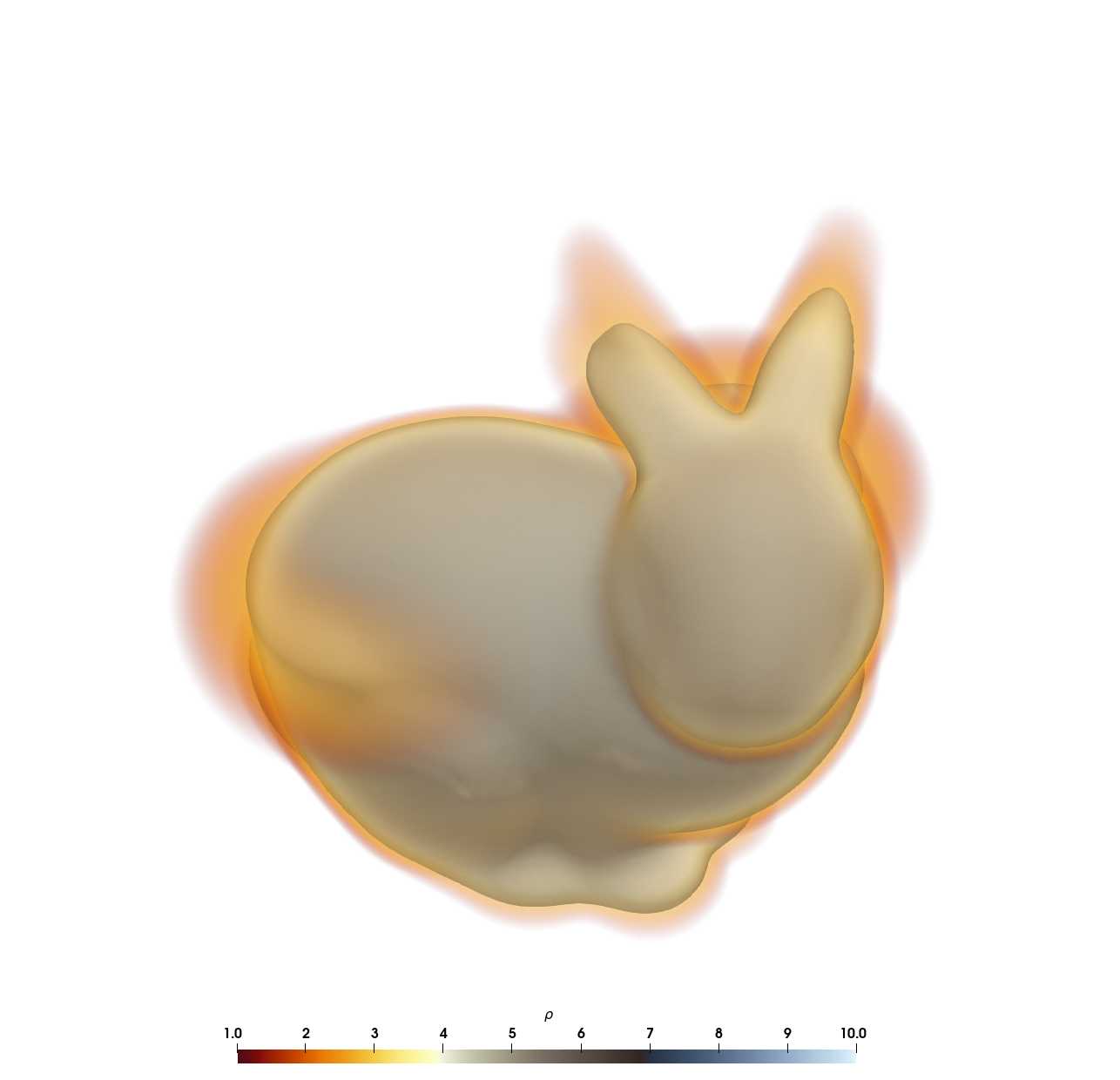}
    \end{subfigure}
    \begin{subfigure}{0.16\textwidth}
        \centering
        \includegraphics[width=\textwidth,trim=40mm 40mm 50mm 70mm, clip]{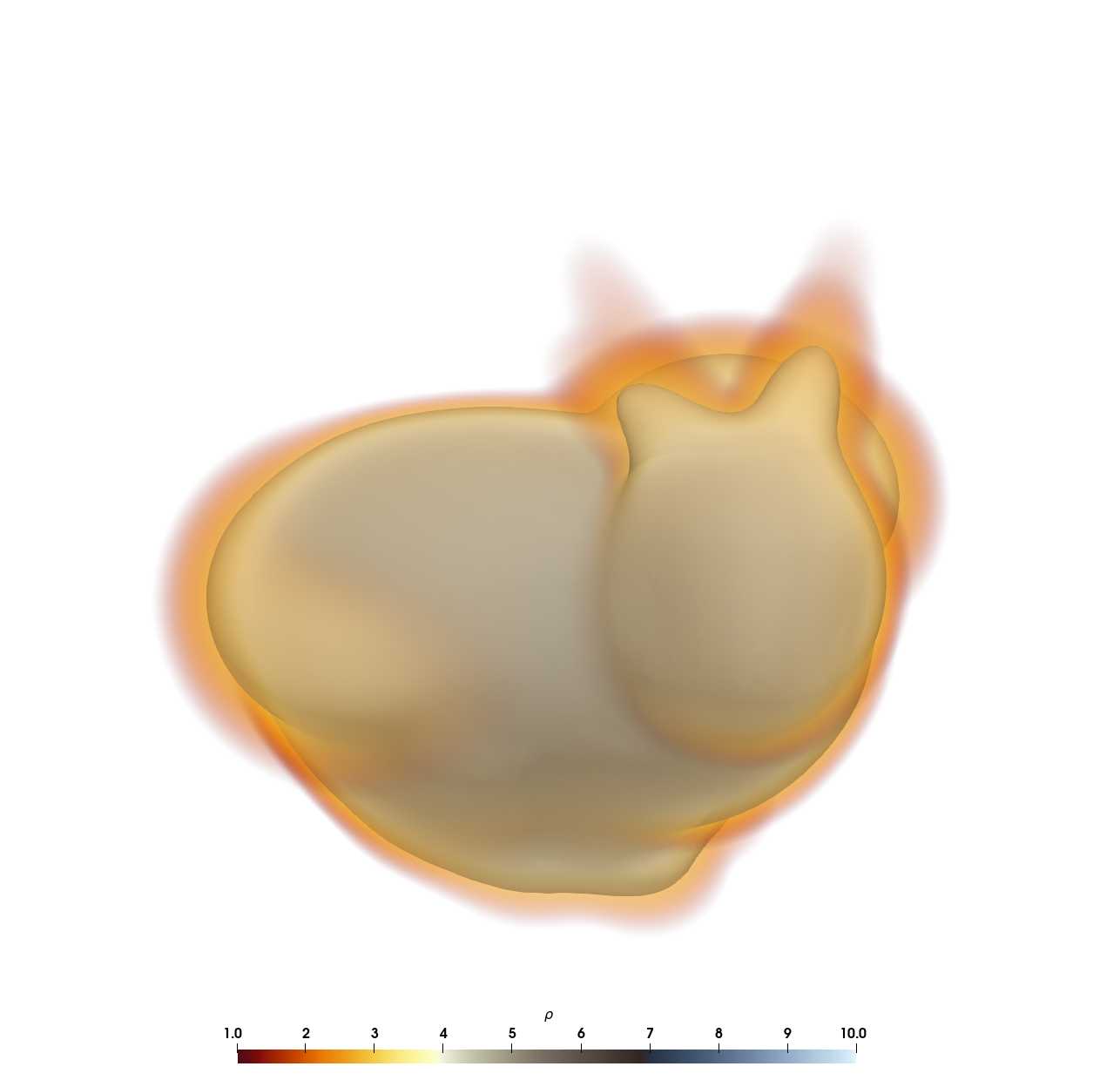}
    \end{subfigure}
    \begin{subfigure}{0.16\textwidth}
        \centering
        \includegraphics[width=\textwidth,trim=40mm 40mm 50mm 70mm, clip]{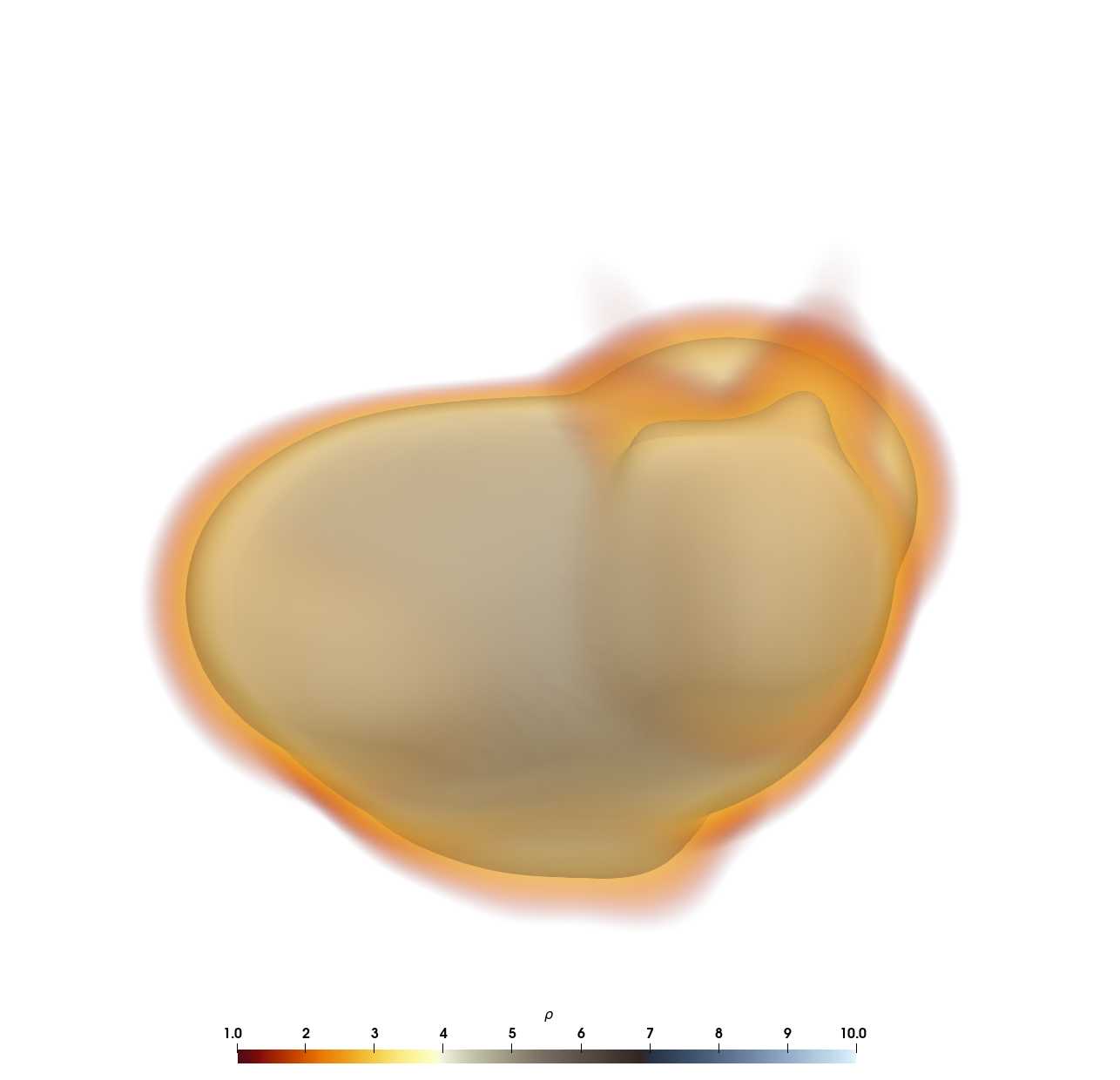}
    \end{subfigure}
    \begin{subfigure}{0.16\textwidth}
        \centering
        \includegraphics[width=\textwidth,trim=40mm 40mm 50mm 70mm, clip]{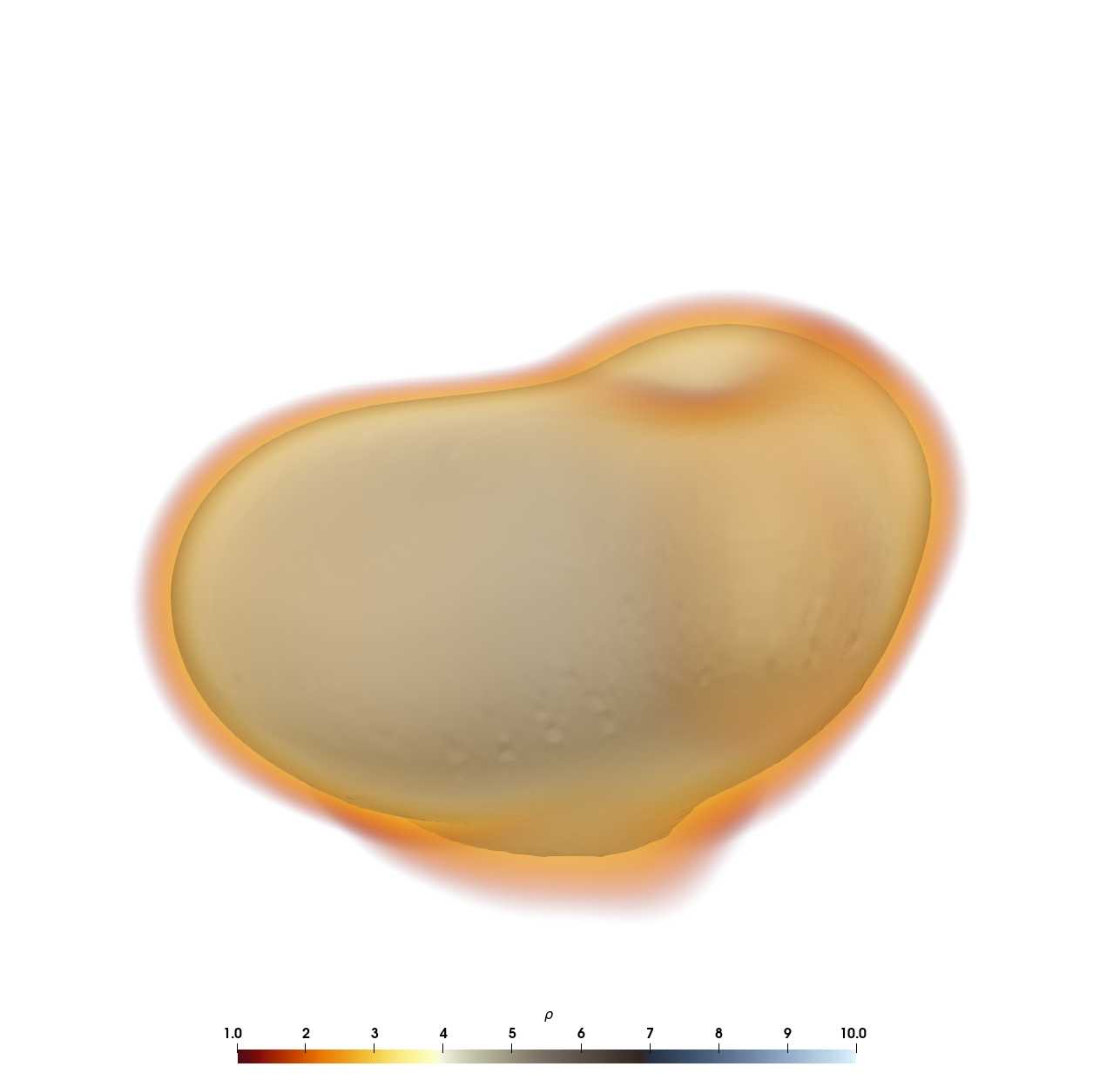}
    \end{subfigure}

    \begin{subfigure}{0.16\textwidth}
        \centering
        \includegraphics[width=\textwidth,trim=40mm 40mm 50mm 70mm, clip]{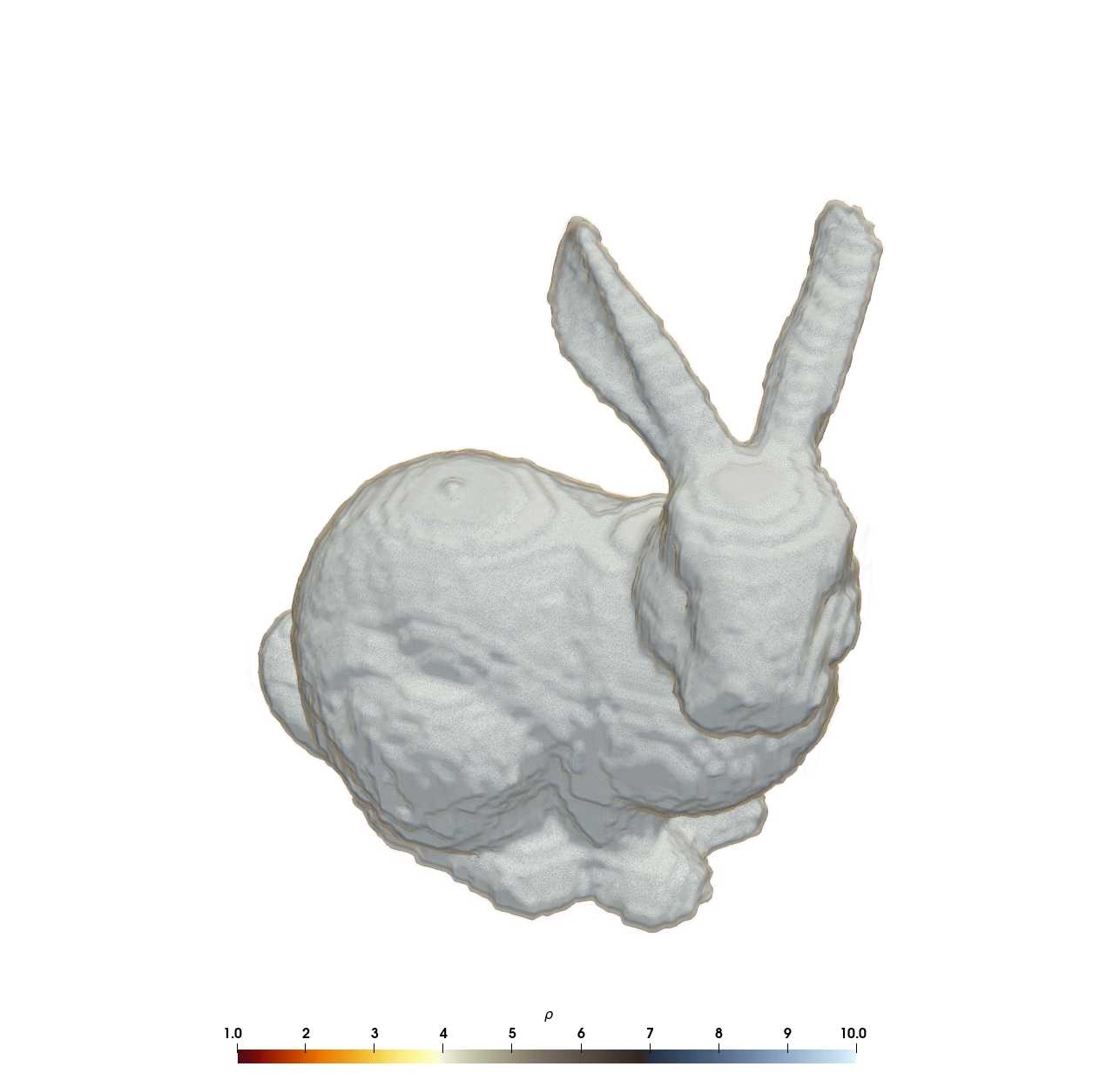}
        \caption*{$t=0$}
    \end{subfigure}
    \begin{subfigure}{0.16\textwidth}
        \centering
        \includegraphics[width=\textwidth,trim=40mm 40mm 50mm 70mm, clip]{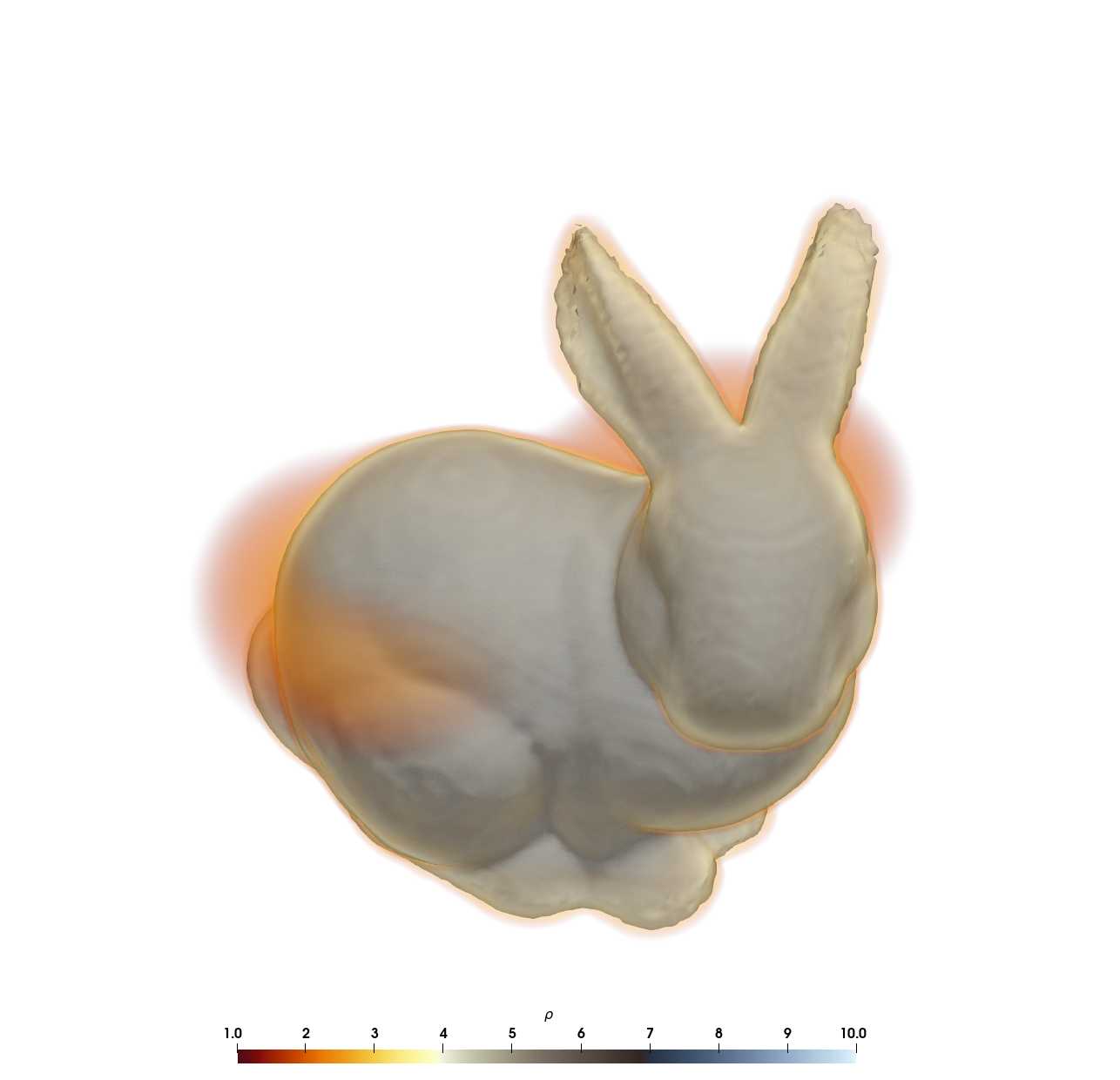}
        \caption*{$t=0.2$}
    \end{subfigure}
    \begin{subfigure}{0.16\textwidth}
        \centering
        \includegraphics[width=\textwidth,trim=40mm 40mm 50mm 70mm, clip]{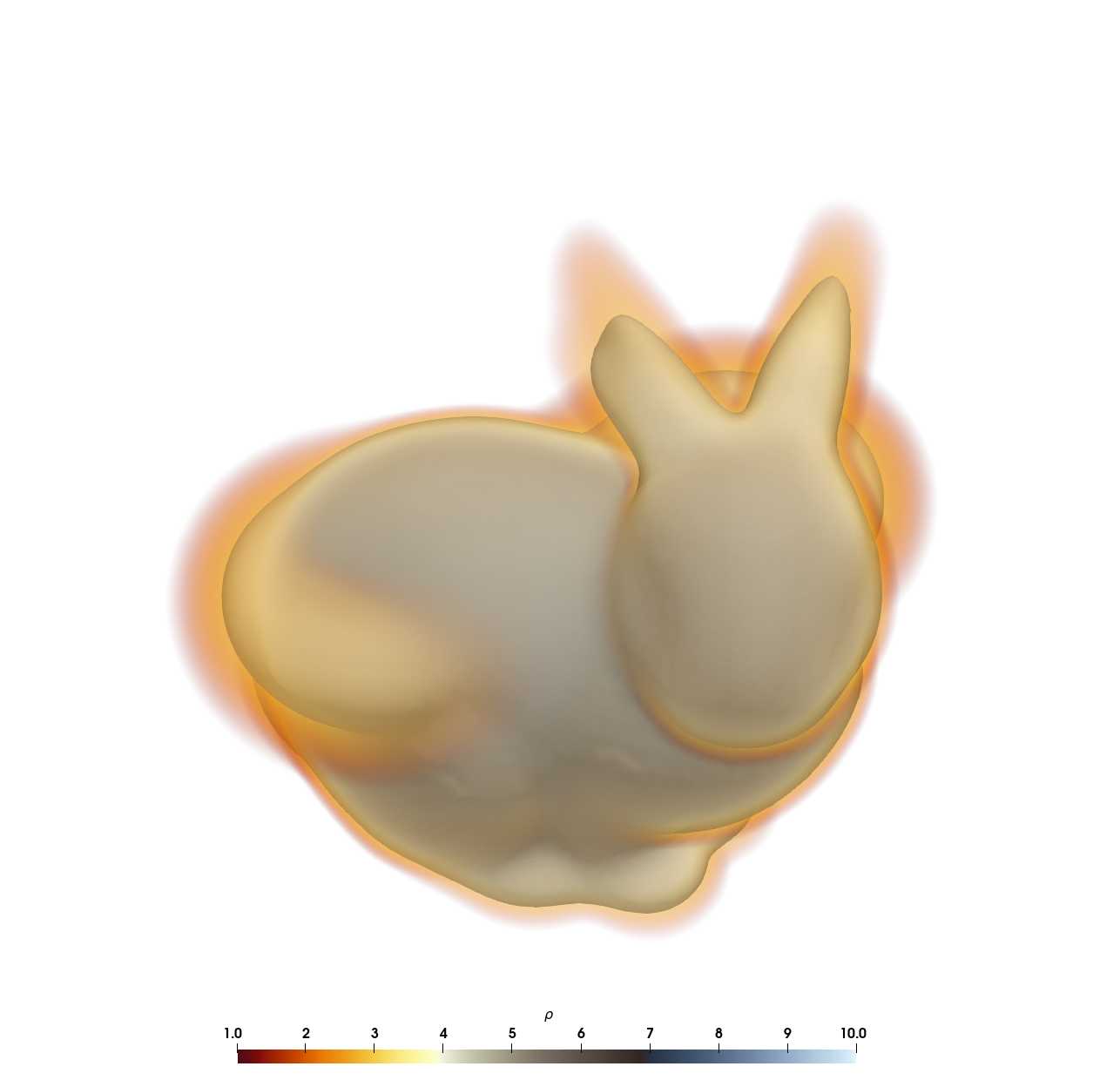}
        \caption*{$t=0.4$}
    \end{subfigure}
    \begin{subfigure}{0.16\textwidth}
        \centering
        \includegraphics[width=\textwidth,trim=40mm 40mm 50mm 70mm, clip]{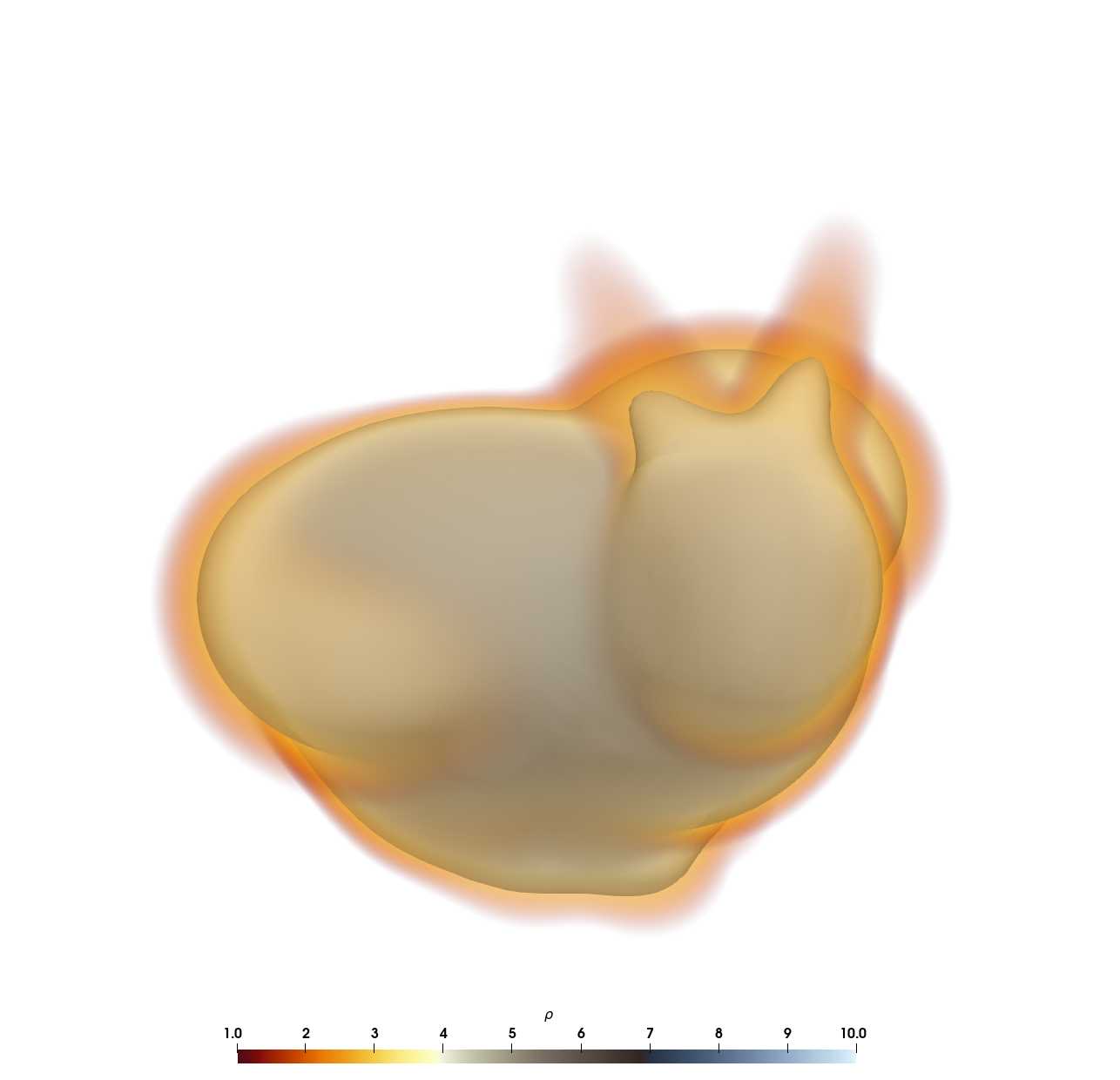}
        \caption*{$t=0.6$}
    \end{subfigure}
    \begin{subfigure}{0.16\textwidth}
        \centering
        \includegraphics[width=\textwidth,trim=40mm 40mm 50mm 70mm, clip]{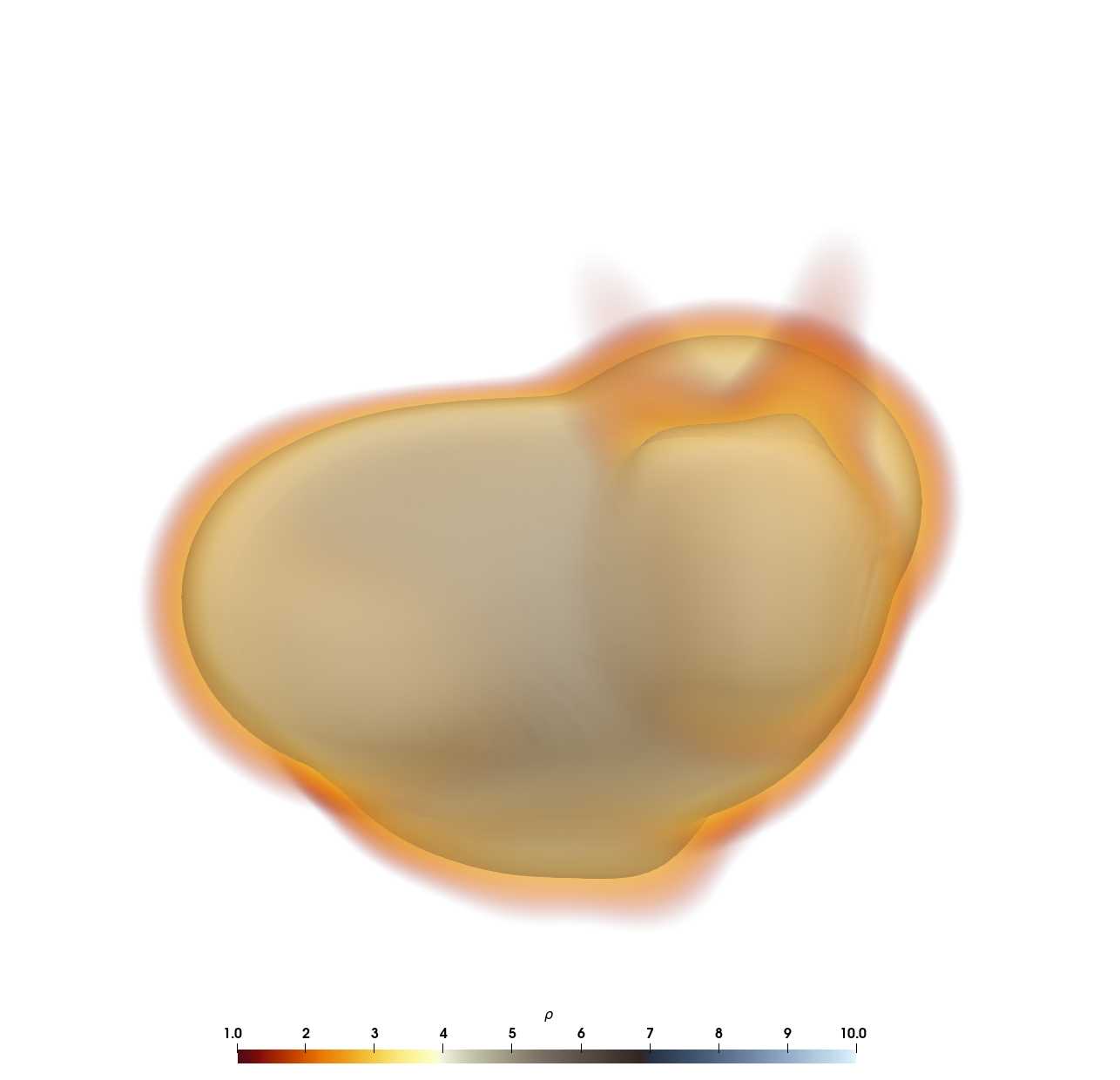}
        \caption*{$t=0.8$}
    \end{subfigure}
    \begin{subfigure}{0.16\textwidth}
        \centering
        \includegraphics[width=\textwidth,trim=40mm 40mm 50mm 70mm, clip]{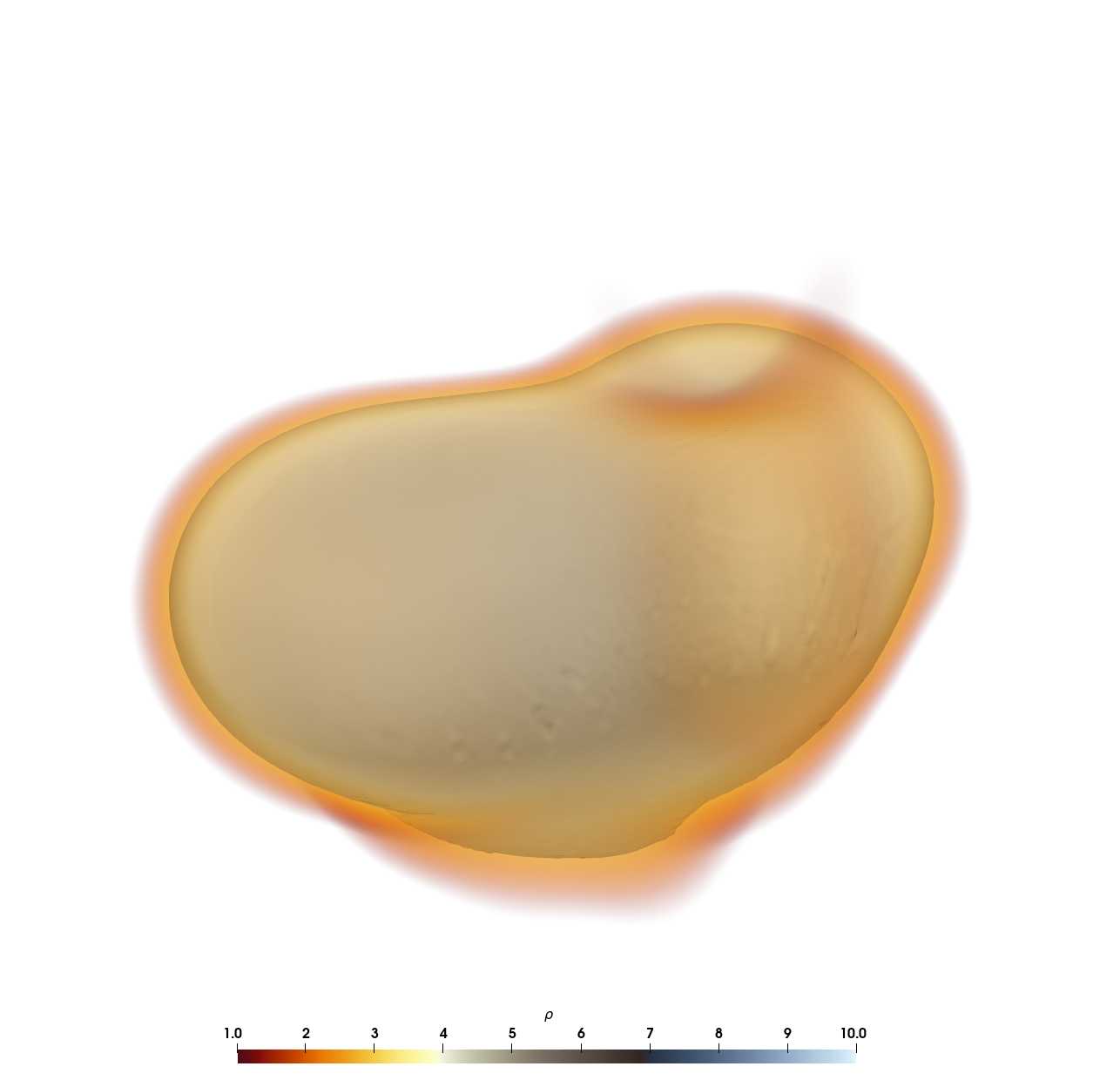}
        \caption*{$t=1.0$}
    \end{subfigure}
\subcaption{$\rho_2$}
\label{f:dtb2}
\end{minipage}
\caption{Shape interpolation in a \textit{double torus}-\textit{bunny} system. Left-right shows the time evolution of density. Top-down shows plots at different reaction strengths: $\alpha=0$, $\alpha=50,$ and $\alpha=100$.}
\label{f:dtb}
\end{figure}

\begin{figure}[h!]
    \centering
\begin{minipage}[b]{\textwidth}

    \begin{minipage}[b]{\textwidth}
    \hfill
        \begin{subfigure}{\textwidth}
            \centering
            \includegraphics[width=\textwidth,trim=0mm 0mm 0mm 400mm, clip]{figures/reaction/doubleTorusBunny/pdhg0.0000..jpg}
        \end{subfigure}
    \end{minipage}
    
    \begin{subfigure}{0.16\textwidth}
        \centering
        \includegraphics[width=\textwidth,trim=30mm 40mm 30mm 70mm, clip]{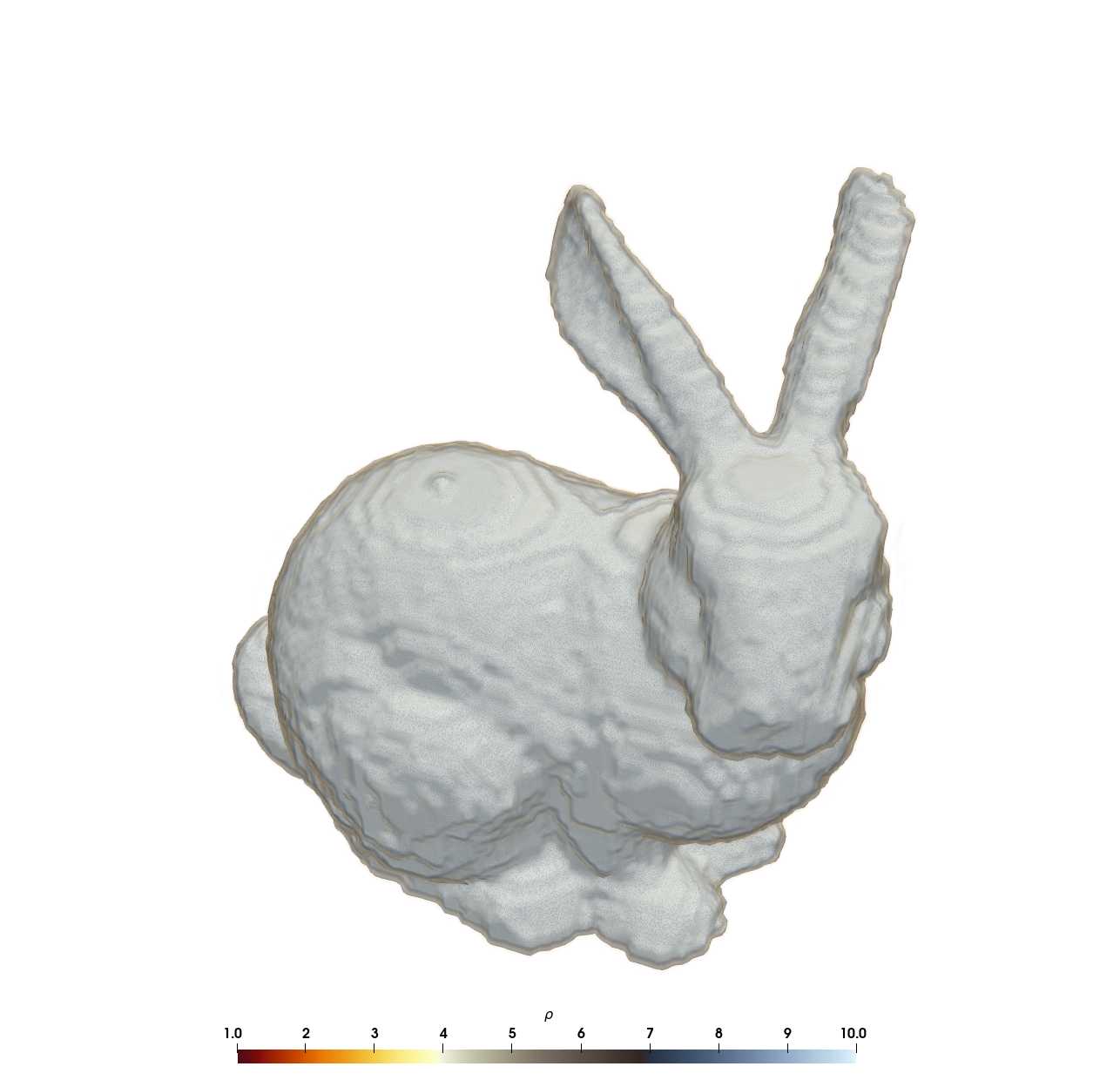}
    \end{subfigure}
    \begin{subfigure}{0.16\textwidth}
        \centering
        \includegraphics[width=\textwidth,trim=30mm 40mm 30mm 70mm, clip]{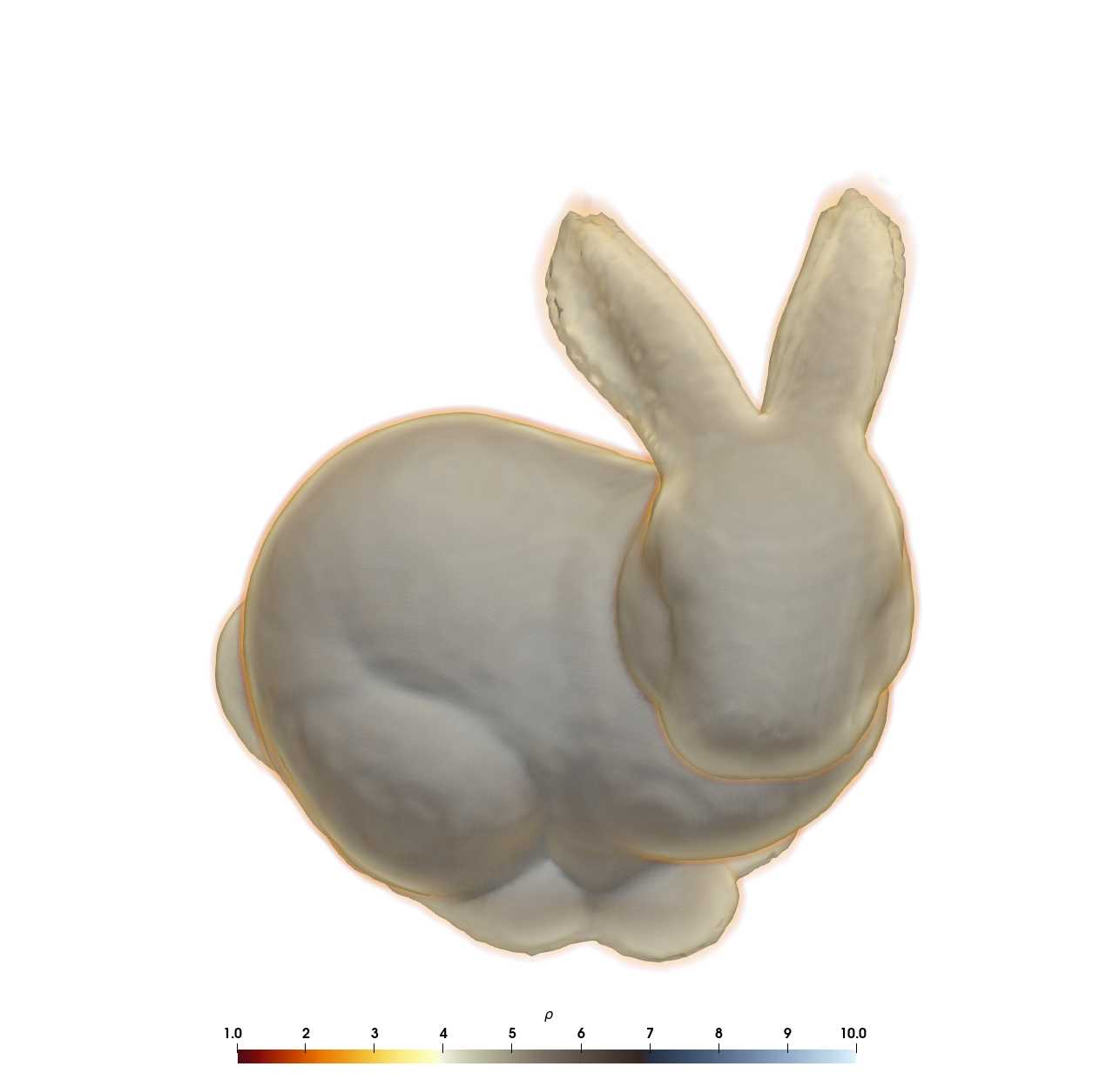}
    \end{subfigure}
    \begin{subfigure}{0.16\textwidth}
        \centering
        \includegraphics[width=\textwidth,trim=30mm 40mm 30mm 70mm, clip]{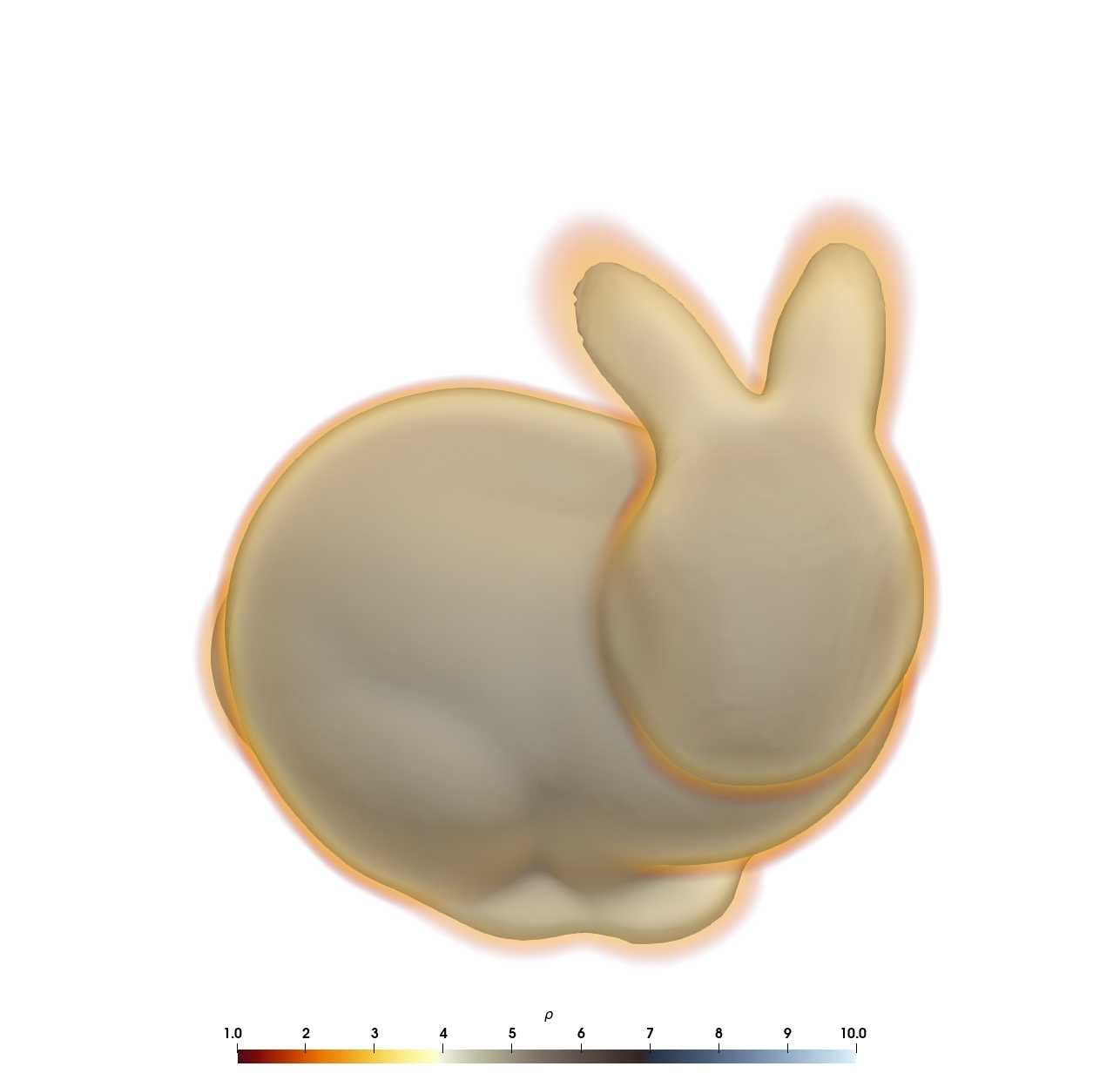}
    \end{subfigure}
    \begin{subfigure}{0.16\textwidth}
        \centering
        \includegraphics[width=\textwidth,trim=30mm 40mm 30mm 70mm, clip]{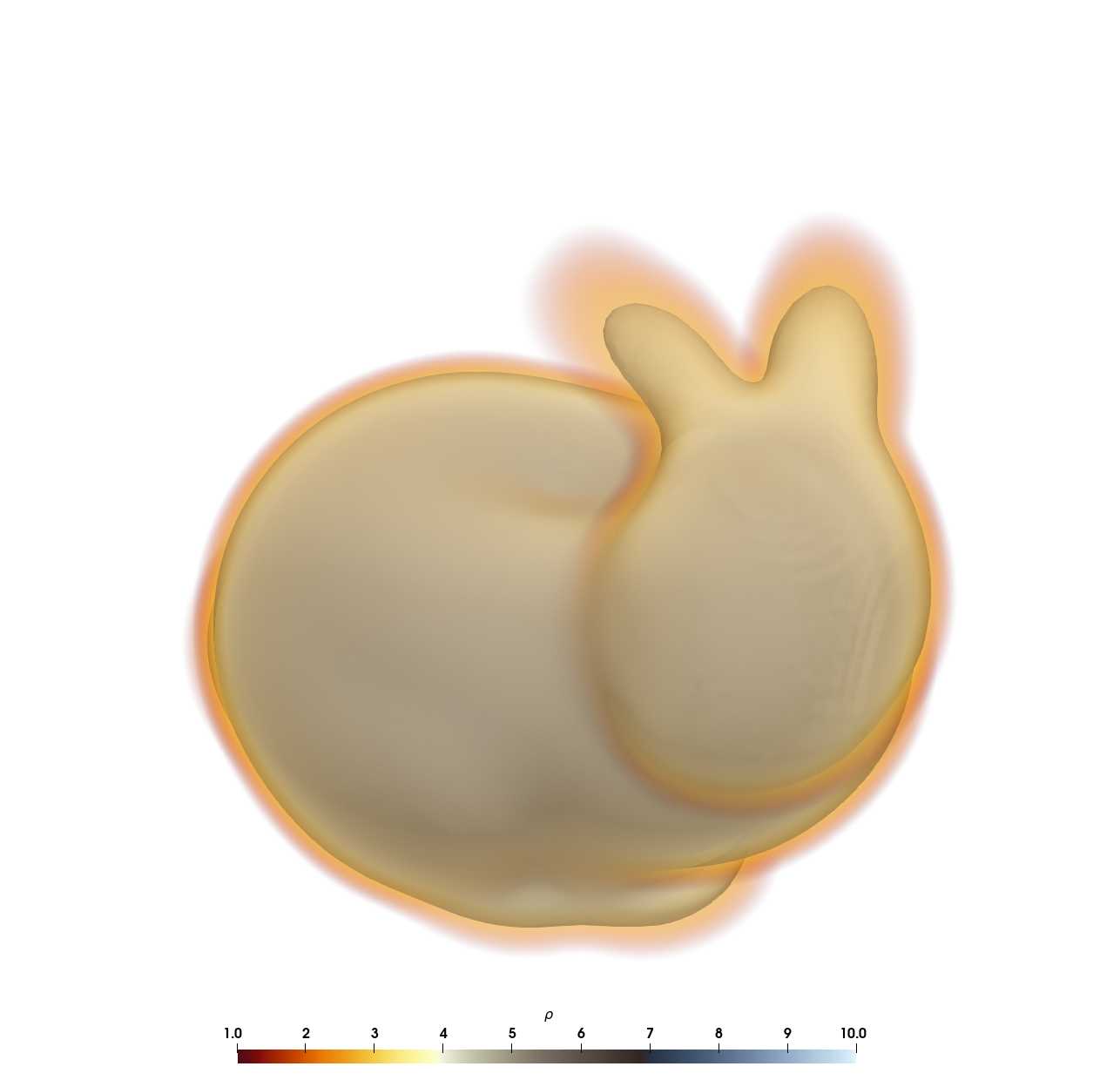}
    \end{subfigure}
    \begin{subfigure}{0.16\textwidth}
        \centering
        \includegraphics[width=\textwidth,trim=30mm 40mm 30mm 70mm, clip]{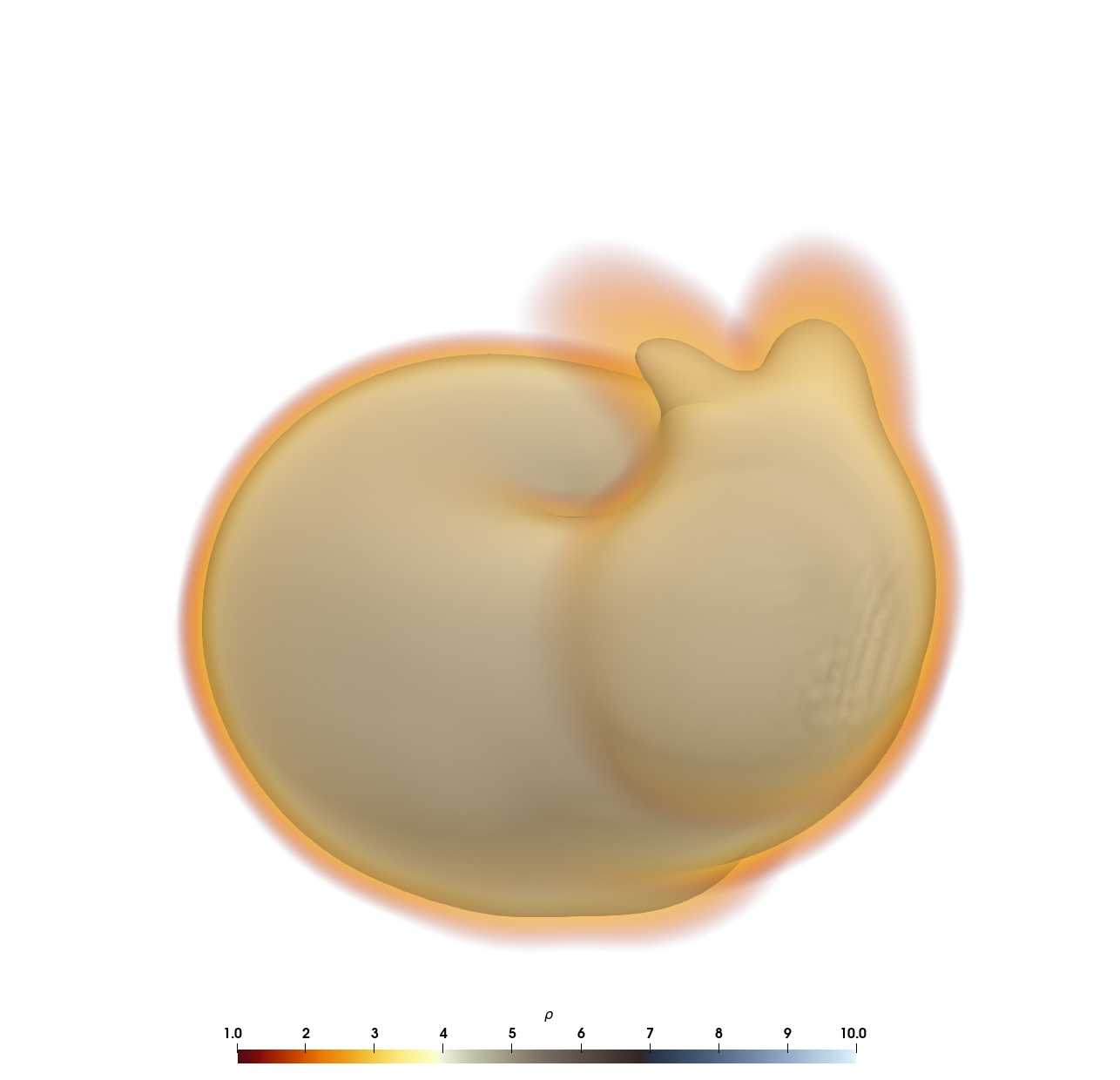}
    \end{subfigure}
    \begin{subfigure}{0.16\textwidth}
        \centering
        \includegraphics[width=\textwidth,trim=30mm 40mm 30mm 70mm, clip]{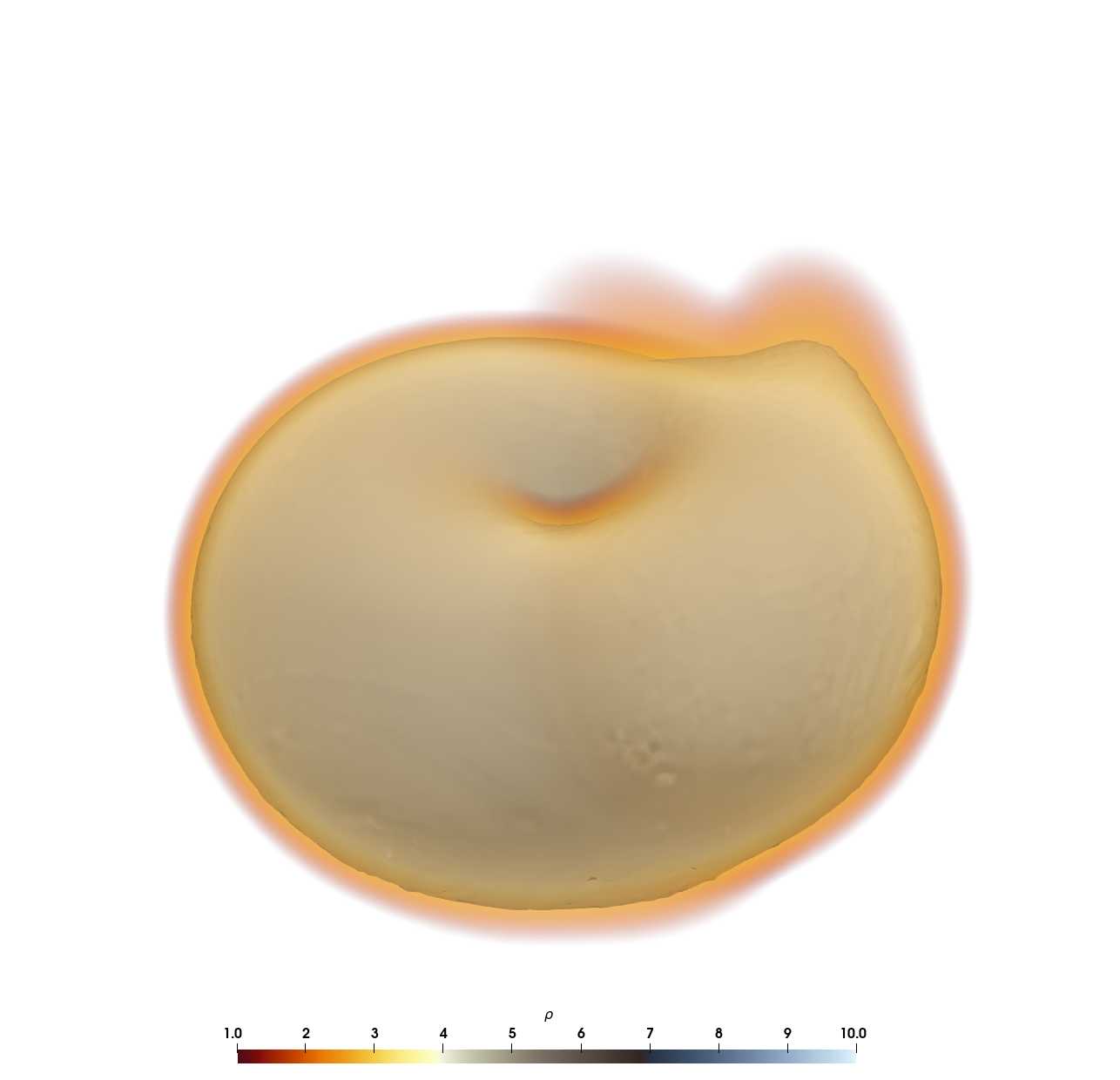}
    \end{subfigure}
    
    \begin{subfigure}{0.16\textwidth}
        \centering
        \includegraphics[width=\textwidth,trim=30mm 40mm 30mm 70mm, clip]{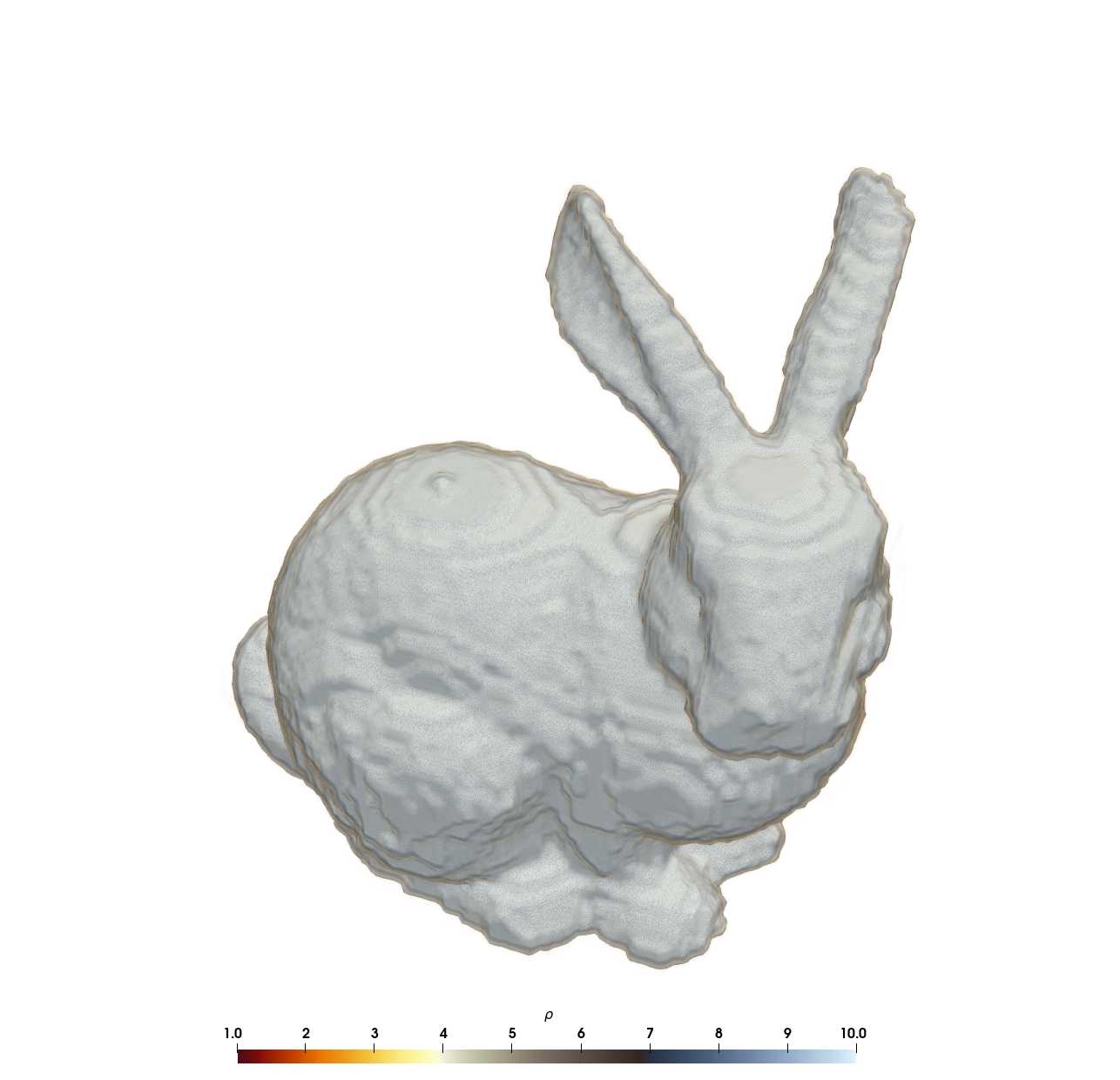}
    \end{subfigure}
    \begin{subfigure}{0.16\textwidth}
        \centering
        \includegraphics[width=\textwidth,trim=30mm 40mm 30mm 70mm, clip]{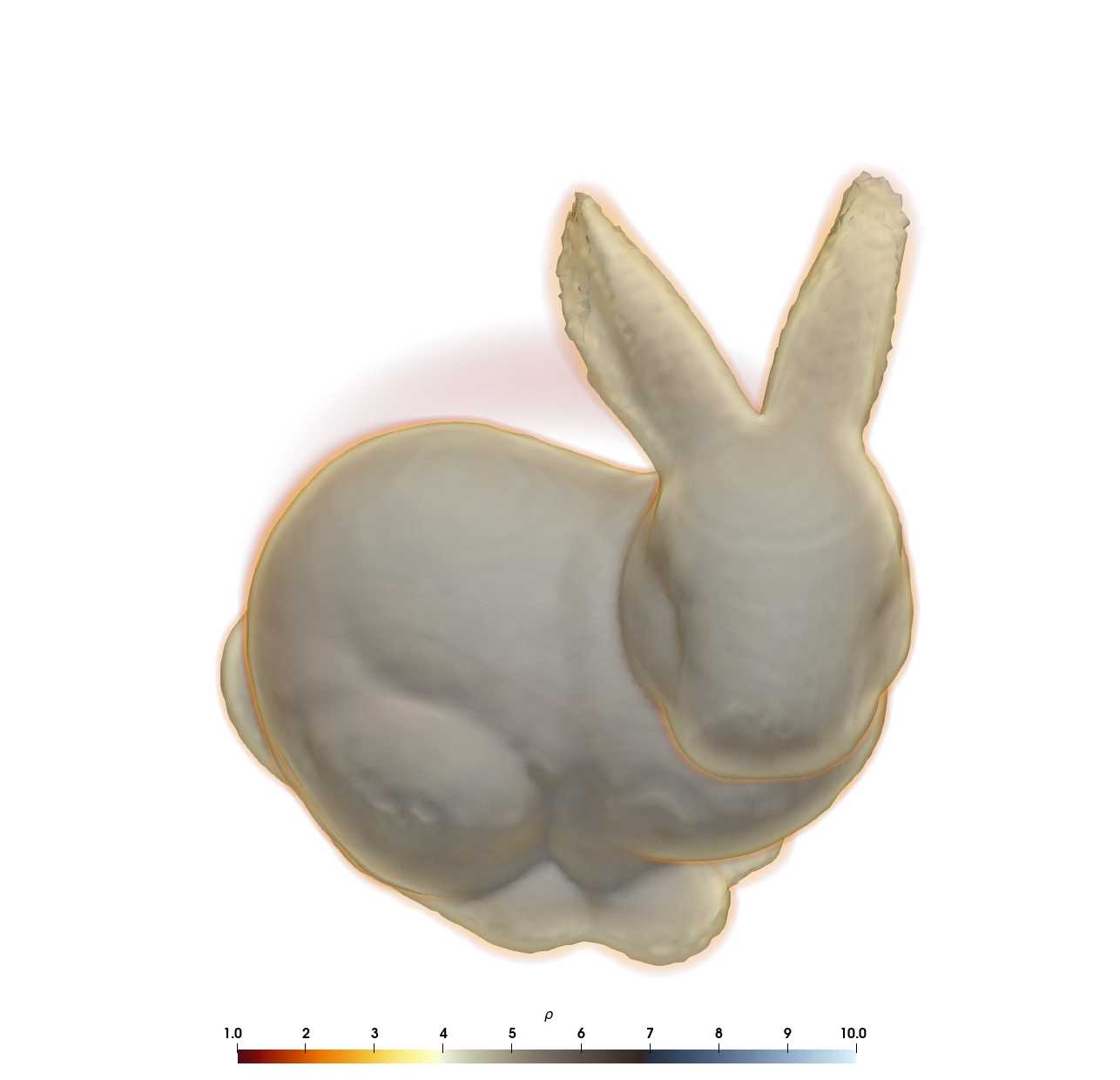}
    \end{subfigure}
    \begin{subfigure}{0.16\textwidth}
        \centering
        \includegraphics[width=\textwidth,trim=30mm 40mm 30mm 70mm, clip]{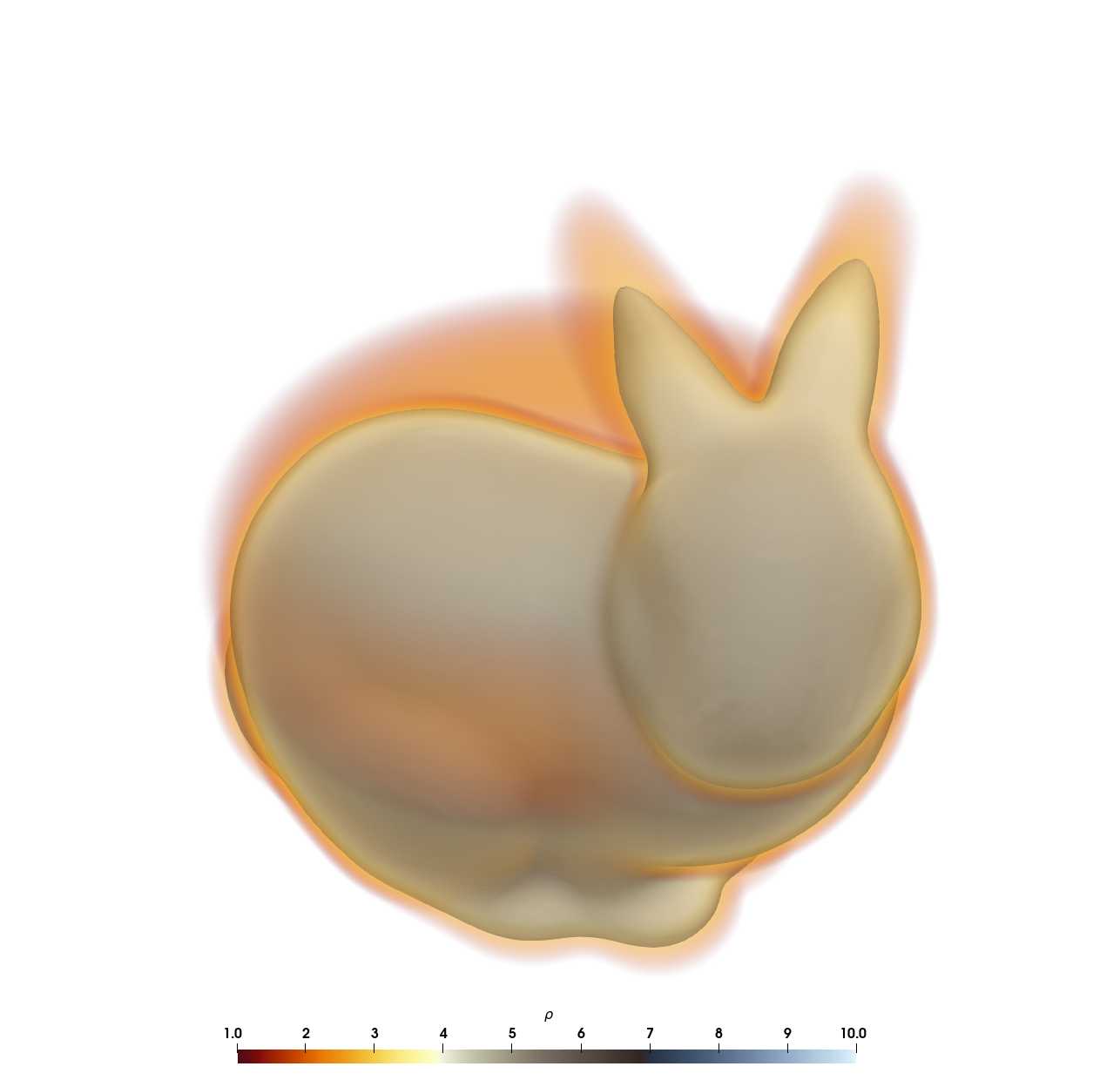}
    \end{subfigure}
    \begin{subfigure}{0.16\textwidth}
        \centering
        \includegraphics[width=\textwidth,trim=30mm 40mm 30mm 70mm, clip]{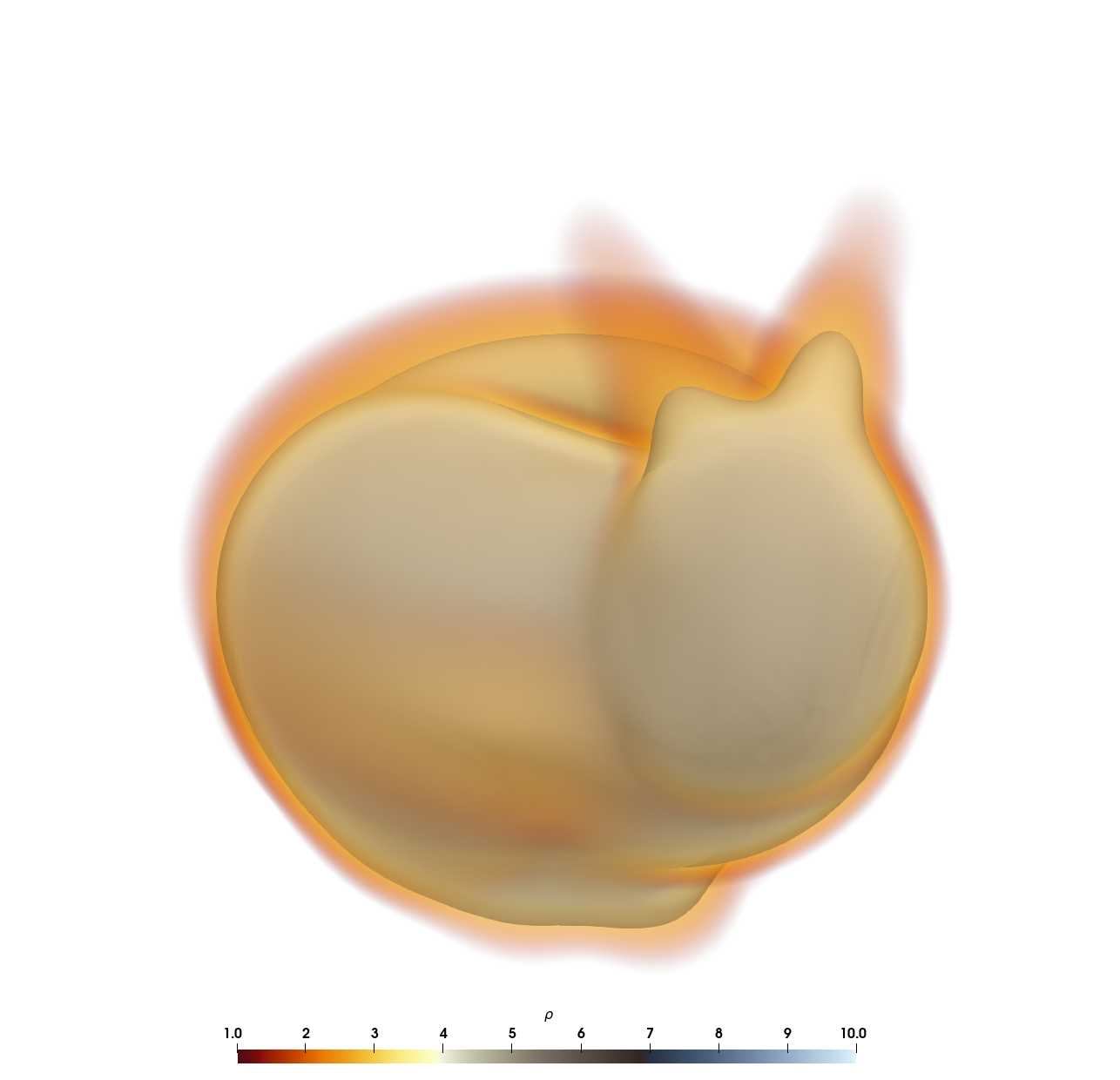}
    \end{subfigure}
    \begin{subfigure}{0.16\textwidth}
        \centering
        \includegraphics[width=\textwidth,trim=30mm 40mm 30mm 70mm, clip]{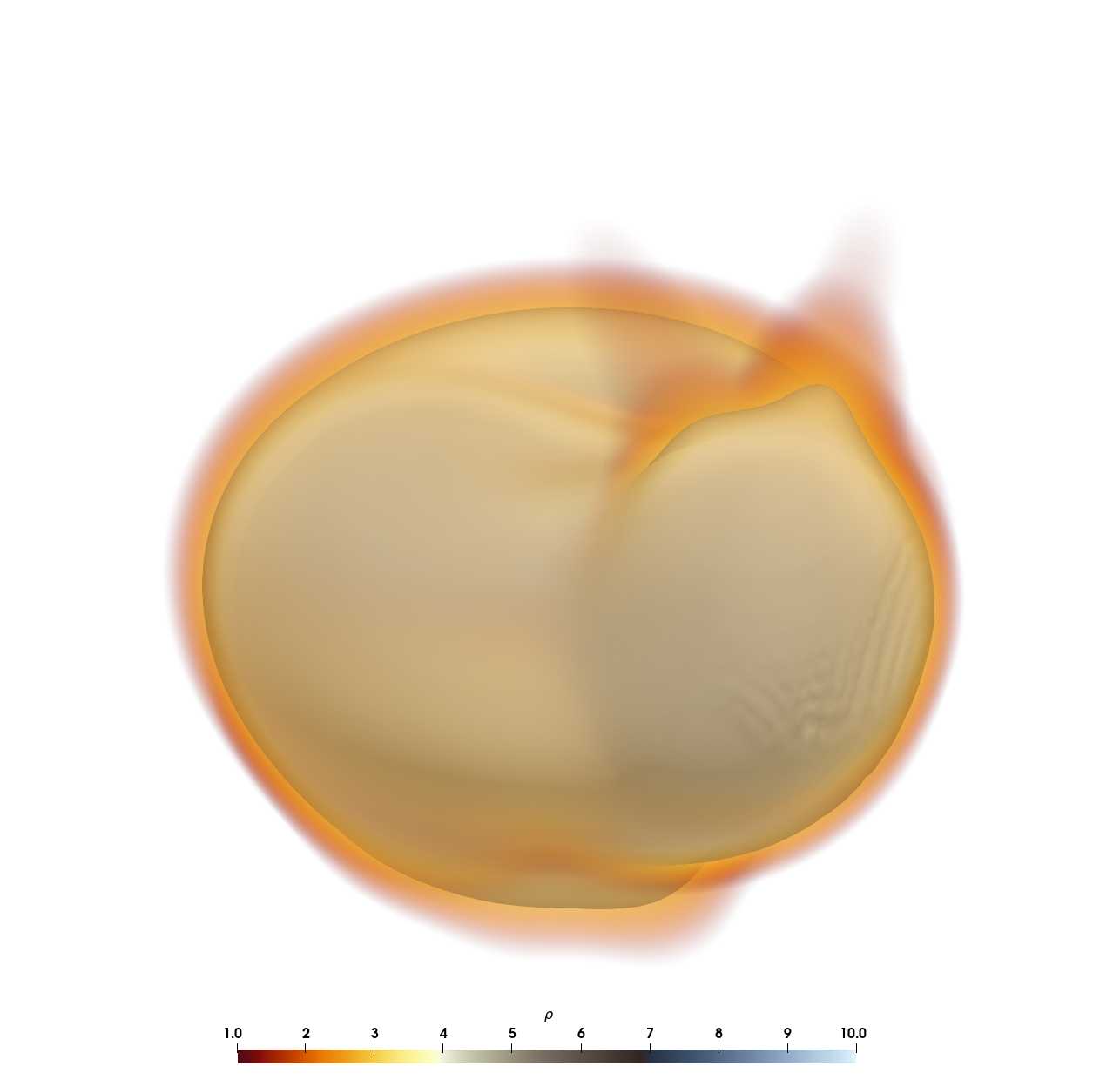}
    \end{subfigure}
    \begin{subfigure}{0.16\textwidth}
        \centering
        \includegraphics[width=\textwidth,trim=30mm 40mm 30mm 70mm, clip]{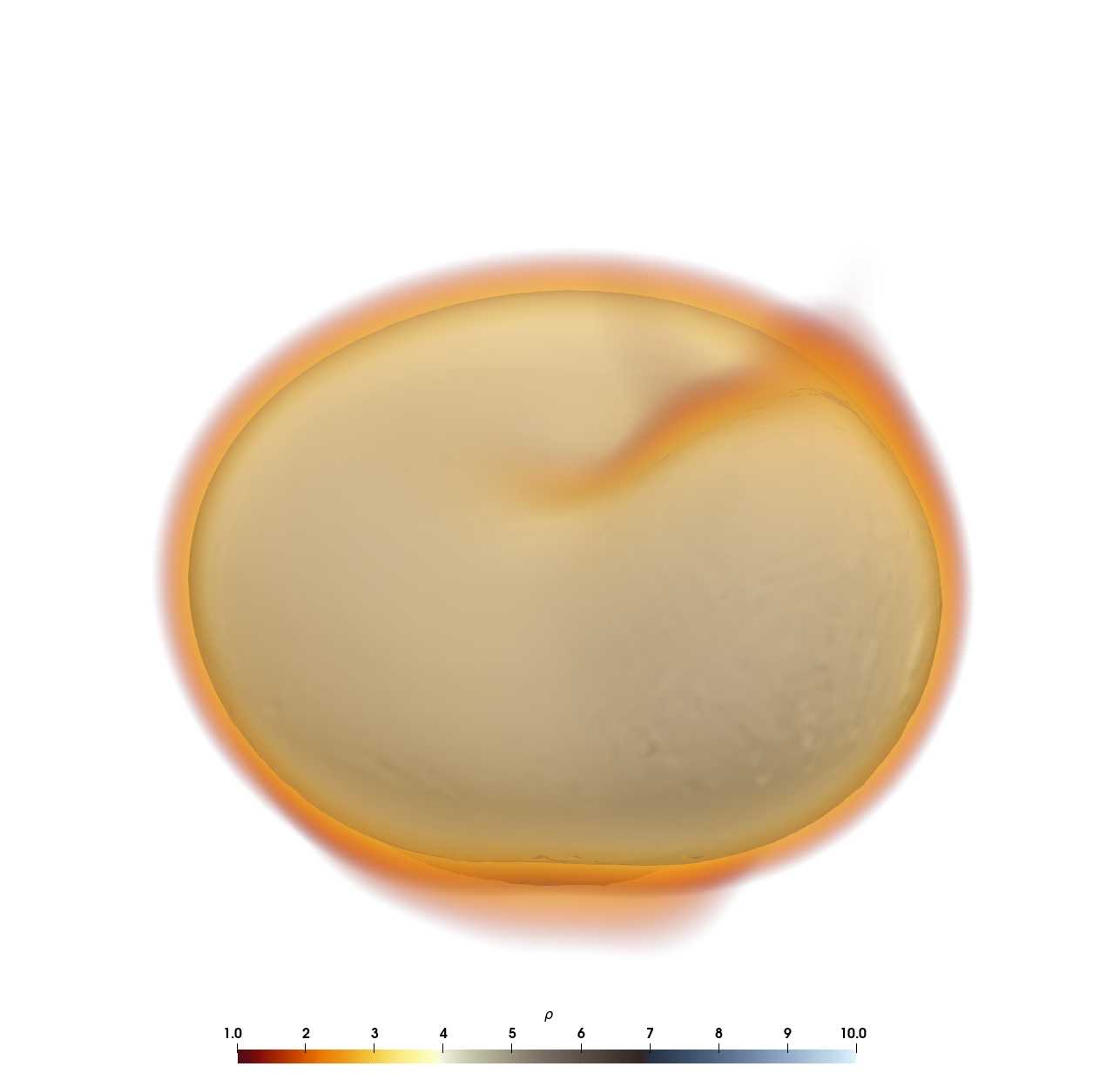}
    \end{subfigure}

    \begin{subfigure}{0.16\textwidth}
        \centering
        \includegraphics[width=\textwidth,trim=30mm 40mm 30mm 70mm, clip]{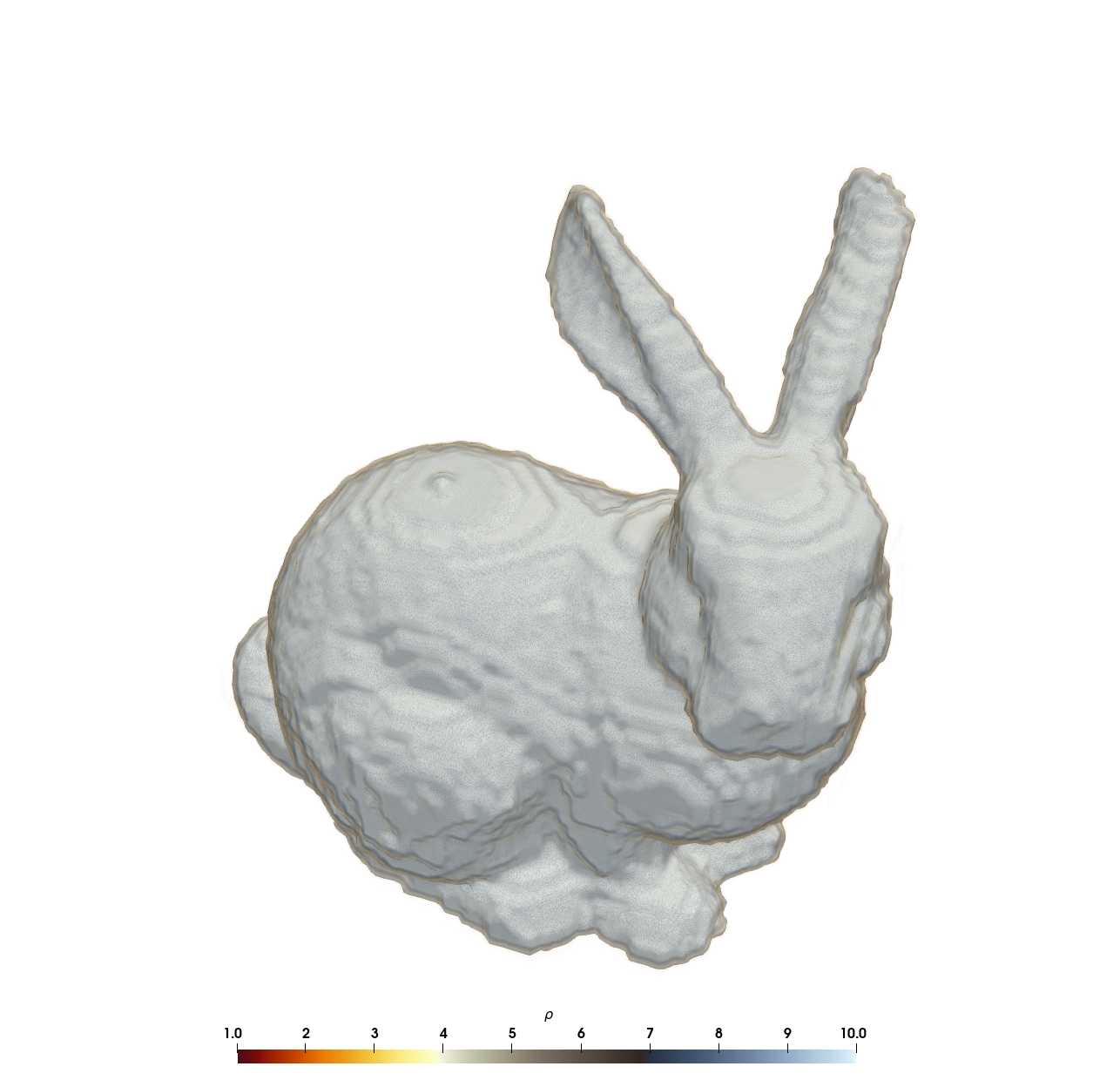}
        \caption*{$t=0$}
    \end{subfigure}
    \begin{subfigure}{0.16\textwidth}
        \centering
        \includegraphics[width=\textwidth,trim=30mm 40mm 30mm 70mm, clip]{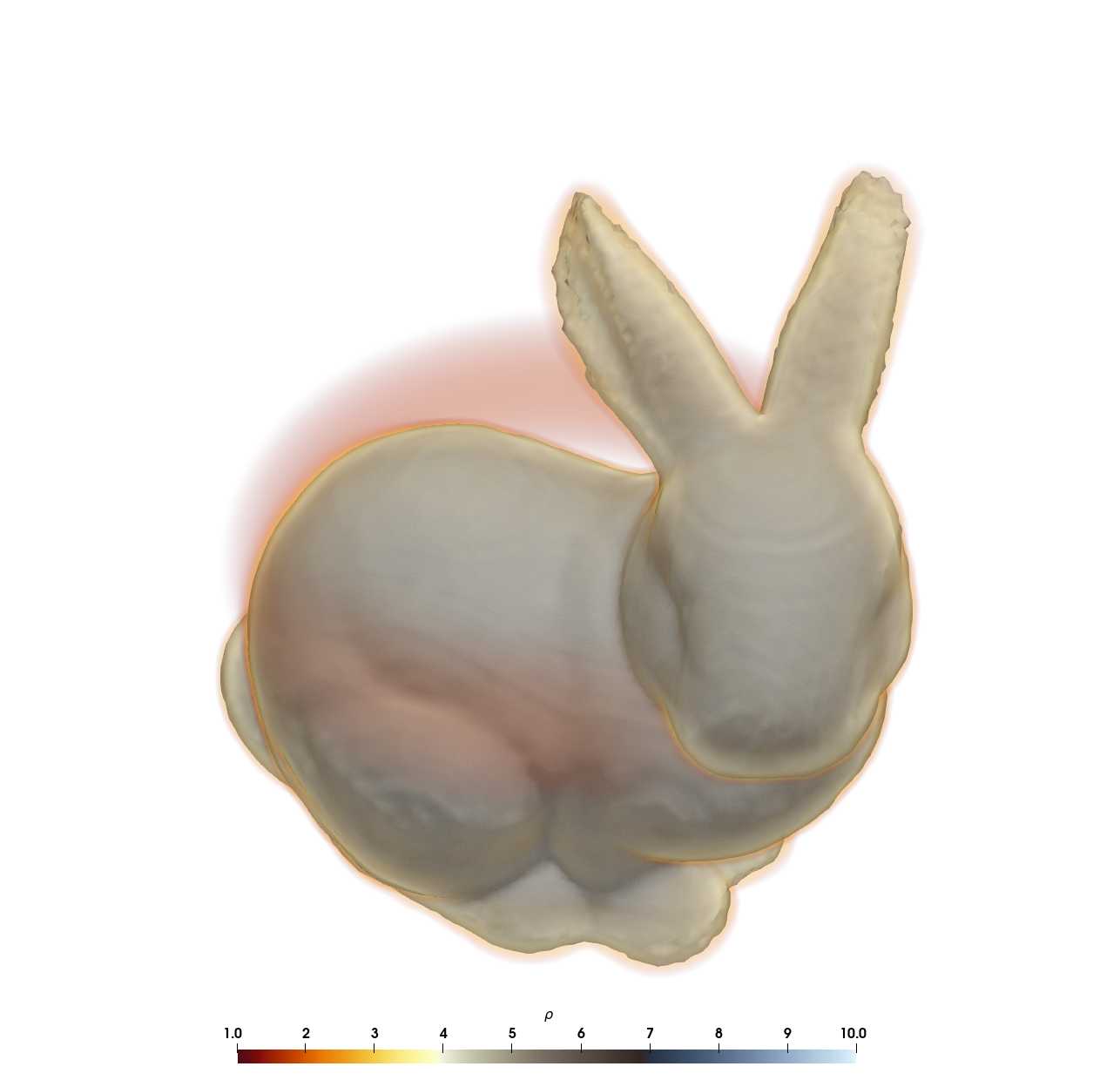}
        \caption*{$t=0.2$}
    \end{subfigure}
    \begin{subfigure}{0.16\textwidth}
        \centering
        \includegraphics[width=\textwidth,trim=30mm 40mm 30mm 70mm, clip]{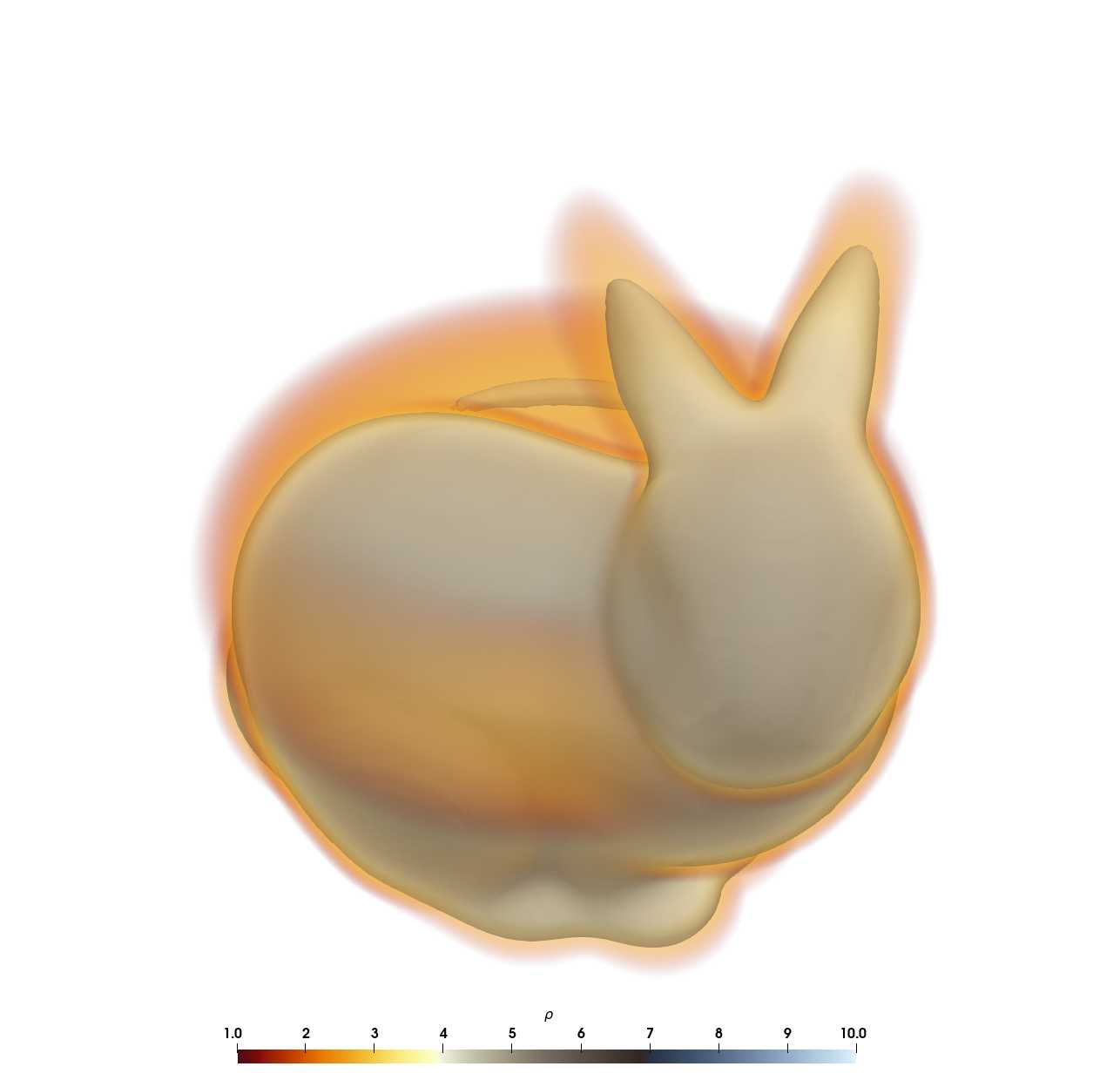}
        \caption*{$t=0.4$}
    \end{subfigure}
    \begin{subfigure}{0.16\textwidth}
        \centering
        \includegraphics[width=\textwidth,trim=30mm 40mm 30mm 70mm, clip]{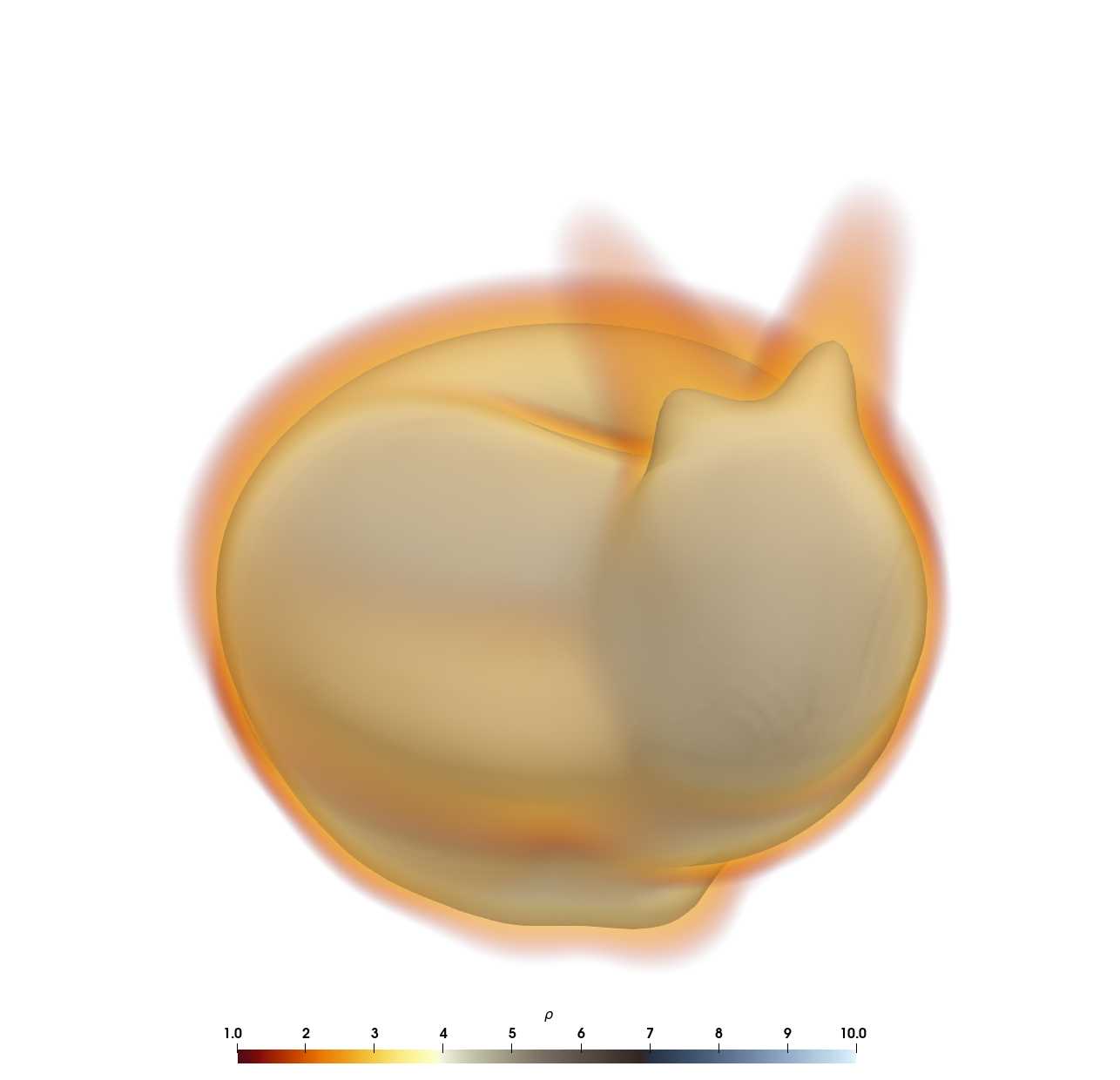}
        \caption*{$t=0.6$}
    \end{subfigure}
    \begin{subfigure}{0.16\textwidth}
        \centering
        \includegraphics[width=\textwidth,trim=30mm 40mm 30mm 70mm, clip]{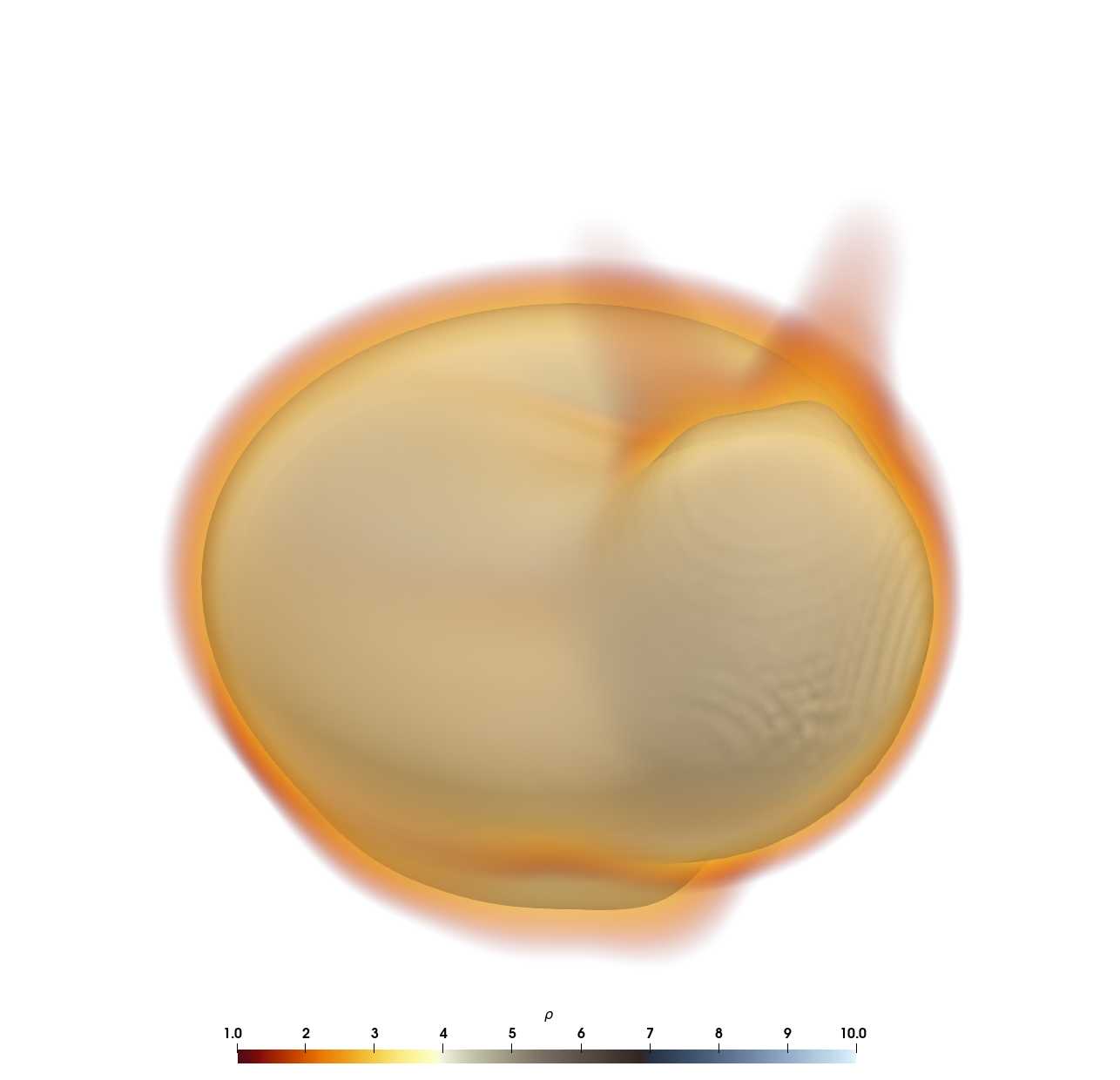}
        \caption*{$t=0.8$}
    \end{subfigure}
    \begin{subfigure}{0.16\textwidth}
        \centering
        \includegraphics[width=\textwidth,trim=30mm 40mm 30mm 70mm, clip]{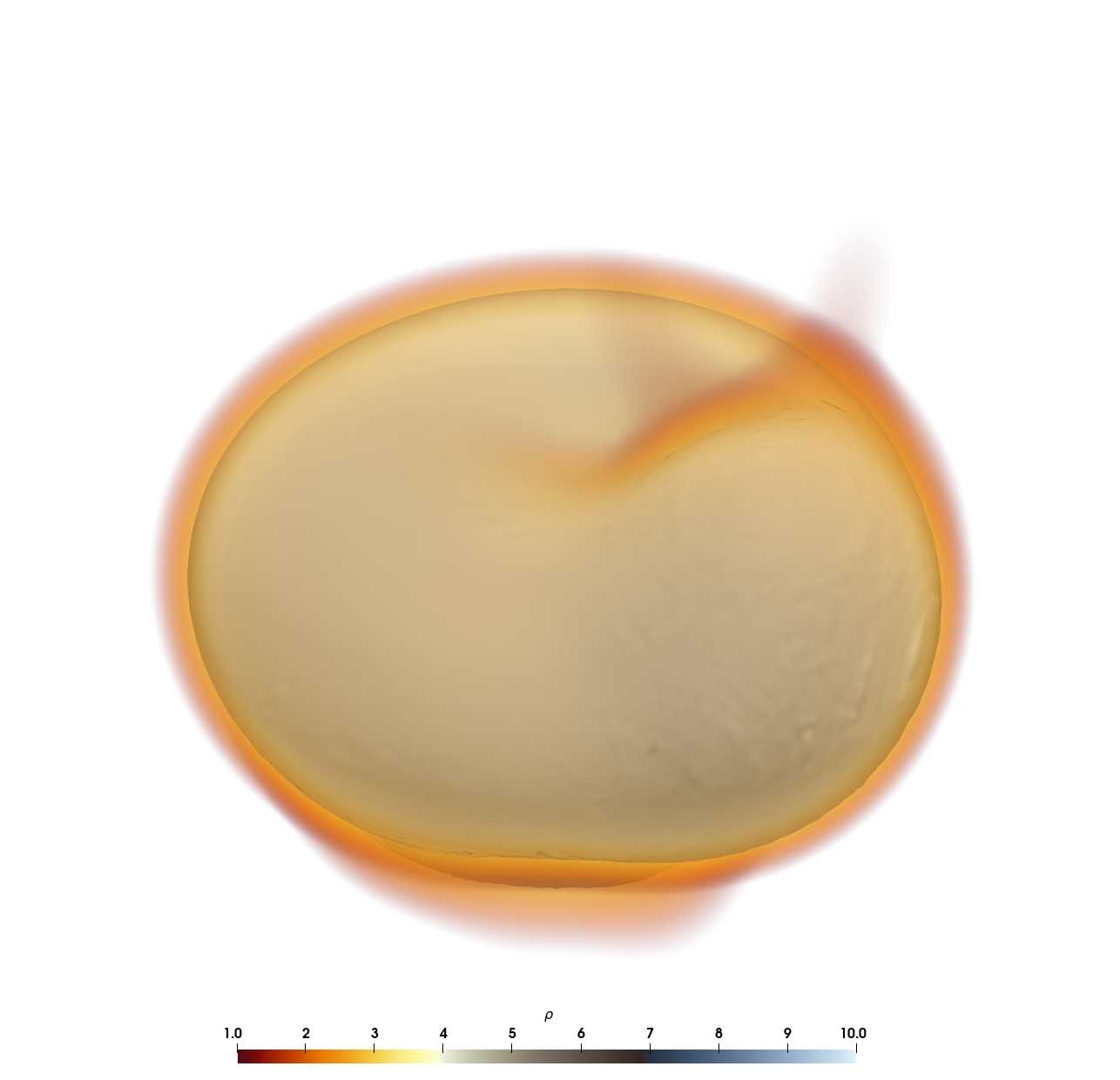}
        \caption*{$t=1.0$}
    \end{subfigure}

\subcaption{$\rho_1$}
\label{f:tb1}
\end{minipage}
\begin{minipage}[b]{\textwidth}

    \begin{minipage}[b]{\textwidth}
    \hfill
        \begin{subfigure}{\textwidth}
            \centering
            \includegraphics[width=\textwidth,trim=0mm 0mm 0mm 400mm, clip]{figures/reaction/doubleTorusBunny/pdhg0.0000..jpg}
        \end{subfigure}
    \end{minipage}
    
    \begin{subfigure}{0.16\textwidth}
        \centering
        \includegraphics[width=\textwidth,trim=30mm 40mm 30mm 70mm, clip]{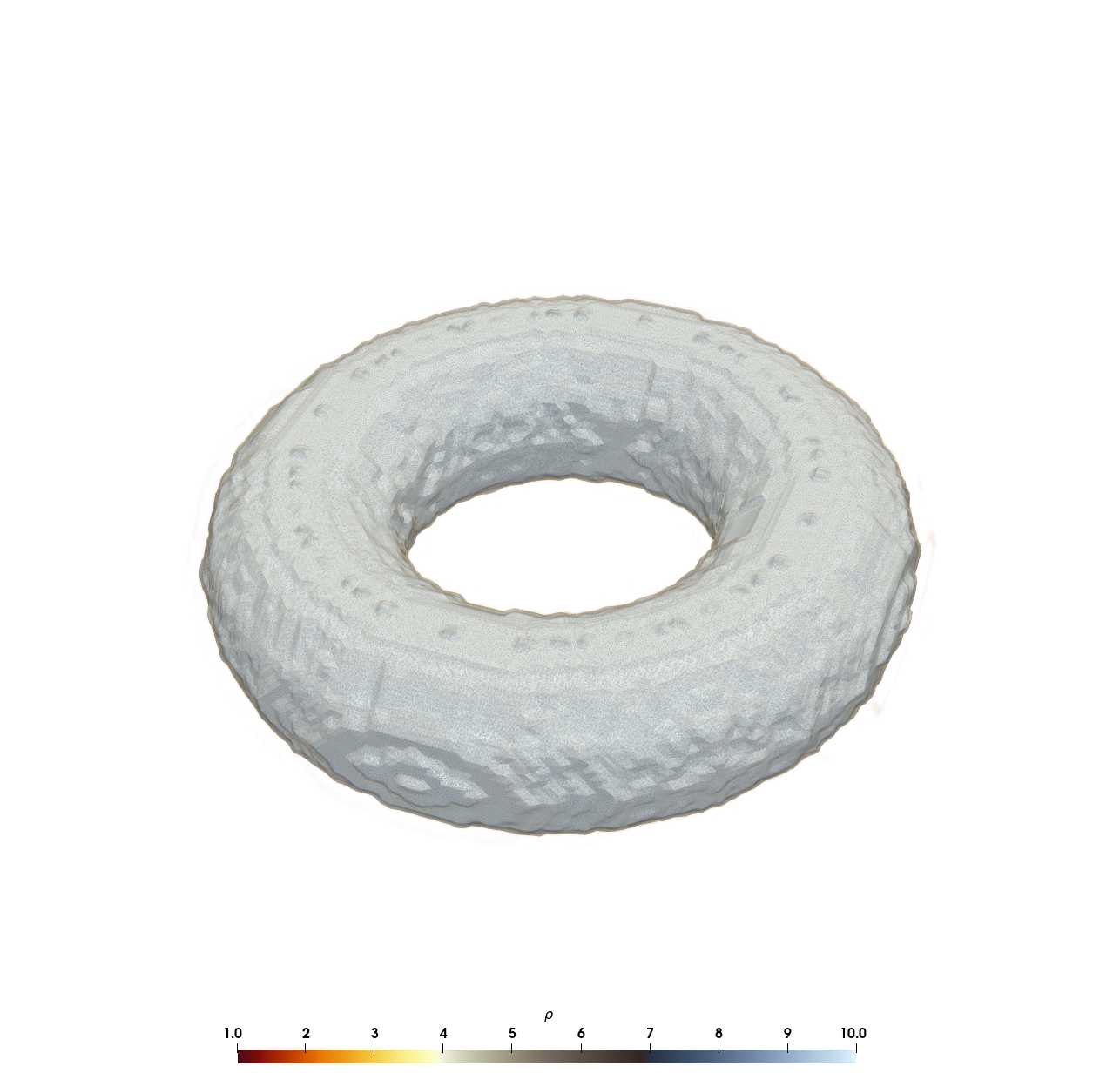}
    \end{subfigure}
    \begin{subfigure}{0.16\textwidth}
        \centering
        \includegraphics[width=\textwidth,trim=30mm 40mm 30mm 70mm, clip]{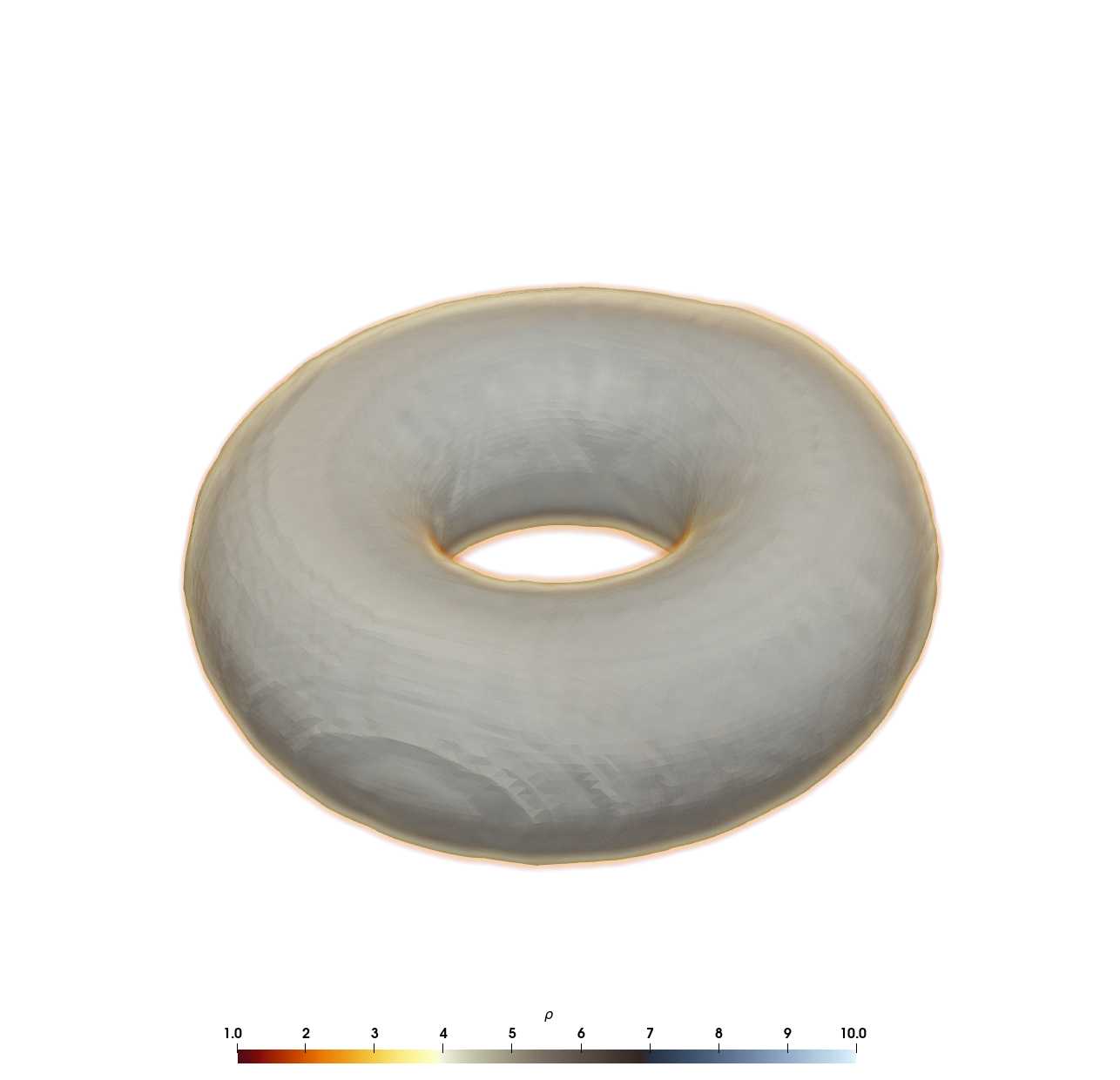}
    \end{subfigure}
    \begin{subfigure}{0.16\textwidth}
        \centering
        \includegraphics[width=\textwidth,trim=30mm 40mm 30mm 70mm, clip]{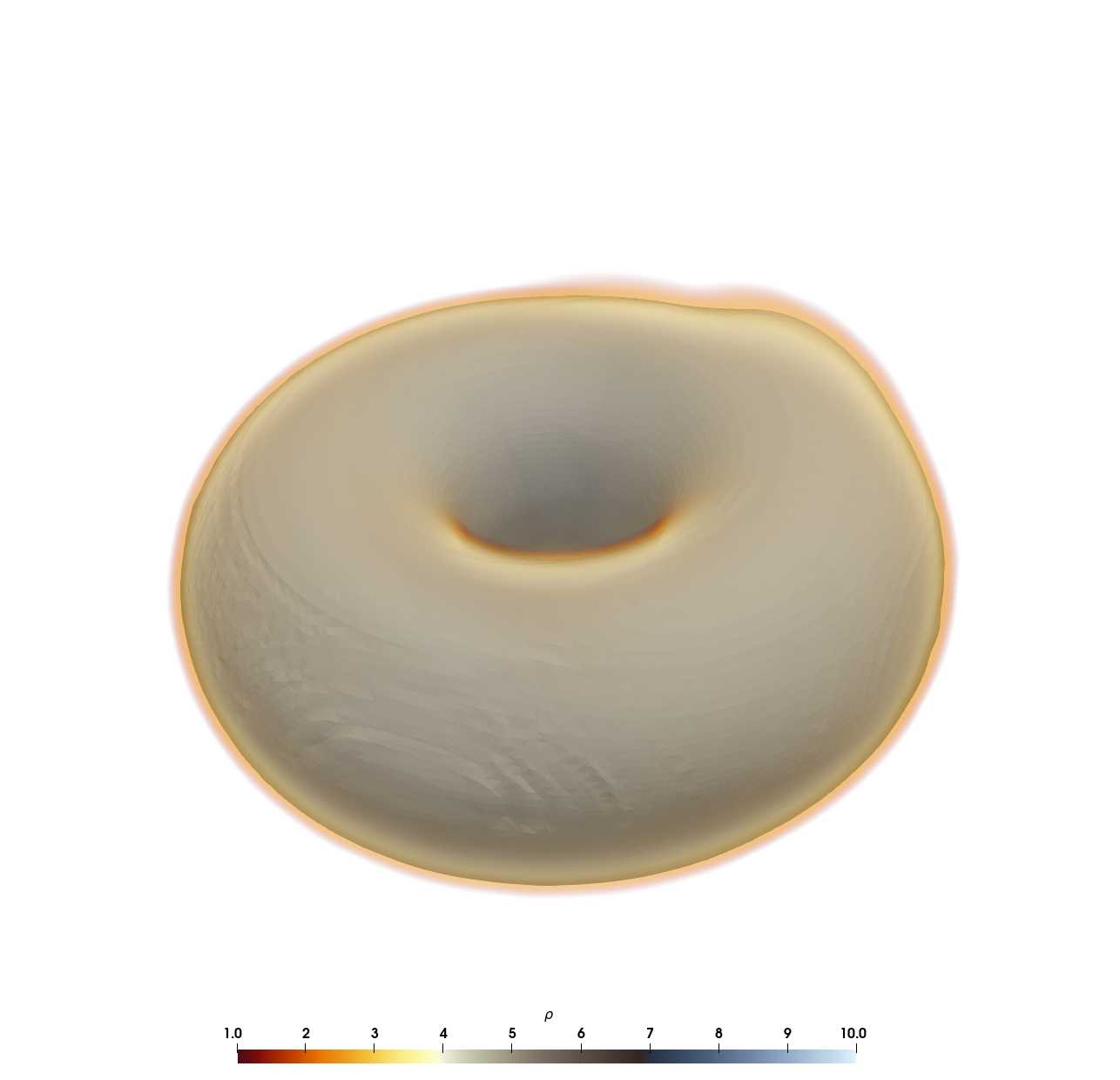}
    \end{subfigure}
    \begin{subfigure}{0.16\textwidth}
        \centering
        \includegraphics[width=\textwidth,trim=30mm 40mm 30mm 70mm, clip]{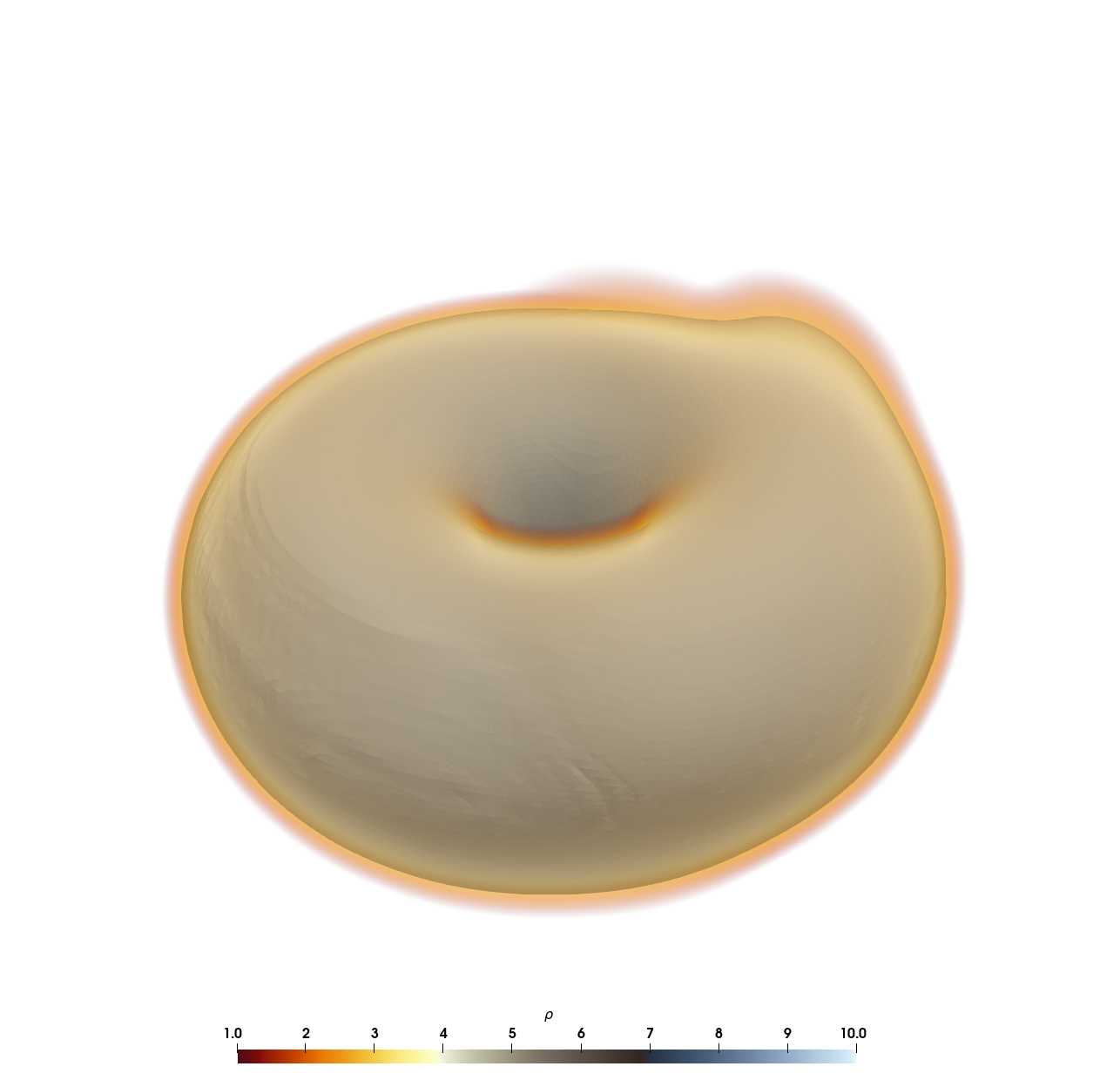}
    \end{subfigure}
    \begin{subfigure}{0.16\textwidth}
        \centering
        \includegraphics[width=\textwidth,trim=30mm 40mm 30mm 70mm, clip]{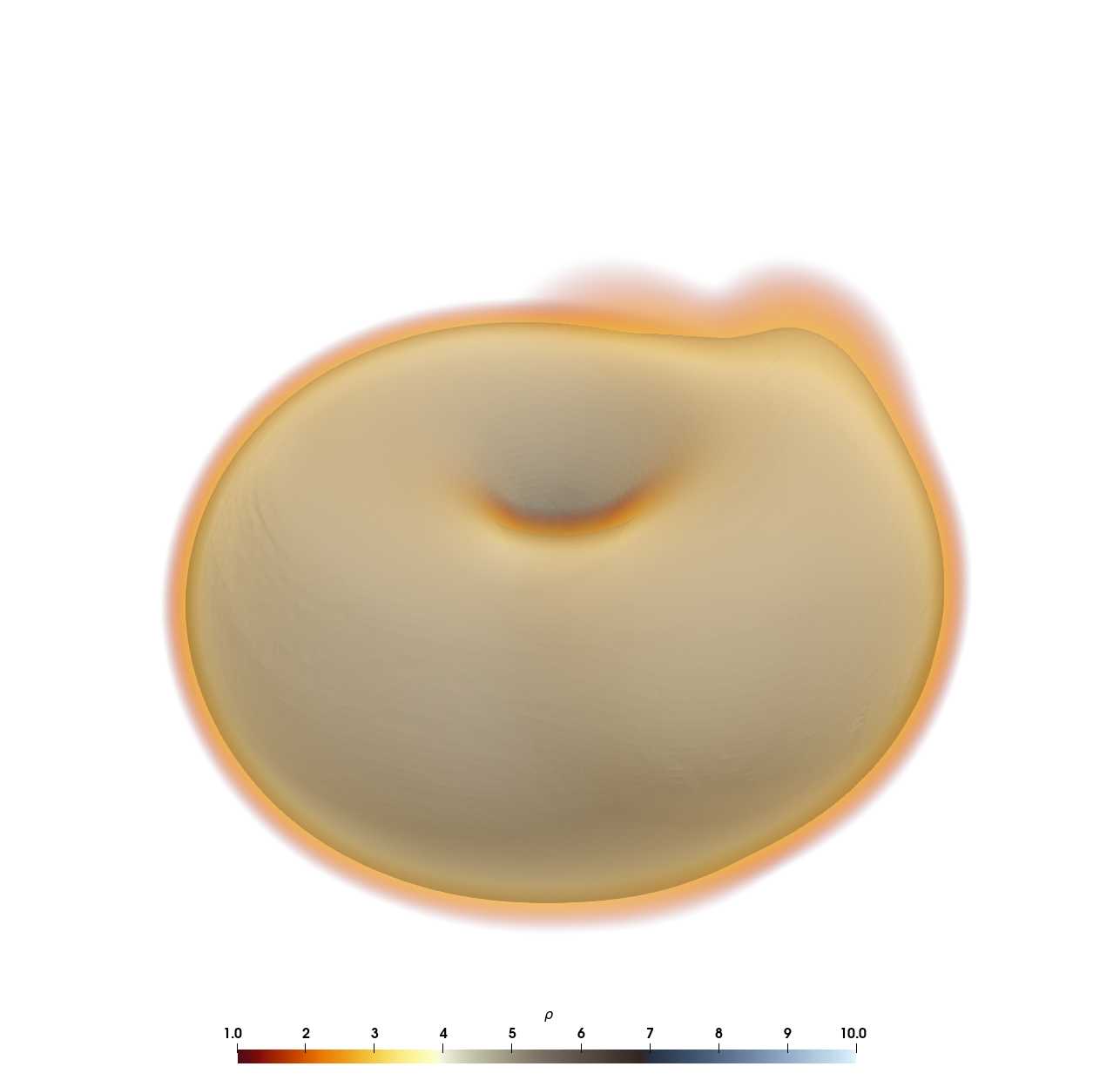}
    \end{subfigure}
    \begin{subfigure}{0.16\textwidth}
        \centering
        \includegraphics[width=\textwidth,trim=30mm 40mm 30mm 70mm, clip]{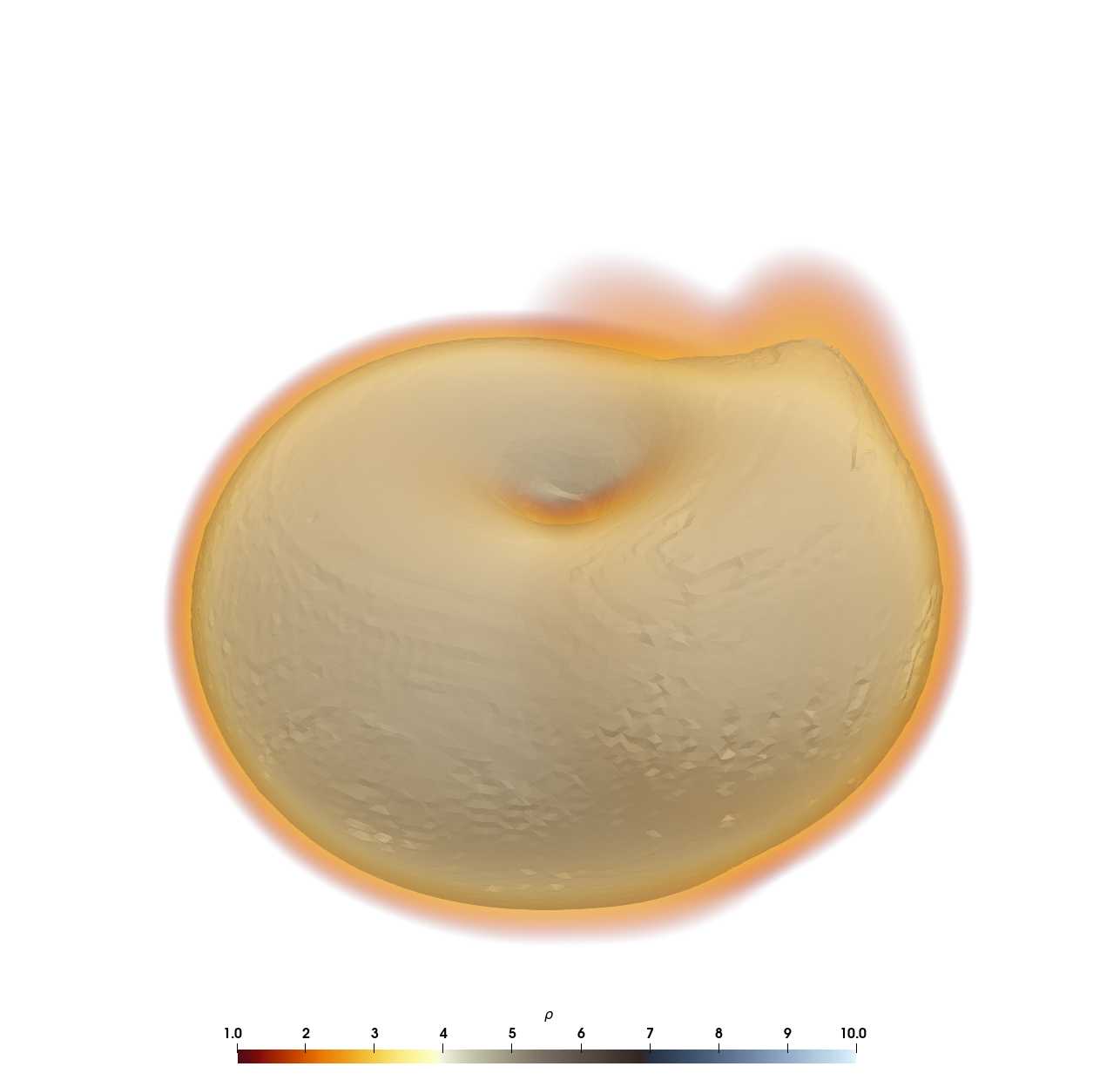}
    \end{subfigure}
    
    \begin{subfigure}{0.16\textwidth}
        \centering
        \includegraphics[width=\textwidth,trim=30mm 40mm 30mm 70mm, clip]{figures/reaction/torusBunny/pdhg0.0000..jpg}
    \end{subfigure}
    \begin{subfigure}{0.16\textwidth}
        \centering
        \includegraphics[width=\textwidth,trim=30mm 40mm 30mm 70mm, clip]{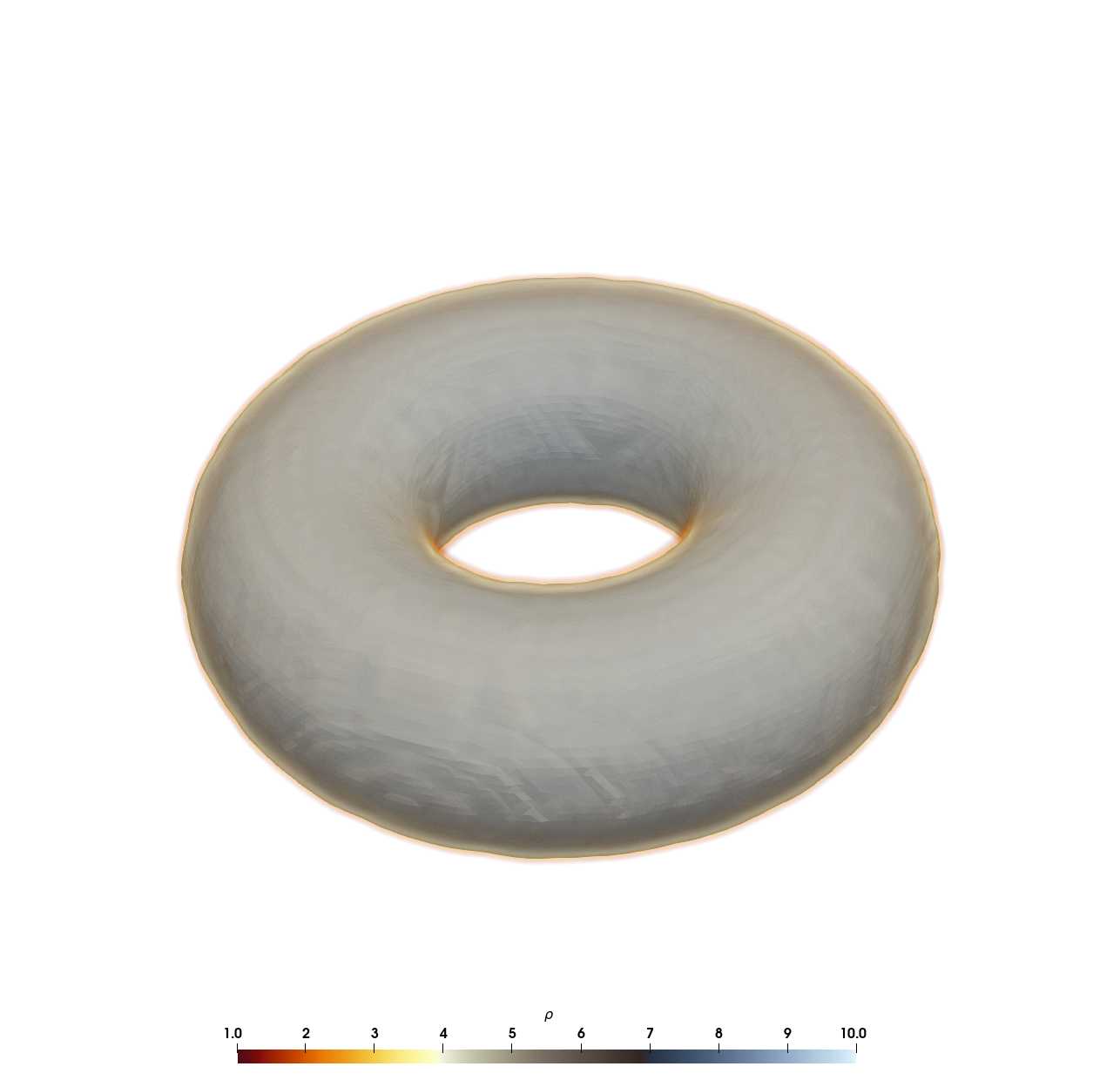}
    \end{subfigure}
    \begin{subfigure}{0.16\textwidth}
        \centering
        \includegraphics[width=\textwidth,trim=30mm 40mm 30mm 70mm, clip]{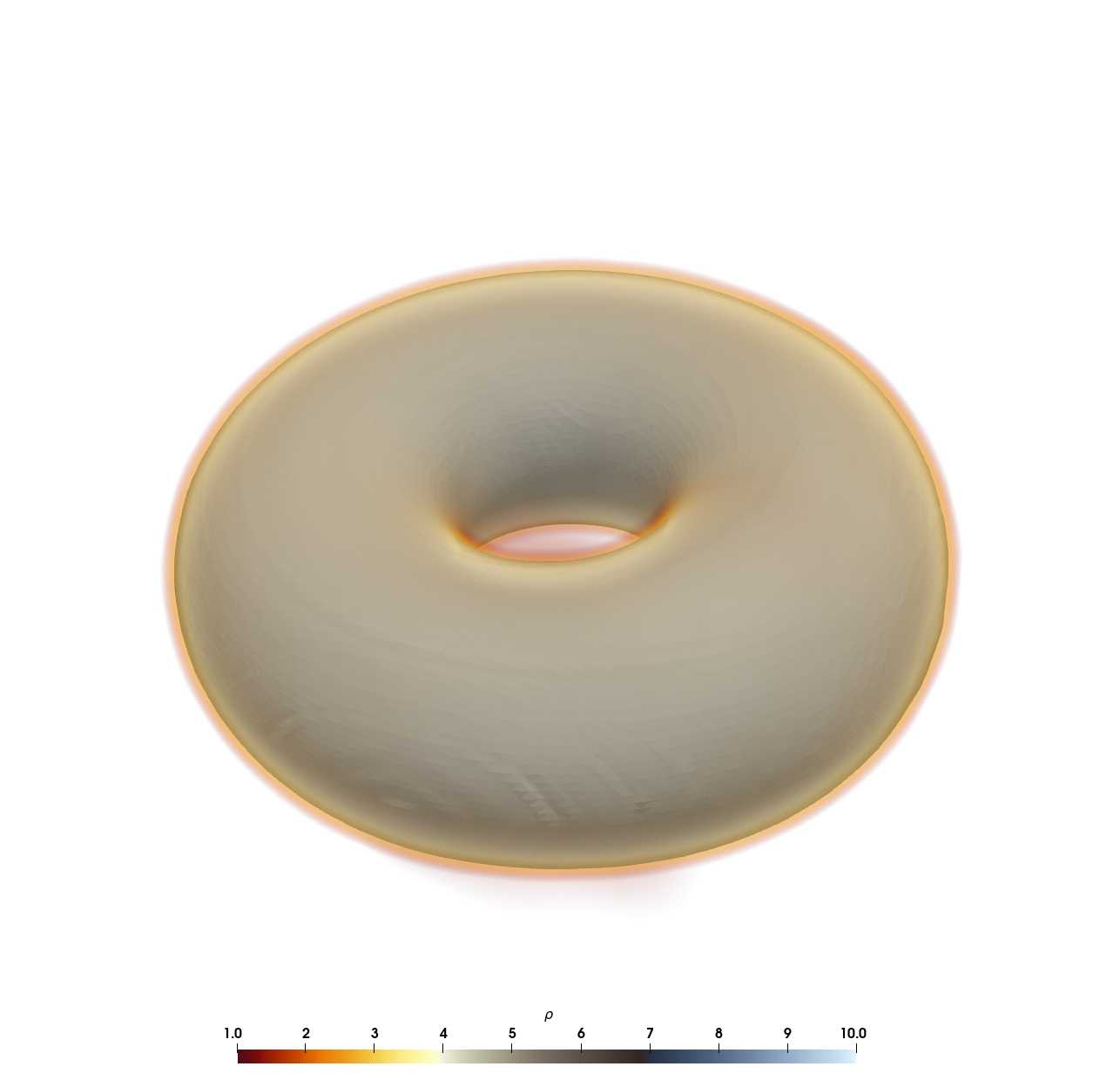}
    \end{subfigure}
    \begin{subfigure}{0.16\textwidth}
        \centering
        \includegraphics[width=\textwidth,trim=30mm 40mm 30mm 70mm, clip]{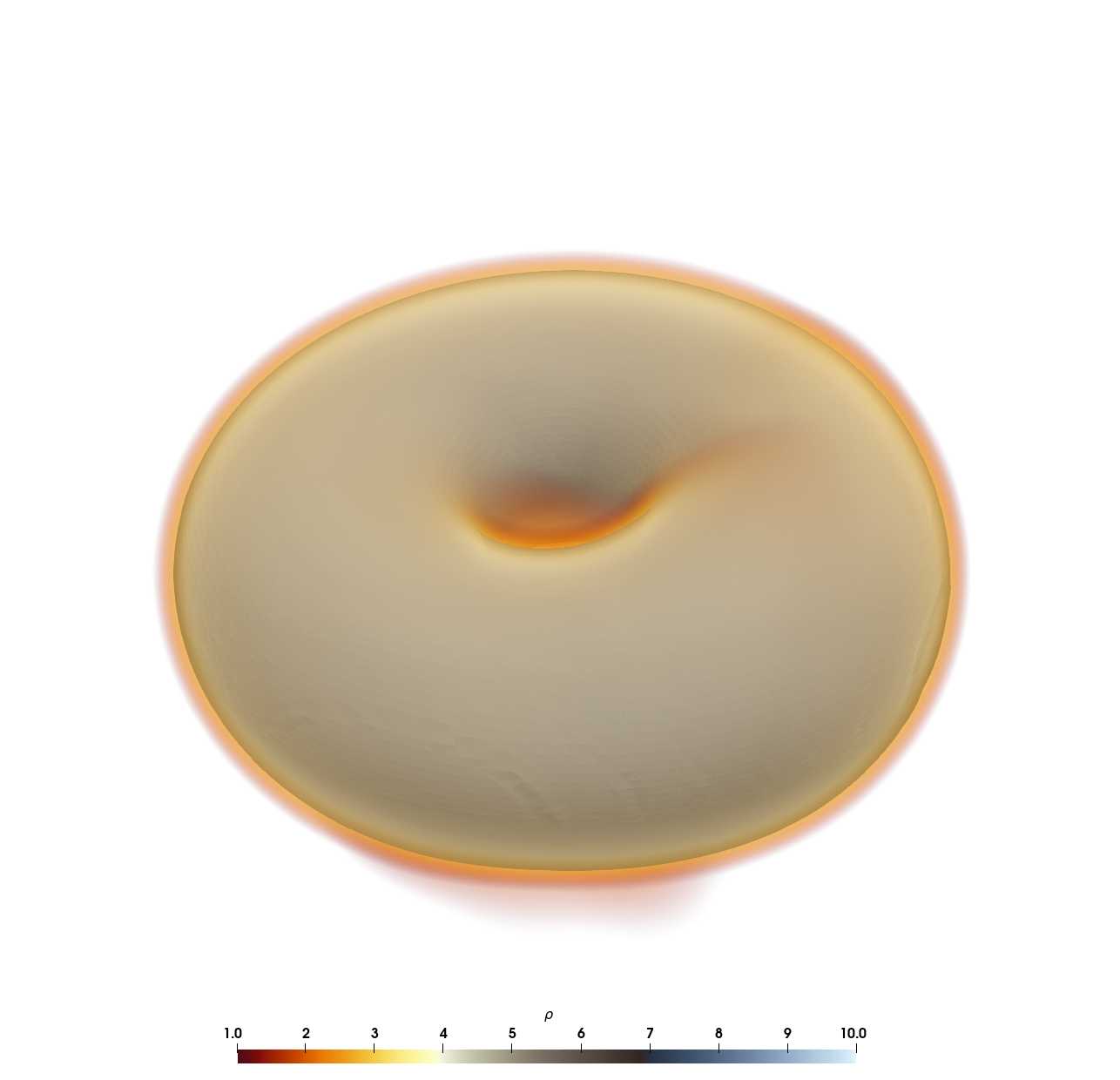}
    \end{subfigure}
    \begin{subfigure}{0.16\textwidth}
        \centering
        \includegraphics[width=\textwidth,trim=30mm 40mm 30mm 70mm, clip]{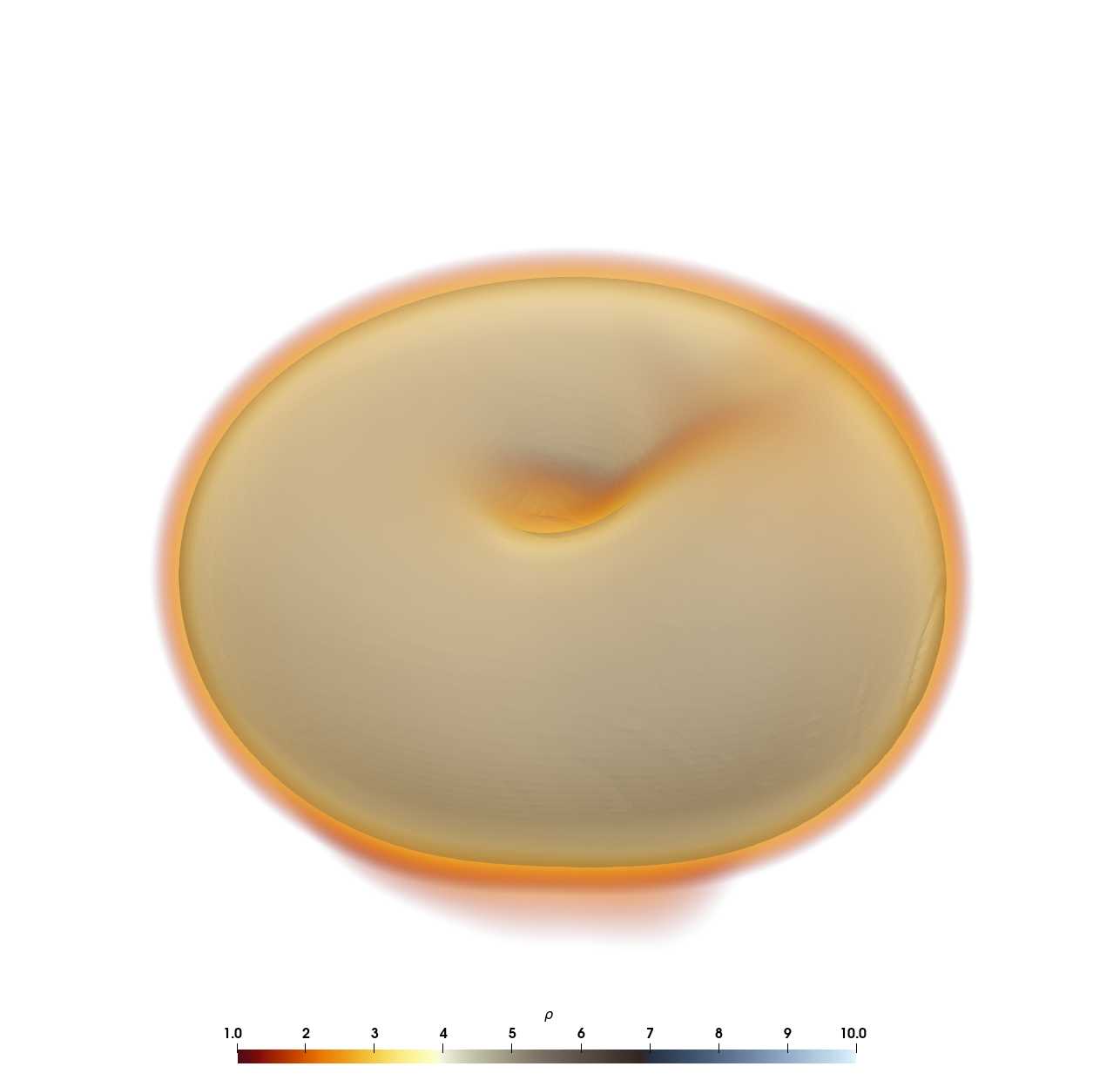}
    \end{subfigure}
    \begin{subfigure}{0.16\textwidth}
        \centering
        \includegraphics[width=\textwidth,trim=30mm 40mm 30mm 70mm, clip]{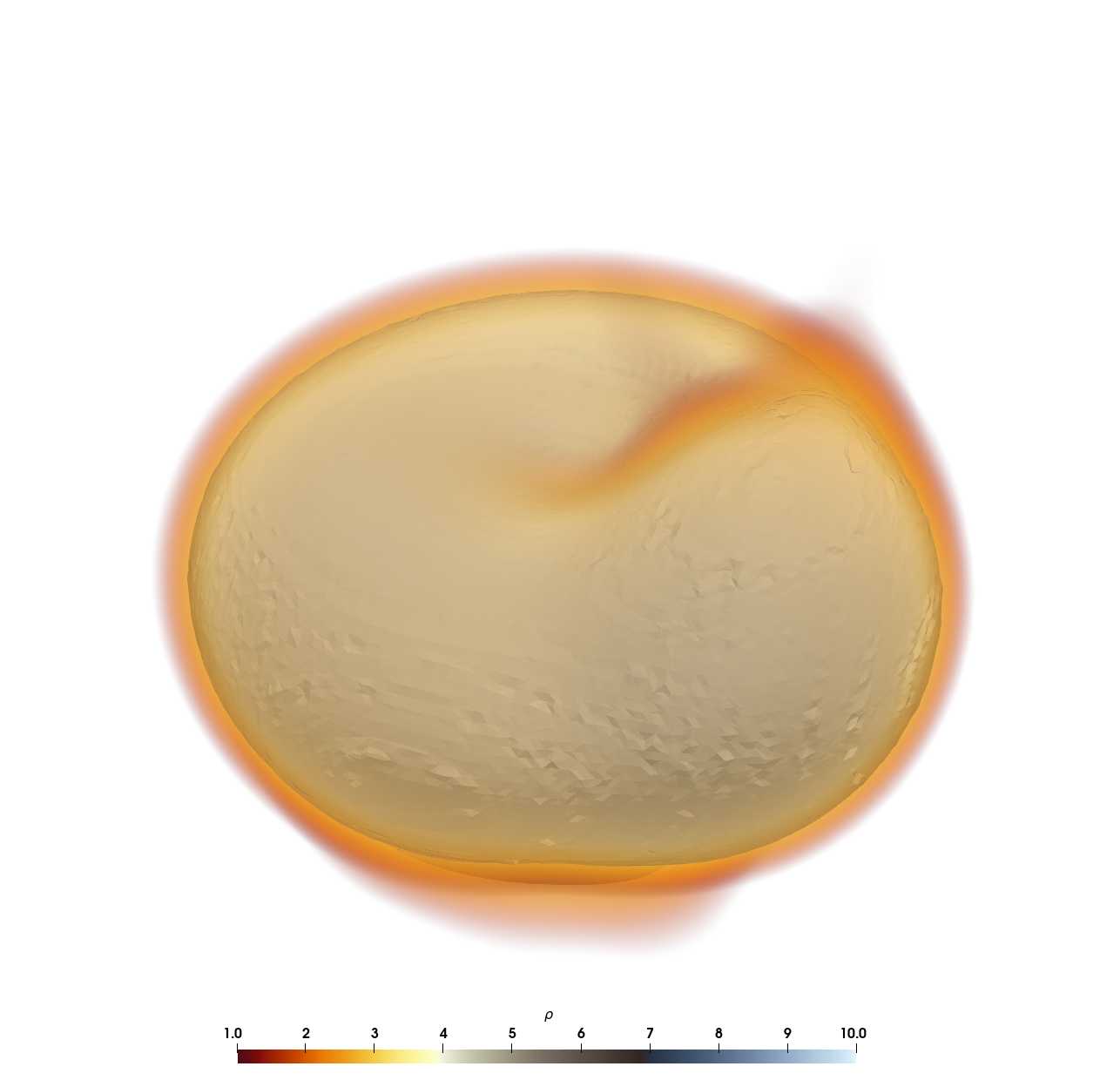}
    \end{subfigure}

    \begin{subfigure}{0.16\textwidth}
        \centering
        \includegraphics[width=\textwidth,trim=30mm 40mm 30mm 70mm, clip]{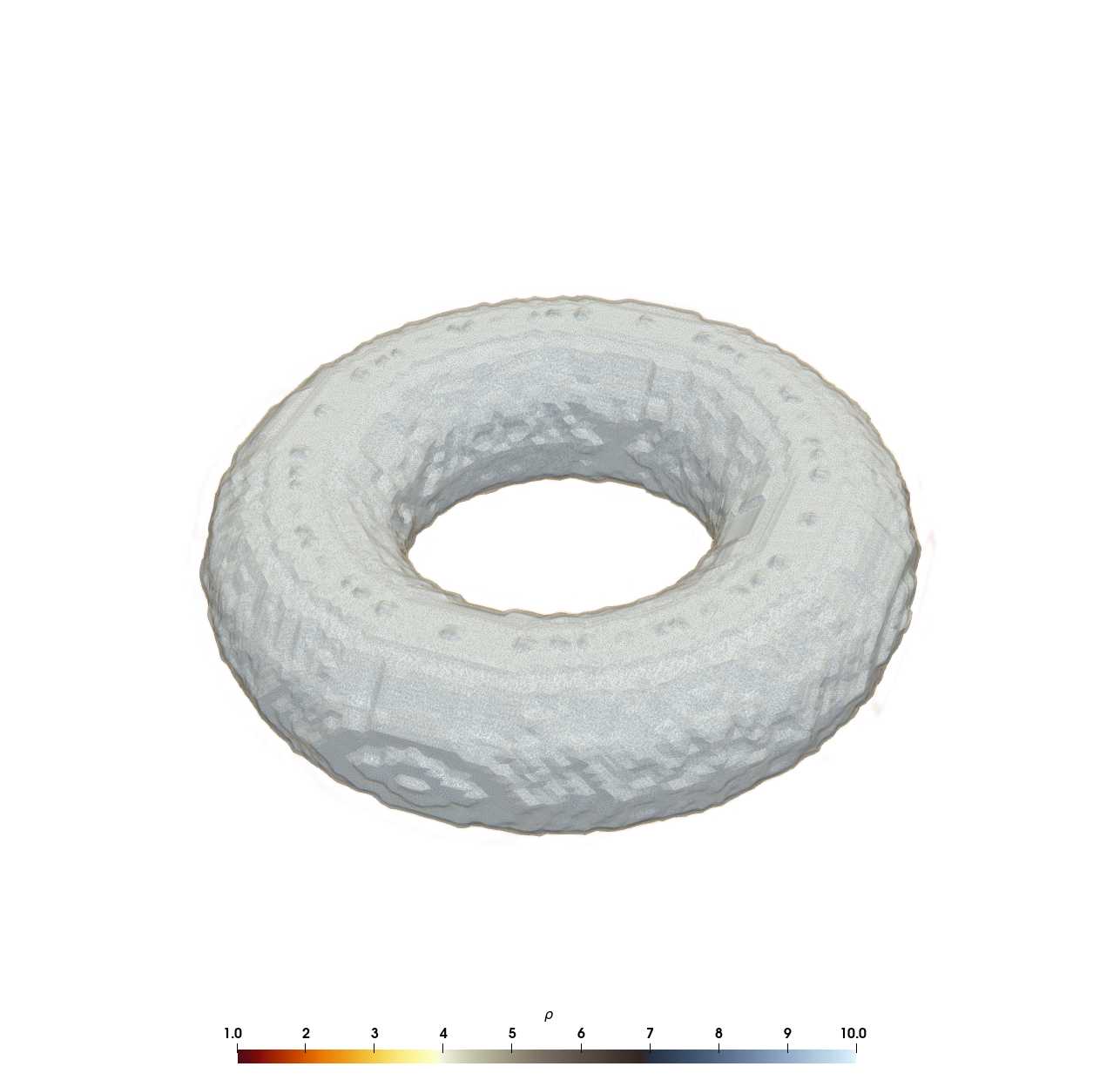}
        \caption*{$t=0$}
    \end{subfigure}
    \begin{subfigure}{0.16\textwidth}
        \centering
        \includegraphics[width=\textwidth,trim=30mm 40mm 30mm 70mm, clip]{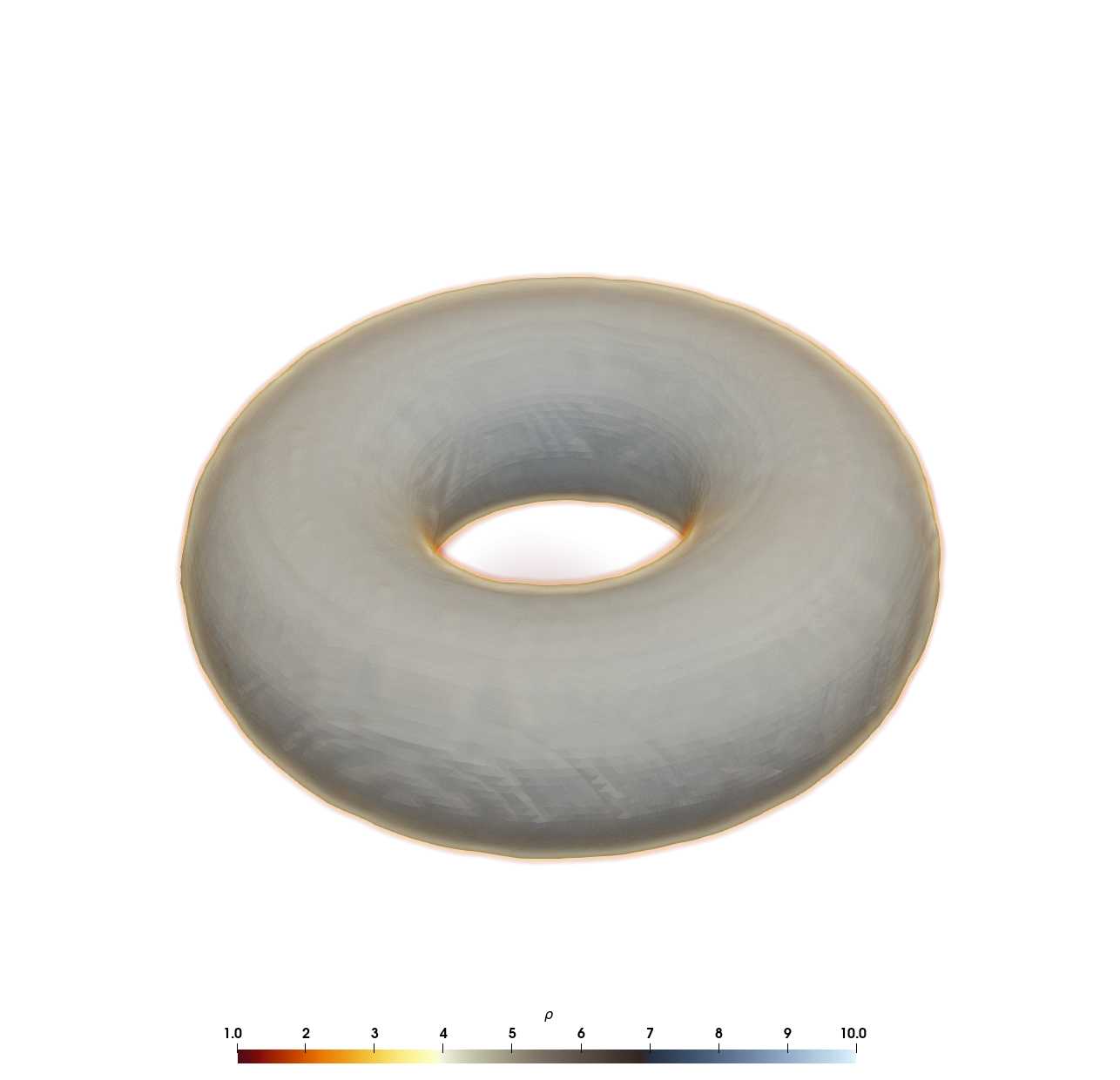}
        \caption*{$t=0.2$}
    \end{subfigure}
    \begin{subfigure}{0.16\textwidth}
        \centering
        \includegraphics[width=\textwidth,trim=30mm 40mm 30mm 70mm, clip]{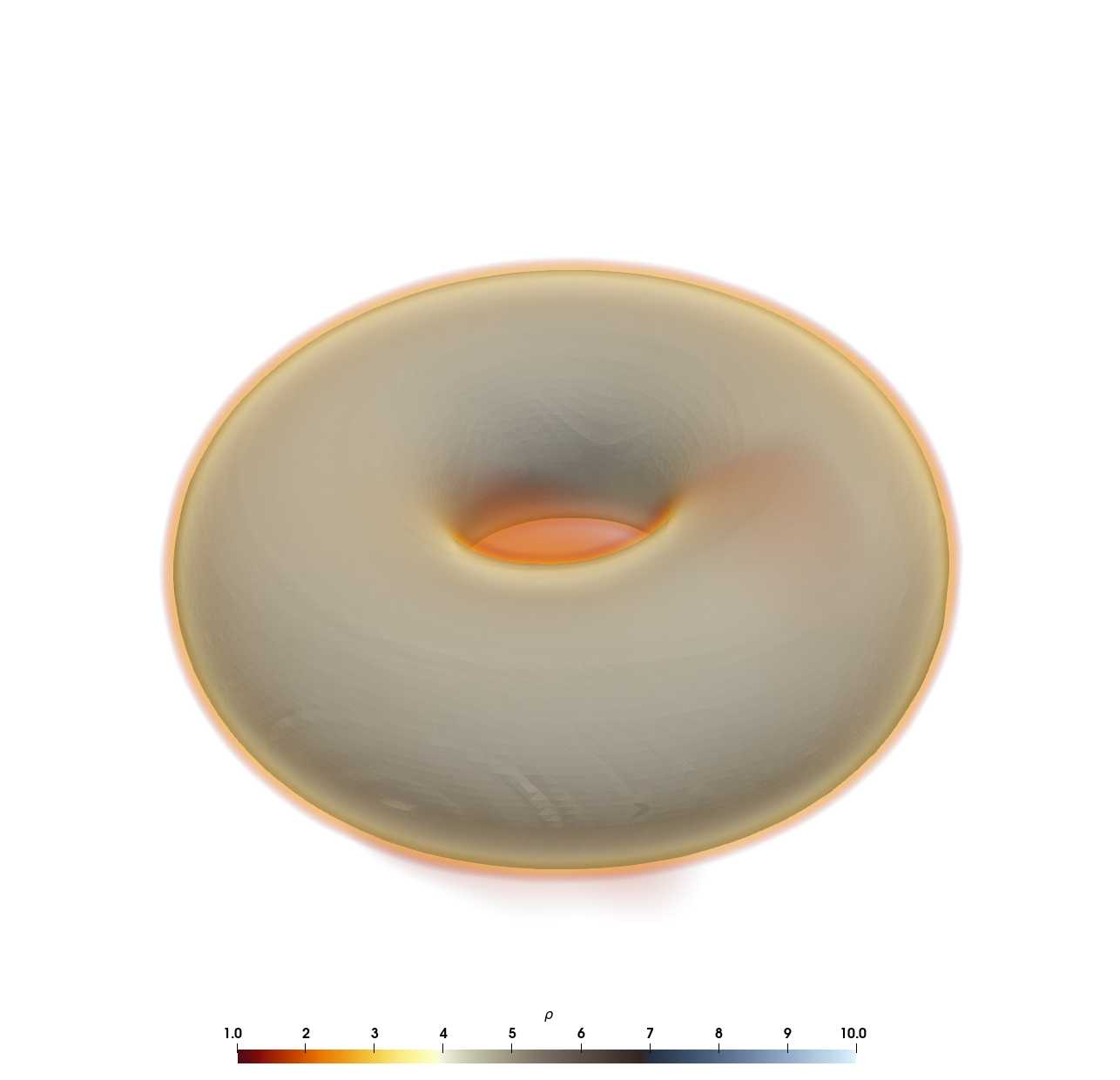}
        \caption*{$t=0.4$}
    \end{subfigure}
    \begin{subfigure}{0.16\textwidth}
        \centering
        \includegraphics[width=\textwidth,trim=30mm 40mm 30mm 70mm, clip]{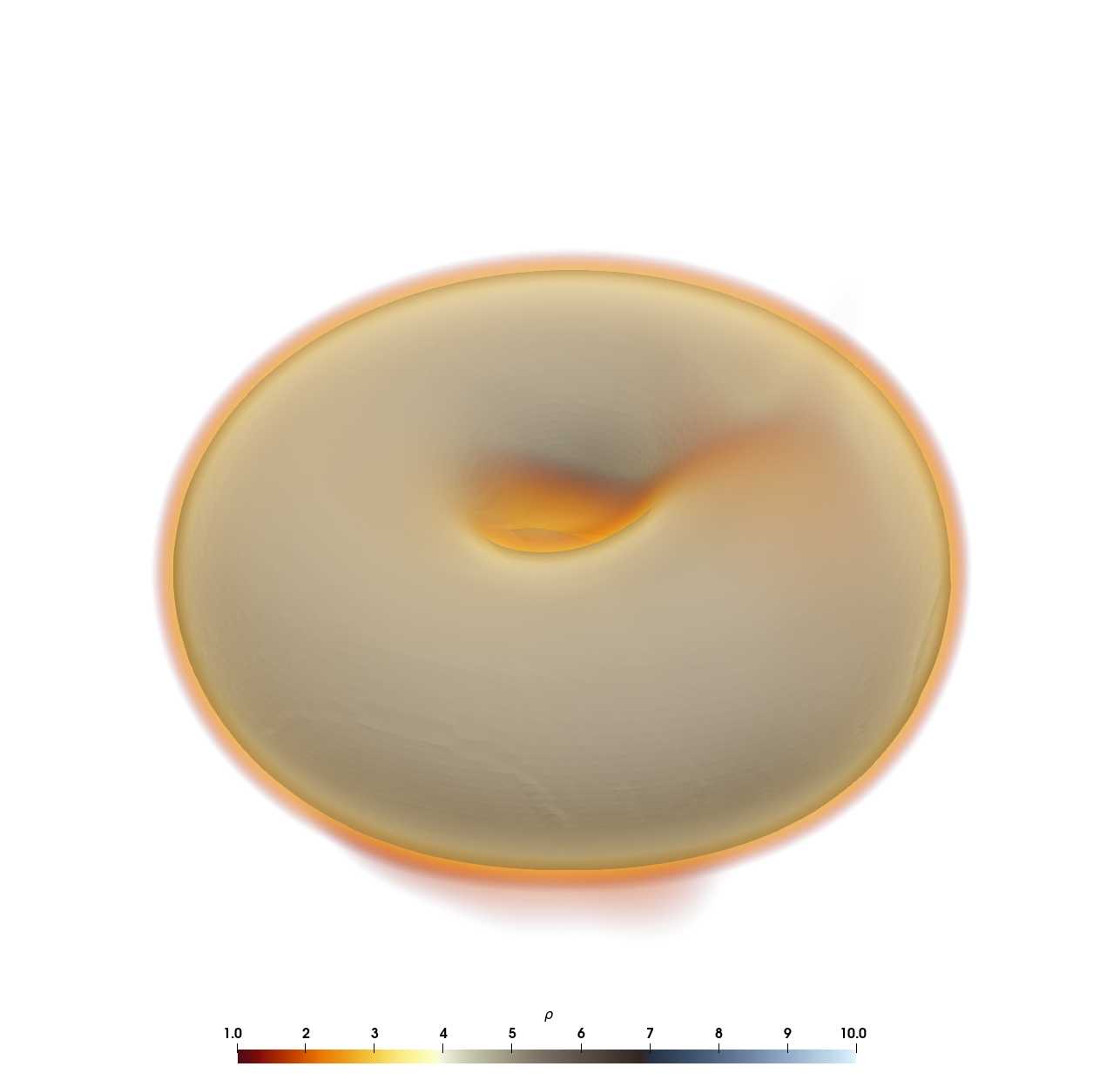}
        \caption*{$t=0.6$}
    \end{subfigure}
    \begin{subfigure}{0.16\textwidth}
        \centering
        \includegraphics[width=\textwidth,trim=30mm 40mm 30mm 70mm, clip]{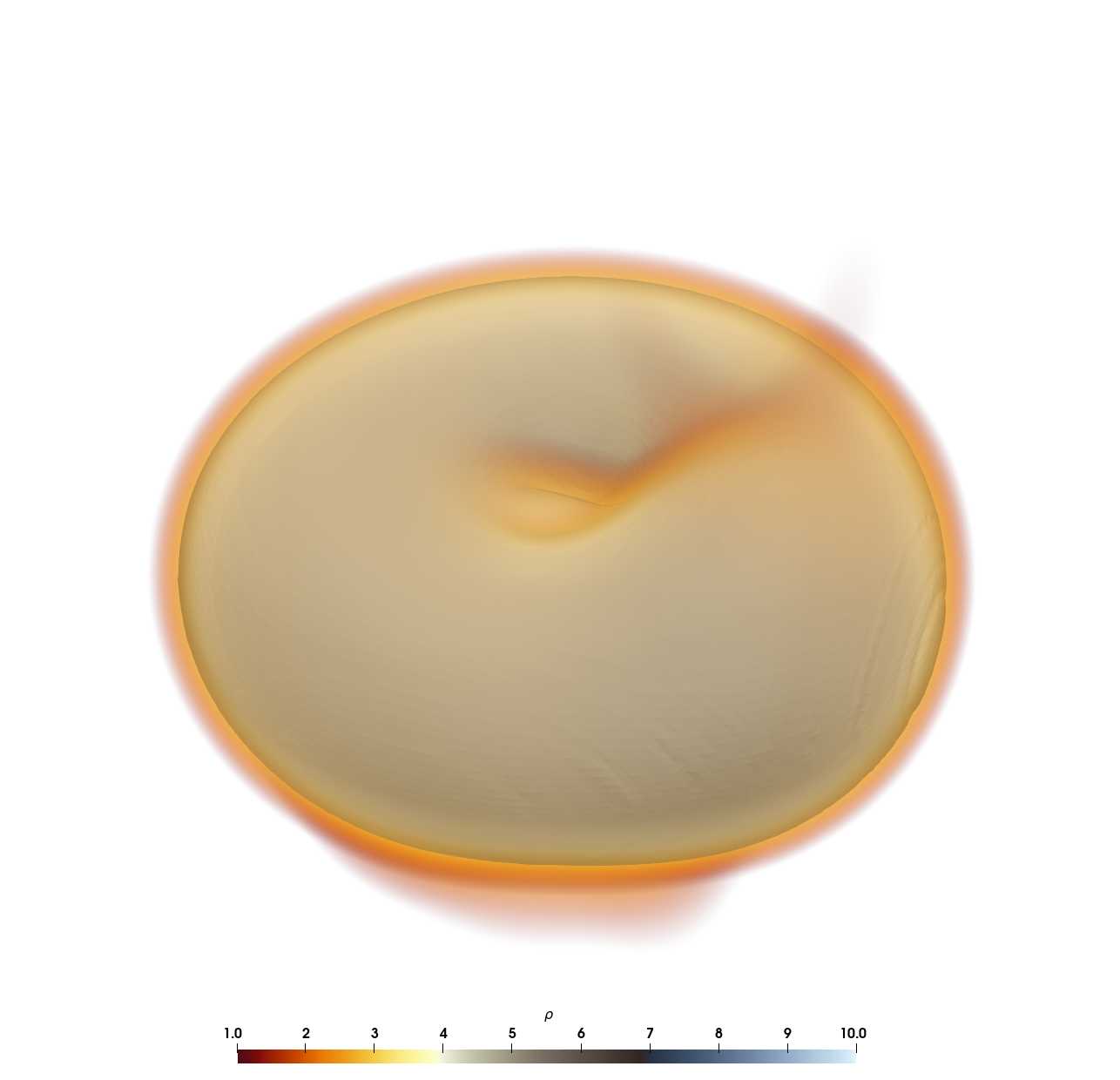}
        \caption*{$t=0.8$}
    \end{subfigure}
    \begin{subfigure}{0.16\textwidth}
        \centering
        \includegraphics[width=\textwidth,trim=30mm 40mm 30mm 70mm, clip]{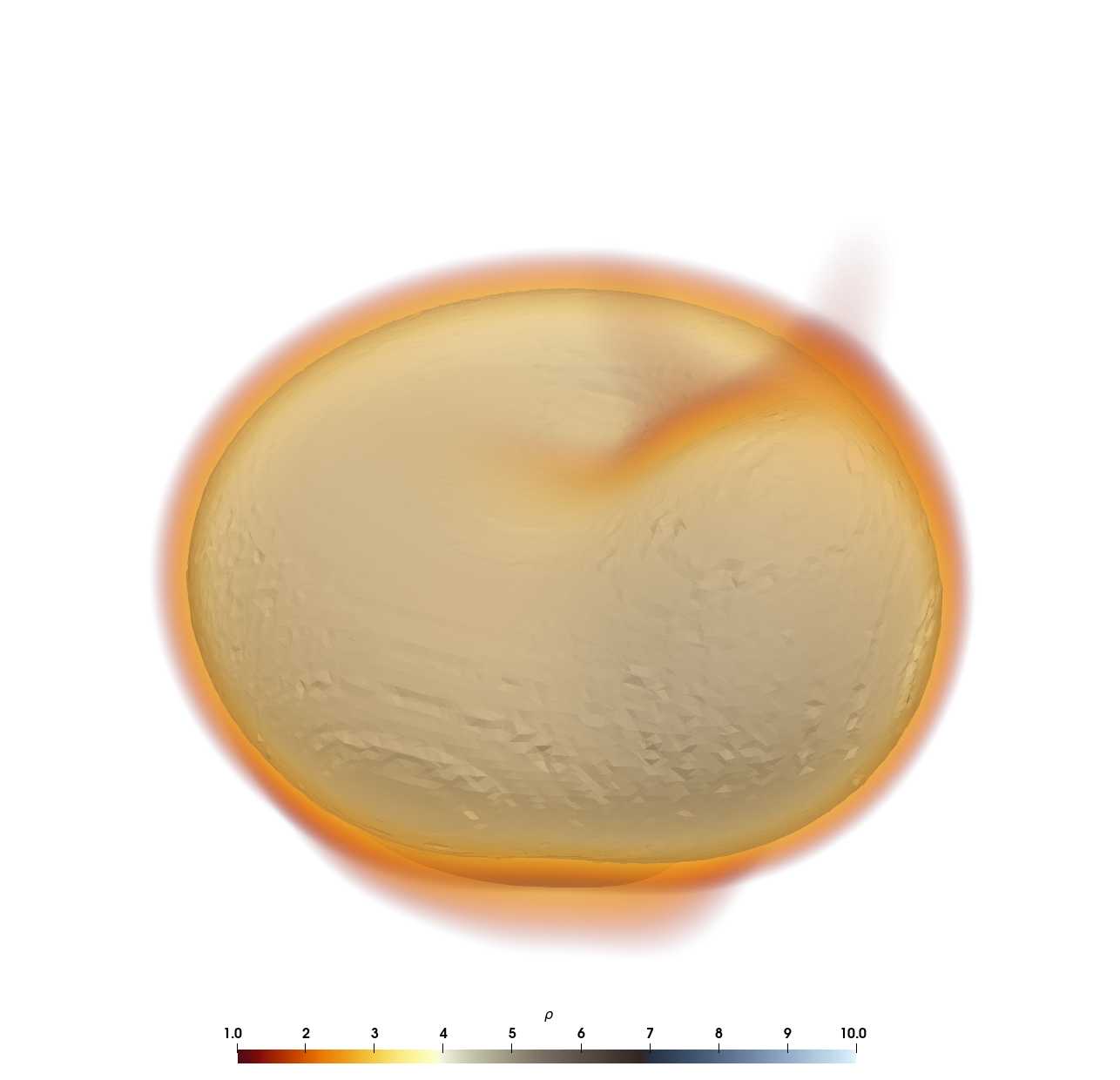}
        \caption*{$t=1.0$}
    \end{subfigure}
\subcaption{$\rho_2$}
\label{f:tb2}
\end{minipage}

\caption{Shape interpolation in a \textit{bunny}-\textit{torus} system. Left-right shows the time evolution of density. Top-down shows plots at different reaction strengths: $\alpha=0$, $\alpha=50,$ and $\alpha=100$.}
\label{f:tb}
\end{figure}

\subsection{3D shape interpolation}
In this subsection, we explore the application of our scheme \eqref{uncon3} to perform 3D shape interpolation. Shape interpolation, a fundamental concept in the realms of computer graphics and computer vision, is a technique used to continuously morph one 3D object into another, enabling a gradual transition between the two shapes and allowing for the creation of smooth animations and visual effects. Shape interpolation by computing Wasserstein barycenters was previously explored in \cite{solomon:hal-01188953} using convolutional Wasserstein distance.

We conduct a series of shape interpolation experiments by computing generalized Wasserstein barycenters of several 2-species systems with varying reaction strengths, namely $\alpha=0$, $\alpha=50$, and $\alpha=100$. The densities are represented by normalized indicator functions on $\Omega=[0,1]^3$ that are created by voxelizing 3D object files \cite{binvox, nooruddin03}. The voxelization process is carried out at a resolution of $256\times256\times256$, and the total mass of the normalized indicator functions is made equal to 1. The 3D objects we use in the experiments are the \textit{torus}, the \textit{double-torus}, and the \textit{bunny}. We use polynomials of degree $k=4$ on a spacetime mesh with $32\times32\times32$ uniform cubical elements in space and four uniform line segments in time to perform the computations. Here, we take the interaction strength $\beta_i$ to be 0.001 for each $i$, which acts as a regularization term in the optimization solver.

We present the results of our experiments using the panels in Figures \ref{f:tdt}, \ref{f:dtb}, \ref{f:tb}. Each row in a figure shows the time evolution of a particular density, with the overall density function plotted along with a single contour surface (in white) where the density equals to 3. These are plotted at regular intervals of $\Delta t=0.2$ from the initial time of 0 until the terminal time of 1. On the other hand, each column shows the density and the contour at the three different reaction strengths. 

Our results reveal that the computed maps $\rho_1(t)$ and $\rho_2(t)$ provide continuous interpolations between the initial densities as each can be seen morphing into their mutual (generalized) barycenter. We also observe that as the reaction strength $\alpha$ increases, the barycenter as well as the intermediate shapes start resembling their Euclidean ($L_2$) equivalents, as features of the adjoining density start appearing at earlier times. Additionally, in cases where $\alpha$ is non-zero, the barycenters share similarities but differ from those in the zero reaction case.

We conclude this subsection with discussions of a shape interpolation experiment involving a 3-density system consisting of all three species we mentioned above. Firstly, we calculate the barycenter with no reaction by setting $\alpha=0$, and then with the reaction by setting $\alpha=50$. The mesh and voxel resolutions, as well as the polynomial degree, remain unchanged in the discretization. The results are displayed in the panels of Figure \ref{f:tdtb}. Similar to the results in the 2 species experiments, we observe that the three initial densities gradually transform into their barycenter, providing continuous intermediate interpolated shapes. In the case where the reaction strength is non-zero, we can again observe the barycenter and the intermediate shapes resembling the corresponding Euclidean averages.

\begin{figure}[h!]
    \centering
\begin{minipage}[b]{\textwidth}

    \begin{minipage}[b]{\textwidth}
    \hfill
        \begin{subfigure}{\textwidth}
            \centering
            \includegraphics[width=\textwidth,trim=0mm 0mm 0mm 400mm, clip]{figures/reaction/doubleTorusBunny/pdhg0.0000..jpg}
        \end{subfigure}
    \end{minipage}
    
    \begin{subfigure}{0.16\textwidth}
        \centering
        \includegraphics[width=\textwidth,trim=30mm 40mm 30mm 70mm, clip]{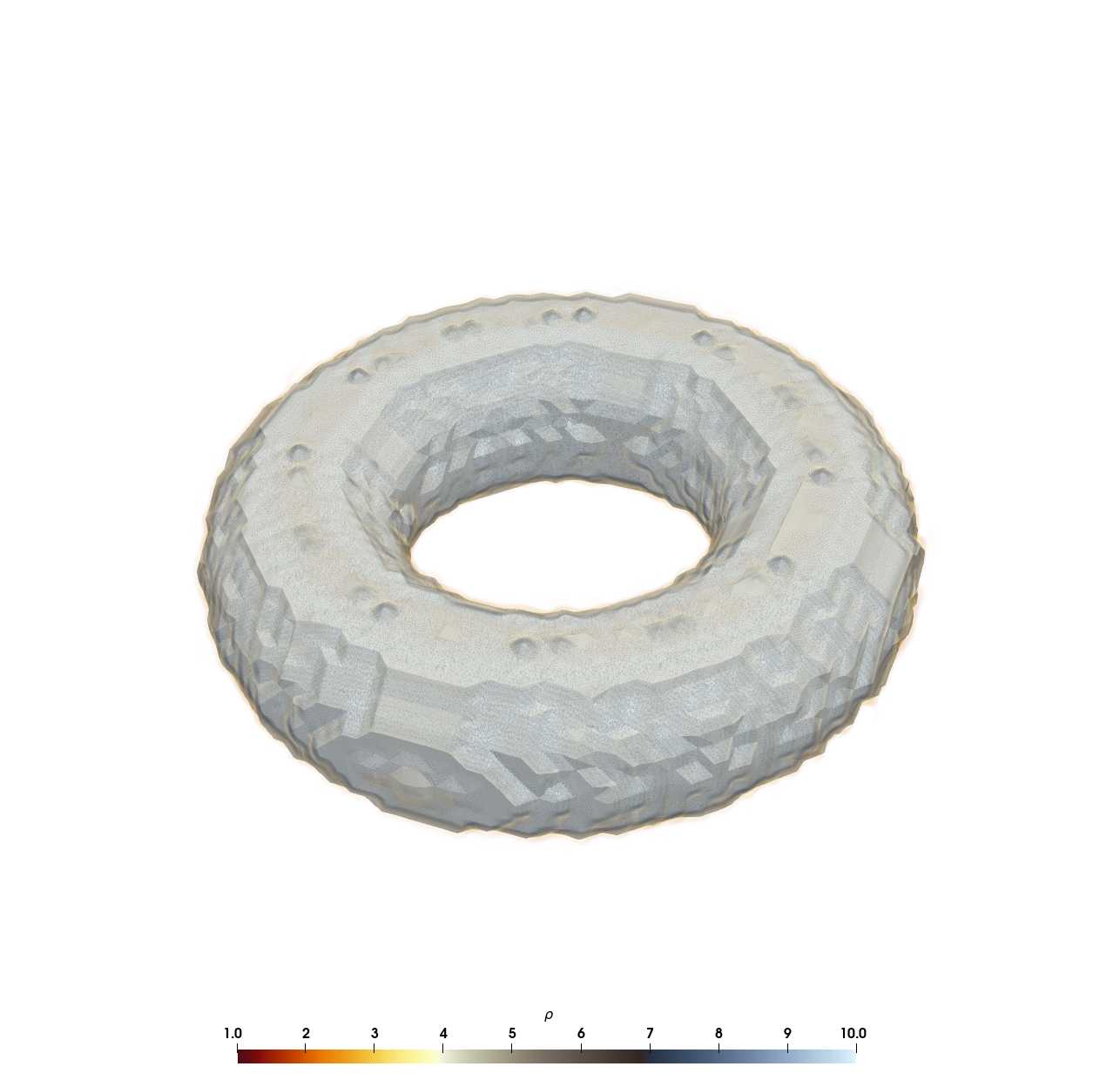}
    \end{subfigure}
    \begin{subfigure}{0.16\textwidth}
        \centering
        \includegraphics[width=\textwidth,trim=30mm 40mm 30mm 70mm, clip]{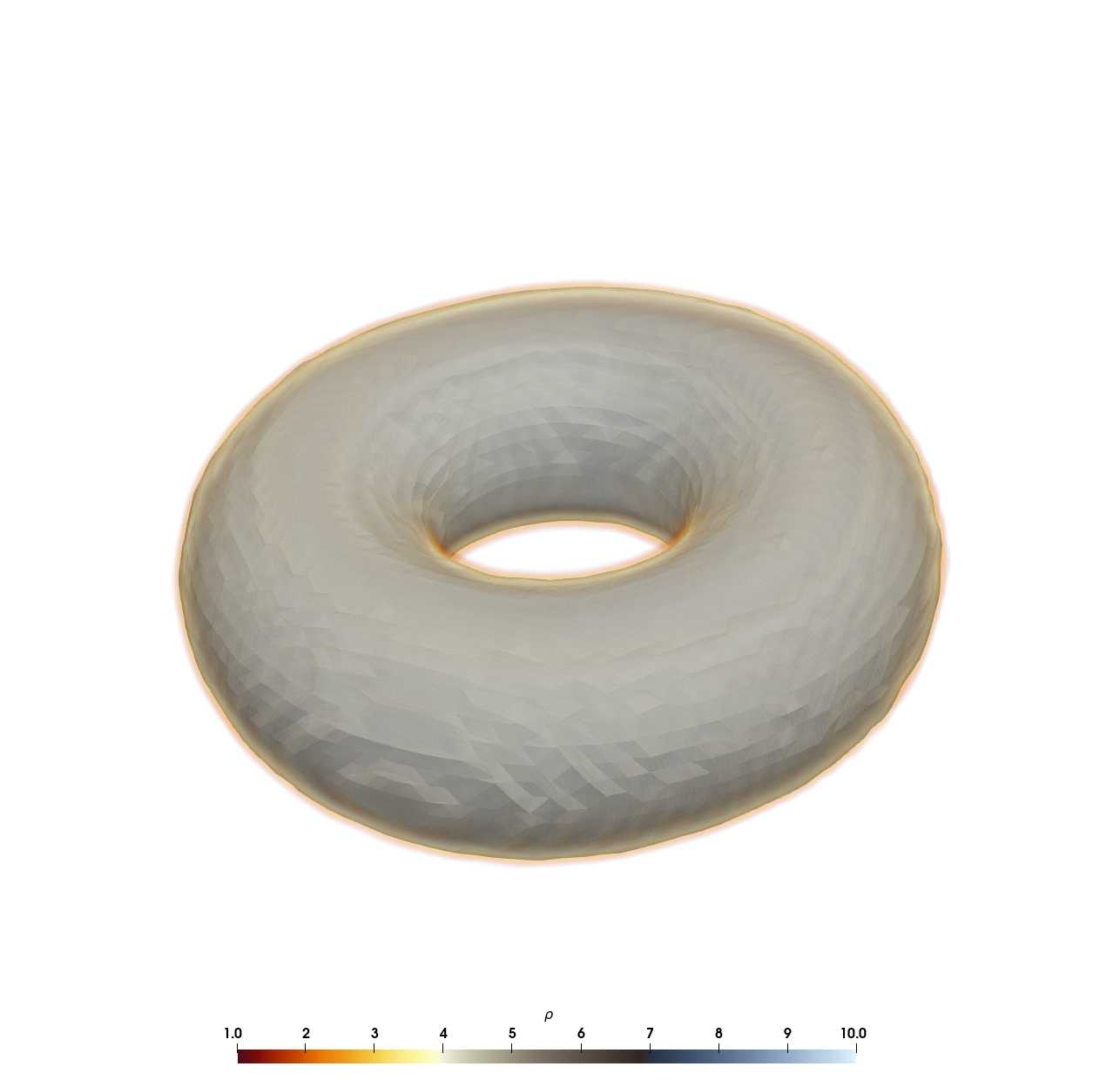}
    \end{subfigure}
    \begin{subfigure}{0.16\textwidth}
        \centering
        \includegraphics[width=\textwidth,trim=30mm 40mm 30mm 70mm, clip]{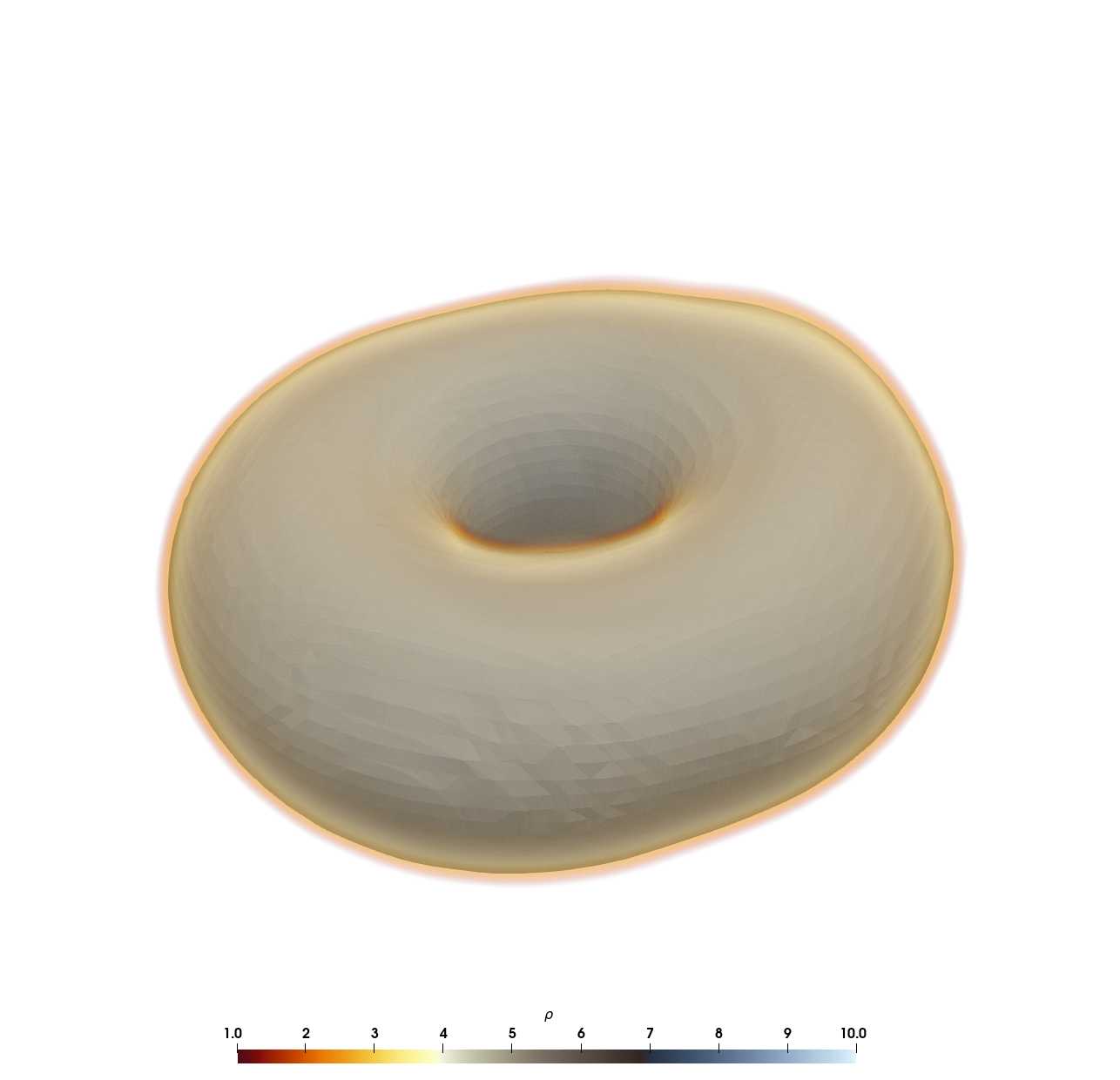}
    \end{subfigure}
    \begin{subfigure}{0.16\textwidth}
        \centering
        \includegraphics[width=\textwidth,trim=30mm 40mm 30mm 70mm, clip]{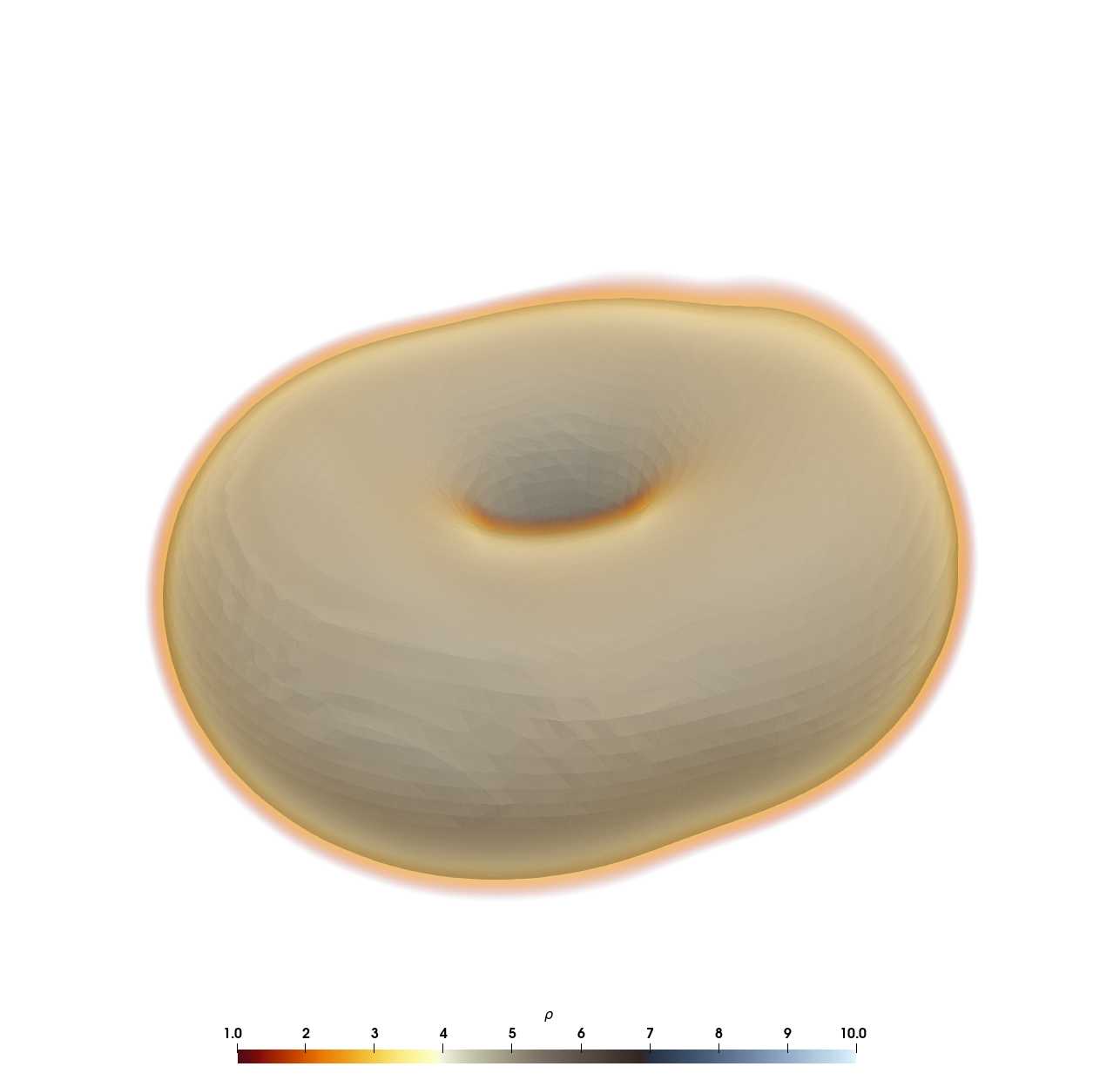}
    \end{subfigure}
    \begin{subfigure}{0.16\textwidth}
        \centering
        \includegraphics[width=\textwidth,trim=30mm 40mm 30mm 70mm, clip]{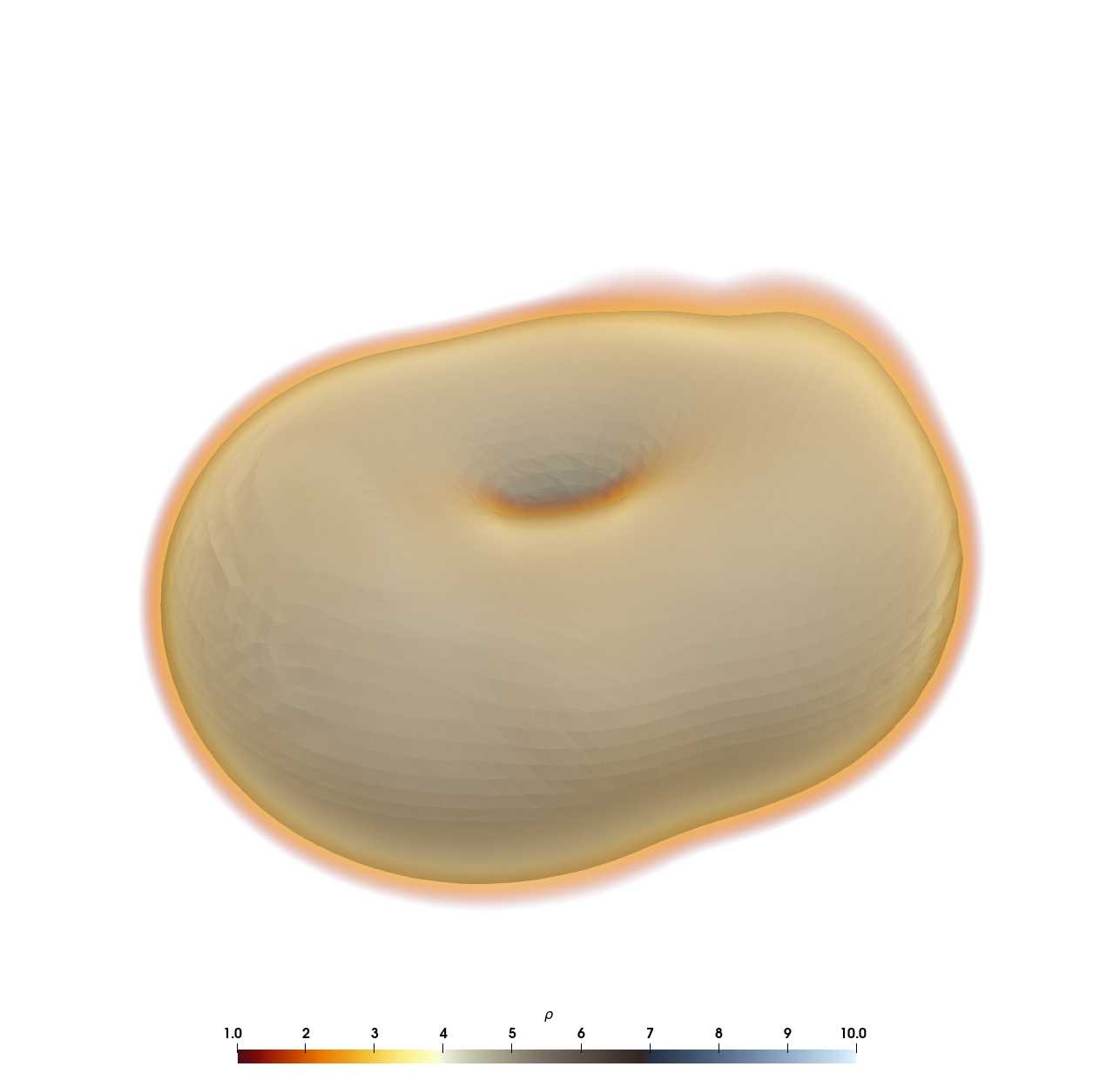}
    \end{subfigure}
    \begin{subfigure}{0.16\textwidth}
        \centering
        \includegraphics[width=\textwidth,trim=30mm 40mm 30mm 70mm, clip]{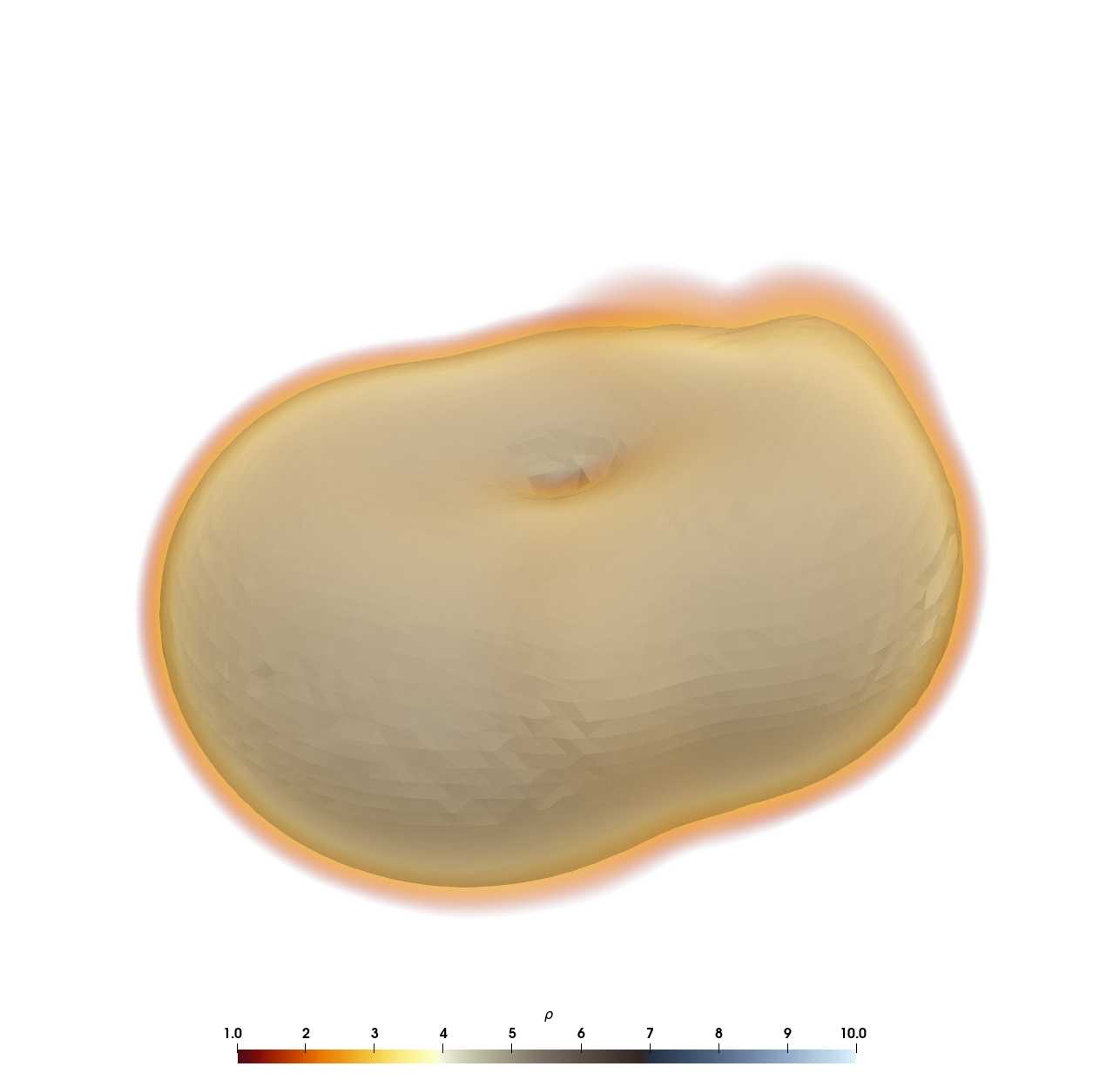}
    \end{subfigure}

    \begin{subfigure}{0.16\textwidth}
        \centering
        \includegraphics[width=\textwidth,trim=30mm 40mm 30mm 70mm, clip]{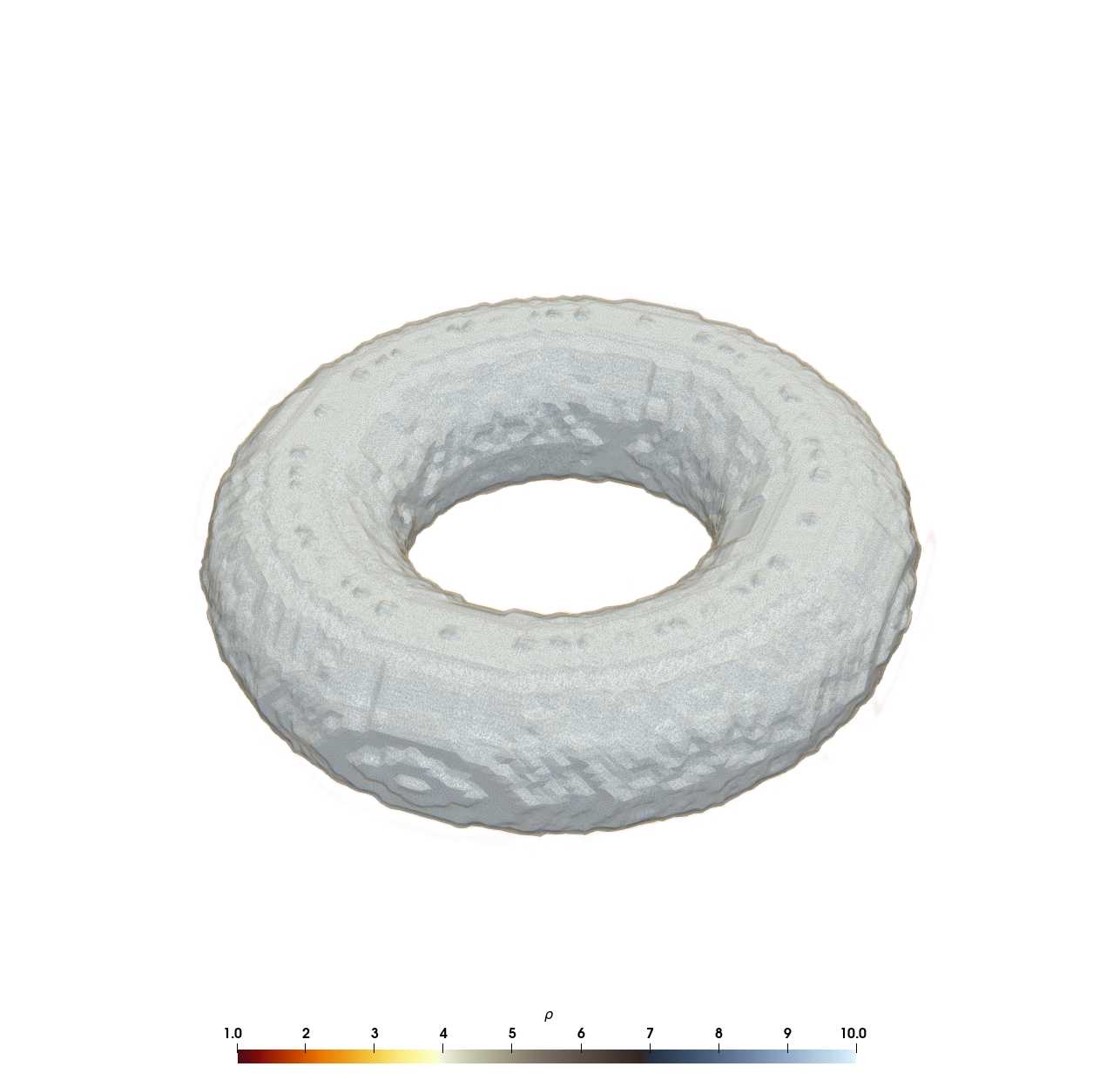}
        \caption*{$t=0$}
    \end{subfigure}
    \begin{subfigure}{0.16\textwidth}
        \centering
        \includegraphics[width=\textwidth,trim=30mm 40mm 30mm 70mm, clip]{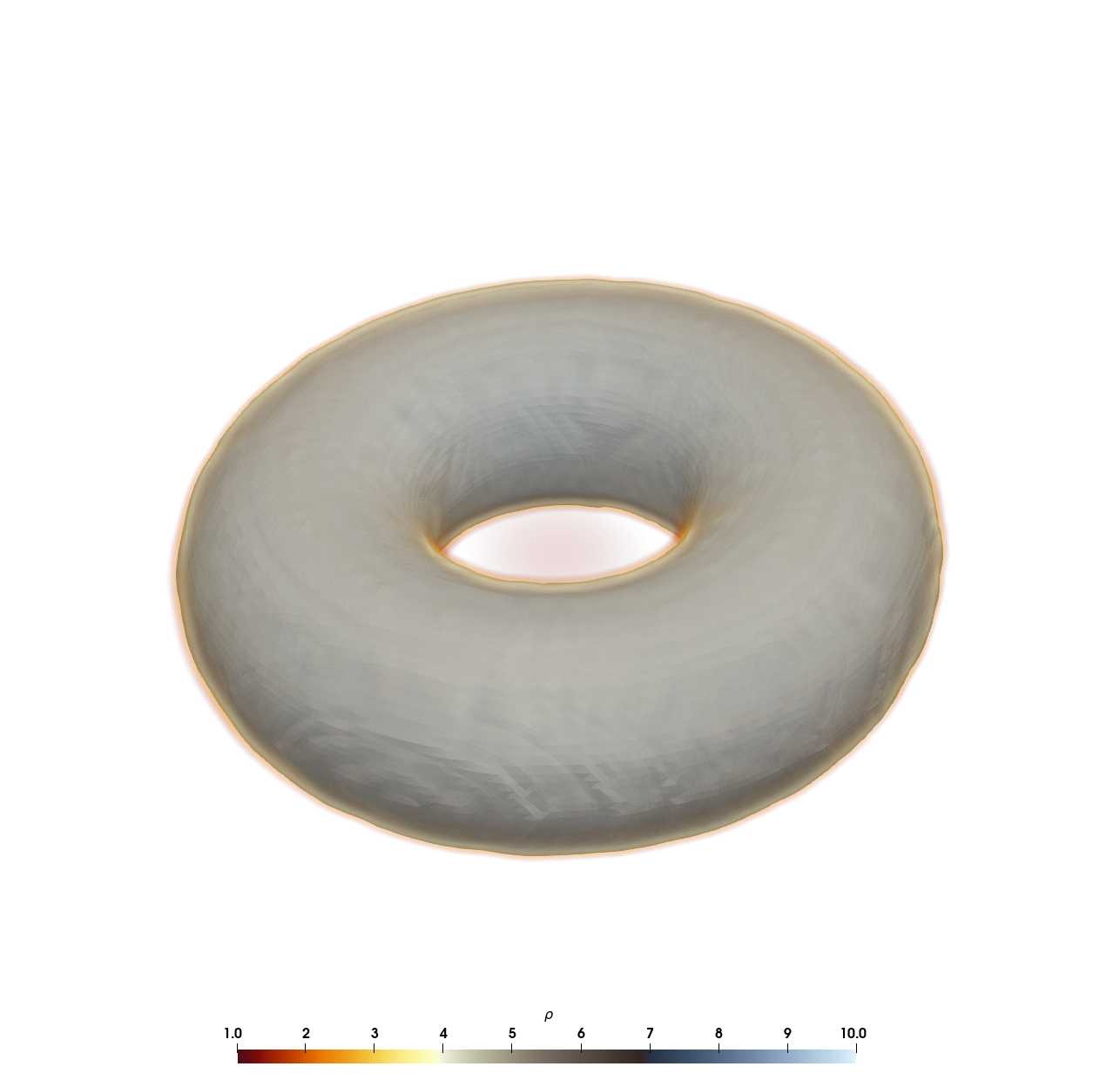}
        \caption*{$t=0.2$}
    \end{subfigure}
    \begin{subfigure}{0.16\textwidth}
        \centering
        \includegraphics[width=\textwidth,trim=30mm 40mm 30mm 70mm, clip]{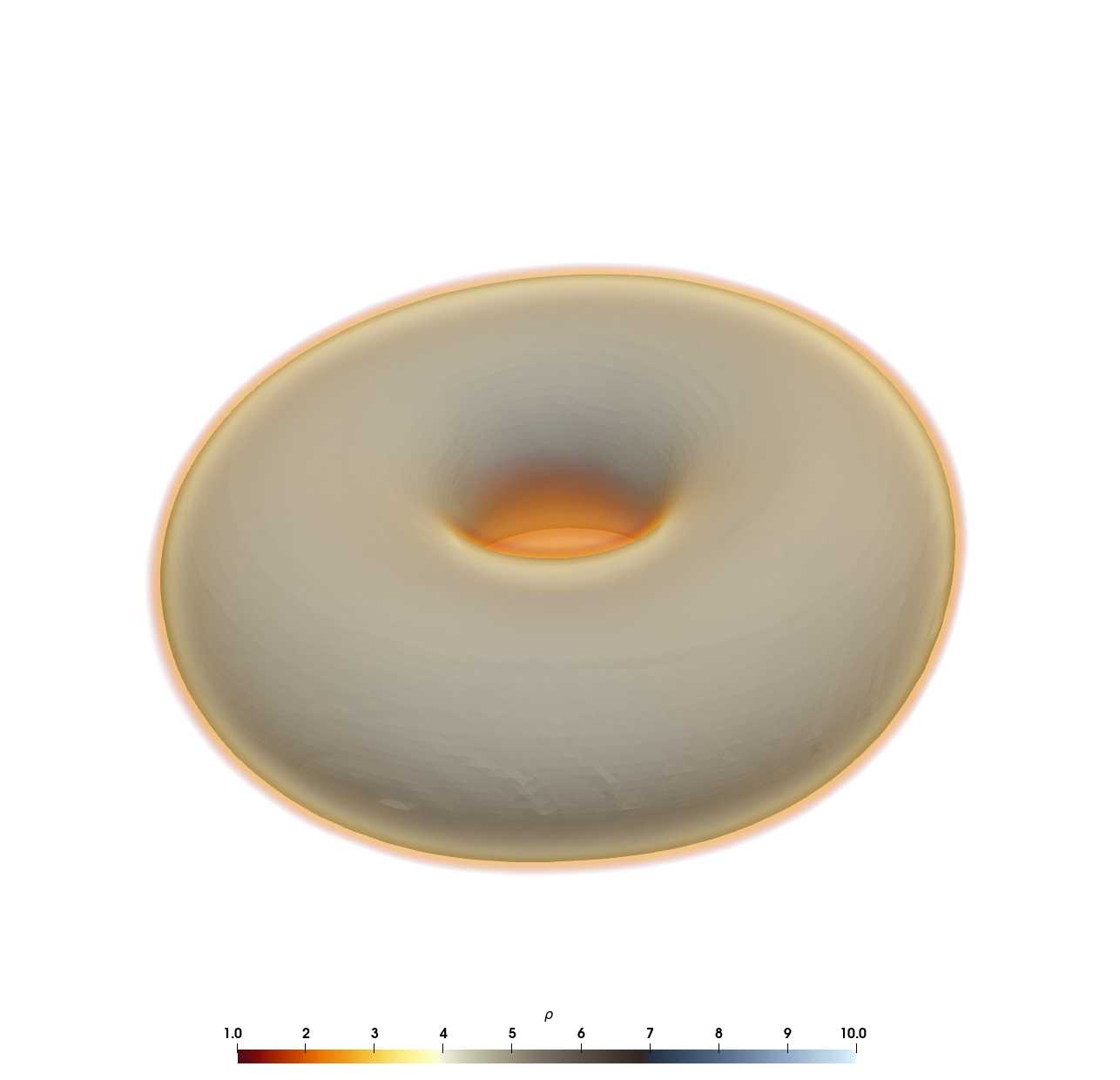}
        \caption*{$t=0.4$}
    \end{subfigure}
    \begin{subfigure}{0.16\textwidth}
        \centering
        \includegraphics[width=\textwidth,trim=30mm 40mm 30mm 70mm, clip]{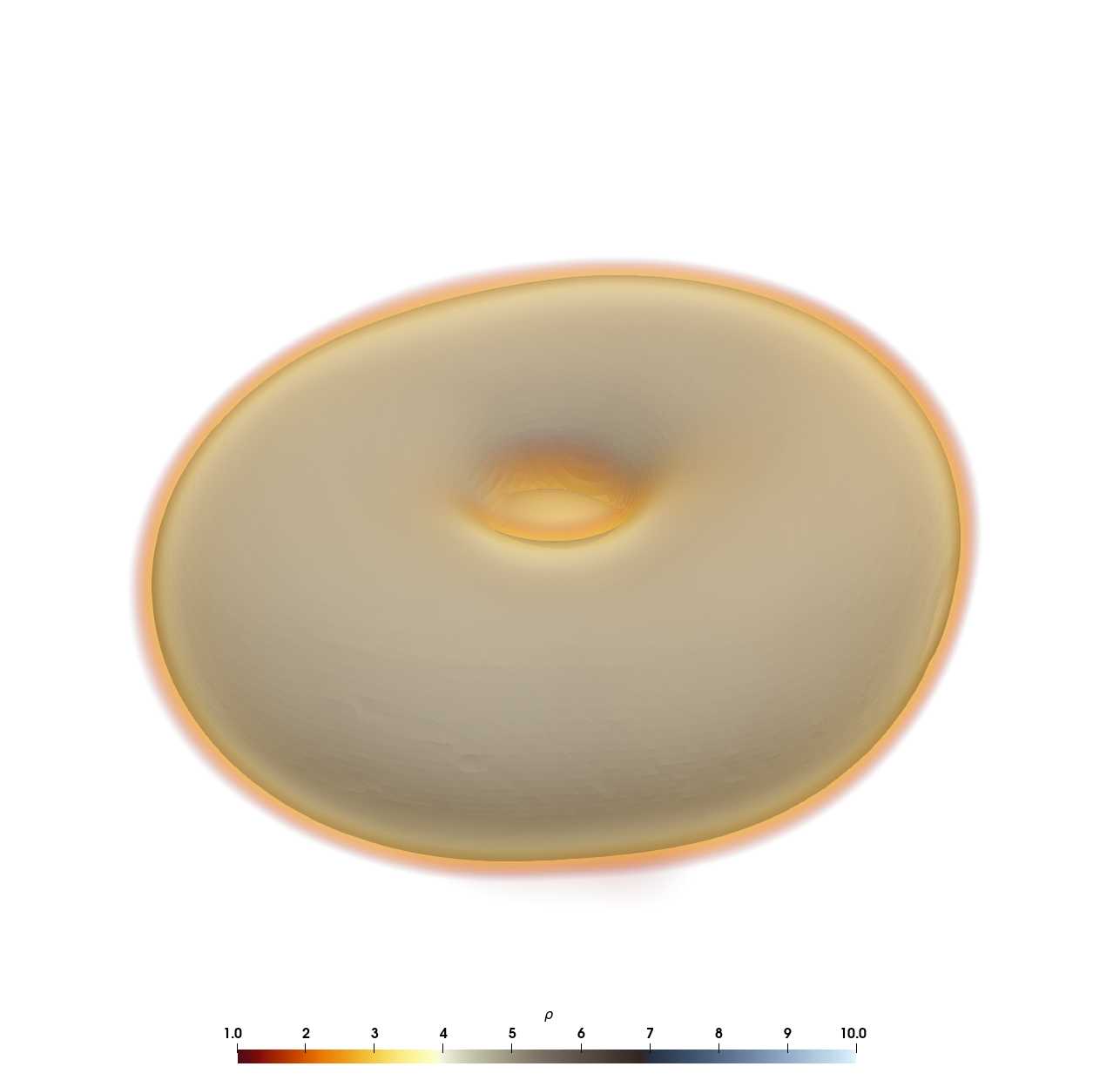}
        \caption*{$t=0.6$}
    \end{subfigure}
    \begin{subfigure}{0.16\textwidth}
        \centering
        \includegraphics[width=\textwidth,trim=30mm 40mm 30mm 70mm, clip]{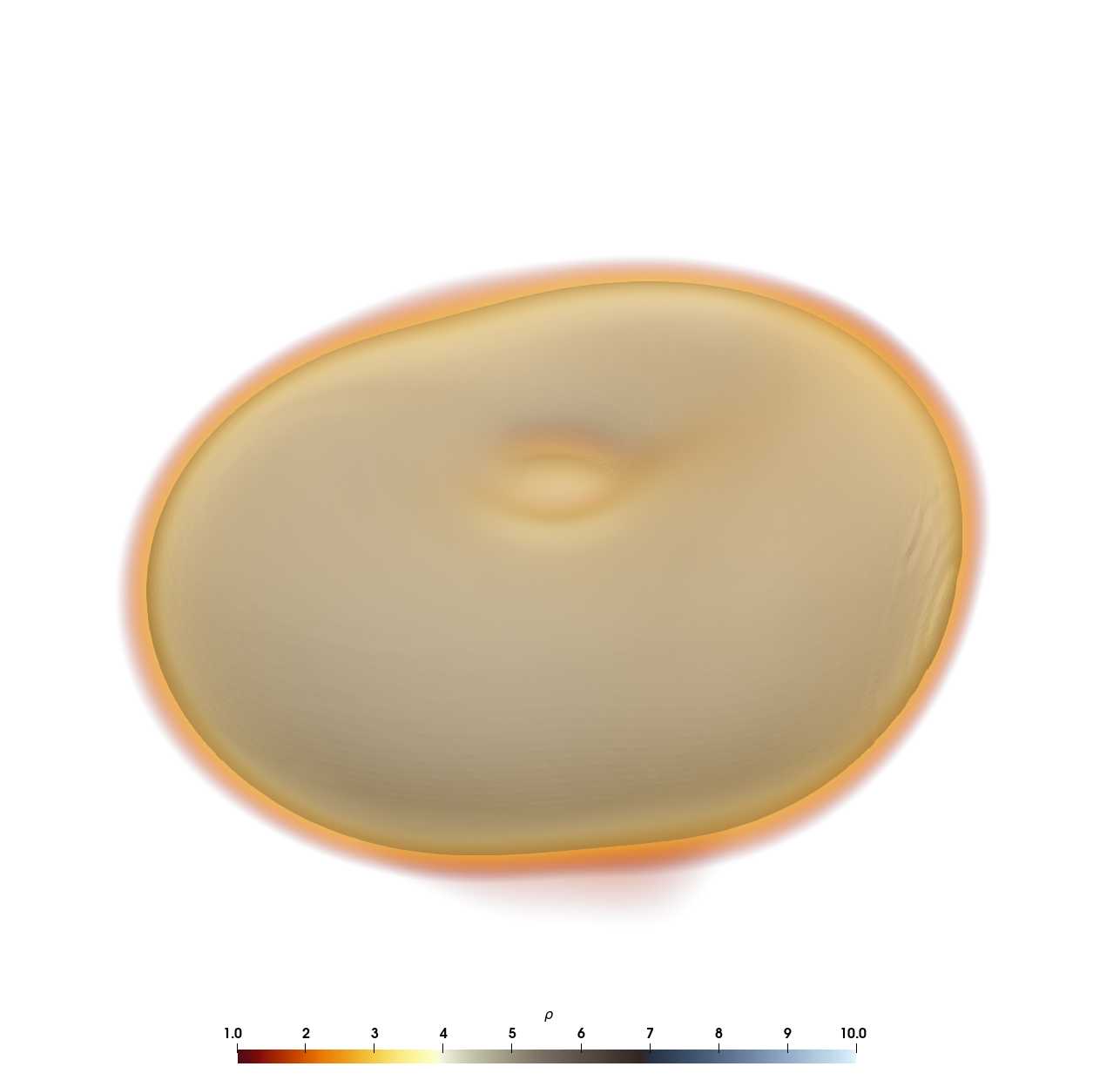}
        \caption*{$t=0.8$}
    \end{subfigure}
    \begin{subfigure}{0.16\textwidth}
        \centering
        \includegraphics[width=\textwidth,trim=30mm 40mm 30mm 70mm, clip]{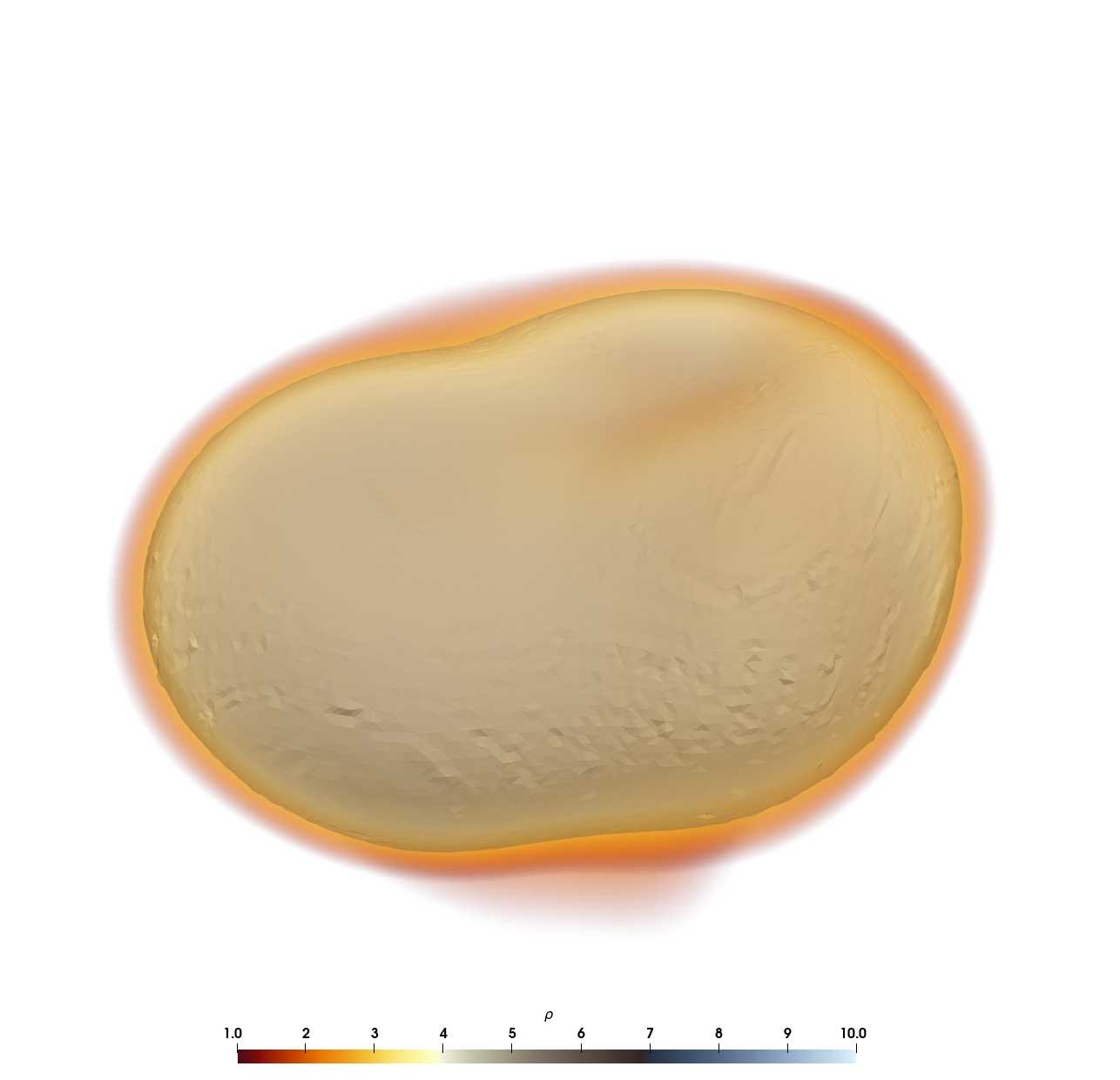}
        \caption*{$t=1.0$}
    \end{subfigure}
\subcaption{$\rho_1$}
\label{f:tdtb1}
\end{minipage}
\begin{minipage}[b]{\textwidth}
    \begin{minipage}[b]{\textwidth}
    \hfill
        \begin{subfigure}{\textwidth}
            \centering
            \includegraphics[width=\textwidth,trim=0mm 0mm 0mm 400mm, clip]{figures/reaction/doubleTorusBunny/pdhg0.0000..jpg}
        \end{subfigure}
    \end{minipage}
    
    \begin{subfigure}{0.16\textwidth}
        \centering
        \includegraphics[width=\textwidth,trim=30mm 40mm 30mm 70mm, clip]{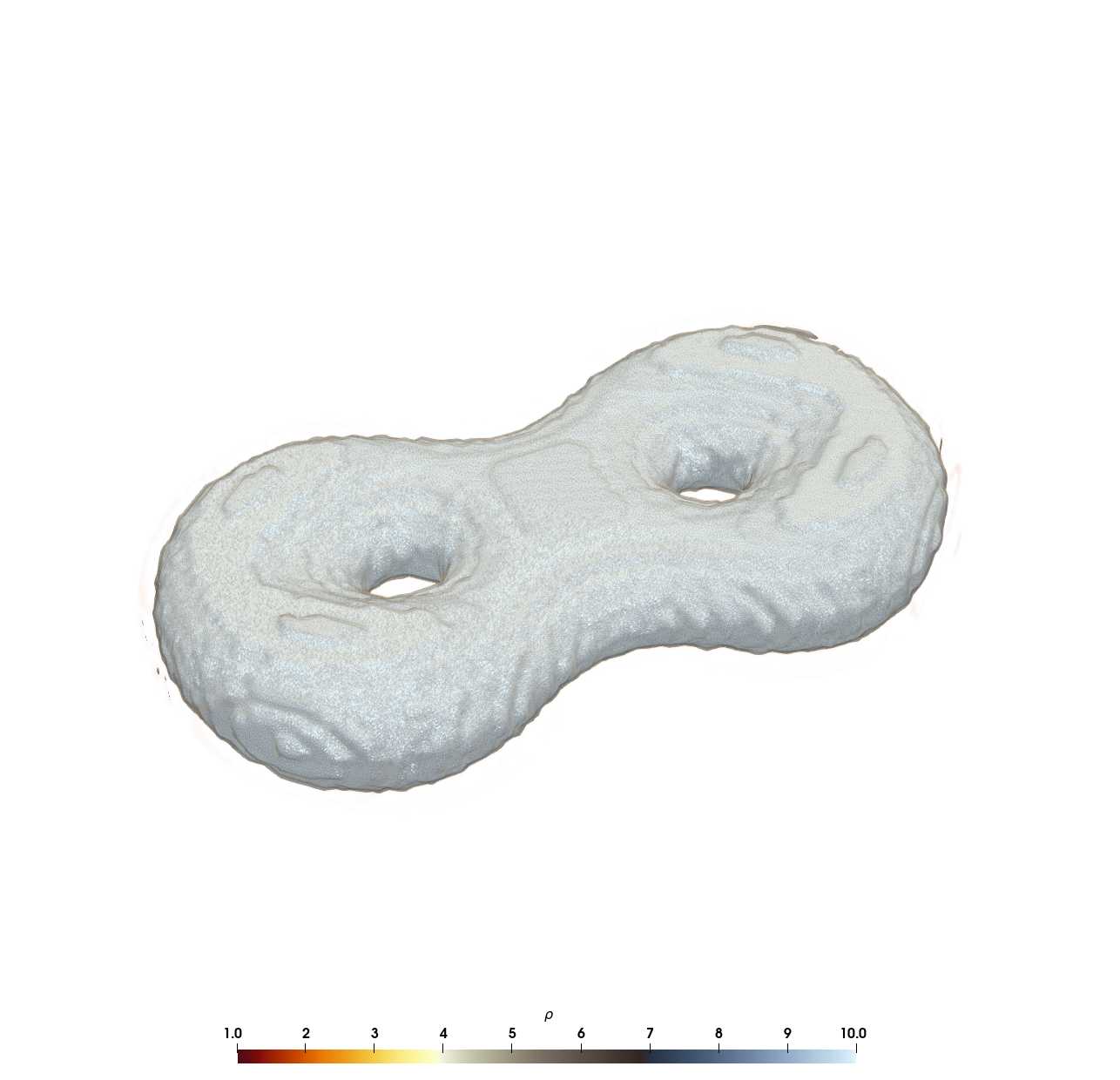}
    \end{subfigure}
    \begin{subfigure}{0.16\textwidth}
        \centering
        \includegraphics[width=\textwidth,trim=30mm 40mm 30mm 70mm, clip]{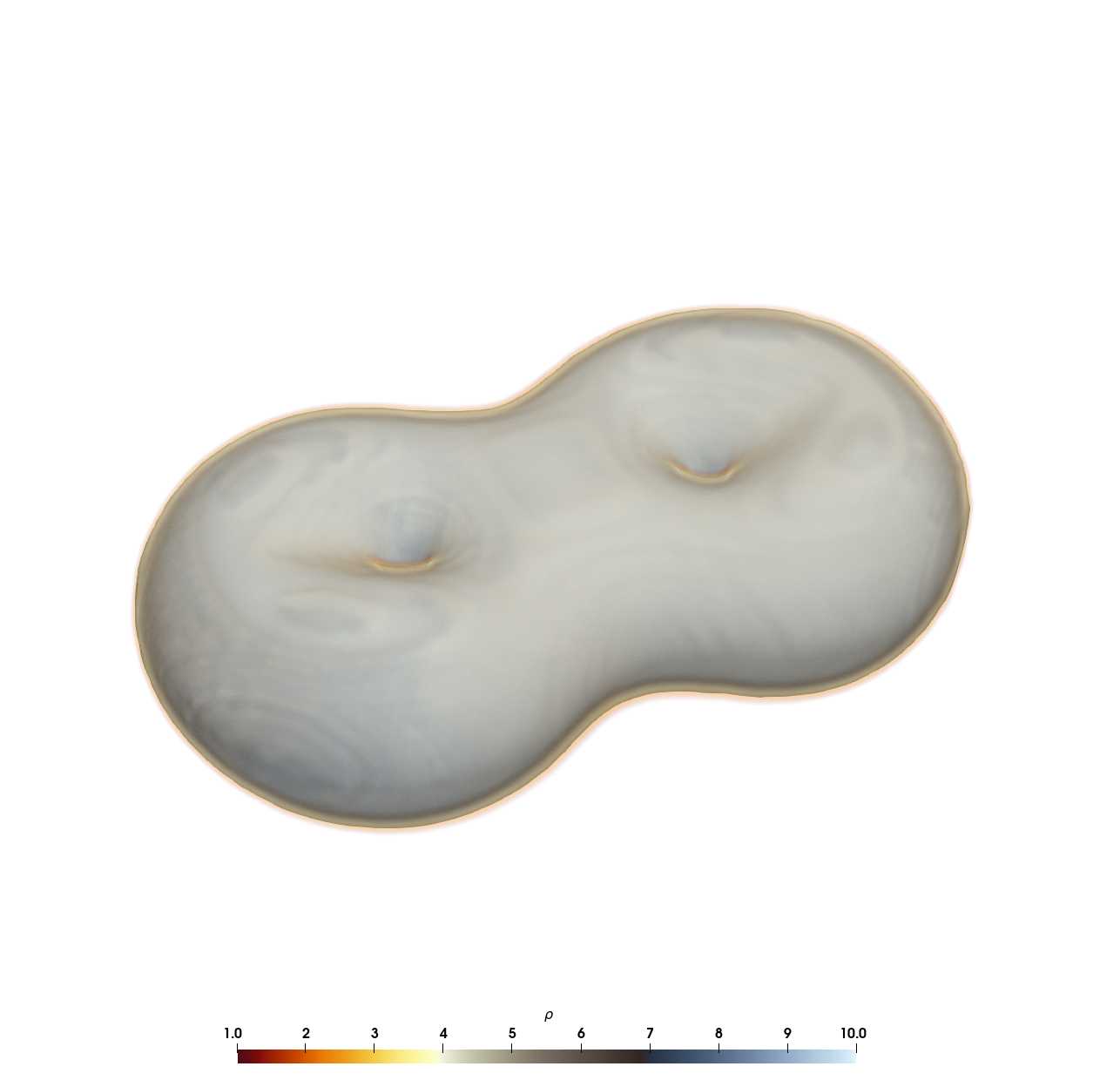}
    \end{subfigure}
    \begin{subfigure}{0.16\textwidth}
        \centering
        \includegraphics[width=\textwidth,trim=30mm 40mm 30mm 70mm, clip]{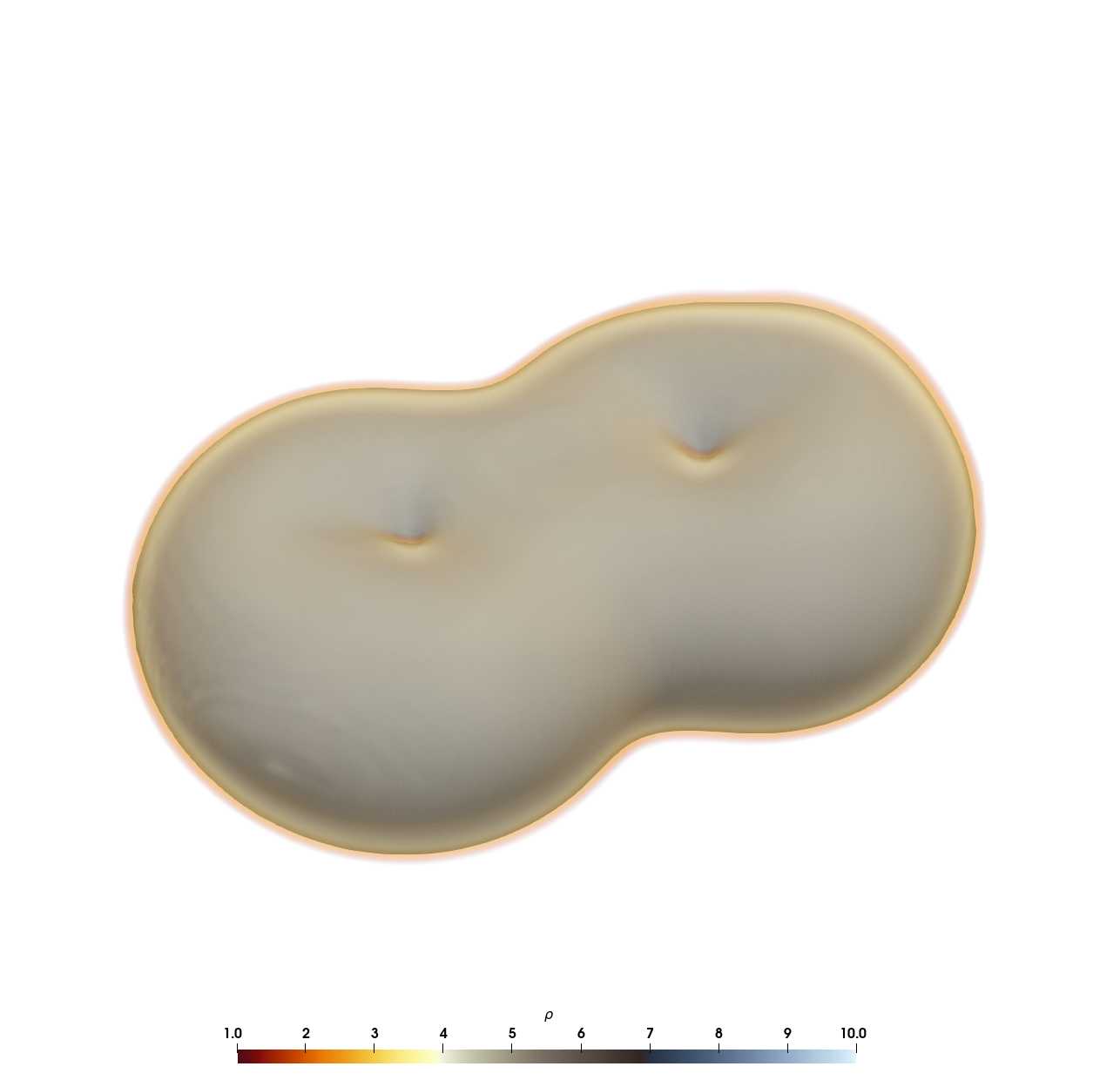}
    \end{subfigure}
    \begin{subfigure}{0.16\textwidth}
        \centering
        \includegraphics[width=\textwidth,trim=30mm 40mm 30mm 70mm, clip]{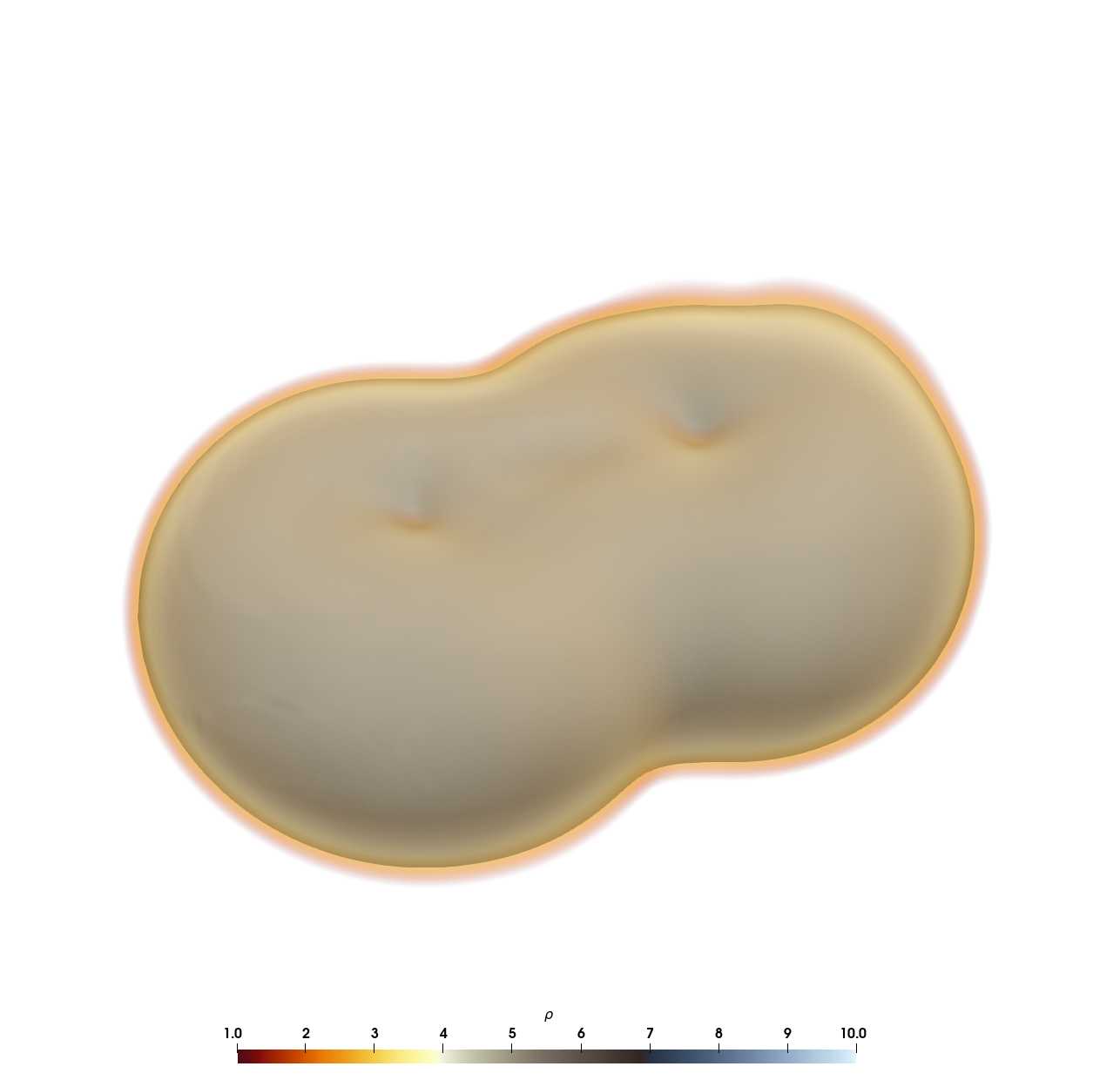}
    \end{subfigure}
    \begin{subfigure}{0.16\textwidth}
        \centering
        \includegraphics[width=\textwidth,trim=30mm 40mm 30mm 70mm, clip]{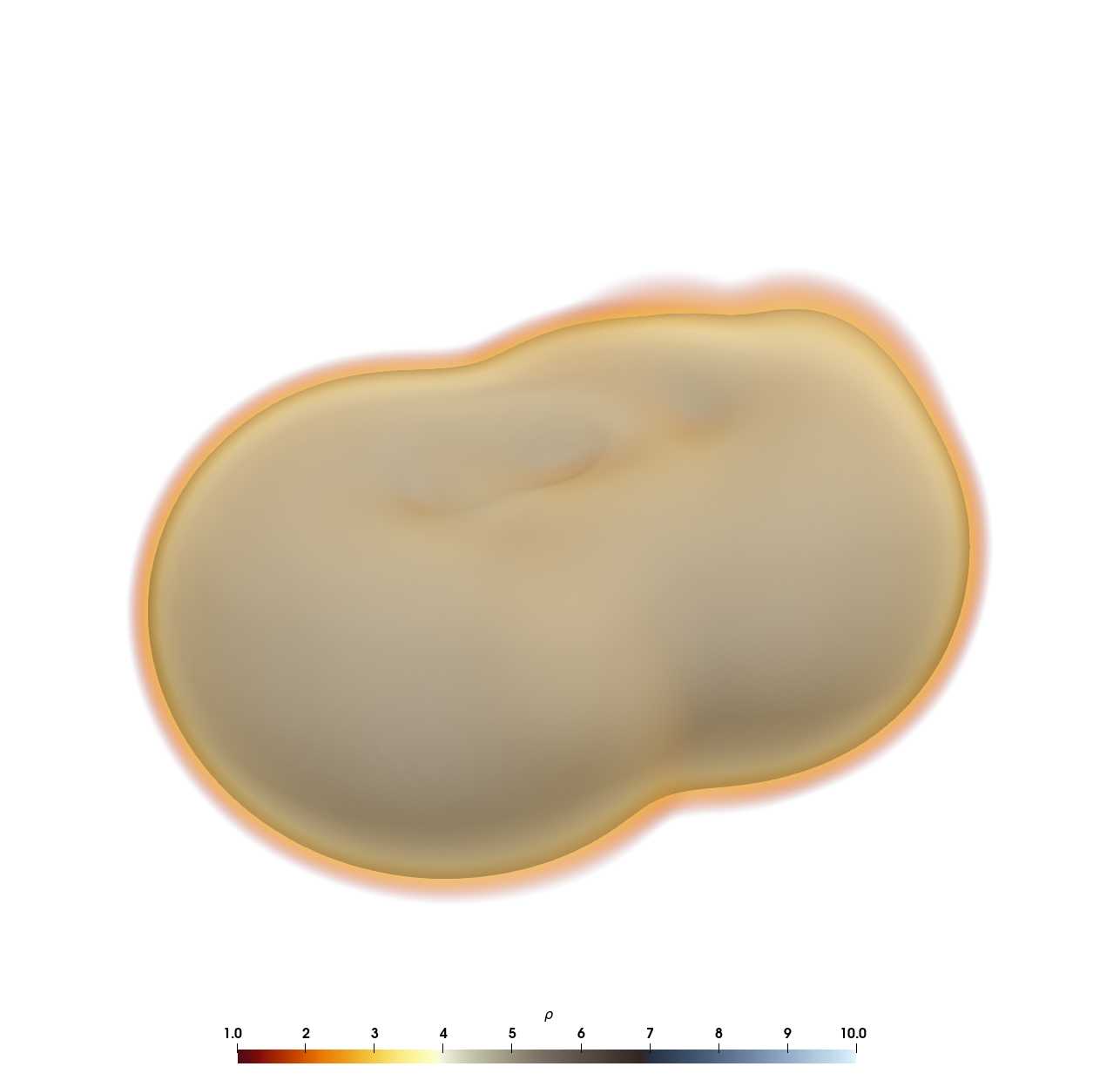}
    \end{subfigure}
    \begin{subfigure}{0.16\textwidth}
        \centering
        \includegraphics[width=\textwidth,trim=30mm 40mm 30mm 70mm, clip]{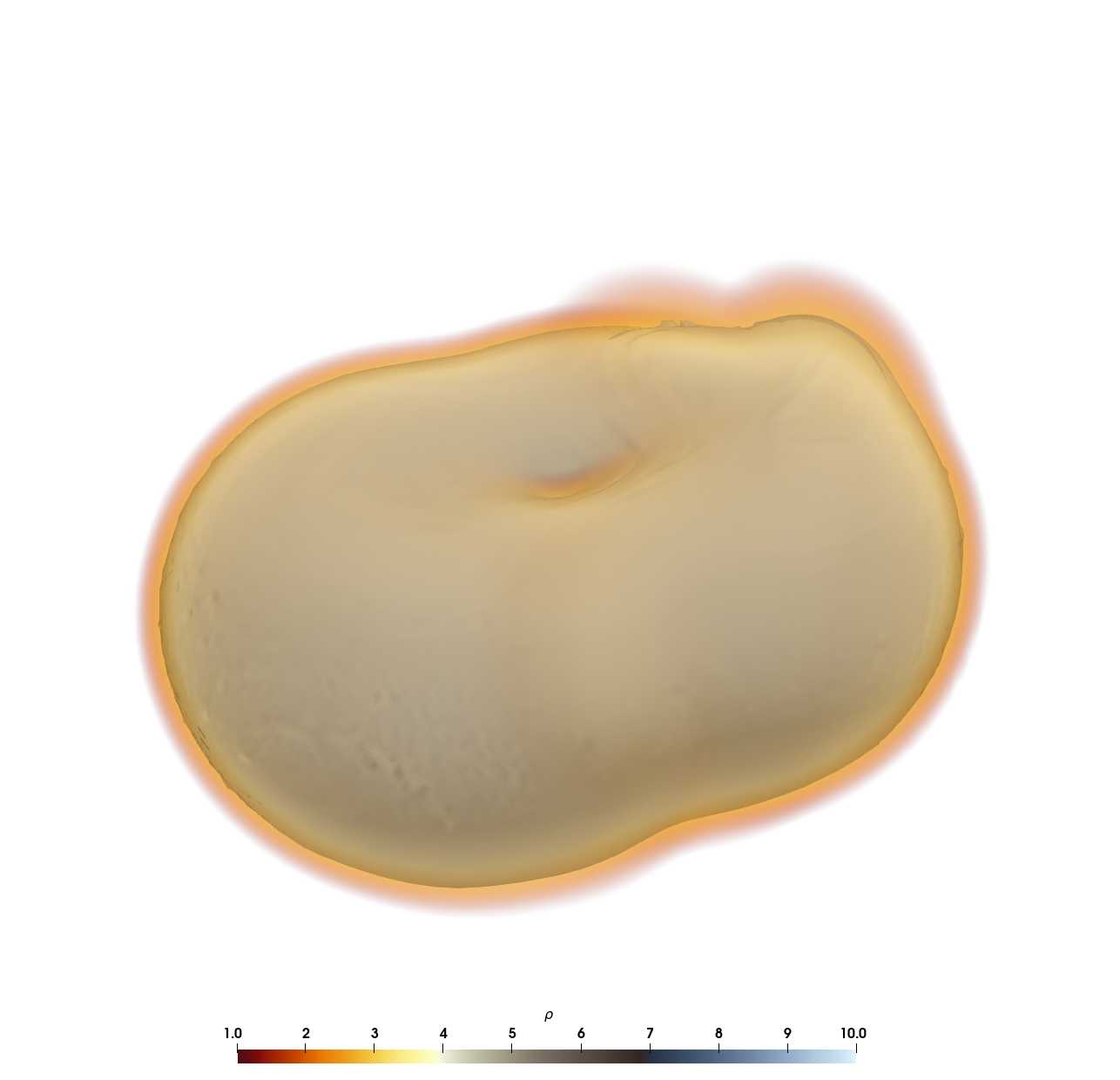}
    \end{subfigure}

    \begin{subfigure}{0.16\textwidth}
        \centering
        \includegraphics[width=\textwidth,trim=30mm 40mm 30mm 70mm, clip]{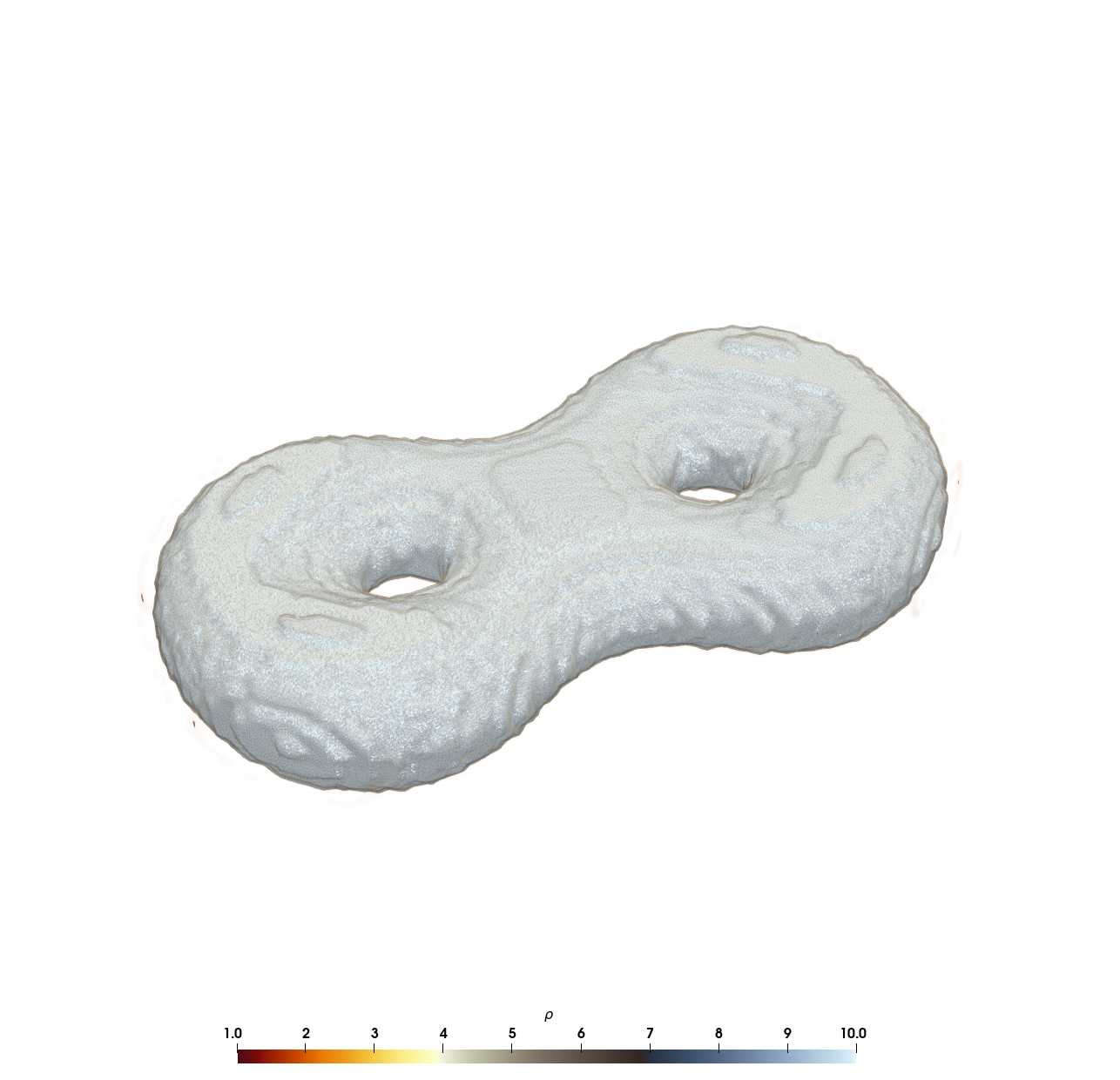}
        \caption*{$t=0$}
    \end{subfigure}
    \begin{subfigure}{0.16\textwidth}
        \centering
        \includegraphics[width=\textwidth,trim=30mm 40mm 30mm 70mm, clip]{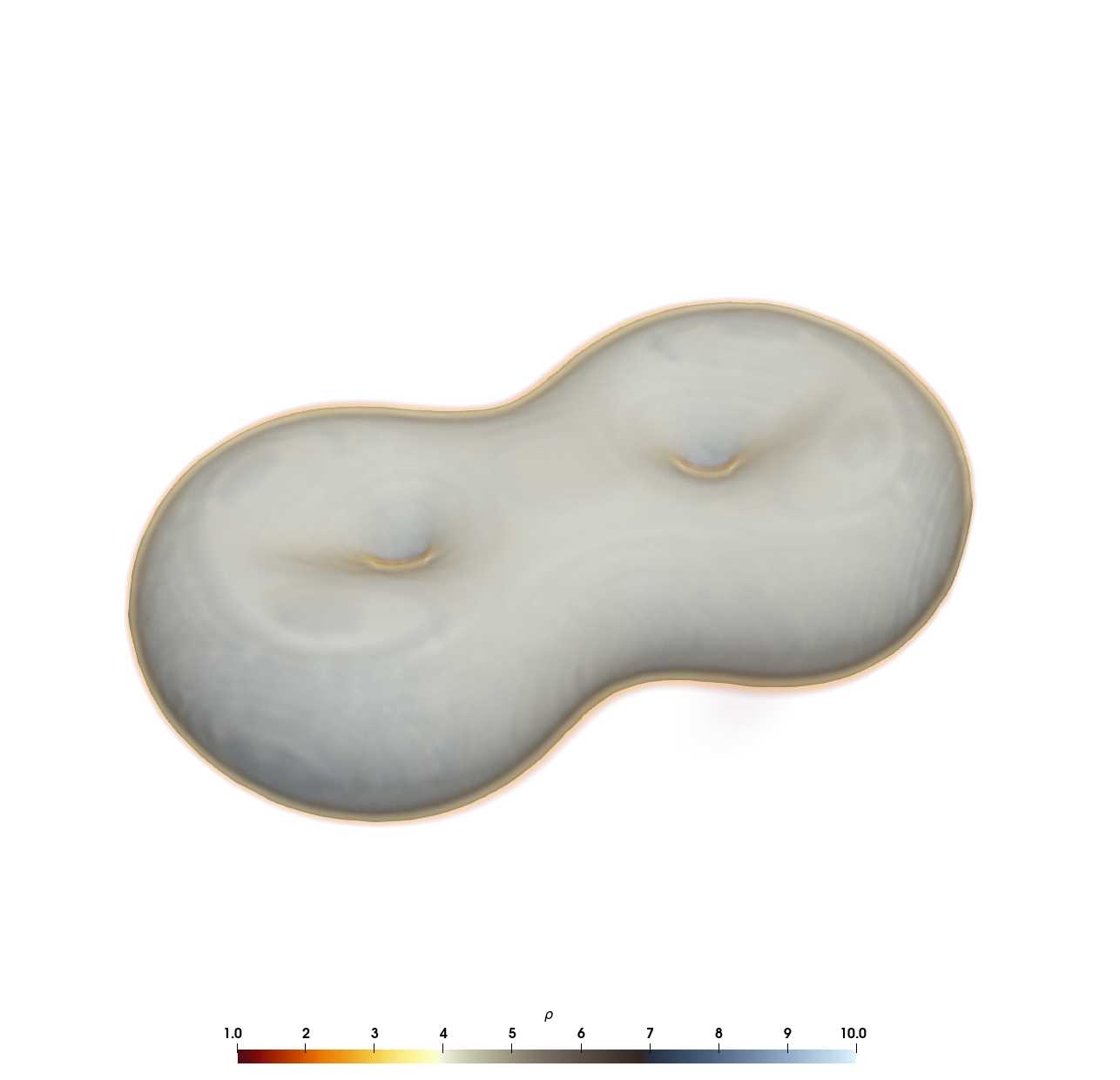}
        \caption*{$t=0.2$}
    \end{subfigure}
    \begin{subfigure}{0.16\textwidth}
        \centering
        \includegraphics[width=\textwidth,trim=30mm 40mm 30mm 70mm, clip]{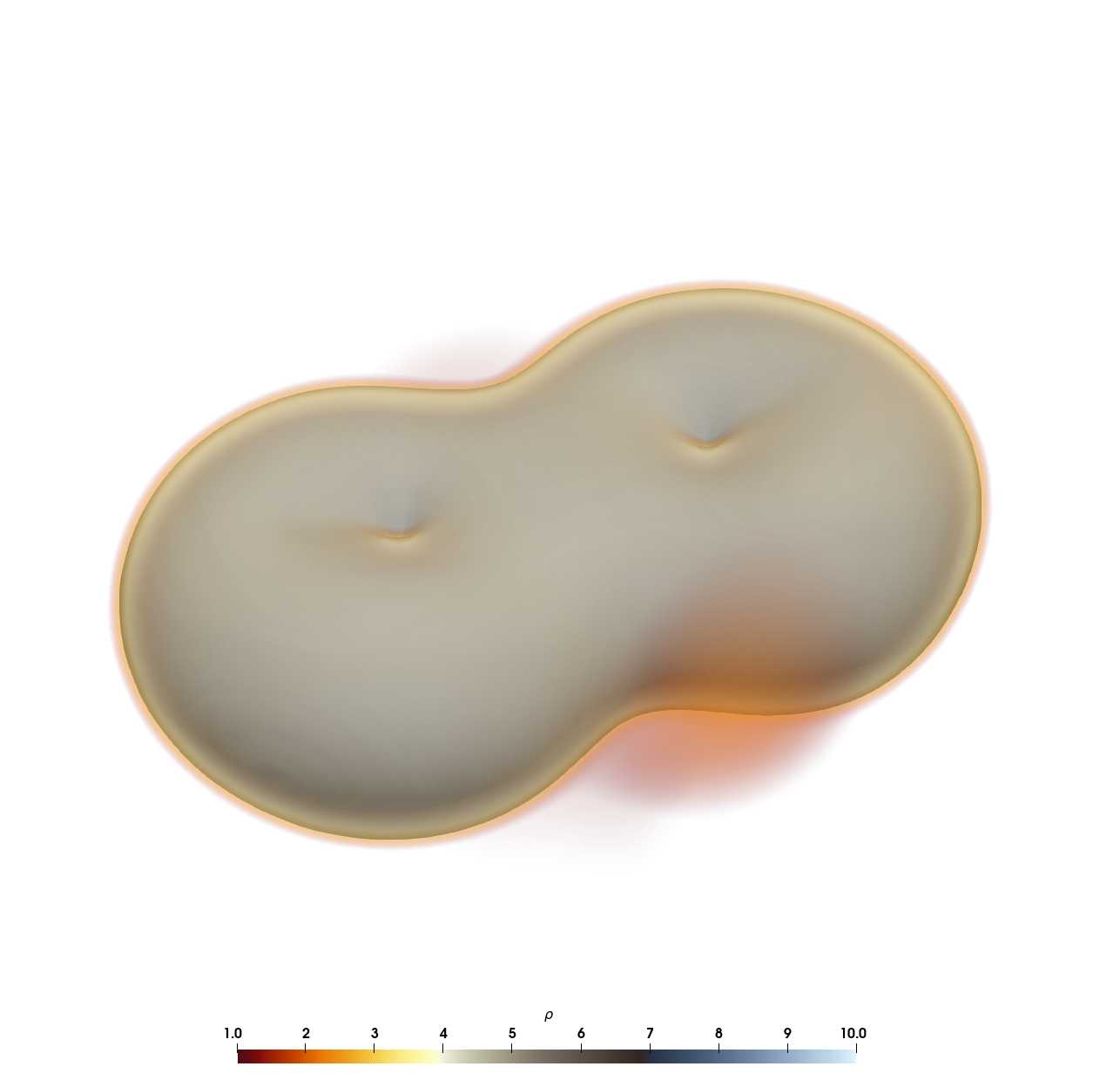}
        \caption*{$t=0.4$}
    \end{subfigure}
    \begin{subfigure}{0.16\textwidth}
        \centering
        \includegraphics[width=\textwidth,trim=30mm 40mm 30mm 70mm, clip]{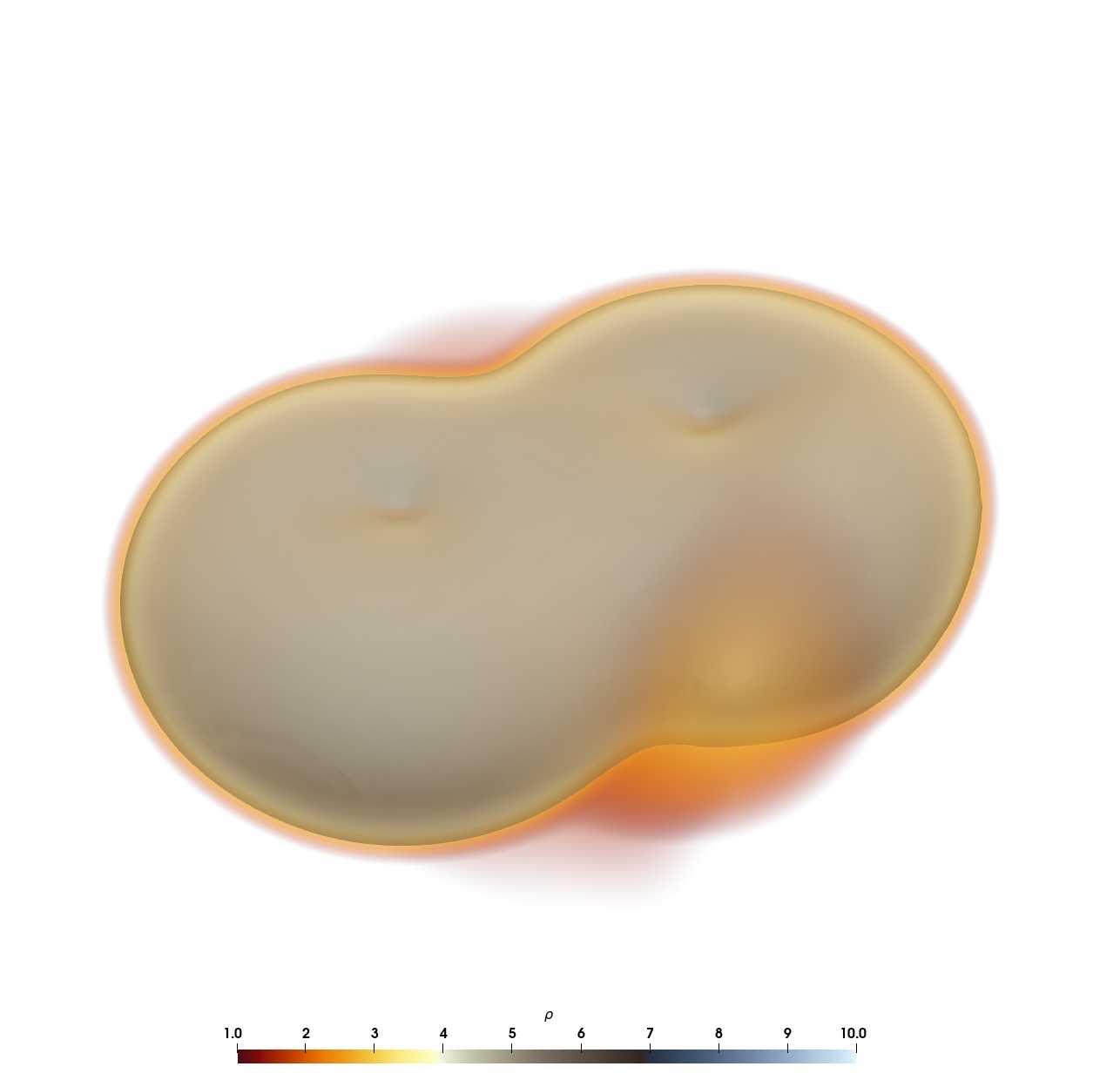}
        \caption*{$t=0.6$}
    \end{subfigure}
    \begin{subfigure}{0.16\textwidth}
        \centering
        \includegraphics[width=\textwidth,trim=30mm 40mm 30mm 70mm, clip]{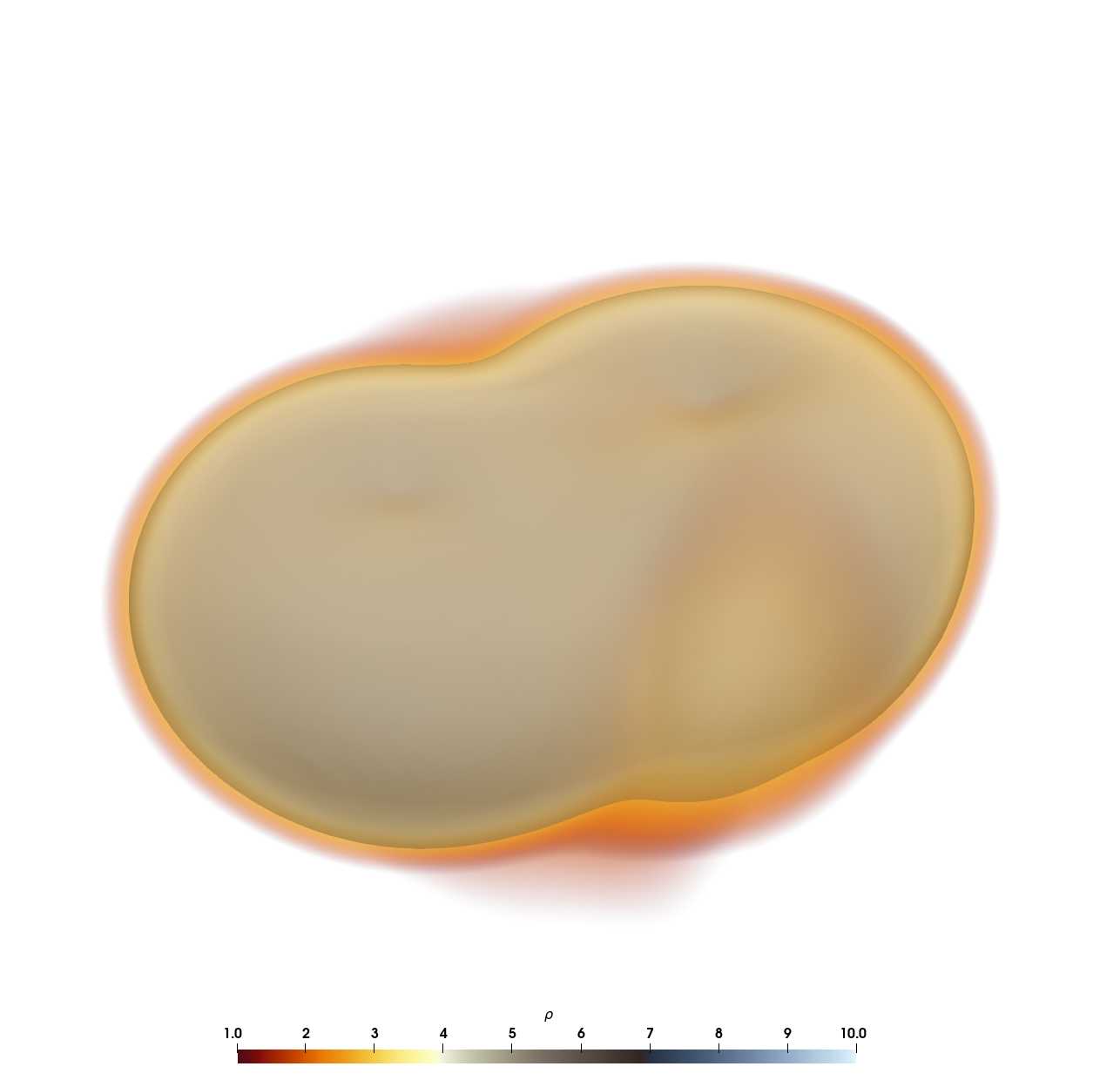}
        \caption*{$t=0.8$}
    \end{subfigure}
    \begin{subfigure}{0.16\textwidth}
        \centering
        \includegraphics[width=\textwidth,trim=30mm 40mm 30mm 70mm, clip]{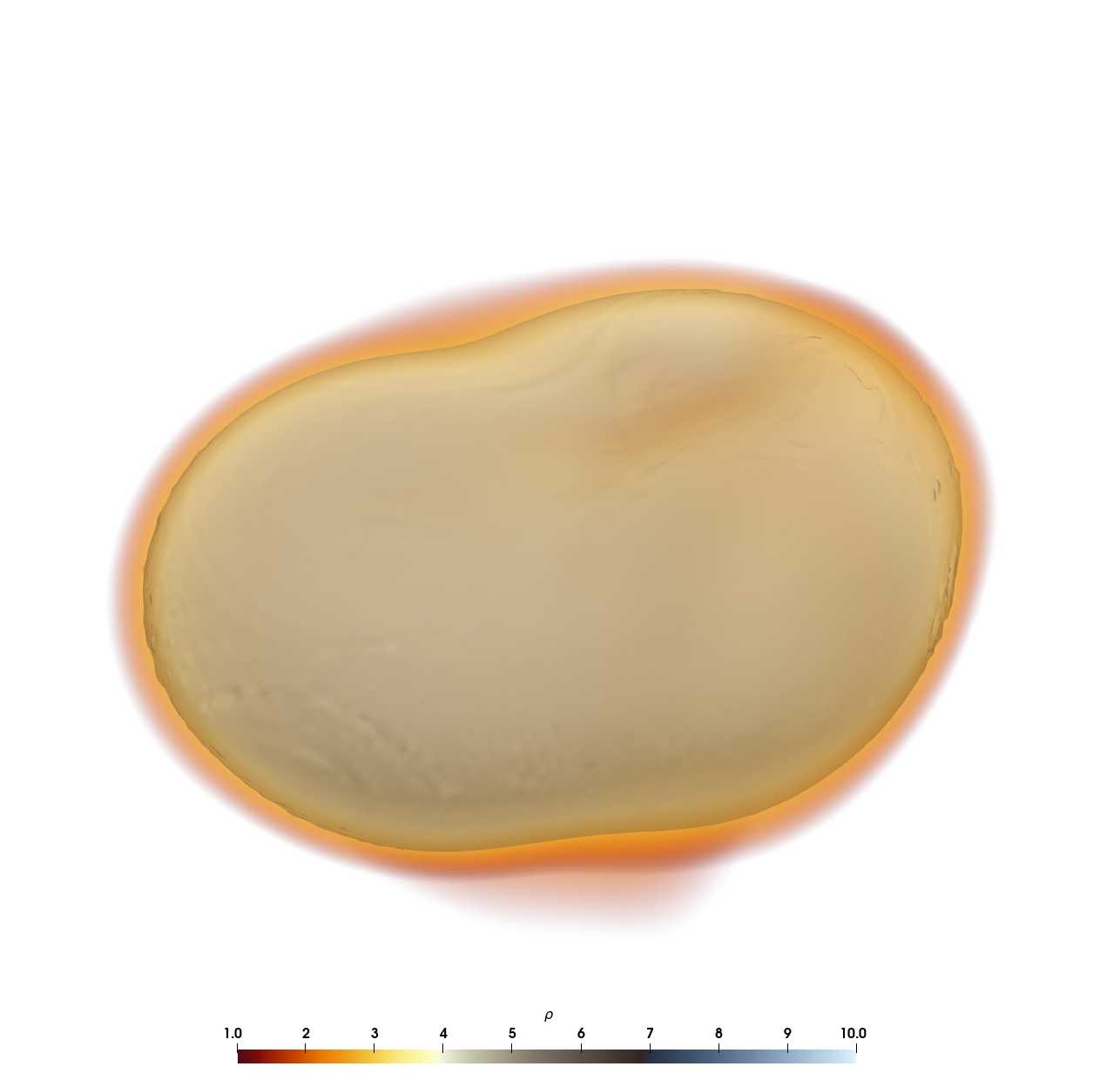}
        \caption*{$t=1.0$}
    \end{subfigure}

\subcaption{$\rho_2$}
\label{f:tdtb2}
\end{minipage}

\begin{minipage}[b]{\textwidth}

    \begin{minipage}[b]{\textwidth}
    \hfill
        \begin{subfigure}{\textwidth}
            \centering
            \includegraphics[width=\textwidth,trim=0mm 0mm 0mm 400mm, clip]{figures/reaction/doubleTorusBunny/pdhg0.0000..jpg}
        \end{subfigure}
    \end{minipage}
    
    \begin{subfigure}{0.16\textwidth}
        \centering
        \includegraphics[width=\textwidth,trim=30mm 40mm 30mm 70mm, clip]{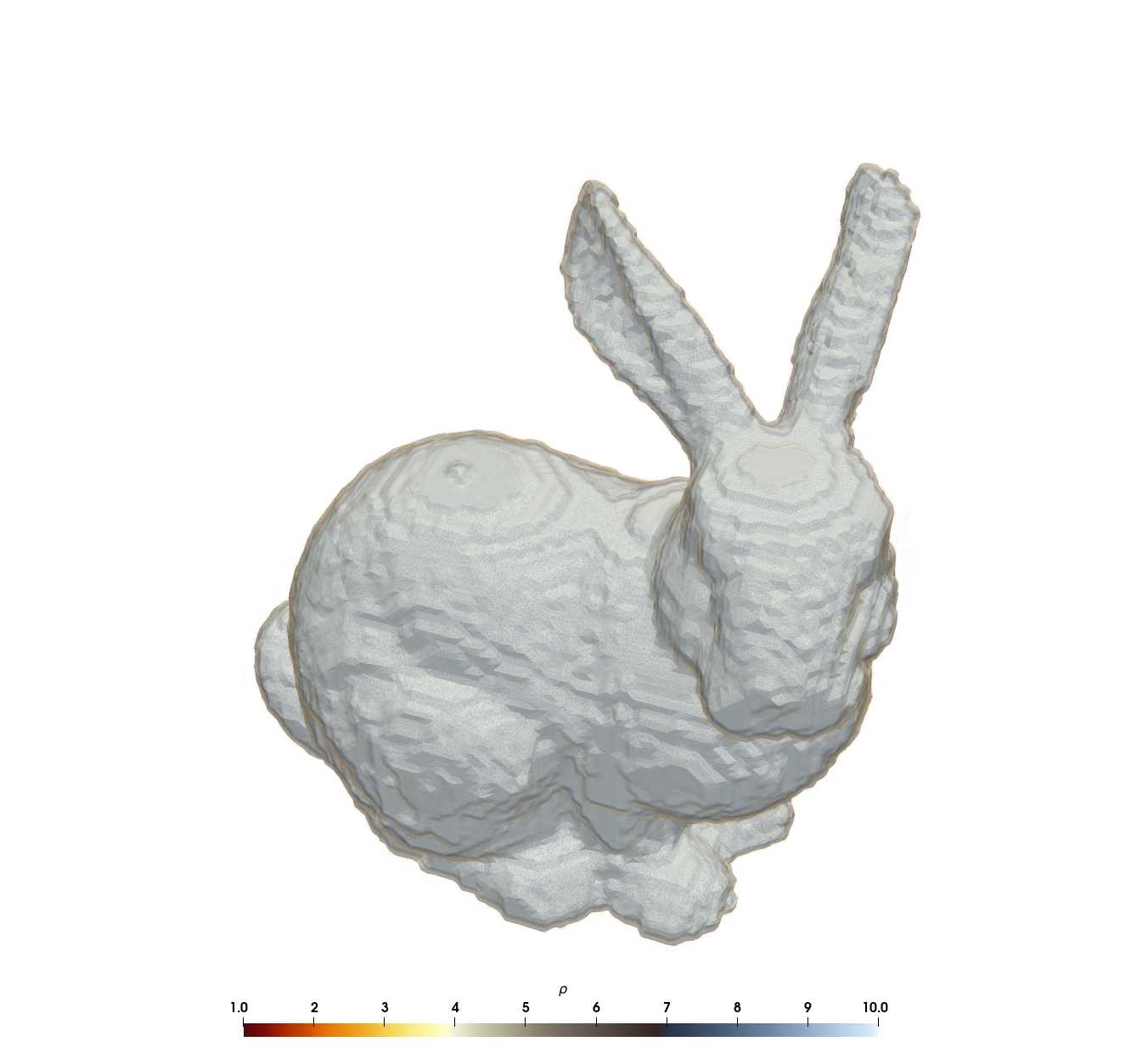}
    \end{subfigure}
    \begin{subfigure}{0.16\textwidth}
        \centering
        \includegraphics[width=\textwidth,trim=30mm 40mm 30mm 70mm, clip]{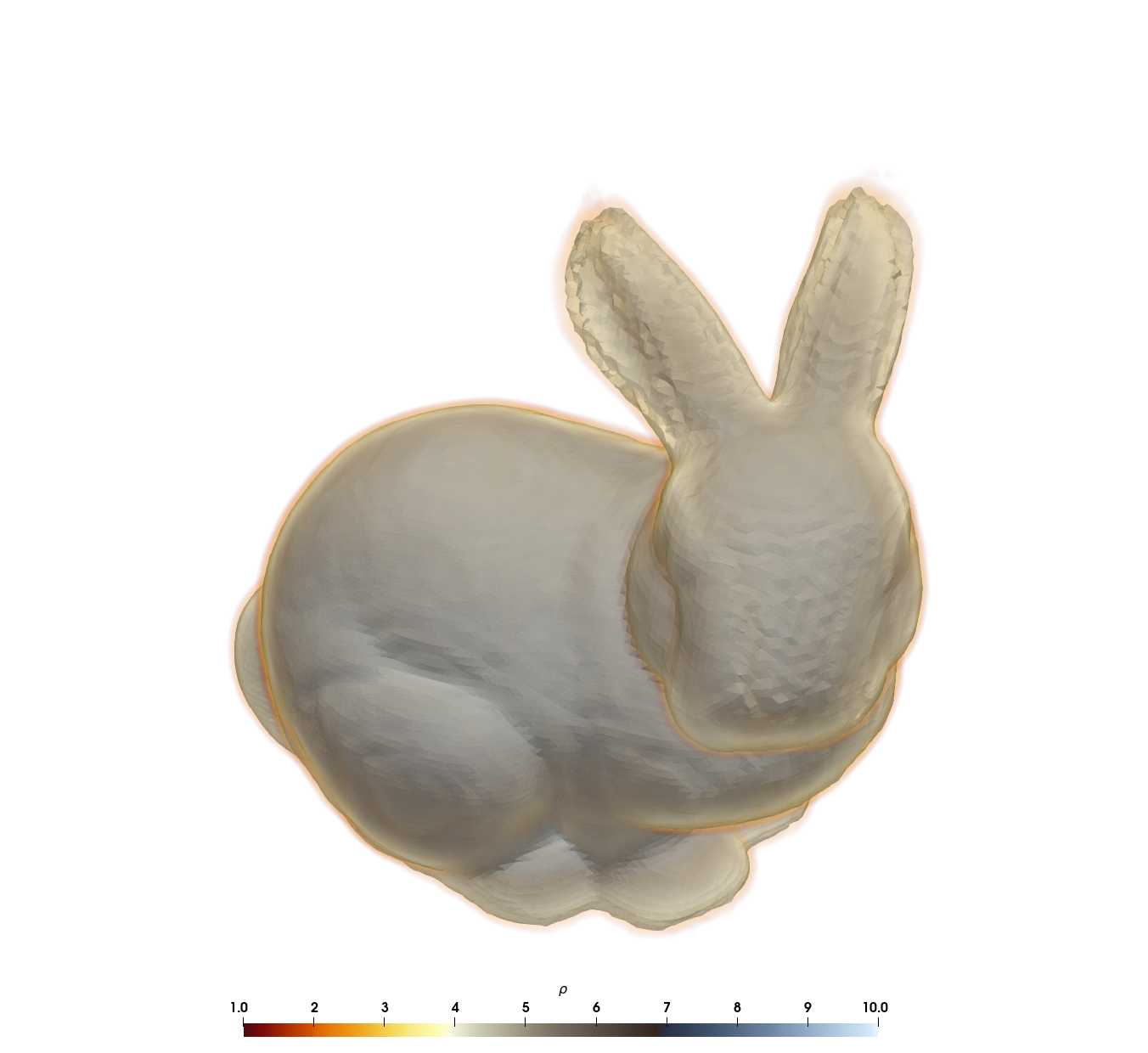}
    \end{subfigure}
    \begin{subfigure}{0.16\textwidth}
        \centering
        \includegraphics[width=\textwidth,trim=30mm 40mm 30mm 70mm, clip]{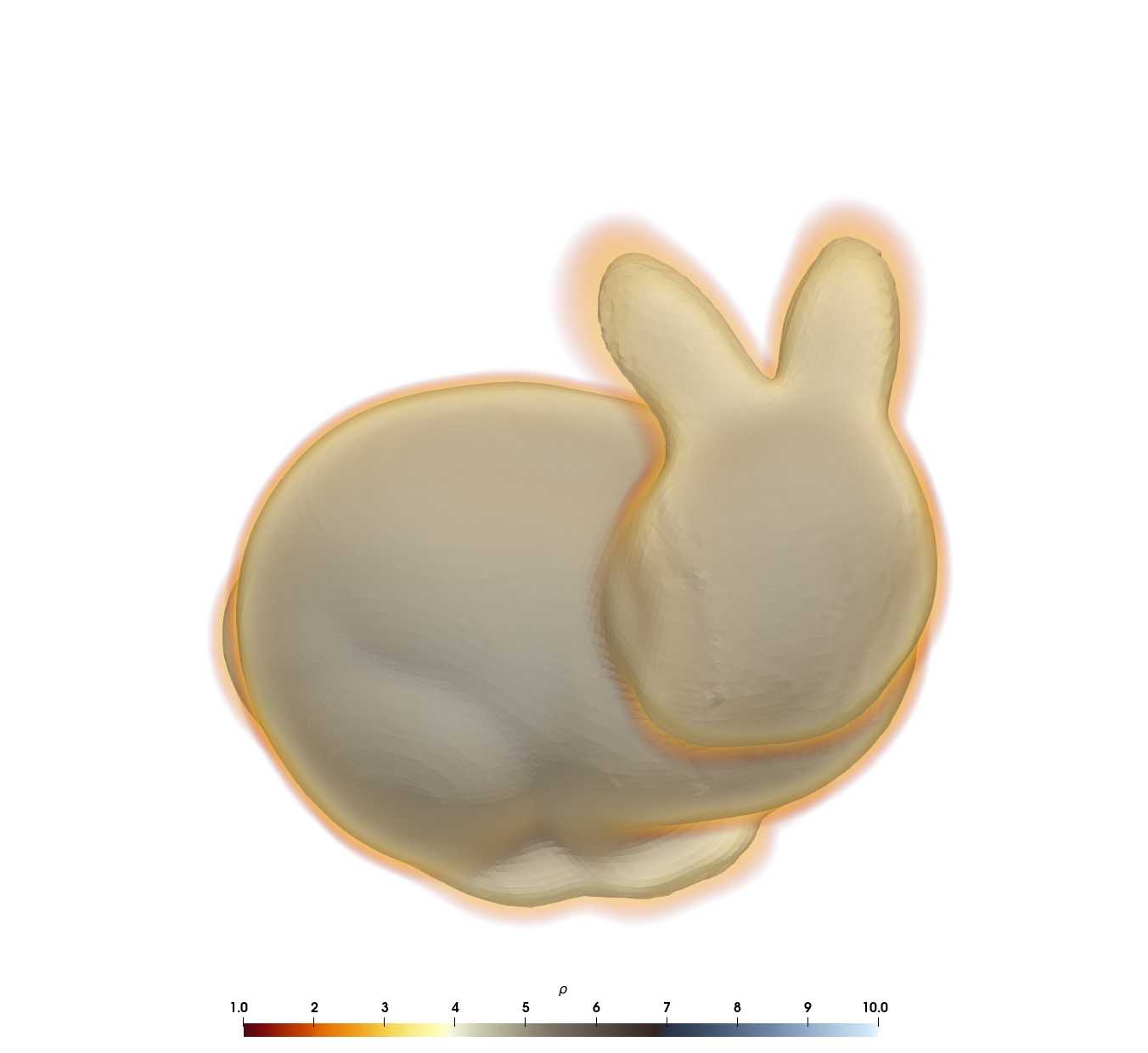}
    \end{subfigure}
    \begin{subfigure}{0.16\textwidth}
        \centering
        \includegraphics[width=\textwidth,trim=30mm 40mm 30mm 70mm, clip]{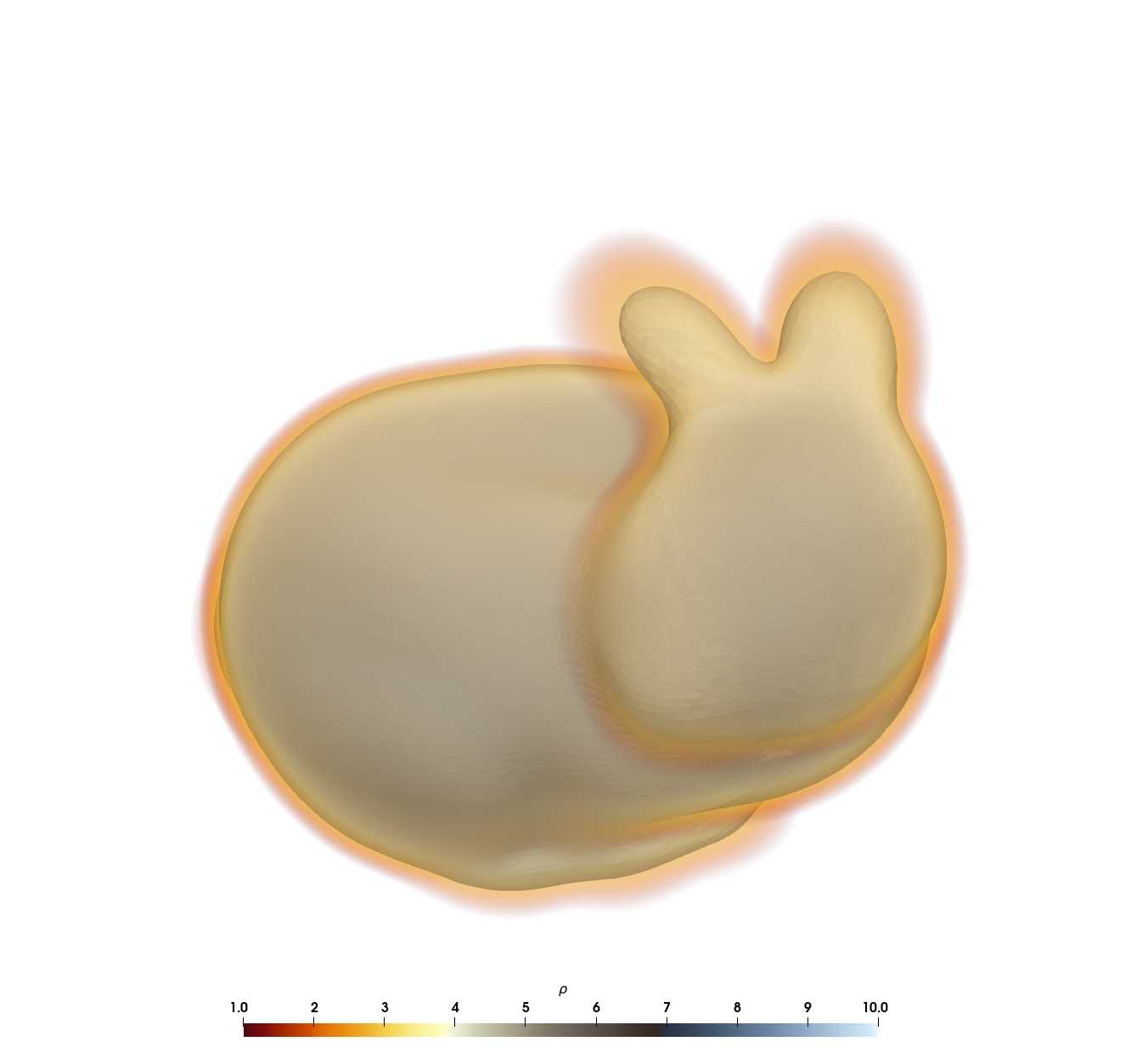}
    \end{subfigure}
    \begin{subfigure}{0.16\textwidth}
        \centering
        \includegraphics[width=\textwidth,trim=30mm 40mm 30mm 70mm, clip]{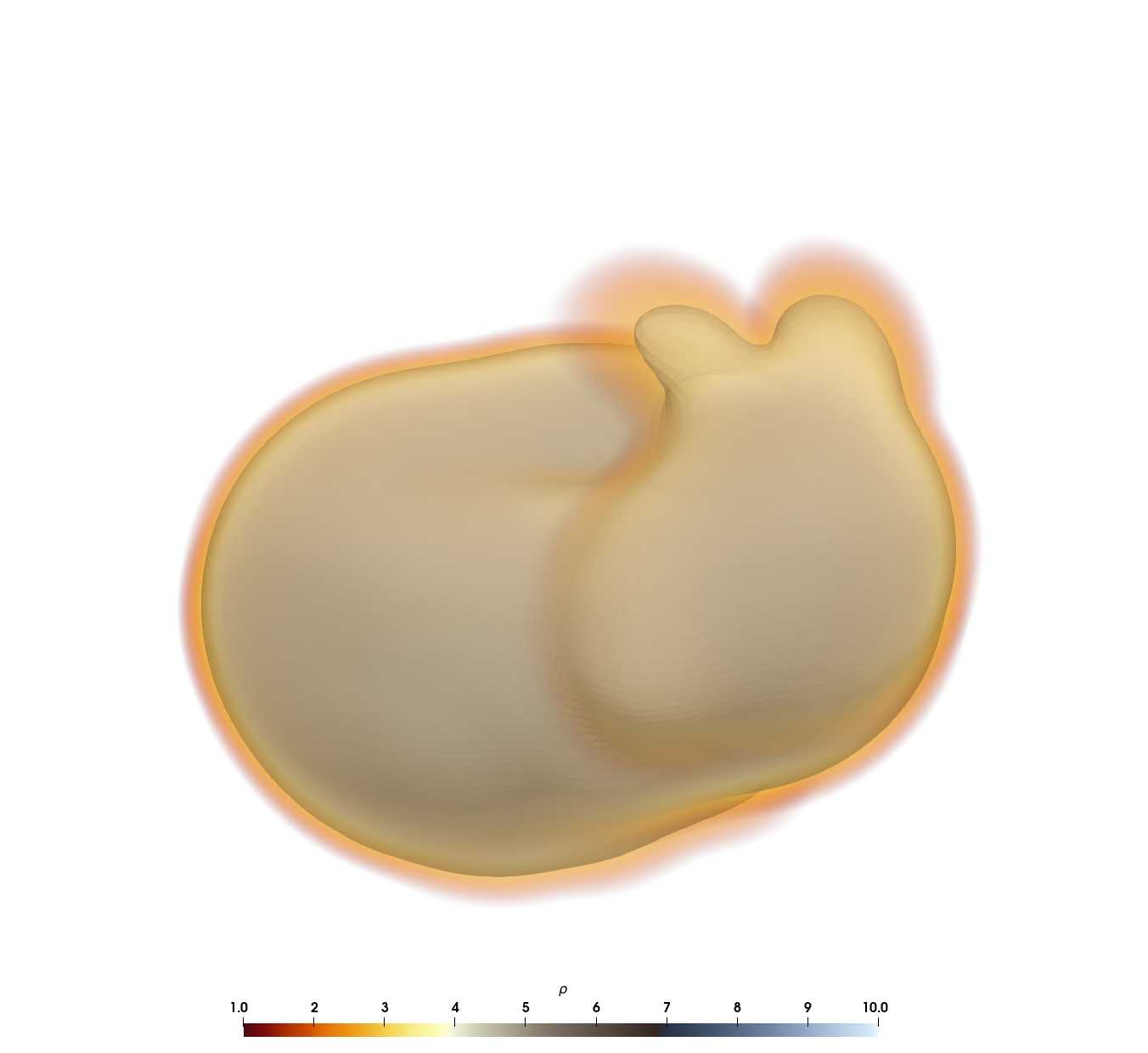}
    \end{subfigure}
    \begin{subfigure}{0.16\textwidth}
        \centering
        \includegraphics[width=\textwidth,trim=30mm 40mm 30mm 70mm, clip]{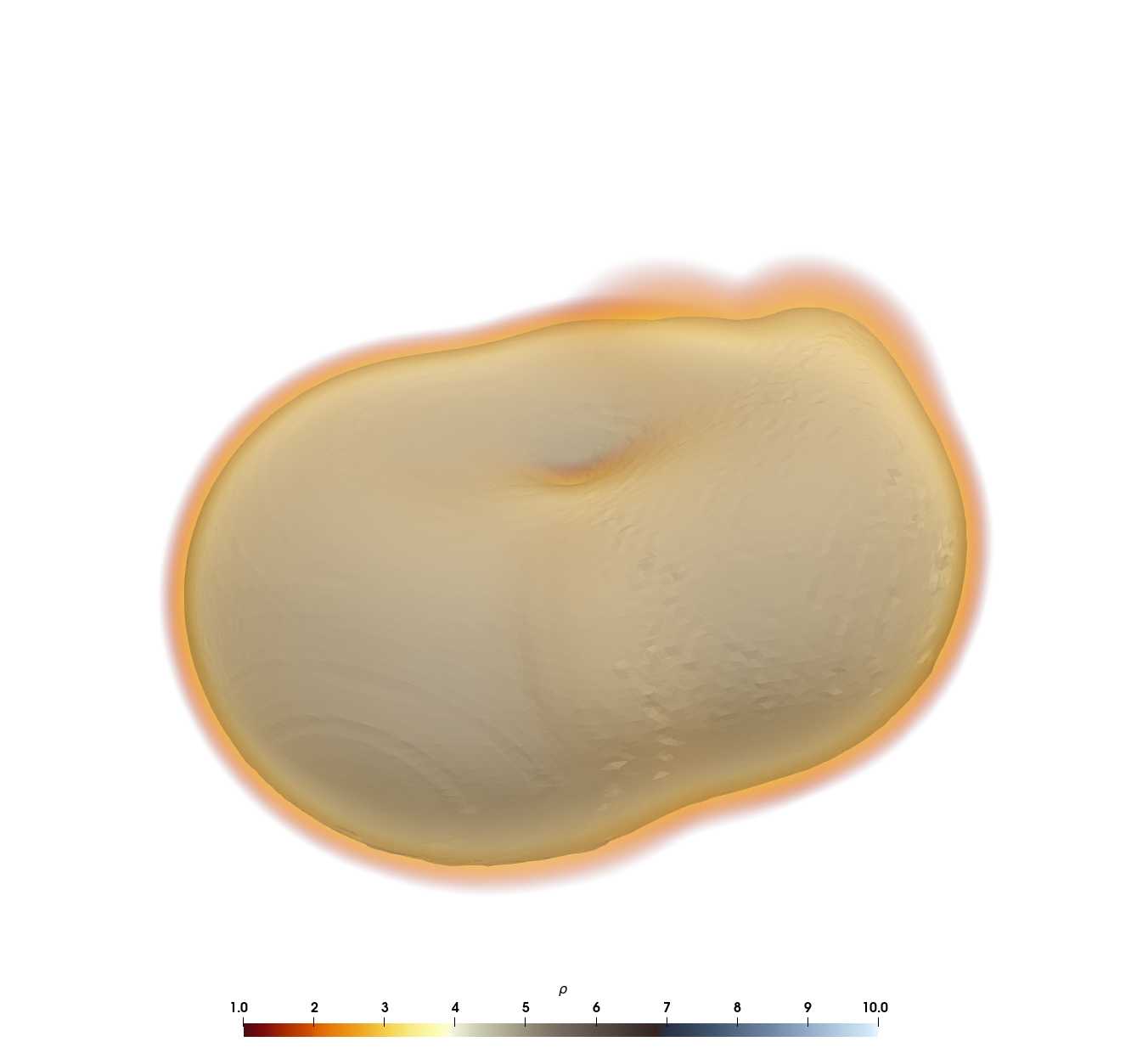}
    \end{subfigure}

    \begin{subfigure}{0.16\textwidth}
        \centering
        \includegraphics[width=\textwidth,trim=30mm 40mm 30mm 70mm, clip]{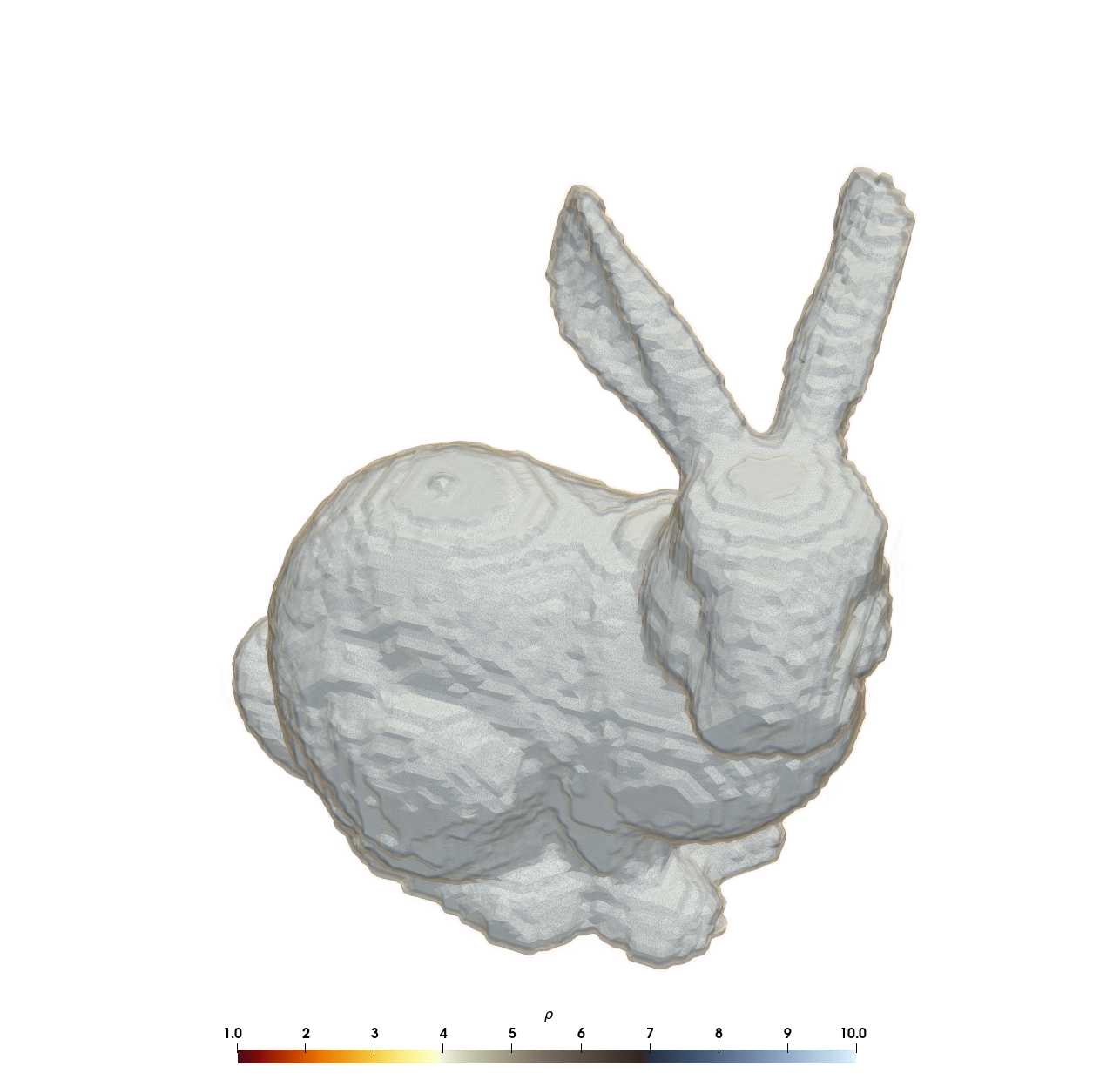}
        \caption*{$t=0$}
    \end{subfigure}
    \begin{subfigure}{0.16\textwidth}
        \centering
        \includegraphics[width=\textwidth,trim=30mm 40mm 30mm 70mm, clip]{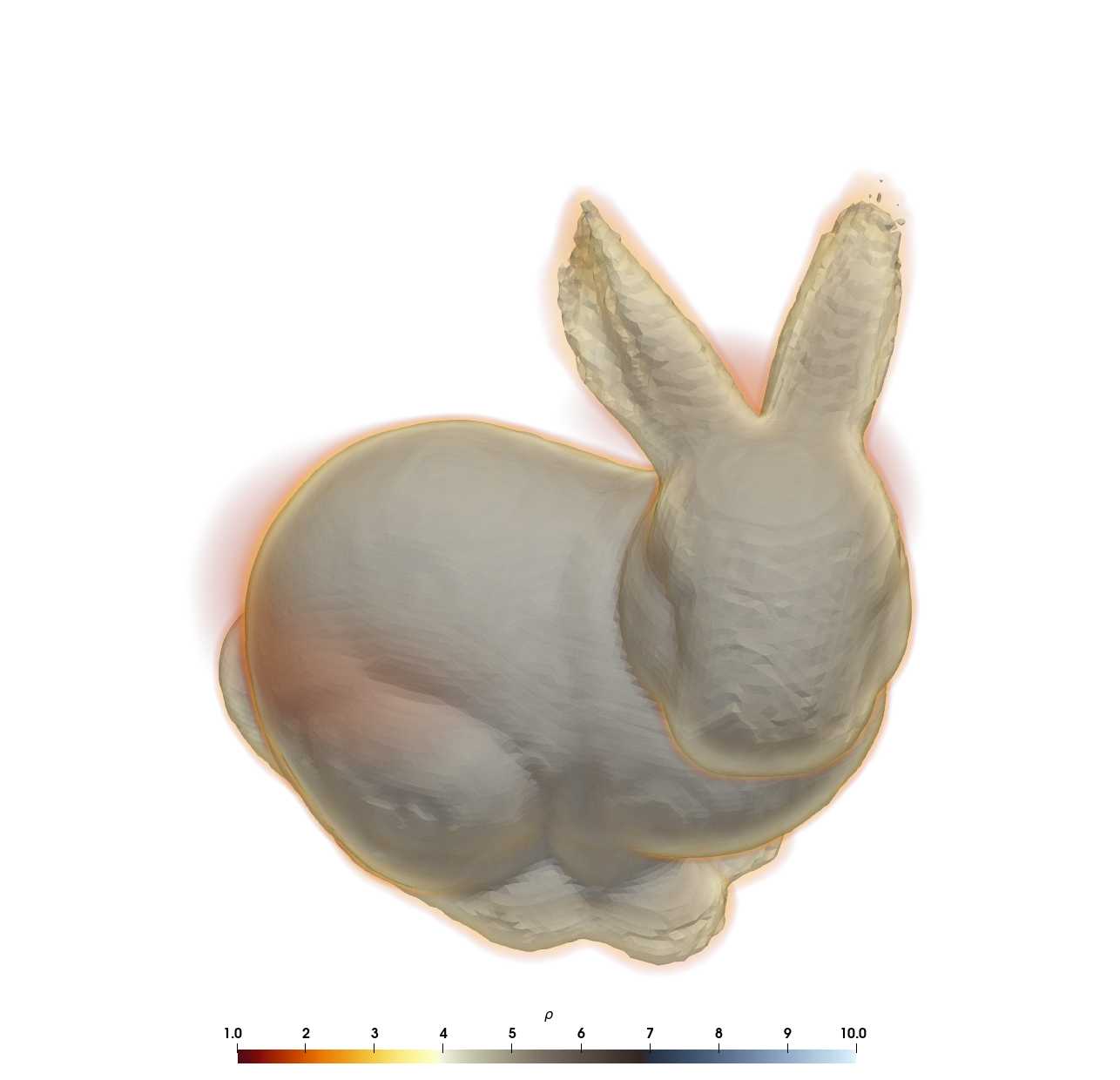}
        \caption*{$t=0.2$}
    \end{subfigure}
    \begin{subfigure}{0.16\textwidth}
        \centering
        \includegraphics[width=\textwidth,trim=30mm 40mm 30mm 70mm, clip]{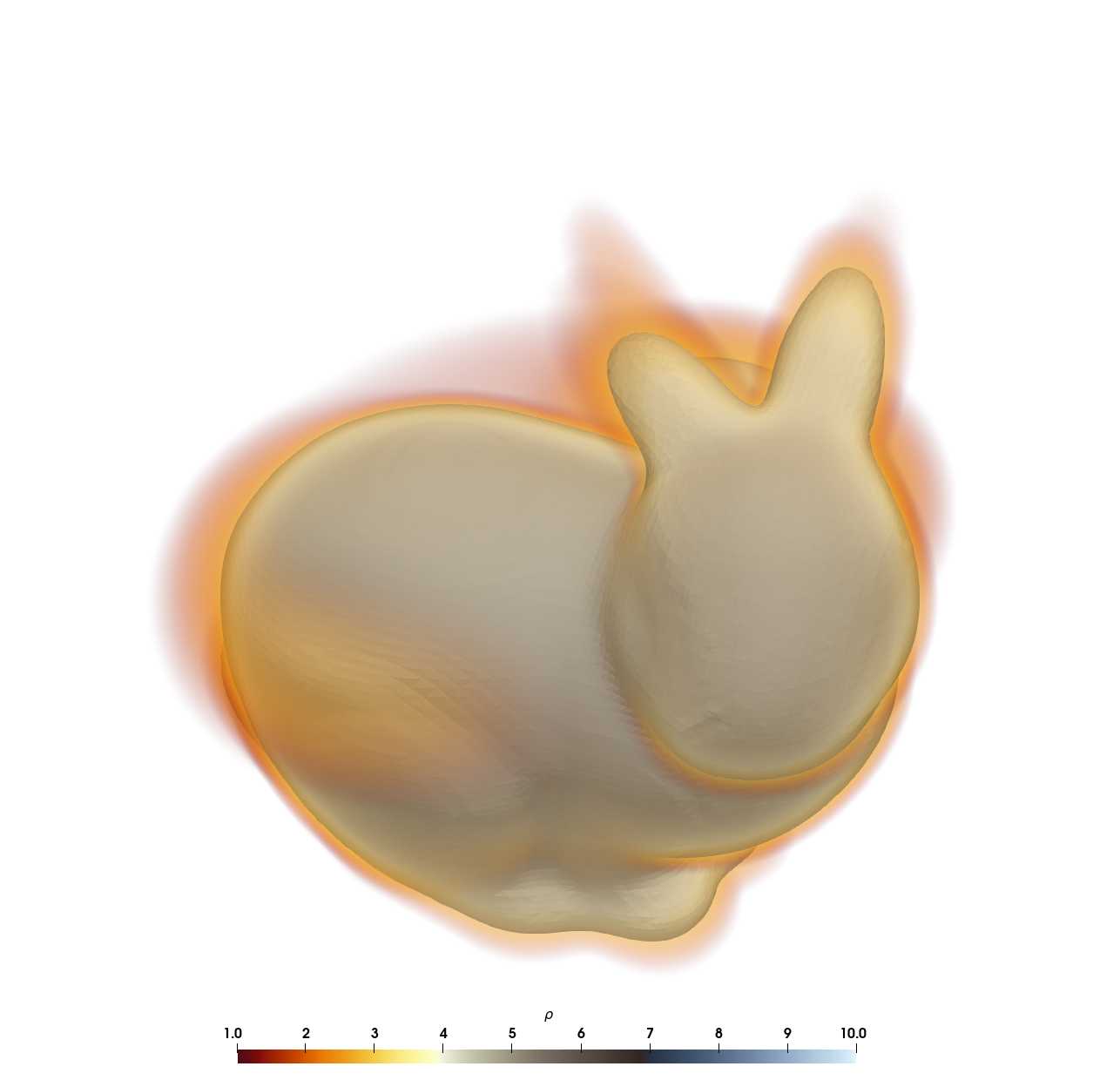}
        \caption*{$t=0.4$}
    \end{subfigure}
    \begin{subfigure}{0.16\textwidth}
        \centering
        \includegraphics[width=\textwidth,trim=30mm 40mm 30mm 70mm, clip]{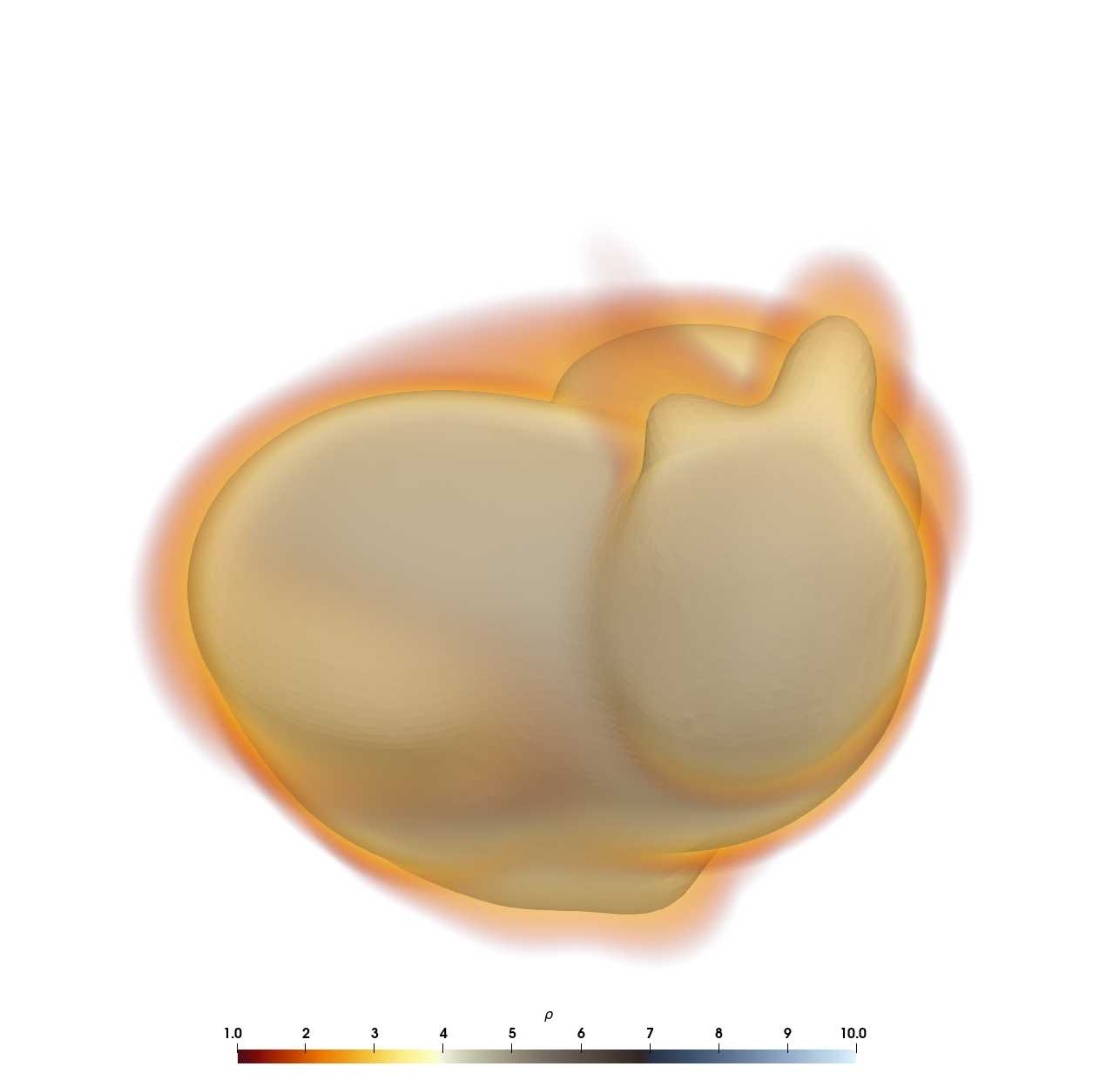}
        \caption*{$t=0.6$}
    \end{subfigure}
    \begin{subfigure}{0.16\textwidth}
        \centering
        \includegraphics[width=\textwidth,trim=30mm 40mm 30mm 70mm, clip]{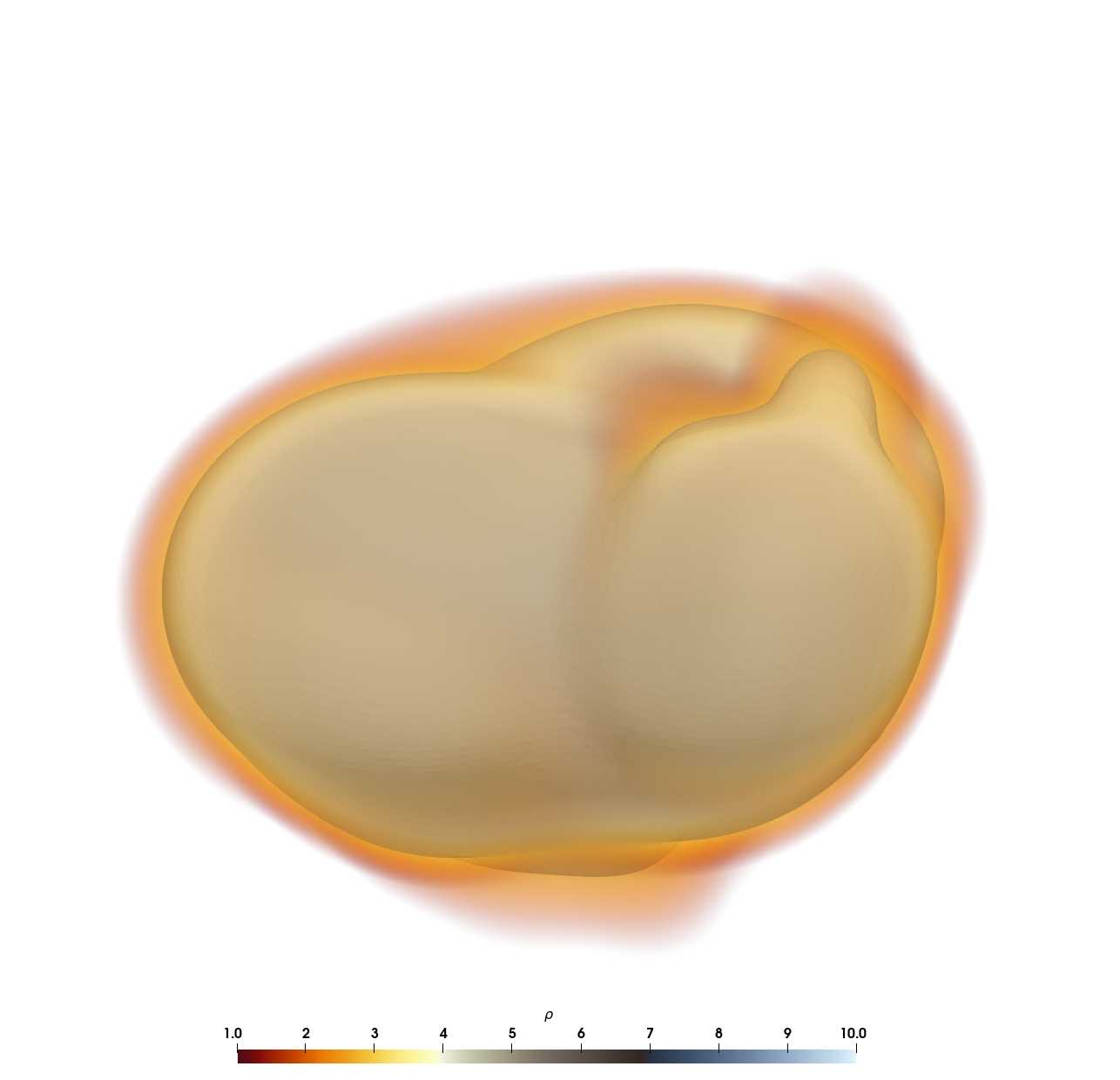}
        \caption*{$t=0.8$}
    \end{subfigure}
    \begin{subfigure}{0.16\textwidth}
        \centering
        \includegraphics[width=\textwidth,trim=30mm 40mm 30mm 70mm, clip]{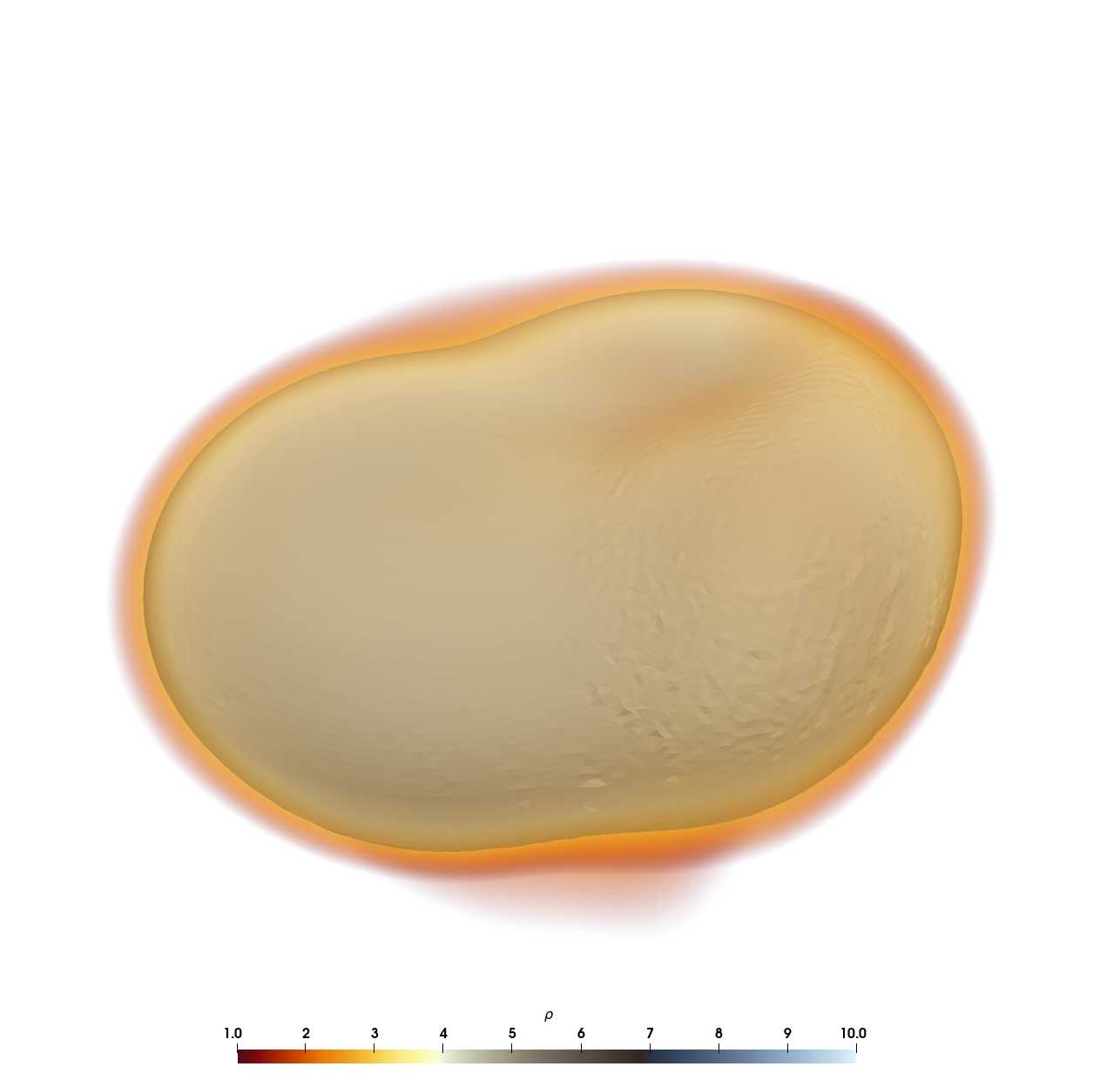}
        \caption*{$t=1.0$}
    \end{subfigure}
    
\subcaption{$\rho_3$}
\label{f:tdtb3}
\end{minipage}
\caption{Shape interpolation in a \textit{bunny}-\textit{torus}-\textit{double torus} system. Left-right shows the time evolution of density. Top-down shows plots at different reaction strengths: $\alpha=0$ and $\alpha=50$.}
\label{f:tdtb}
\end{figure}

\subsection{Barycenter on 3D surface geometry}
In this subsection, we employ our methodology to calculate generalized Wasserstein barycenters on two-dimensional surfaces embedded in $\mathbb{R}^3$. In our first illustration, we determine the barycenter of three Gaussian densities situated on the cylindrical surface $C \times [-1,1]$, where $C = {(\cos\theta,\sin\theta) : \theta \in [0,2\pi]}$. The expressions for the densities are as follows:

\begin{subequations}
    \begin{align}
        \rho_1 =&\; e^{-20\left( x^2 + (y+1)^2 + 2z^2 + 2xz\right)},\\
        \rho_2 =&\; e^{-20\left( (x-1)^2 + 2y^2 + z^2 \right)}, \\
        \rho_3 =&\; e^{-20\left( 2x^2 + (y-1)^2 + z^2 -2xz \right)}. 
    \end{align}
\end{subequations}

We employ $k=4$ degree polynomials on a spacetime mesh with $32\times 32$ spatial surface elements and four temporal line elements. We set the reaction strength $\alpha$ and the interaction strength $\beta_i$ to 0 for this computation. The results are visualized in Figure \ref{f:surfgauss}, where the density contours of all three densities are plotted on the same ambient mesh at regular intervals of 0.2 from the initial to the final time. One can see the three Gaussian densities gradually undergo rotation and translation and meet at their mutual barycenter density at the terminal time $t=1$.

Through our final numerical experiment, we further demonstrate the robustness of our scheme by computing the barycenter of a 4-density system on a complex surface mesh shaped like a dinosaur, obtained from the quadwild project  \href{https://github.com/nicopietroni/quadwild}{https://github.com/nicopietroni/quadwild}. The initial densities are indicator functions located at the head, the tail, and two of the feet. For this experiment, we use $k=4$ polynomials on a spacetime mesh with 5758 spatial surface elements and four temporal elements. Additionally, we take $\beta_i=0.001$ and $\alpha=0.1$. The results are displayed in Figure $\ref{f:dino}$, where contours of all four densities are simultaneously plotted on the dinosaur mesh from $t=0$ to $t=1$. The initial indicator functions are smoothened upon evolution 
and meet at their barycenter at the terminal time. The barycenter density is localized around the dinosaur’s abdomen and appears to have a bimodal spread with a stronger peak on the dinosaur’s right.

\begin{figure}[h!]
    \centering
    \begin{subfigure}{0.32\textwidth}
        \centering
        \includegraphics[width=\textwidth,trim=20mm 20mm 40mm 0mm, clip]{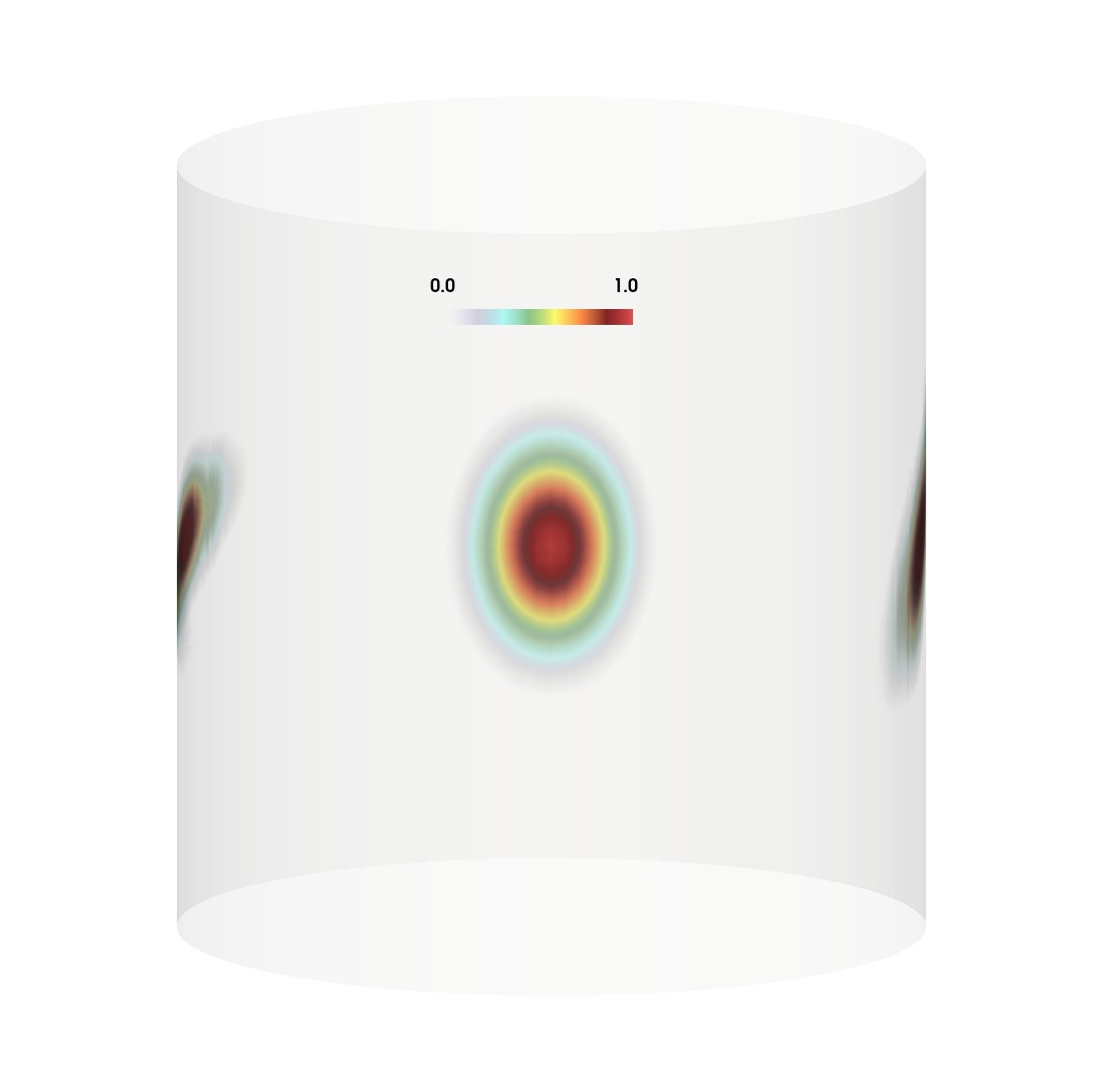}
        \caption{$t=0$}
    \end{subfigure}
    \begin{subfigure}{0.32\textwidth}
        \centering
        \includegraphics[width=\textwidth,trim=20mm 20mm 40mm 0mm, clip]{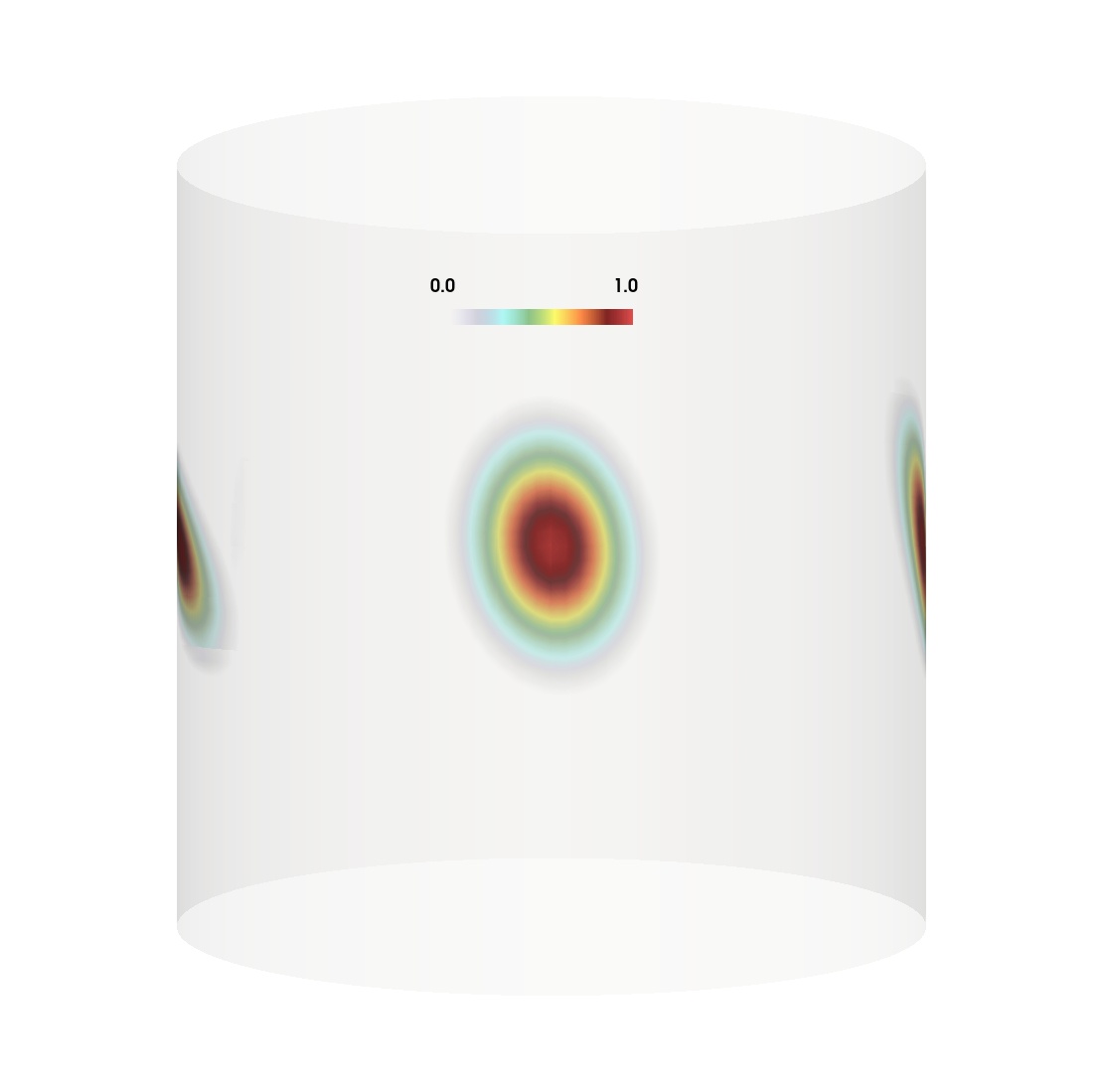}
        \caption{$t=0.2$}
    \end{subfigure}
    \begin{subfigure}{0.32\textwidth}
        \centering
        \includegraphics[width=\textwidth,trim=20mm 20mm 40mm 0mm, clip]{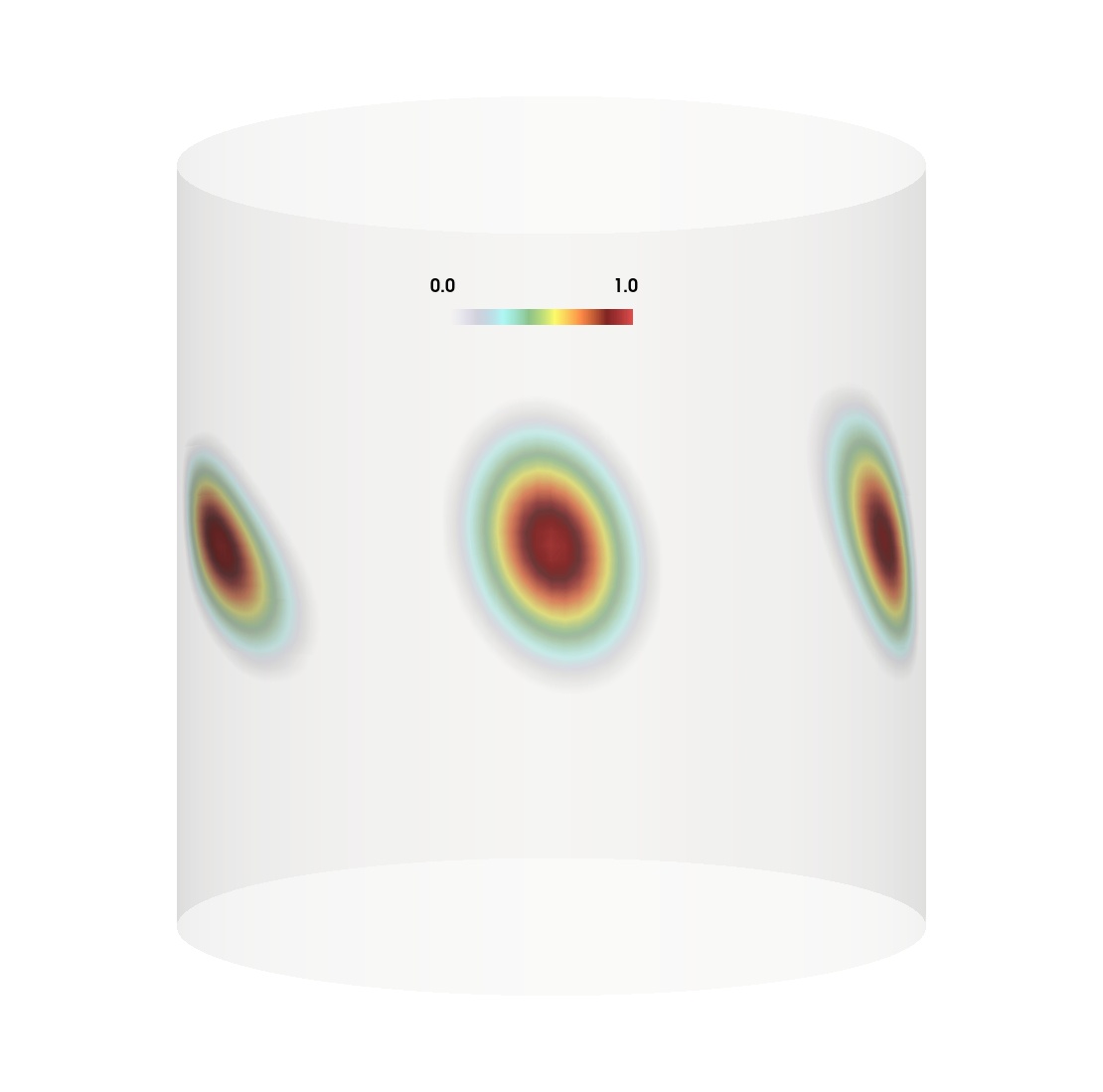}
        \caption{$t=0.4$}
    \end{subfigure}

    \begin{subfigure}{0.32\textwidth}
        \centering
        \includegraphics[width=\textwidth,trim=20mm 20mm 40mm 0mm, clip]{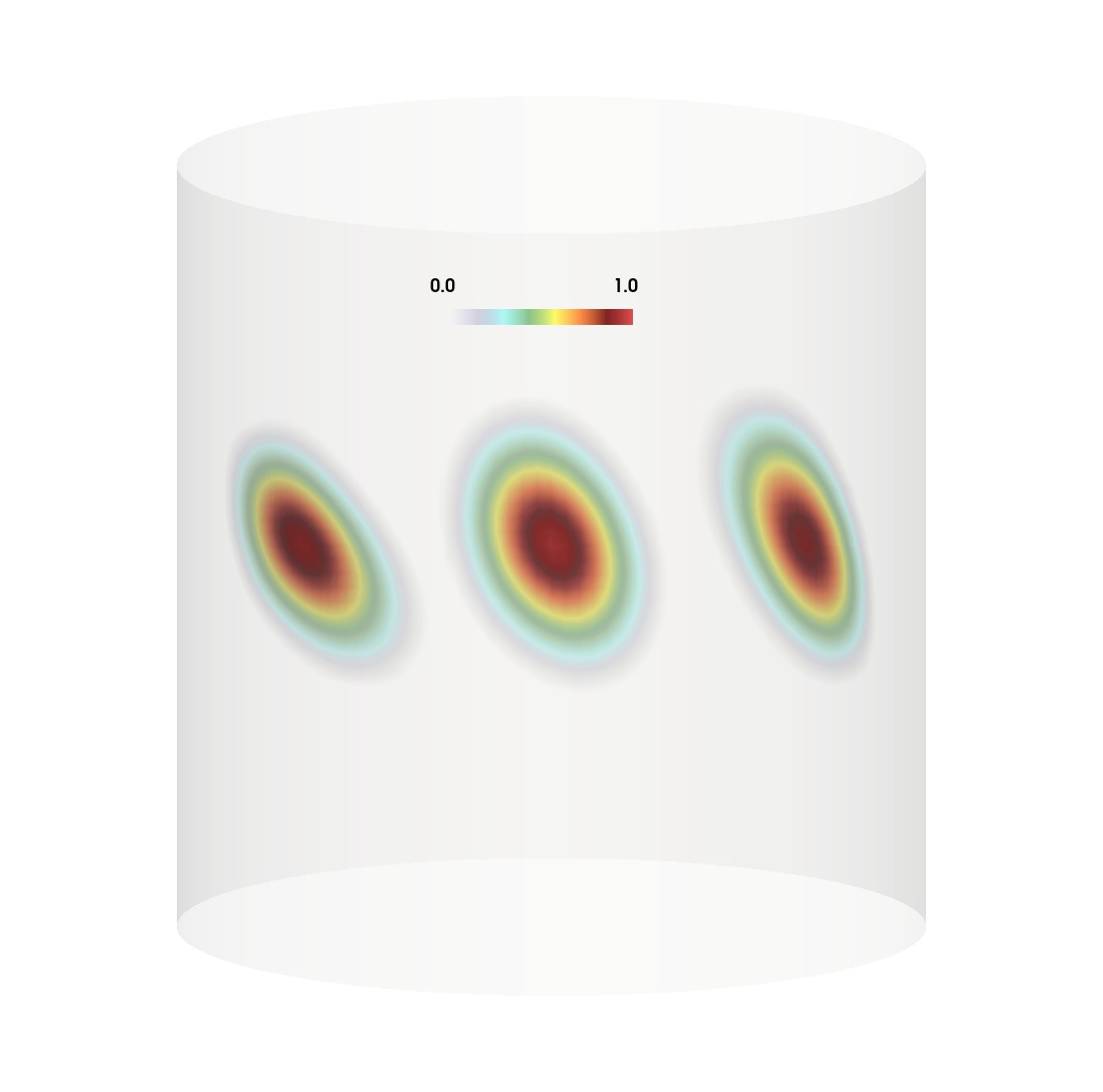}
        \caption{$t=0.6$}
    \end{subfigure}
    \begin{subfigure}{0.32\textwidth}
        \centering
        \includegraphics[width=\textwidth,trim=20mm 20mm 40mm 0mm, clip]{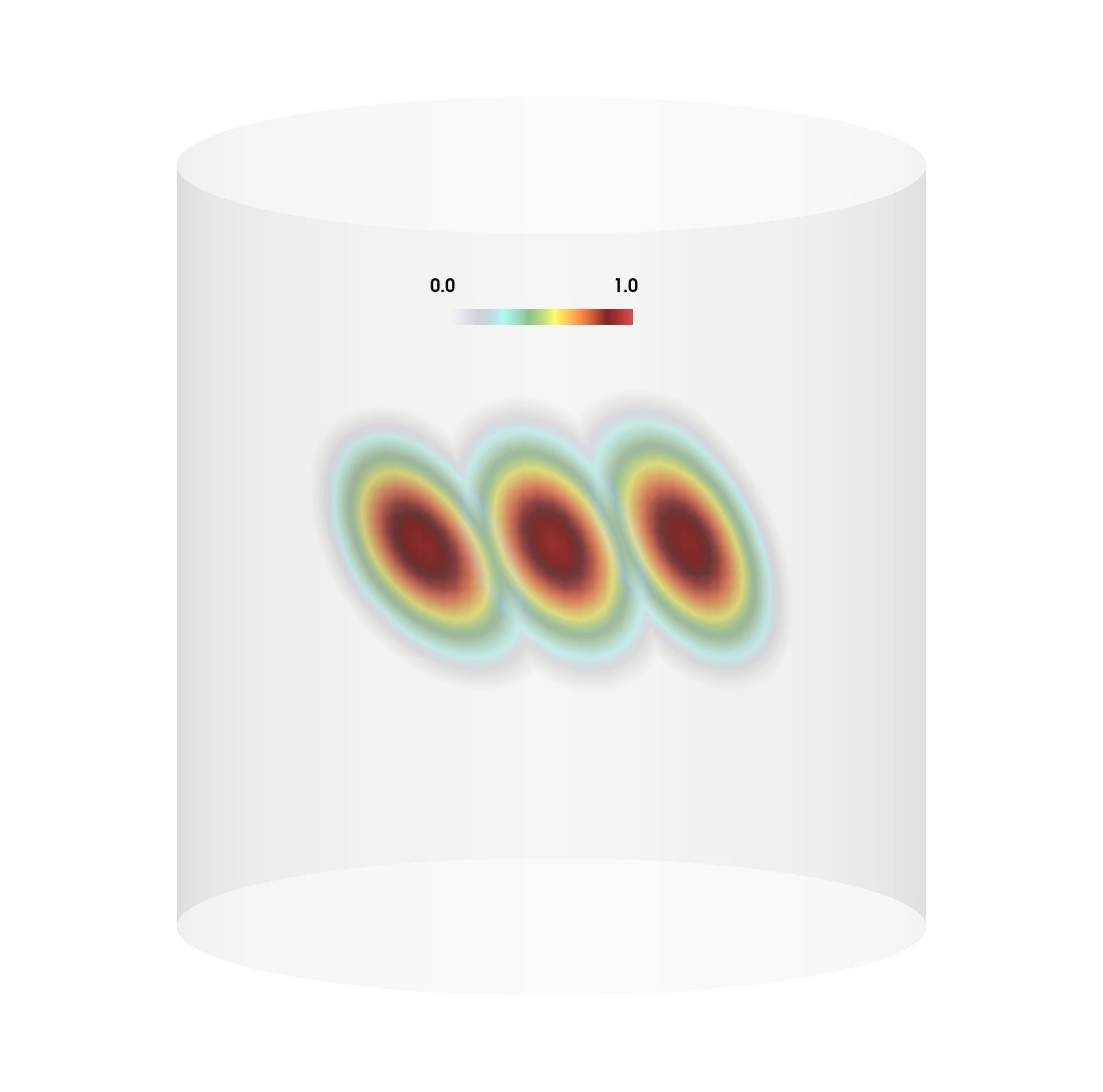}
        \caption{$t=0.8$}
    \end{subfigure}
    \begin{subfigure}{0.32\textwidth}
        \centering
        \includegraphics[width=\textwidth,trim=20mm 20mm 40mm 0mm, clip]{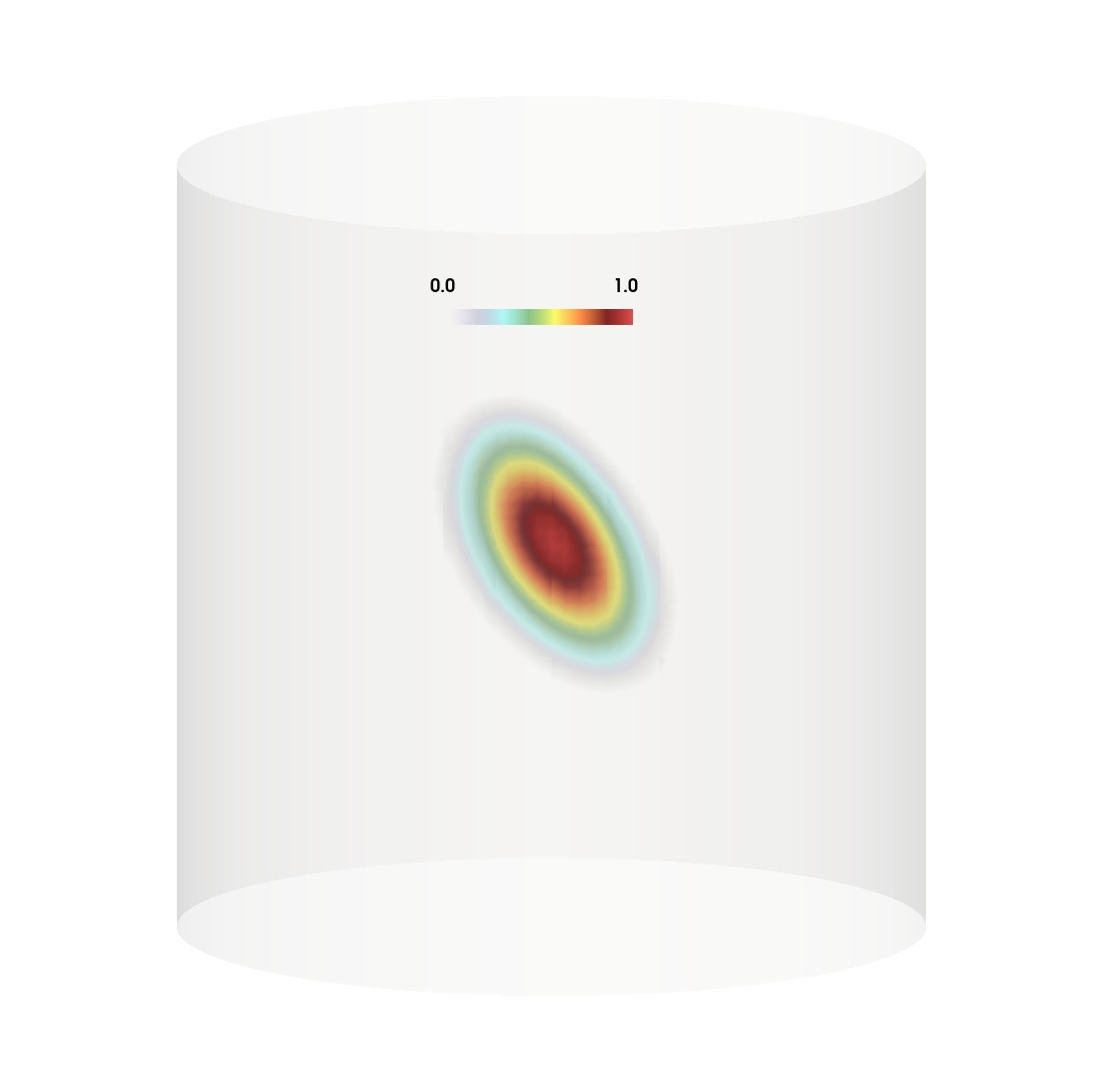}
        \caption{$t=1.0$}
    \end{subfigure}
  
    \caption{Density evolution of three Gaussian distributions on a cylindrical surface mesh.}
    \label{f:surfgauss}
\end{figure}

\begin{figure}[h]
    \centering
    \begin{subfigure}{0.32\textwidth}
        \centering
        \includegraphics[width=\textwidth,trim=20mm 20mm 90mm 0mm, clip]{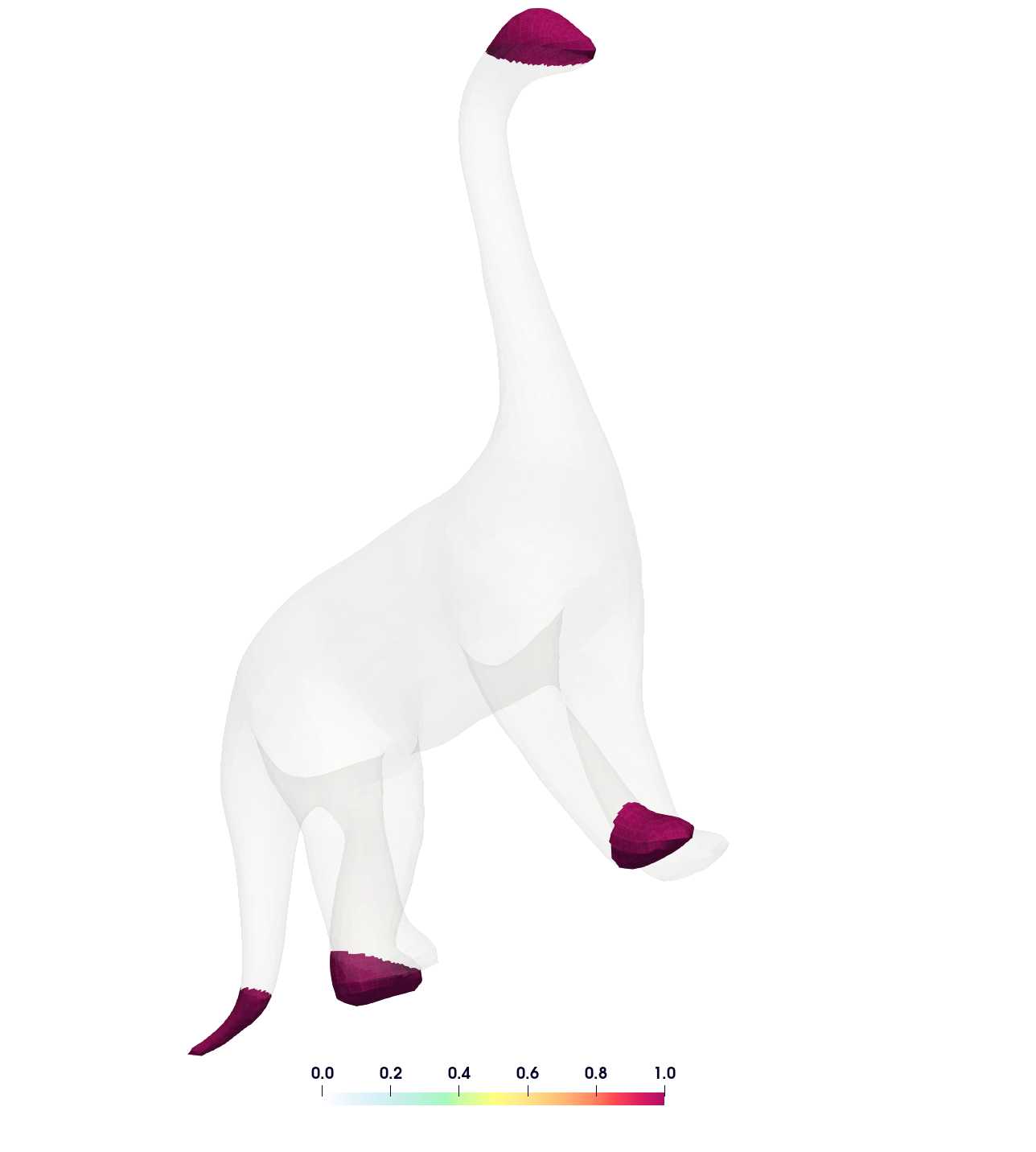}
        \caption{$t=0$}
    \end{subfigure}
    \begin{subfigure}{0.32\textwidth}
        \centering
        \includegraphics[width=\textwidth,trim=20mm 20mm 90mm 0mm, clip]{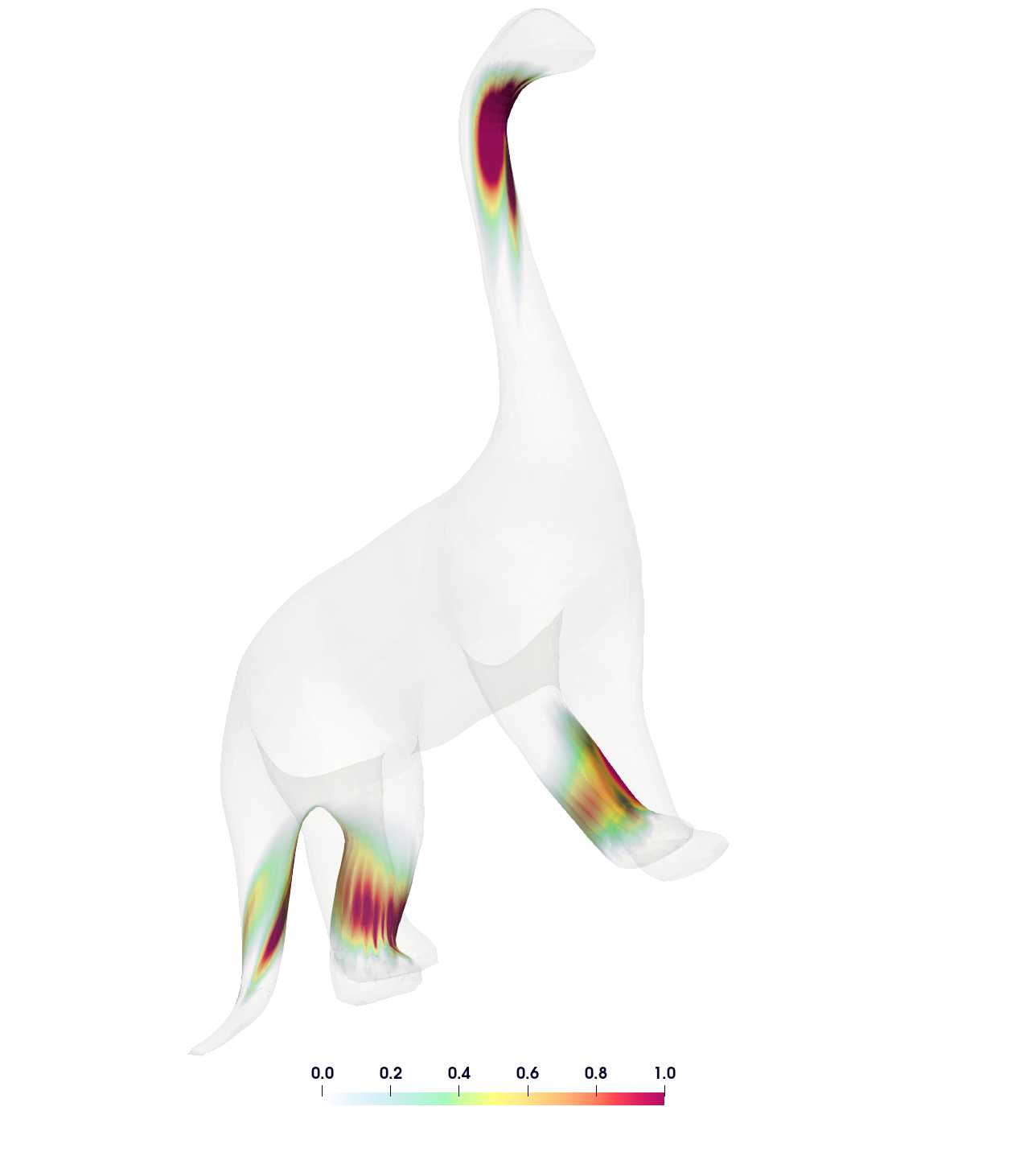}
        \caption{$t=0.2$}
    \end{subfigure}
    \begin{subfigure}{0.32\textwidth}
        \centering
        \includegraphics[width=\textwidth,trim=20mm 20mm 90mm 0mm, clip]{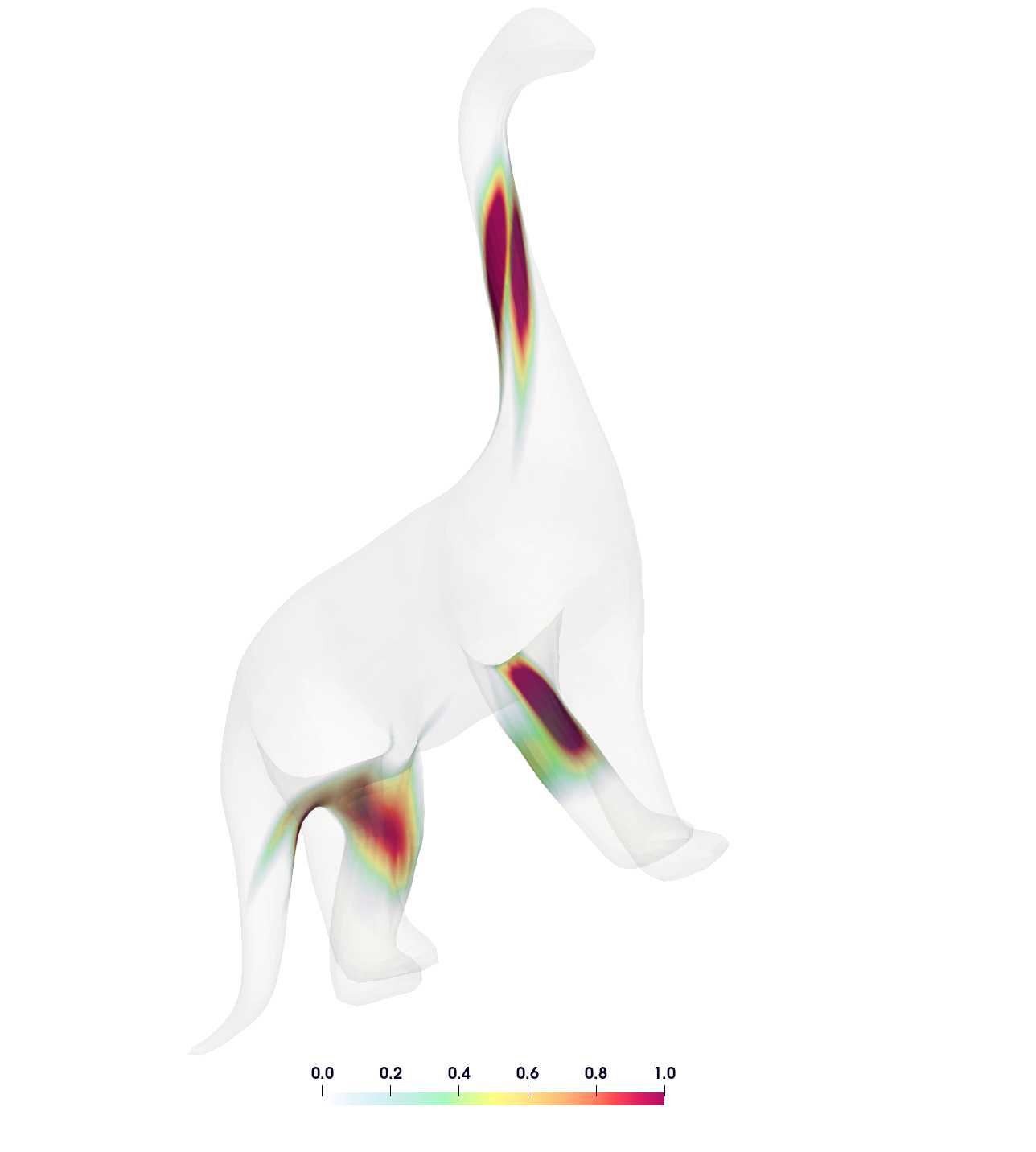}
        \caption{$t=0.4$}
    \end{subfigure}

    \begin{subfigure}{0.32\textwidth}
        \centering
        \includegraphics[width=\textwidth,trim=20mm 20mm 90mm 0mm, clip]{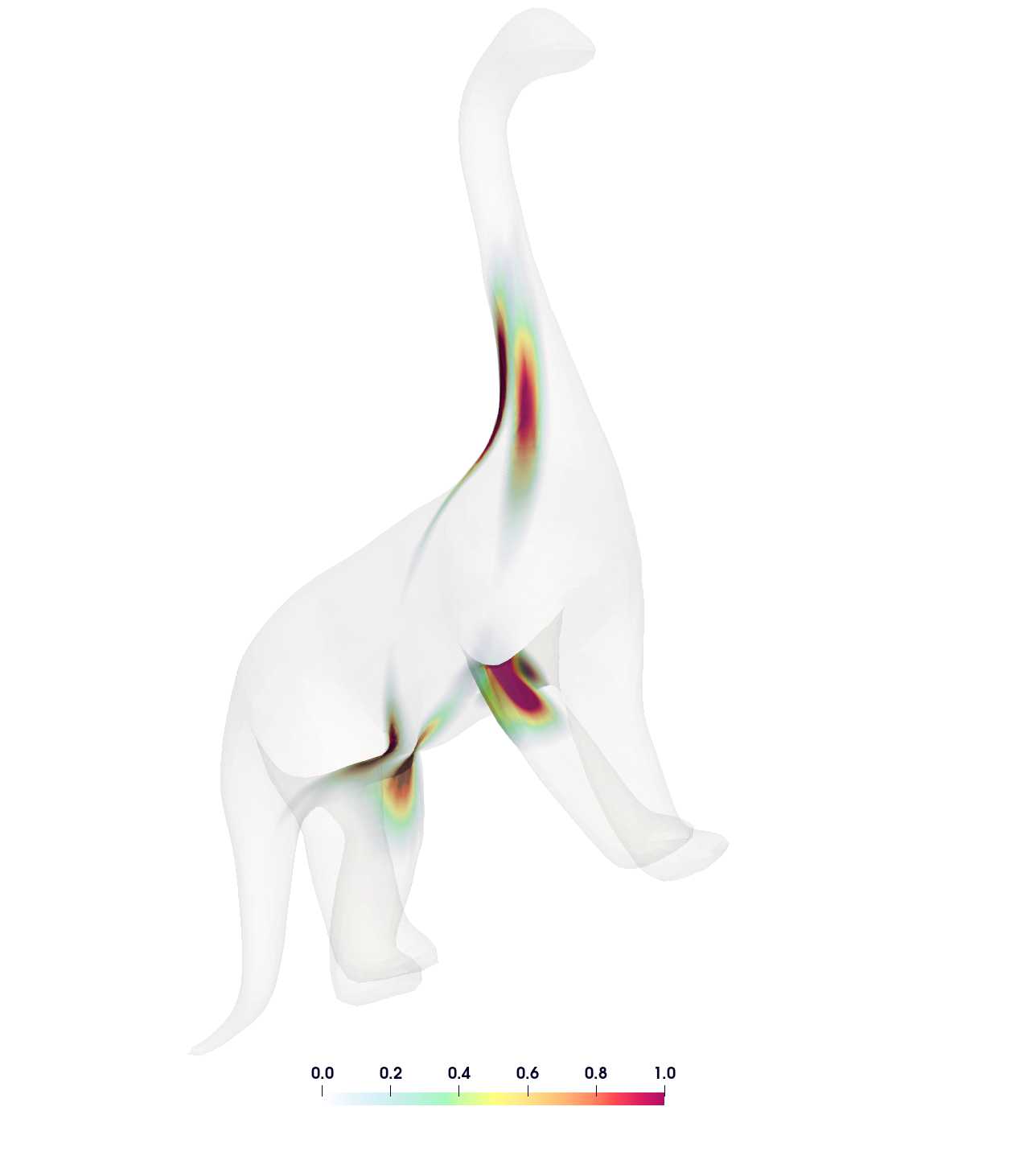}
        \caption{$t=0.6$}
    \end{subfigure}
    \begin{subfigure}{0.32\textwidth}
        \centering
        \includegraphics[width=\textwidth,trim=20mm 20mm 90mm 0mm, clip]{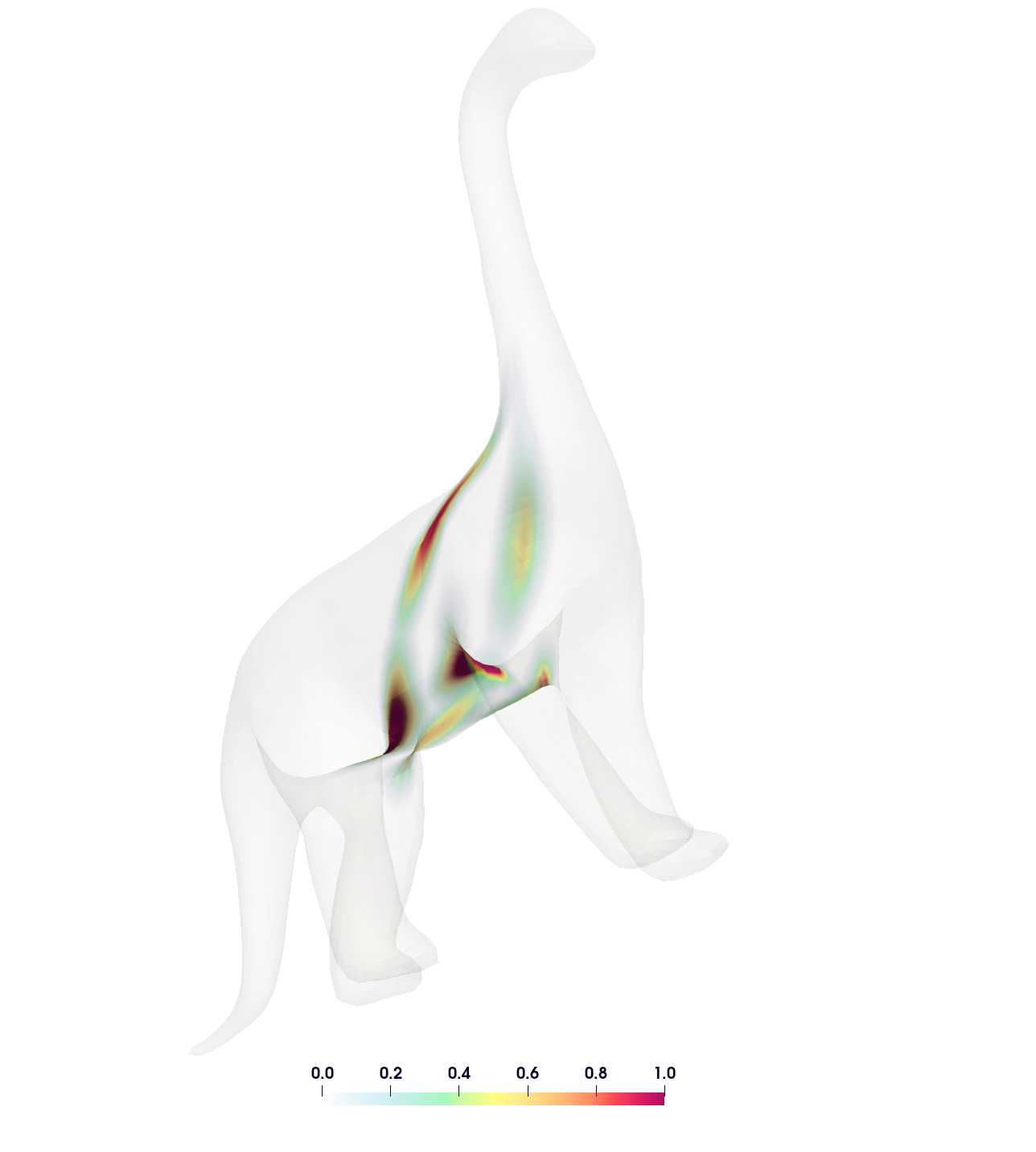}
        \caption{$t=0.8$}
    \end{subfigure}
    \begin{subfigure}{0.32\textwidth}
        \centering
        \includegraphics[width=\textwidth,trim=20mm 20mm 90mm 0mm, clip]{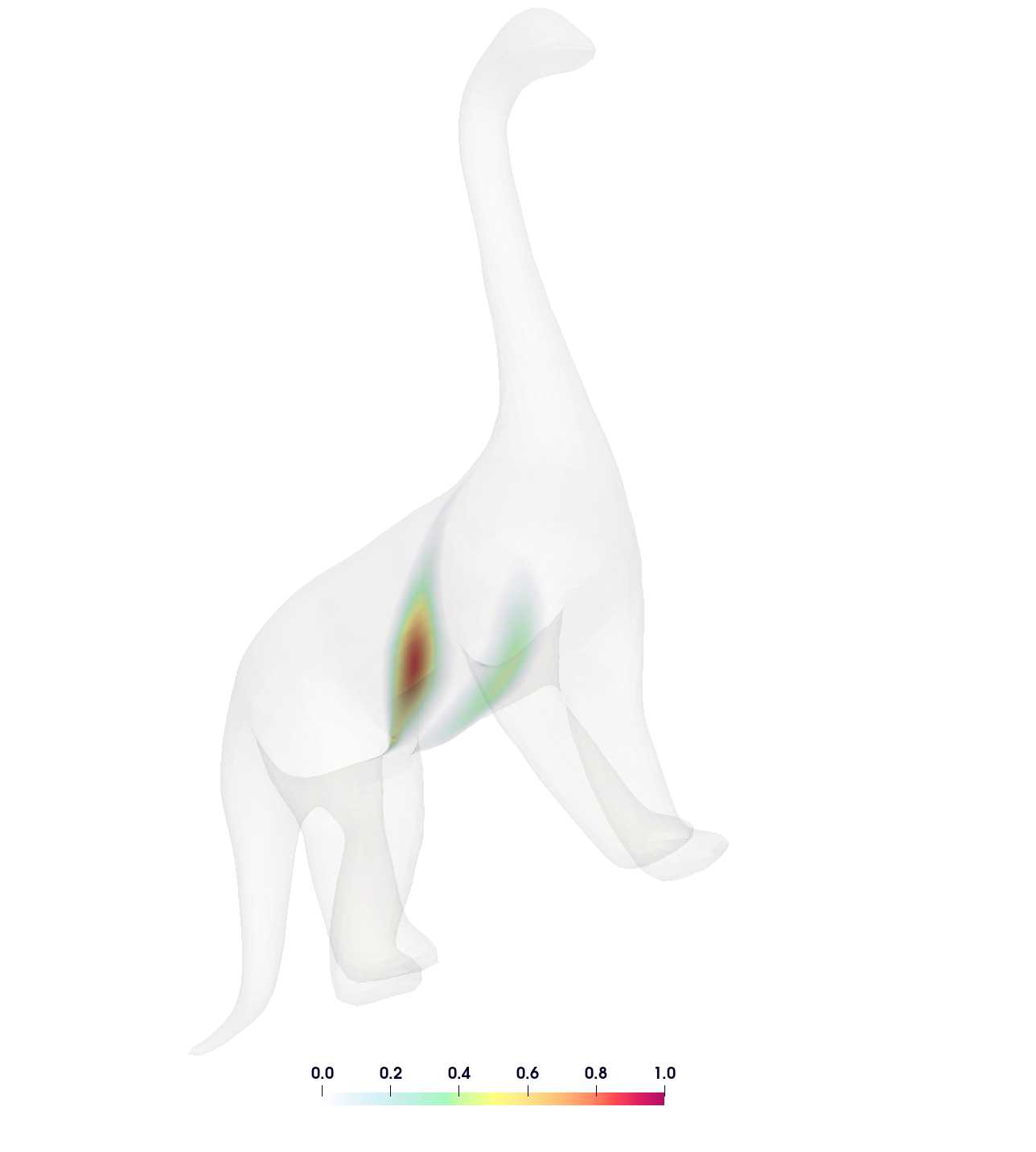}
        \caption{$t=1.0$}
    \end{subfigure}
    \caption{Density evolution of 4 indicator functions on a dinosaur surface mesh.}
    \label{f:dino}
\end{figure}

\newpage
\section{Conclusion}\label{sec5}
This paper studies a model of generalized Wasserstein barycenter problems based on MFC variational formulations. We apply the generalized gradient flow formulation for reaction-diffusion systems to obtain the generalized Wasserstein-2 distances with both transportation and reaction mobilities. The averaging optimization of these generalized Wasserstein-2 distances introduces a new set of Barycenter problems. We apply the high-order finite element methods to discretize spatial and time domains and then use the PDHG algorithm to compute the discretized MFC-barycenter problems. Numerical examples in the two-dimensional surface domain and three-dimensional volume domain demonstrate the effectiveness of the proposed method. 

In future work, we shall investigate the MFC-barycenter of general dynamics, including both reaction-diffusion equations/systems and conservation laws. It can bring general physics equations and modeling into the computational average of densities from application problems in computer vision and data sciences. We also plan to apply the current MFC-barycenter models to compute realistic barycenters of density vectors on 2D or 3D complicated spatial domains.

\vspace{0.2cm}

\noindent\textbf{Funding}  S. Osher is partially funded by AFOSR MURI FA9550-18-502 and ONR N00014-20-1-2787. G. Fu is supported by AFOSR YIP award No. FA9550-23-1-0087. W. Li is supported by AFOSR YIP award No. FA9550-23-1-0087, NSF RTG: 2038080, and NSF DMS: 2245097.

\noindent\textbf{Data Availability} Enquiries about data availability shall be directed to the authors.

\section*{Declaration}
\noindent\textbf{Conflict of interest}
The authors declare that there are no known conflicts of interest associated with this work.

\clearpage
\bibliography{ref}
\bibliographystyle{siam}

\end{document}